\documentclass{amsart}
\pdfoutput=1
\usepackage[draft,all,hyperref,numberbysection]{bi-discrete}
\usepackage[normalem]{ulem}

\usepackage[backend=biber,maxbibnames=5,maxalphanames=5,style=alphabetic,bibencoding=utf8,giveninits,url=true,isbn=false]{biblatex}
\AtBeginBibliography{\small}
\DeclareFieldFormat{url}{%
  \iffieldundef{doi}{%
    \mkbibacro{URL}\addcolon\space\url{#1}%
  }{%
  }%
}

\newcommand{\pvint}{\,\mathrm{p.v.}\!\!\int}
\newcommand{\DeltaA}{-\Delta_{\bbA}}
\newcommand{\DeltaAs}{(-\Delta_{\bbA})^s}
\newcommand{\Deltas}{(-\Delta)^s}

\newcommand{\Js}{\mathcal{J}^s}
\newcommand{\Jse}{\mathcal{J}^s_\epsilon}
\newcommand{\Jsed}{\mathcal{J}^s_{\epsilon,\delta}}

\newcommand{\kse}{k_{s,\epsilon}}
\newcommand{\bbAe}{\bbA_{\epsilon}}
\newcommand{\bbAse}{\bbA_{s,\epsilon}}

\begin{document}

\author[Balci]{Anna Kh.~Balci}
\address{Anna Kh.~Balci, Charles University in Prague, Department of Mathematical Analysis Sokolovsk\'a 49/83, 186 75 Praha 8, Czech Republic and University of Bielefeld, Department of Mathematics,  Postfach 10 01 31, 33501 Bielefeld, Germany.}
\email{akhripun@math.uni-bielefeld.de}
\author[Diening]{Lars Diening}
\address{Lars Diening, Fakult\"at f\"ur Mathematik, Universit\"at Bielefeld, Postfach 100131, D-33501 Bielefeld, Germany}
\email{lars.diening@uni-bielefeld.de}
\author[Kassmann]{Moritz Kassmann}
\address{Moritz Kassmann, Fakult\"at f\"ur Mathematik, Universit\"at Bielefeld, Postfach 100131, D-33501 Bielefeld, Germany}
\email{moritz.kassmann@uni-bielefeld.de}
\author[Lee]{Ho-Sik Lee}
\address{Ho-Sik Lee, Fakult\"at f\"ur Mathematik, Universit\"at Bielefeld, Postfach 100131, D-33501 Bielefeld, Germany}
\email{ho-sik.lee@uni-bielefeld.de}

\makeatletter
\@namedef{subjclassname@2020}{\textup{2020} Mathematics Subject Classification}
\makeatother

\subjclass[2020]{Primary: 
35R09, 
35B65; 
Secondary: 
35A15, 
35D30, 
47G20. 
}

\keywords{nonlocal equations, fractional Laplacian, regularity}
\thanks{Anna Kh. Balci is funded by the Deutsche Forschungsgemeinschaft (DFG, German Research Foundation) - SFB 1283/2 2021 - 317210226, Charles University in Prague PRIMUS/24/SCI/020 and Research Centre program No. UNCE/24/SCI/005. Lars Diening and Moritz Kassmann gratefully acknowledge funding by the Deutsche Forschungsgemeinschaft (DFG, German Research Foundation) - SFB 1283/2 2021 - 317210226. Ho-Sik Lee is funded by the Deutsche Forschungsgemeinschaft (DFG, German Research Foundation) - GRK 2235/2 2021 - 282638148.
}

\begin{abstract}
  We present nonlocal variants of the famous Meyers' example of limited higher integrability and differentiability. In the limit $s \nearrow 1$ we recover the standard Meyers' example. We consider the fractional Laplacian based on differences as well as the one based on fractional derivatives defined by Riesz potentials.
\end{abstract}

\title{Nonlocal Meyers' Example}

\maketitle

\tableofcontents

\section{Introduction}\label{sec:intro}

Critical regularity results for partial differential equations have been studied for a long time and have led to fruitful research. A regularity result stating an additional regularity property of solutions to some equation under a certain condition on the data (coefficients, right-hand side, boundary of domain) is called \emph{optimal} or \emph{sharp} if the solutions do not possess this additional property in general if the aforementioned condition on the data fails. Usually, it is rather challenging to produce counterexamples proving that additional regularity is limited in general. It is the purpose of this paper to open up this analysis for integrodifferential operators of fractional order $2s \in (0,2)$ and to address the limits of self-improving properties for such nonlocal operators. 

Let us explain self-improving properties in the fractional case provided in \cite{BaRe13, KuuMinSir15, Sch16, ABES19}. Given a coefficient function $a: \RR^d \times \RR^d \to [0, +\infty]$, which is symmetric and satisfies $\lambda \leq a(x,y) \leq \Lambda$ for all $x,y \in \RR^d$, the self-improving property means that a weak solution $u \in W^{s,2}_{\loc} (\Omega)$ (with a certain tail condition) to 
\begin{align}\label{eq:wei-intro}
	\mathcal{L}u(x) = \pvint_{\setR^d} \dfrac{u(x)-u(y)}{\abs{x-y}^{d+2s}} \, a(x,y) \,dy=0
\end{align}
in an open set $\Omega \subset \RR^d$ satisfies $u \in W^{s+\epsilon,2+\epsilon}_{\loc} (\Omega)$ for some $\epsilon > 0$. One very challenging question is to find upper bounds for $\epsilon$, which are sharp with regard to assumptions on the coefficient function $a(x,y)$. Our main result contributes to this question by providing examples of coefficient functions $a(x,y)$ together with explicit solutions to the equation allowing for sharp regularity statements. One of our results is the following: 

\medskip

\textbf{Theorem} [see Theorem \ref{thm:Meyers}]. \emph{Let $d\geq2$ and $s \in (0,1)$. For every $\epsilon \in (0,\frac12]$ there is a coefficient function $a_\epsilon: \RR^d \times \RR^d \to [0, +\infty]$, and $\delta=\delta(\epsilon) \in (0,\frac12]$ such that for $u_\delta(x):=\abs{x}^{-\delta}x_1$ we have $u_{\delta}\in W^{s,2}_{\loc} (\RR^d)$ and $\mathcal{L}_\epsilon u_\delta = 0$ in $B_1(0)$ with $\mathcal{L}_\epsilon$ being defined as in \eqref{eq:wei-intro} replacing $a(x,y)$ with $a_\epsilon(x,y)$. Furthermore, the following properties hold:
\begin{enumerate}
\item The range and the oscillation of the coefficient functions $a_{\epsilon}$ decrease as $\epsilon$ decreases, with $a_0(x,y)\equiv 1$.
\item The map $\epsilon\mapsto\delta$ defined above is strictly increasing with $\epsilon\leq \overline{c}(d,s)\delta$ and $\delta(0^+)=0$. 
\item For all $t>0$ and $q\in(1,\infty)$, we have
	\begin{align}\label{eq:thm-intro}
		u_{\delta}\in W_{\loc}^{t,q}\,\,\text{for}\,\, 1-\delta> t-\frac{d}{q}\quad\text{and}\quad u_{\delta}\not\in W_{\loc}^{t,q}\,\,\text{for}\,\, 1-\delta\leq t-\frac{d}{q}.
	\end{align}	
\end{enumerate}
}

\medskip

Note that the function $u_{\delta}:\setR^d\rightarrow\setR$ is modeled after \cite{Me1,BDGN22}. Theorem \ref{thm:Meyers} contains several more details, in particular on the regularity of the coefficient function $a_\epsilon$. Let us collect some important remarks.

\medskip

\begin{itemize}
	\item[(i)] A detailed discussion of property (a) in the above theorem is important because solutions are known to be more regular if the coefficient $a_{\epsilon}$ satisfies additional assumptions (see \cite{KuuNowSir22,DeMinNow25}). The significance of the above result is that it indicates that sharp regularity results are possible even for nonlocal operators of fractional order. Of course, it has direct consequences, too. For example, \eqref{eq:thm-intro} says that solutions, in general, are not Lipschitz-continuous in the interior of $\Omega$, since $\delta>0$.
	\item[(ii)] The coefficient functions $a_\epsilon$ that we work with are defined by $$a_\epsilon(x,y)
		= \tfrac12 \bigskp{\left(\bbM_\epsilon^2(x)+\bbM_\epsilon^2(y)\right)(x-y)}{x-y} \,,$$ with an appropriately chosen matrix-valued function $\bbM_\epsilon$. In Theorem \ref{thm:Meyers} we provide a link between the regularity of $\bbM_\epsilon$ and the regularity of $u_\delta$. To be precise, there is a number $c(d,s) > 1$ such that for every $t \in (0,1)$ and $q \in [2, \infty)$ there is a uniformly elliptic matrix-valued function $\bbM$ satisfying $\norm{\log \bbM}_{L^{\infty}} \leq c(d,s) (1-t + \frac{d}{q})$ and a solution $u$ to the equation $\mathcal{L} u = 0$ in $B_1(0)$ with coefficients $a(x,y)$ defined by $\bbM$ with $u \notin W_{\loc}^{t,q}$. This observation is analogous to the one made in \cite{BDGN22} for linear partial differential equation of the form $-\divergence (\bbM^2 \nabla u) = 0$. Given $q \in (2, \infty)$, the assumption $\norm{\log\bbM}_{L^{\infty}} \leq \frac{c}{q}$ implies $\nabla u \in L^q$. The corresponding Meyers-type counterexample yields that the linear dependence of this implication with respect to $q$ is optimal. 
\item[(iii)] Theorem \ref{thm:Meyers} is formulated for a fixed choice of $s \in (0,1)$. However, we point out that the results in this paper are robust as $s \nearrow 1$ in several ways. On the one hand we show in Lemma \ref{lem:gamma} that the corresponding energies converge to a local gradient-type energy for a given function $u$ as $s \nearrow 1$. In Theorem \ref{thm:s1} we also study the convergence of the operators applied to $u_\delta$.
\item[(iv)] The regularity of $u_\delta$ is the worst for $d=2$ in the sense that the additional regularity of $u_\delta$ compared to a $W^{s,2}_{\loc}$ function becomes greater for larger~$d$. This is in accordance with~\cite{Me1}, where the case $d\geq 3$ is constructed by the method of descent: the two dimensional example is extended constantly in the additional variables.
\end{itemize}

\subsection{Regularity results for nonlocal equations} In the recent decades, the study of the regularity theory for linear nonlocal problems involving operators of the form \eqref{eq:wei-intro}, which are versions of the fractional Laplace operator
\begin{align*}
(-\Delta)^su(x)=\pvint_{\setR^d}\dfrac{u(x)-u(y)}{\abs{x-y}^{d+2s}}\,dy\quad(d\in\setN,\,\,s\in(0,1)) \,,
\end{align*}
has evolved in several directions. Early contributions concerning H\"{o}lder regularity and Harnack inequality include the works \cite{BasLev02, CheKum03, Sil06, Kas09}. Imposing some regularity assumptions on the coefficients, interior higher order regularity 
is proved in \cite{CafSil09,CafSil11, ChaDav12,KimLee12,Kri13, KimLee21}. 

Calder\'{o}n-Zygmund type estimates are proved in \cite{MenSchYee21} assuming that, for every $x \in \setR^d$, the function $a(x,\cdot)$ is H\"{o}lder continuous, see also \cite{FalMenSchYee22}. Nonlocal self-improving properties were shown in \cite{BaRe13, KuuMinSir15, Sch16, ABES19}. \cite{Now1_21} studies regularity properties of nonlocal equations with continuous coefficients and shows higher H\"{o}lder regularity. \cite{Now231, Now232} show higher Sobolev-regularity by assuming a VMO-condition for $a(\cdot,\cdot)$. Potential and gradient potential estimates are proved in \cite{KarPetUlu11,KuuMinSir15,NNSW} and \cite{KuuNowSir22,DieKimLeeNow24}, respectively. We do not comment on interesting regularity results for recently studied nonlinear problems including the fractional $p$-Laplace equation. For the H\"{o}lder continuity of solutions to problems with degenerate coefficients, we refer to \cite{BehDieOkRol24}.

\medskip

Let us mention that another possibility to formulate nonlocal equations involving operators of differentiability order $2s$ is given by 
\begin{align}\label{eq:divs}
	-\divergence^s(\mathbb{A}(x) \nabla^s u)=0,
\end{align}
where $\mathbb{A}(x)$ is a uniformly elliptic matrix, $\divergence^s$ and $\nabla^s$ are defined with the help of the Riesz potential, see \cite{ShiSpe15}. In Section \ref{sec:Riesz} we provide a Meyers-type example for equations of the form \eqref{eq:divs}, too. Again we show limited regularity of the solution in terms of $s,d$, and $\epsilon$, similar to the theorem above.

\medskip

For some of our computations it will be very important that the fractional Laplace operator $(-\Delta)^s$ can represented as a simple multiplication operator when the Fourier transform is applied. To be precise, for every Schwartz function $v$ and every $\xi \in \RR^d$
\begin{align}\label{eq:frac.lap.ft}
\mathcal{F}[(-\Delta)^sv](\xi)=(2\pi\abs{\xi})^{2s}\mathcal{F}[v](\xi)
\end{align}
The study of fractional operators with the help of the Fourier transform and harmonic analysis has been intensively carried out for very long. We refer the reader to the references in \cite{Jaco01} and \cite{Grub15}. Note that fractional operators for vector-valued functions such as for fractional harmonic maps or O'hara knot energies have also been considered \cite{DaRiv11,Sch12,BlaReiSch16}.

\medskip

One interesting question appearing in the study of operators of fractional order $2s \in (0,2)$ is, whether the results are robust as $s \nearrow 1$. Here, robustness means that the constants in the estimates under consideration remain bounded when $s$ converges to $1$. This question is closely related to results like \cite{BoBrMi02} and the study of function spaces of fractional order. Note that several of the aforementioned results are robust whereas other results are not. Let us mention one example. As explained above, weak solutions $u \in W^{s,2}_{\loc} (\Omega)$ to \eqref{eq:wei-intro} in an open set $\Omega \subset \RR^d$ satisfy $u \in W^{s+\epsilon,2+\epsilon}_{\loc} (\Omega)$ for some $\epsilon > 0$. The corresponding estimate is not $s$-robust and it is incorrect if $s=1$.

\subsection{Higher integrability for elliptic local equations} \label{sec:weighted-local-energy}
The concept of weak solutions includes integrability of the gradient of the solutions. Let us list a few classical results proving higher integrability in the local case. To this end, let $\bbA(x)$ denote for every $x \in \RRd$ a symmetric matrix, which is uniformly positive definite in $x$. As shown in \cite{Boy57} for the two-dimensional case, weak solutions $u$ to the equation $-\divergence(\bbA(x) \nabla u) = 0$ satisfy higher integrability for $|\nabla u|$. The case of arbitrary dimensions and a class of counterexamples is covered in \cite{Me1,BDGN22}. Note that these counterexamples are the main inspiration of the current article. Let us also mention that additional assumptions on the coefficients lead to higher regularity. \cite{D1} assumes that the mean oscillation of $\bbA$ vanishes, whereas \cite{CafPer98} assumes some Cordes-type condition and \cite{ByuWan04} assumes that the the BMO-norm of $\bbA$ is small. Analogous regularity results have been proved also or solutions to degenerate elliptic equations with Muckenhoupt weights, see \cite{FKS,CMP1,BDGN22}. Note that the results in \cite{BDGN22} show optimal higher integrability exponents measured in terms of BMO norm of $\log \bbA$, thus complementing Meyers' example. Similar results relating the boundary regularity and the optimal higher integrability have been proved in \cite{BalByuDieLee22}. We also like to mention the interesting results of \cite{Tru71,BelSch21} where different conditions on the coefficients are assumed.

\subsection{Meyers' example in the local case}
\label{sec:coefficiented-local-energy}
Let us begin with a review of this classical example based on~\cite{Me1}. The details for all statements below can be found in the recent presentation in~\cite[Example~24]{BDGN22}.  Let~$B_1(0) \subset \RRd$ denote the unit ball of~$\RRd$ with $d\geq 2$, and write $\widehat{x}_1=\frac{x_1}{\abs{x}}$ and $\widehat{x}=\frac{x}{\abs{x}}$. With~$\delta ,\epsilon\in (0,\frac 12]$, we define $u_\delta \,:\, B_1(0) \to \RR$ and the matrix-valued coefficient $\bbAe\,:\, B_1(0) \to \RRddsym$ by
\begin{align}
\label{eq:meyer-u-M}
\begin{aligned}
u_\delta(x) &\coloneqq \abs{x}^{1-\delta} \widehat{x}_1,
\\
\bbAe(x) &\coloneqq (1-\epsilon) \identity + \epsilon \widehat{x} \otimes \widehat{x}.
\end{aligned}
\end{align}
Then for any $x \in B_1(0)$ and $\xi \in \RRd$, 
\begin{alignat*}{4}
\tfrac 14 |\xi|^2 &\leq (1-\epsilon) |\xi|^2 &&\leq \skp{\bbAe(x)\xi}{\xi} &&\leq |\xi|^2
\end{alignat*}
holds and so $\bbAe$ is uniformly elliptic. Moreover, $\norm{\log \bbAe}_{L^{\infty}} \leq \epsilon$, where $\log$ is the matrix logarithm and $\norm{\cdot}_{L^{\infty}}$ is the BMO-seminorm on~$\RRd$. If we impose
\begin{align}
\label{eq:meyer-epsilon}
\epsilon=\frac{d-\delta}{d-1}\delta,
\end{align}
then $u_{\delta} \in W^{1,2}(B_1(0))$ satisfies
\begin{align}
\label{eq:meyer-eq}
(-\Delta_{\bbAe}) u_\delta(x)=-\divergence\big(\bbAe(x) \nabla u_\delta(x)\big) &= 0 \qquad \text{on $B_1(0)$}
\end{align}
in the weak sense. Although $\bbAe$ is uniformly elliptic, the integrability of~$\nabla u_{\delta}$ is limited. In fact, $\nabla u_{\delta}$ only belongs to the Marcinkiewicz space~$L^{\frac d\delta,\infty}(B_1(0))$. As a consequence $u_{\delta} \in W^{1,\rho}(B_1(0))$ if and only if $\rho < \frac d \delta$.

This example was used in~\cite{Me1} to show that there is no minimal gain of higher integrability of $\nabla u$ if the condition number of the coefficient $\mathbb{A}(\cdot)$ goes to infinity, where $u$ is a solution of $-\divergence\big(\bbA(x) \nabla u(x)\big) = 0$. In~\cite[Theorem~1]{BDGN22} it has been shown more explicitly that the smallness of $\norm{\log \bbA}_{L^{\infty}}$ is directly linked to the exponent of integrability of~$\nabla u$. In fact, $\norm{\log \bbA}_{L^{\infty}} \leq \epsilon$ implies $\nabla u \in L^{q}(B_1(0))$ for all~$q \in [2,\infty)$ with $q \leq \frac{\kappa}{\epsilon}$ for some $\kappa=\kappa(d)$. The example of Meyers' shows that this linear dependence of~$\norm{\log \bbA}_{L^{\infty}}$ and $\frac 1 q$ is sharp (up to the choice of~$\kappa$).

\subsection{Strategy in the nonlocal setting}\label{subsec:strategy}

Let us explain the main challenges in the proof of our main result together with ideas how to approach them. We concentrate on three main issues: computation of the appropriate coefficients $a(x,y)$ in the operator given in \eqref{eq:wei-intro}, adjusting the energy to the growth of the function $u_\delta(x)=\abs{x}^{-\delta} x_1$ at infinity and how to apply the operator under consideration to $u_\delta$.

\medskip 

\textit{Computation of the coefficients $a(x,y)$ in the operator $\mathcal{L}$ from \eqref{eq:wei-intro}}. One challenge is to incorporate the matrix function $\mathbb{A}_\varepsilon(\cdot)$ from \eqref{eq:meyer-u-M} into the fractional operator $\mathcal{L}$ from $\eqref{eq:wei-intro}$. One option is to apply the Fourier transform. For a moment, consider a given positive definite symmetric constant matrix $\mathbb{A}$ and $\bbM^2 = \mathbb{A}$ and define an operator $\DeltaA$ by $\DeltaA v \coloneqq \divergence(\bbA \nabla v)=\divergence(\bbM^2 \nabla v)$. Then  
\begin{align*}
\mathcal{F}\big[\DeltaA v\big]=\left(2\pi \abs{\bbM\xi}\right)^2 \mathcal{F}[v].
\end{align*}
Inspired by this, one could define $\left(-\Delta_{\mathbb{A}}\right)^s$ by 
\begin{align*}
\mathcal{F}\big[ \left(-\Delta_{\mathbb{A}}\right)^s v\big] &\coloneqq \left(2\pi\abs{\bbM\xi}\right)^{2s}\mathcal{F}[v] \,,
\end{align*}
or, equivalently, for a sufficiently regular function $v$  
\begin{align*}
\left(-\Delta_{\mathbb{A}}\right)^s v(x)
&=
-2\kappa_{d,s} \pvint_{\RRd} \frac{v(x+\bbM h) -v(x)}{\abs{h}^{d+2s}}\,dh\quad(x\in\setR^d)
\end{align*}
with some constant $\kappa_{d,s}$ depending on $d$ and $s$. This would justify to work with the non-translation invariant versions of the form 
\begin{align*}
\left(-\Delta_{\mathbb{A}_{\epsilon}(\cdot)}\right)^s v(x)
&=
-2\kappa_{d,s} \pvint_{\RRd} \frac{v(x+\bbM_{\epsilon}(x) h) -v(x)}{\abs{h}^{d+2s}}\,dh,
\end{align*}
with $\mathbb{A}_{\epsilon}$ as in \eqref{eq:meyer-u-M} and $\mathbb{M}^{2}_{\epsilon}=\mathbb{A}_{\epsilon}$. The corresponding nonlocal energy would be 
\begin{align*}
  \mathcal{J}^s_{\bbM_{\epsilon}(\cdot)}(v)
  &= \frac{\kappa_{d,s}}{2}
  \int_{\setR^{d}}\int_{\setR^{d}} \frac{\abs{v(x+\bbM_{\epsilon}(x) h)-v(x)}^2}{\abs{h}^{d+2s}}\,dh \,dx.
\end{align*}
However, to verify that a certain function is a minimizer of this energy seems an impossible task and thus we choose another approach. We apply Taylor approximation with respect to $\epsilon$, see Section \ref{sec:model-using-taylor}, and arrive at the following more accessible energy:
\begin{align*}
	\mathcal{J}^{s}_\varepsilon(v)\coloneqq \frac{\kappa_{d,s}}{2}\int_{\setR^{d}}\int_{\setR^{d}} \frac{\big(v(x)-v(y)\big)^2}{|y-x|^{d+2s}} a_\varepsilon (x,y)\,dy\,dx \,,
\end{align*}
where the coefficient function $a_\varepsilon(x,y):\setR^{2d}\rightarrow\setR$ is defined by 
\begin{align}
	a_{\varepsilon}(x,y)\coloneqq \Bigskp{\left(\tfrac{\bbA_{s,\varepsilon}(x)+\bbA_{s,\varepsilon}(y)}{2}\right)\widehat{x-y}}{\widehat{x-y}} \,,
  \\
  \intertext{and}
 \label{eq:As}
	\bbAse(x)\coloneqq \left(1-\tfrac{1+2s}{2}\epsilon\right)\identity+\tfrac{d+2s}{2}\epsilon\widehat{x}\otimes\widehat{x},
\end{align}
where $\hat{x}$ and $\hat{x}\otimes\hat{x}$ is defined in Subsection \ref{sec:coefficiented-local-energy}. Let us point out some nice properties of the quadratic form $v \mapsto \mathcal{J}^{s}_\varepsilon (v)$:
\begin{itemize}
	\item The energy $\mathcal{J}^{s}_\varepsilon (v)$ converges to $\tfrac 12\int_{\RRd}  \skp{\bbAe \nabla v }{\nabla v }$ as $s \nearrow 1$, see Lemma \ref{lem:gamma}.
	\item The corresponding nonlocal operator applied to the function $u_\delta$ satisfies scaling and rotation properties, see Remark \ref{rem:prop}.
	\item The action of the nonlocal operator can be explicitly computed, see Section \ref{sec:nonl-meyers-example}, the Appendix \ref{app} and the auxiliary computations in Appendix \ref{app2}.
\end{itemize}

\medskip

\textit{How to adjust the energy to the growth of the function $u_\delta$ at infinity}. Similarly to \cite{Me1}, we use the function $u_\delta(x)=\abs{x}^{1-\delta} \widehat{x}_1$ for some $\delta\in[0,\frac{1}{2}]$ as a solution to our problem. However, $\mathcal{J}_\varepsilon^{s}(u_\delta)$ with $u_{\delta}(x)=\abs{x}^{1-\delta}\widehat{x}_1$ is infinite because of the growth of $u_\delta(x)$ at infinity. Analogously to  \cite{CMY17, DPV19}, we renormalize the energy $\mathcal{J}_\varepsilon^s (v)$ as follows:
\begin{align*}
\mathcal{J}_{\varepsilon, \delta}^s (v) \coloneqq\dfrac{\kappa_{d,s}}{2}\int_{\RRd} \int_{\RRd} \dfrac{\big(\abs{v(y)-v(x)}^2 - \abs{u_\delta(y)-u_\delta(x)}^2\big)}{\abs{y-x}^{d+2s}}a_{\epsilon}(x,y) \,dy\,dx.
\end{align*}
Additionally for $s\in(0,\frac{1}{2}]$, we need to interpret the energy by a principal value integral, see Section \ref{sec:renormalized-energy}. Furthermore, by Proposition \ref{prop:functional-nice} we guarantee that there exists a unique minimizer and we can consider the Euler-Lagrange equation. Moreover, due to  Lemma \ref{lem:basic} the object $(-\Delta_{\bbA_{s,\epsilon}})^s u$, which is originally defined in the sense of distributions, turns out to be a locally integrable function, see Lemma \ref{eq:estimate-Deltas-ud}. The technical difficulties with the growth at infinity are the main reason that we have to restrict the range of~$\delta$ to $\delta \in [0,\frac 12]$.

\medskip

\textit{How to evaluate the operator under consideration applied to $u_\delta$}. Our main result Theorem \ref{thm:Meyers} shows that both, the higher integrability and differentiability of the solution of the nonlocal fractional equation are limited in a quantitative way by the condition on the coefficients. The analogous result in the non-fractional case is given in \cite[Example 24]{BDGN22}. We would like to point out that there are only very few examples involving explicit computations for the equation involving fractional Laplacian, e.g., in \cite{Dyd12,Gru24}. The presence of the integral in the nonlocal equation makes the computation difficult. 

We approach these difficulties by finding a convolution structure of our equation after making use of scaling and rotation properties. This enables us to work with the Fourier transform. Overall, computing $(-\Delta_{\mathbb{A}_{s,\epsilon}})^s u_\delta(x)$ is reduced to calculating
\begin{align*}
(-\Delta_{\bbA_{s,\epsilon}})^s u_\delta(e_1)=-\frac{\kappa_{d,s}}{2}\Big[\left(1-\tfrac{1+2s}{2}\epsilon\right)\tfrac{2}{\kappa_{d,s}}(-\Delta)^s u_\delta(e_1)+\tfrac{d+2s}{2}\epsilon f_2(d,s,\delta)\Big],
\end{align*}
where
\begin{align*}
\begin{split}
&f_2(d,s,\delta)=\frac{1}{2}\pvint_{\setR^d}\bigg(\dfrac{\skp{e_1+h}{h}^2}{|e_1+h|^2|h|^2}+\dfrac{\skp{e_1}{h}^2}{|h|^2}\bigg)\dfrac{|e_1+h|^{-\delta}(1+h_1)-1}{|h|^{d+2s}}\,dh \,.
\end{split}
\end{align*}
See \eqref{eq:f2} and \eqref{eq:f2'} below for $f_2(d,s,\delta)$. The last expression may look complicated but it contains a convolution structure thus it can be computed precisely via the Fourier transform, see Appendix \ref{app}.

\subsection{Organization of the paper}
Section \ref{sec:weighted-nonlocal-equations} is devoted to introducing the nonlocal energy based on the Fourier transform, and to derive its Euler-Lagrange equation. In Section \ref{sec:strong-formulation} we prove that our operator $(-\Delta_{\mathbb{A}_{s,\epsilon}})^s$ applied to $u_{\delta}$ is actually a locally integrable function. In Section \ref{sec:nonl-meyers-example} we provide the nonlocal Meyers' example. Adjusting $\epsilon$ and $\delta$ will give $(-\Delta_{\mathbb{A}_{s,\epsilon}})^su=0$. Section \ref{sec:Riesz} contains an example for \eqref{eq:divs} using Riesz fractional derivatives. In Appendix \ref{app} we precisely compute the term $f_2$ from above with the help of the Fourier transform. In Appendix \ref{app2} we provide justifications of convolutions and convolution theorems which are necessary for the proof.

\section{Nonlocal equations with coefficients}
\label{sec:weighted-nonlocal-equations}

Different from the local case it is a~priori not obvious, what an equation such as~\eqref{eq:meyer-eq} should look like in the nonlocal case. In this section we derive nonlocal models with coefficients that are inspired by the idea that one should recover~\eqref{eq:meyer-eq} in the limit~$s\nearrow 1$.

\subsection{Standard fractional Laplacian}
\label{sec:stand-fract-lapl}

Let us briefly recall the definition of the fractional Laplacian and introduce some notation.

Let $d \geq 2$ denote the dimension. By $B_r(x)$ we denote the open ball of radius~$r$ with center~$x$. By $\mathcal{S}(\RRd)$ we denote the standard space of Schwartz functions.  For $u\in\mathcal{S}(\setR^d)$, let us use the Fourier transform $\mathcal{F}[u] \in \mathcal{S}(\RRd)$ as follows: 
\begin{align*}
  \mathcal{F}[u](\xi) = \int_{\RRd} u(x) e^{-2\pi i x \cdot \xi}\,dx.
\end{align*}
The Fourier transform is extended to the space of tempered distributions $\mathcal{S}'(\RRd)$.

The Fourier symbol of $-\Delta u$ is $(2\pi\abs{\xi})^2$. This allows us to define the fractional Laplacian $\Deltas u$ for $s \in (0,1)$ by
\begin{align}\label{eq:fractional.lap}
  \mathcal{F}\big[(-\Delta)^s u\big](\xi) = (2 \pi \abs{\xi})^{2s} \mathcal{F}[u](\xi).
\end{align}
The fractional Laplacian $\Deltas$ also can be defined equivalently in terms of differences of functions. For this we consider the nonlocal energy 
\begin{align*}
  \Js(v)
  &= \frac{\kappa_{d,s}}{2}
  \int_{\RRd} \int_{\RRd} \frac{\abs{v(x+h)-v(x)}^2}{\abs{h}^{d+2s}}\,dh \,dx
\end{align*}
with the constant
\begin{align}\label{eq:kappa}
  \kappa_{d,s} =  \frac{2^{2s-1} \Gamma(\frac d2 +s)}{\pi^{\frac d2} \abs{\Gamma(-s)}}.
\end{align}
Then the fractional Laplacian can be defined as the operator appears in the Euler-Lagrange equation, i.e., for $v,w \in \mathcal{S}(\RRd)$ we have
\begin{align}
  \label{eq:euler-lagrange}
  (\Js)'(v)(w) = \int_\RRd \Deltas v\, w\,dx.
\end{align}
The constant $\kappa_{d,s}$ above is chosen such that no additional constant appears in~\eqref{eq:euler-lagrange}. We refer to~\cite{Kwa17} for comparison of many possible equivalent definitions of~$\Deltas$.

It is possible to compute~$\Deltas v$ by
\begin{align*}
  (-\Delta)^s v(x)
  &= \kappa_{d,s} \int_{\RRd} \frac{2v(x) -v(x+h)-v(x-h) }{\abs{h}^{d+2s}}\,dh
  \\
  &= -2 \kappa_{d,s} \pvint_{\RRd} \frac{v(x+h) -v(x)}{\abs{h}^{d+2s}}\,dh,
\end{align*}
where the second integral is a principal value integral. The principal value integral is defined as the limit of the integral over $\RRn \setminus B_r(0)$ for $r \to 0$ or even $B_{1/r}(0) \setminus B_r(0)$ for $r \to 0$. 
The equivalence of the two definitions can be seen by the following computation in the distributional sense
\begin{align}\label{eq:exp}
  \begin{split}
    \mathcal{F} \bigg[ \pvint_{\RRd} \frac{v(\cdot+h) -v(\cdot)}{\abs{h}^{d+2s}}\,dh  \bigg](\xi)
    &= \bigg( \pvint_{\RRd} \frac{\exp(2 \pi i \xi \cdot h) - 1}{\abs{h}^{d+2s}}\,dh\bigg)\mathcal{F}[v](\xi)
    \\
    &=  \mathcal{F}( \abs{\cdot}^{-d-2s})(\xi)\mathcal{F}[v](\xi).
  \end{split}
\end{align}
Note that for all $t \notin 2 \setN\cup\{0\}$ the following is true (as a tempered distribution):
\begin{align}
  \label{eq:F-of-power}
  \mathcal{F}( \abs{\cdot}^{-d-t} )(\xi)
  &= \pi^{t+\frac d2} \frac{\Gamma(-\frac t2)}{\Gamma(\frac{d+t}{2})} \abs{\xi}^t.
\end{align}
Thus, we get with $t=2s$
\begin{align*}
  \mathcal{F} \bigg[ -2\kappa_{d,s}\pvint_{\RRd} \frac{v(\cdot+h) -v(\cdot)}{\abs{h}^{d+2s}}\,dh  \bigg](\xi)
  &= -2\kappa_{d,s} \pi^{2s+\frac d2} \frac{\Gamma(-s)}{\Gamma(\frac{d}{2}+s)}  \abs{\xi}^{2s}\mathcal{F}[v](\xi)\\
  &= (2\pi \abs{\xi})^{2s}\mathcal{F}[v](\xi).
\end{align*}
This briefly verifies the equivalence of the two approaches. For the rigorous proof, see the proof of \cite[Proposition 3.3]{NGV12}.

\subsection{Fractional Sobolev spaces}\label{sec:pre}
We always denote by $s\in(0,1)$ the fractional differentiability. For $\Omega\subset\setR^d$ being an open set with continuous boundary, we define the fractional Sobolev space $W^{s,2}(\Omega)$ as the set of all measurable functions $v:\setR^d\rightarrow\setR$ with $v\in L^2(\Omega)$ such that the following seminorm
\begin{align*}
\abs{v}_{W^{s,2}(\Omega)}\coloneqq\int_{\Omega}\int_{\Omega}\abs{x-y}^{-d-2s} \abs{v(x)-v(y)}^2 \,dx \,dy
\end{align*}
is finite. We denote the norm in $W^{s,2}(\Omega)$ as $\norm{v}_{W^{s,2}(\Omega)}\coloneqq\norm{v}_{L^{2}(\Omega)}+\abs{v}_{W^{s,2}(\Omega)}$. The other fractional Sobolev space $W^{s,2}_0(\Omega)$ is defined as 
\begin{align*}
  W^{s,2}_0(\Omega) &\coloneqq \left\{v\in W^{s,2}(\setR^d)\,:\,v\equiv 0\,\,\text{in}\,\,\Omega^\complement\right\}.
\end{align*}
Then $W^{s,2}(\Omega)$ and $W^{s,2}_0(\Omega)$ are Banach spaces, and $W^{s,2}_0(\Omega)\subset W^{s,2}(\Omega)$ holds. Moreover,  we have continuous embeddings $W^{s,2}(\Omega) \embedding L^2(\Omega)$ and $W^{s,2}_0(\Omega) \embedding L^2(\Omega)$. Moreover, $C^\infty_0(\Omega)$ is dense in $W^{s,2}_0(\Omega)$ by~\cite[Theorem~6]{FisSerVal15}.

\subsection{Constant matrix-valued coefficient}
\label{subsec:3.1}

To get some inspiration for a good nonlocal model we start with the case of a constant matrix coefficient. For this let $\bbA \in \RRddpos$, where $\RRddpos$ is the set of symmetric positive-definite matrices in~$\RRdd$. Let $\bbM \coloneqq \sqrt{\bbA}$ be its positive-definite square root.  We define
\begin{align}
  \label{eq:def-deltaA}
  \DeltaA u &\coloneqq \divergence(\bbA \nabla u)=\divergence(\bbM^2 \nabla u).
\end{align}
The fractional Laplacian can be obtained by taking the Fourier symbol $(2\pi \abs{\xi})^2$ of the Laplacian $-\Delta$ to a fractional power~$s \in (0,1)$. Then the new Fourier symbol $(2\pi \abs{\xi})^{2s}$ defines the fractional Laplacian~$(-\Delta)^s$. We want to proceed similarly in order to construct $\DeltaAs$, in this section for a constant matrix~$\bbA$. The Fourier symbol of $\DeltaA$ is $(2\pi \abs{\bbM\xi})^2$ and can be calculated as follows:
\begin{align}\label{eq:weight.frac}
  \mathcal{F}\big[\DeltaA u\big](\xi)\!=\!\mathcal{F}\big[\!-\!\divergence(\bbA \nabla u)\big](\xi)\!=\!4 \pi^2 \skp{\bbA\xi}{\xi} \mathcal{F}[u](\xi)\!=\!\left(2\pi \abs{\bbM\xi}\right)^2 \mathcal{F}[u](\xi).
\end{align}
Hence, we define $\DeltaAs$ by
\begin{align*}
  \mathcal{F}\big[ \DeltaAs u\big](\xi) &\coloneqq 4^s \pi^{2s}\skp{\bbA\xi}{\xi} \mathcal{F}[u](\xi) =
  \left(2\pi\abs{\bbM\xi}\right)^{2s}\mathcal{F}[u](\xi),
\end{align*}
where the variable~$\xi$ of the Fourier symbol of $(-\Delta)^s$ is replaced by~$\bbM \xi$ to obtain $\DeltaAs$. Similar to~\eqref{eq:exp} we compute
\begin{align*}
  \mathcal{F}\big[\DeltaAs u\big](\xi)
  &= 2 \kappa_{d,s} \bigg(\pvint_{\RRd} \frac{\exp(2 \pi i \bbM \xi \cdot h) - 1}{\abs{h}^{d+2s}}\,dh\bigg)\mathcal{F}[u](\xi)
  \\
  &= 2 \kappa_{d,s} \bigg(\pvint_{\RRd} \frac{\exp(2 \pi i \xi \cdot \bbM h) - 1}{\abs{h}^{d+2s}}\,dh\bigg)\mathcal{F}[u](\xi),
\end{align*}
where we used the symmetry of $\bbM$. This and the positive definiteness of~$\bbM$ imply
\begin{align}
\label{eq:DeltaAsconstant}
\begin{split}
\DeltaAs u(x)
&=
-2\kappa_{d,s} \pvint_{\RRd} \frac{u(x+\bbM h) -u(x)}{\abs{h}^{d+2s}}\,dh
\\
&=
-2\kappa_{d,s} \pvint_{\RRd} \frac{u(x+h) -u(x)}{\abs{\bbM^{-1} h}^{d+2s} \det(\bbM)}\,dh.
\end{split}
\end{align}
Note that $\DeltaAs u(x)=0$ is the Euler-Lagrange equation of the following energy:
\begin{align}
  \label{eq:JAsconstant}
  \begin{aligned}
    \Js_\bbA(v)
    &= \frac{\kappa_{d,s}}{2} \!
    \int_{\RRd} \!\int_{\RRd}\! \frac{\abs{v(x+\bbM h)\!-\!v(x)}^2}{\abs{h}^{d+2s}}\,dh \,dx
    \\
    &= \frac{\kappa_{d,s}}{2} \!
    \int_{\RRd}\! \int_{\RRd}
    \!\frac{\abs{v(x+ h)\!-\!v(x)}^2}{\abs{\bbM^{-1} h}^{d+2s} \det(\bbM)}\,dh \,dx.
  \end{aligned}
\end{align}
In the next subsection, we introduce the coefficient as the matrix-valued function $\bbM(x)$ on $x\in\setR^d$, instead of the constant matrix $\bbM$.

\subsection{Model with matrix-valued coefficient}
\label{subsec:var}

Inspired by the nonlocal energy~\eqref{eq:JAsconstant} for a constant matrix-valued coefficient, we now consider an energy for a non-constant matrix-valued coefficient. For $\bbA \,:\, \RRd \to \RRddpos$ and $\bbM(x) = \sqrt{\bbA(x)}$ we define the energy
\begin{align}\label{eq:J.weighted}
  \begin{split}
    \Js_{\bbA(\cdot)}(v) &\coloneqq \frac{\kappa_{d,s}}{2} \int_{\RRd} \int_{\RRd} \frac{\abs{v(x+\bbM(x) h)-v(x)}^2}{\abs{h}^{d+2s}}\,dh \,dx
    \\
    &= \frac{\kappa_{d,s}}{2} \int_{\RRd} \int_{\RRd} \frac{\abs{v(x+ h)-v(x)}^2}{\abs{\bbM^{-1}(x) h}^{d+2s} \det(\bbM(x))}\,dh \,dx.
  \end{split}
\end{align}
Then the Euler-Lagrange equation of $\Js_{\bbA(\cdot)}$ is given as
\begin{align*}
 \pvint_{\RRd} \frac{\widetilde{k}(x;h) + \widetilde{k}(x+h;-h)}{2} \big(v(x+h) -v(x)\big) \,dh = 0
\end{align*}
with
\begin{align*}
  \widetilde{k}(x;h) = \frac{\kappa_{d,s}}{2} \frac{1}{\abs{\bbM^{-1}(x)h}^{d+2s} \det(\bbM(x))}
  =  \frac{\kappa_{d,s}}{2} \frac{1}{\skp{\bbA^{-1}(x)h}{h}^{\frac{d+2s}{2}} \sqrt{\det(\bbA(x))}}.
\end{align*}
\begin{remark}
  Suppose that $\lambda \identity \leq \bbM(\cdot) \leq \Lambda \identity$ for some $0<\lambda \leq \Lambda$. Then it is possible to show that for any $v \in C^\infty_0(\RRd)$, the energy $\Js_{\bbA(\cdot)}(v)$ converges to a multiple of $\frac 12 \int_{\RRd} \abs{\bbM(x) \nabla v}^2\,dx$ for $s\to 1$. The arguments are similar to~\cite{Voi17,FKV20}.
  Note that the energy in~\eqref{eq:J.weighted} has a positive kernel~$\widetilde{k}(x;h)$ for every~$\bbA(x) > 0$. 
\end{remark}

\subsection{Model using Taylor approximation}
\label{sec:model-using-taylor}

Unfortunately, the model in Section~\ref{subsec:var} is too difficult to analyze with respect to Meyers-type example. The nonlinear denominator makes it impossible to calculate the integral using the Fourier transform. Therefore, we use for our Meyers-type example a Taylor approximation of the energy in~\eqref{eq:J.weighted} and derive another simpler nonlocal energy with coefficients.

For this we consider the matrix coefficient $\bbAe$ from the classical Meyers' example, see~\eqref{eq:meyer-u-M} as a perturbation of the identity matrix, i.e., for $x\in\setR^d$, we write
\begin{align*}
\widehat{x}_1=\frac{x_1}{\abs{x}}\quad\text{and}\quad \widehat{x}=\frac{x}{\abs{x}}.
\end{align*}
With $(\widehat{x}\otimes \widehat{x})_{i,j}=\widehat{x}_i\widehat{x}_j$ for $1\leq i,j\leq d$, we get
\begin{align*}
  \bbAe(x) = (1-\epsilon)\identity + \epsilon \widehat{x} \otimes \widehat{x} = \identity - \epsilon \big(\identity - \widehat{x} \otimes \widehat{x}\big) \eqqcolon \identity - \epsilon \bbG.
\end{align*}
Let us perform a Taylor approximation of the energy~\eqref{eq:J.weighted} with respect to~$\epsilon$:
\begin{align*}
  \bbAe^{-1}&=(\identity-\epsilon\bbG)^{-1}=\sum_{i\geq 0} (\epsilon\bbG)^i=1+\epsilon\bbG+O(\epsilon^2)
\end{align*}
and $\det\bbAe=\det(\identity-\epsilon\bbG)=1-\epsilon\textrm{tr}\bbG+O(\epsilon^2)$. Hence, we can approximate $\widetilde{k}(x;h)$ from Section~\ref{subsec:var} by the following calculation
\begin{align}\label{eq:OG}
  \begin{aligned}
    \lefteqn{\skp{\bbAe^{-1}h}{h}^{-\frac d2-s}(\det\bbAe)^{- \frac 12}} \qquad &
    \\
    &=\abs{h}^{-d-2s}\big(1+\epsilon\skp{\bbG\widehat{h}}{\widehat{h}}+O(\epsilon^2)\big)^{-\frac d2-s}\big(1-\epsilon\textrm{tr}\bbG+O(\epsilon^2)\big)^{- \frac 12}
    \\
    &=\abs{h}^{-d-2s}\big(1-\tfrac{d+2s}{2} \epsilon\skp{\bbG\widehat{h}}{\widehat{h}}+O(\epsilon^2)\big) \big(1+\tfrac{1}{2} \epsilon \textrm{tr}\bbG+O(\epsilon^2)\big)
    \\
    &= \abs{h}^{-d-2s}\big(1+\tfrac{1}{2} \epsilon\textrm{tr}\bbG-\tfrac{d+2s}{2}\epsilon\skp{\bbG\widehat{h}}{\widehat{h}}+O(\epsilon^2)\big)
    \\
    &=\abs{h}^{-d-2s}\big(\tfrac{2-2s}{2} - \tfrac 12\textrm{tr}\bbAe + \tfrac{d+2s}{2}\skp{\bbAe\widehat{h}}{\widehat{h}}+O(\epsilon^2)\big)
    \\
    &=\abs{h}^{-d-2s}\big(\tfrac{2-2s}{2} + \tfrac d2\bigskp{(\bbAe - \tfrac 1d (\mathrm{tr} \bbAe) \identity\big)\widehat{h}}{\widehat{h}} + \tfrac{2s}{2}\skp{\bbAe\widehat{h}}{\widehat{h}}+O(\epsilon^2)\big)
    \\
    &= \abs{h}^{-d-2s}\big[\bigskp{\big(\big(1-\tfrac{\epsilon(1+2s)}{2}\big)\identity+\tfrac{\epsilon(d+2s)}{2}\widehat{x}\otimes\widehat{x}\big)\widehat{h}}{\widehat{h}}+O(\epsilon^2)\big].
  \end{aligned}
\end{align}
Neglecting the term $O(\abs{\bbG}^2)$ we obtain a new energy for our Meyers-type result.

\begin{definition}[Meyers-type energy]
  \label{def:J-Meyers}
  For $s \in (0,1)$ and $\epsilon \in [0,\frac 12]$ we define
  \begin{align*}
    \Jse(v) \coloneqq \dfrac{\kappa_{d,s}}{2}\int_{\RRd} \int_{\RRd} \kse(x,y) \abs{v(x)-v(y)}^2\,dx\,dy,
  \end{align*}
  with $\kse\,:\, \RRd \times \RRd \to [0,\infty)$ and  $\bbAse(x):\setR^d\rightarrow\RRddpos$ given by
  \begin{align*} 
    \kse(x,y) &\coloneqq |x-y|^{-d-2s} \bigskp{\tfrac{\bbAse(x)+\bbAse(y)}{2}\widehat{x-y}}{\widehat{x-y}},
    \\
    \bbAse(x) &\coloneqq \big(1-\tfrac{1+2s}{2}\epsilon\big)\identity+\tfrac{d+2s}{2}\epsilon\widehat{x}\otimes\widehat{x}.
  \end{align*}
\end{definition}
Then we note that 
\begin{itemize}
\item $\bbAse(x)$ is a uniformly elliptic coefficient, since $\bbAse(x)$ has eigenvalue $1+\frac{d-1}{2}\epsilon$ with eigenvector $\widehat{x}$, and eigenvalue $1-\frac{1+2s}{2}\epsilon$ with eigenspace $(\linearspan\set{\widehat{x}})^{\perp}$.
\item $\kse(x,y)$ is symmetric, i.e., $\kse(x,y)=\kse(y,x)$,
\item and $\kse(x,y)\eqsim |x-y|^{-d-2s}$.
\end{itemize}

Here, we assume ~$\epsilon \in [0,\frac 12]$ since it ensures that $1-\frac{1+2s}{2}\epsilon>0$.

\subsection{Limit $s \to 1$}

We begin our analysis by showing that the limit $s\nearrow 1$ recovers the energy used in the classical Meyers' example.
\begin{lemma}\label{lem:gamma}
  For $v \in C^\infty_0(\RRd)$, with $\bbAe(x)$ as in \eqref{eq:meyer-u-M}, we have
  \begin{align*}
    \lim_{s \to 1}\Jse(v)  = \tfrac 12\int_{\RRd}  \skp{\bbAe(x)\nabla v(x)}{\nabla v(x)}\,dx.
  \end{align*}
\end{lemma}
\begin{proof}
  By the definition of kernel~$k_{s,\epsilon}$ in Definition \ref{def:J-Meyers} and \cite[Remark 4.5]{Voi17},
  \begin{align*}
    \lim_{s \to 1}\Jse(v)  = \tfrac 12\int_{\RRd}  \skp{\bbAe(x)\nabla v(x)}{\nabla v(x)}\,dx
  \end{align*}
  holds. We have to show that $\bbAe(x)=\widetilde{\bbA}_\epsilon(x)$ with 
  \begin{align*}
    \widetilde{\bbA}_\epsilon(x) &\coloneqq
    \frac{\Gamma\big(\tfrac{d}{2}+1\big)}{\pi^{\frac{d}{2}}}\int_{S^{d-1}}\sigma \otimes \sigma \bigskp{\bbA_{1,\epsilon}(x)\sigma}{\sigma}\,d\sigma.
  \end{align*}
  Note that for each orthogonal matrix~$Q$,
\begin{align*}
\widetilde{\bbA}_\epsilon(Qx) = Q \widetilde{\bbA}_\epsilon(x)Q^{t},\quad\text{and}\quad\widetilde{\bbA}_\epsilon(Qx) = Q \widetilde{\bbA}_\epsilon(x)Q^{t}.
\end{align*}
Moreover, $\widetilde{\bbA}_\epsilon$ and 
$\bbAe$ are 0-homogeneous. Thus, it suffices to prove $\bbAe(e_1) = \widetilde{\bbA}_\epsilon(e_1)$. By definition of~$\bbA_{1,\epsilon}$  we have
  \begin{align*}
    \big(\widetilde{\bbA}_\epsilon(e_1)\big)_{ij} 
    &=\frac{\Gamma\big(\tfrac{d}{2}+1\big)}{\pi^{\frac{d}{2}}}\bigg[ \Big(1-\frac{3}{2}\epsilon\Big)\int_{S^{d-1}}\sigma_i\sigma_j\,d\sigma+\frac{d+2}{2}\epsilon\int_{S^{d-1}}\sigma_i\sigma_j\sigma_1^2\,d\sigma\bigg].
  \end{align*}
  Here, from \cite{Fol01} we have
  \begin{align*}
    \int_{S^{d-1}}\sigma_i\sigma_j\,d\sigma &=
    \begin{cases}
      \frac{\Gamma\left(\frac{1}{2}\right)^d}{ \Gamma\left(\frac{d}{2}+1\right)}\delta_{ij} &\quad\;\,\, \text{if $i=j$},
      \\
      0 &\quad\;\,\, \text{else},
    \end{cases}
    \\
    \int_{S^{d-1}}\sigma_i\sigma_j\sigma_1^2\,d\sigma
    &=
    \begin{cases}
      0&\quad\text{if }i\neq j,\\
      \frac{3}{d+2}\frac{\Gamma\left(\frac{1}{2}\right)^d}{\Gamma\left(\frac{d}{2}+1\right)}&\quad\text{if }i=j=1,\\
      \frac{1}{d+2}\frac{\Gamma\left(\frac{1}{2}\right)^d}{\Gamma\left(\frac{d}{2}+1\right)}&\quad\text{if }i=j\neq 1.
    \end{cases}
  \end{align*}
  By merging the above three computations and recalling $\Gamma\left(\frac{1}{2}\right)=\pi^{\frac{1}{2}}$, we have
  \begin{align*}
    \widetilde{\bbA}_\epsilon(e_1) =\left[\left(1-\frac{3}{2}\epsilon\right)\identity+\frac{d+2}{2}\epsilon\left(\dfrac{1}{d+2}\identity+\dfrac{2}{d+2}(e_1\otimes e_1)\right)\right]= \bbAe(e_1).
  \end{align*}
  The proof is complete.
\end{proof}

\subsection{Renormalized energy}
\label{sec:renormalized-energy}

We use the same function as in the local Meyers' example, namely $u_\delta(x) := \abs{x}^{1-\delta} \widehat{x}_1$ for some $\delta \in [0,\tfrac 12]$.  Our goal is that $u_\delta$ is a minimizer of~$\Jse$ for a suitable coupling of~$\delta$ and~$\epsilon$.

Unfortunately, $u_\delta$ is not in the energy space, i.e., $\Jse(u_\delta)=\infty$ because of the growth of $u_\delta$ at infinity.  However, this technical problem can be solved by renormalization. For this we renormalize the energy by means of the boundary values. Since $u_\delta$ is supposed to be the solution later (for suitable~$\epsilon=\epsilon(\delta,d,s)$), we use $u_\delta$ itself here as the boundary value. 
We define the renormalized energy $\Jsed$ by
\begin{align}
  \label{eq:J-renormalized}
  \Jsed(v)
  &\coloneqq \dfrac{\kappa_{d,s}}{2} \lim_{R \to \infty} \int_{B_R(0)} \int_{B_R(0)} \hspace{-2mm}\kse(x,y)  \big(\abs{v(y)\!-\!v(x)}^2\!-\!\abs{u_\delta(y)\!-\!u_\delta(x)}^2\big) \,dy\,dx.
\end{align}
Later we have to choose a special combination of $\delta$ and $\epsilon$ depending on~$s$ such that $u$ will indeed be a minimizer of~$\Jsed$.  By the definition of $\Jsed$ we have $\Jsed(u_\delta)=0$. In Proposition~\ref{prop:functional-nice} we will show that $\Jsed$ is well defined on~$u_\delta + W^{s,2}_0(B)$ with 
\begin{align*}
B\coloneqq B_1(0).
\end{align*}
Note that the limit $R\to \infty$ is only necessary for values of~$s \in (0,\frac 12]$. For $s \in (\frac 12,1)$ one could directly take integrals over~$\RRd$ without a limit.

It would be possible in the renormalization to replace $u_\delta$ by any function that agrees with $u_\delta$ on $(2B)^c$ and is in $W^{s,2}(4B)$. For example the function could be chosen to be smooth on~$2B$.
\begin{remark}
	\label{rem:renormalized}
	The renormalized energy is defined such that it has the same first variation as the standard energy $\Jsed(v)$, but allows for a larger class of boundary values and arguments. It is necessary only, since the boundary value function has infinite energy. We refer to \cite{CMY17, DPV19} for further discussions.

  An alternative method of renormalization is to restrict the integral over $\RRd \times \RRd$ to $(B^c \times B^c)^c$, see for example~\cite{DydKas19}. However, the almost linear growth of~$u_\delta$ (for $\delta$ small) would imply~$\Jsed(u_\delta)<\infty$ which still requires $s \in (\frac 12,1)$. The renormalization in~\eqref{eq:J-renormalized}, however, allows to consider the full range $s \in (0,1)$.
\end{remark}

For future reference let us recall the following lemma.
\begin{lemma}\label{lem:basic}
  Let $\alpha\in\setR$ be given. For any $z,w\in\setR^d$ with $|z-w|\leq \frac 12 |w|$, we have
  \begin{align*}
    \left|\dfrac{z}{|z|^{\alpha}}-\dfrac{w}{|w|^{\alpha}}\right|\leq c(\alpha)\dfrac{|z-w|}{|w|^{\alpha}}.
  \end{align*}
\end{lemma}
Using this, we prove basic properties of $\Jsed$. 
\begin{proposition}
	\label{prop:functional-nice}
	Let $d\geq 2$, $s\in(0,1)$, $\epsilon\in[0,\frac{1}{2}]$, and $\delta\in[0,\frac{1}{2}]$. The functional $\Jsed$ is well-defined and finite on $u_\delta + W^{s,2}_0(B)$.  It is coercive and uniformly convex on $u_\delta + W^{s,2}_0(B)$ in the sense that for all $v_1,v_2 \in u_\delta + W^{s,2}_0(B)$ we have
  \begin{align}
    \label{eq:unifconvex}
    \tfrac 12 \Jsed(v_1) + \tfrac 12 \Jsed(v_2) -  \Jsed\Big( \frac{v_1+v_2}{2}\Big) 
    \eqsim \abs{v_1-v_2}_{W^{s,2}(\RRd)}
  \end{align}
with the implicit constant $c=c(d,s)$.
\end{proposition}
\begin{proof}
  We first show that $\Jsed$ is well-defined and bounded from below. Let $v = u_\delta + w$ with $w \in W^{s,2}_0(B)$. Since $s-\frac{d}{2} <0 < 1-\delta$ we have $u_\delta \in W^{s,2}(2B)$. Moreover, $\norm{u_\delta}_{W^{s,2}(2B)}$ is bounded uniformly for $\delta \in [0,\frac 12]$. Using $v= u_\delta +w$,
  \begin{align*}
    \abs{v(x)\!-\!v(y)}^2 \!-\! \abs{u_\delta(x)\!-\!u_\delta(y)}^2
    &= \abs{w(x)\!-\!w(y)}^2 \!-\! 2 \big(w(x)\!-\!w(y)\big) \big(u_\delta(x)\!-\!u_\delta(y)\big)
  \end{align*}
  is obtained. This proves that
  \begin{align*}
    \Jsed(v) &= \Jse(w) - \kappa_{d,s} \lim_{R \to \infty} \int_{B_R(0)} \int_{B_R(0)} \!\!\kse(x,y) \big(w(x)\!-\!w(y)\big) \big(u_\delta(x)\!-\!u_\delta(y)\big) \,dy\,dx
    \\
    &\eqqcolon \Jse(w) - \mathrm{I}.
  \end{align*}
  We will below show that $\mathrm{I}$ is well defined and
  \begin{align}
    \label{eq:coercive1}
    \mathrm{I} \lesssim \abs{w}_{W^{s,2}(\RRd)} \quad \text{ uniformly for $\delta \in [0, \tfrac 12]$.}
  \end{align}
  This then implies that $\Jsed(w)$ is well defined. Moreover, $\Jse(w) \gtrsim \abs{w}_{W^{s,2}(\RRd)}^2$ also implies the coercivity of~$w \mapsto \Jsed(u_\delta+w)$ with respect to $\norm{w}_{W^{s,2}(\RRd)}^2$.

  Note that
  \begin{align*}
    \int_{2B} \int_{2B} \kse(x,y)  \bigabs{w(x)\!-\!w(y)}\, \bigabs{u_\delta(x)\!-\!u_\delta(y)}\,dx\,dy &\lesssim \abs{w}_{W^{s,2}(2B)} \abs{u_\delta}_{W^{s,2}(2B)}
    \\
    &\lesssim \abs{w}_{W^{s,2}(\RRd)}
  \end{align*}
with the implicit constant $c=c(d,s,\delta)$.
  Define $A_R \coloneqq (B_R(0) \times B_R(0)) \setminus (2B \times 2B)$. By symmetry, we now consider
  \begin{align*}
    \mathrm{II} \coloneqq \lim_{R \to \infty}
    \iint_{A_R} \kse(x,y)  w(x) \big(u_\delta(x)\!-\!u_\delta(y)\big) \,dy\,dx.
  \end{align*}
  Using $w=0$ on $B^c$ and symmetry we obtain
  \begin{align*}
    \mathrm{II} &= 2\lim_{R \to \infty} \,
    \int_B \int_{B_R(0)\setminus 2B} \kse(x,y) w(x) \big(u_\delta(x)\!-\!u_\delta(y)\big)\,dy\,dx.
  \end{align*}
  Note that
  \begin{align*}
    \int_B \int_{(2B)^c} \kse(x,y)  \abs{w(x)}\abs{u_\delta(x)} \,dy\,dx
    &\lesssim \norm{w}_{L^2(B)} \norm{u_\delta}_{L^2(B)} \lesssim \abs{w}_{W^{s,2}(\RRd)}.
  \end{align*}
  It remains to consider
  \begin{align*}
    \mathrm{III} &\coloneqq -2\lim_{R \to \infty} \,
    \int_B \int_{B_R(0)\setminus 2B} \kse(x,y) w(x) u_\delta(y)\,dy\,dx.
  \end{align*}
  If $s \in (\frac 12,1)$, then $\abs{u_\delta(y)} \leq \abs{y}^{1-\delta} \leq \abs{y}$ on $(2B)^c$ would allow us to control $\mathrm{III}$ similar to the calculations above. However, for the full range $s \in(0,1)$ we need some cancellation property.
  Recall that
  \begin{align*}
    \kse(x,y)
    &= \bigskp{ \tfrac{\bbAse(x)+\bbAse(y)}{2}
       \tfrac{x-y}{\abs{x-y}^{d/2+s+1}}
    }{
      \tfrac{x-y}{\abs{x-y}^{d/2+s+1}}
      }.
  \end{align*}
  Let us define
  \begin{align*}
    g(x,y)
    &= \bigskp{ \tfrac{\bbAse(x)+\bbAse(y)}{2}
      \tfrac{y}{\abs{y}^{d/2+s+1}}
    }{
      \tfrac{y}{\abs{y}^{d/2+s+1}}
    }.
  \end{align*}
  Since $u_\delta$ is odd, $\bbAse(y)$ and $y\mapsto g(x,y)$ is even, and the set $B_R(0)\setminus 2B=B_R(0)\setminus 2B_{1}(0)$ is symmetric, we have
  \begin{align*}
    \int_B \int_{B_R(0)\setminus 2B} g(x,y) \bigskp{ \tfrac{\bbAse(x)+\bbAse(y)}{2} \widehat{y}}{\widehat{y}} w(x)u_\delta(y)\,dy\,dx &= 0.
  \end{align*}
  Thus,
  \begin{align*}
    \mathrm{III} &= -2\lim_{R \to \infty} 
    \int_B \int_{B_R(0)\setminus 2B} \big(\kse(x,y) -g(x,y)\big) w(x)u_\delta(y)\,dy\,dx.
  \end{align*}
  It follows from $\abs{\bbAse}\lesssim 1$ and Lemma~\ref{lem:basic} that for $x \in B$ and $y \in (2B)^c$,
  \begin{align}\label{eq:diff.ker}
\begin{split}
\abs{\kse(x,y)-g(x,y)} &\lesssim
    \Bigabs{\tfrac{x-y}{\abs{x-y}^{d/2+s+1}} - \tfrac{y}{\abs{y}^{d/2+s+1}}}
    \Bigabs{\tfrac{x-y}{\abs{x-y}^{d/2+s+1}} + \tfrac{y}{\abs{y}^{d/2+s+1}}}
    \\
    &\lesssim
    \abs{x} \abs{y}^{-d-2s-1}.
\end{split}
  \end{align}
  This implies that
  \begin{align*}
    \lefteqn{\int_B \int_{(2B)^c} \bigabs{\kse(x,y) -g(x,y)} \abs{w(x)} \abs{u_\delta(y)}\,dy\,dx} \qquad
    \\
    &\lesssim
    \int_B \int_{(2B)^c}     \abs{x} \abs{y}^{-d-2s-1} \abs{w(x)} \abs{y}^{1-\delta}\,dy\,dx\lesssim \norm{w}_{L^1(B)} \lesssim \abs{w}_{W^{s,2}(\RRd)}.
  \end{align*}
  Overall, this shows that $\mathrm{III}$, $\mathrm{II}$ and $\mathrm{I}$ are well defined and that~\eqref{eq:coercive1} holds.  

  A straightforward calculation shows that  for all $v_1,v_2 \in u_\delta + W^{s,2}_0(B)$ we have
  \begin{align*}
    \lefteqn{\tfrac 12 \Jsed(v_1) + \tfrac 12 \Jsed(v_2) -  \Jsed\Big( \frac{v_1+v_2}{2}\Big)} \quad
    \\
    &= \frac{\kappa_{d,s}}{2} \int_{\RRd} \int_{\RRd} \kse(x,y) \abs{(v_1-v_2)(y)-(v_1-v_2)(x)}^2 \,dy\,dx\eqsim \abs{v_1-v_2}_{W^{s,2}(\RRd)}.
  \end{align*}
  This proves the uniform convexity in the sense of~\eqref{eq:unifconvex}.
\end{proof}

\subsection{Euler--Langrange equation}\label{subsec:EL}
Proposition~\ref{prop:functional-nice} ensures that we can apply the standard theory of variational integrals to~$\Jsed$, which ensures the existence of a unique minimizer of
\begin{align}
  \label{eq:def-u}
  \bar{u}_{\epsilon,\delta} \coloneqq \argmin_{v \in u_\delta+W^{s,2}_0(\Omega)} \Jsed(v).
\end{align}
We remark that later in our nonlocal Meyers' example $\delta$ and $\epsilon$ are chosen such that $u_\delta$ itself is the solution, i.e., $u_\delta=\bar{u}_{\epsilon,\delta}$. Let us now derive the Euler-Lagrange equation. For this we need the derivative of~$\Jsed$. Let $v \in u_\delta + W^{s,2}_0(B)$ and $\eta \in W^{s,2}_0(B)$. Then for $\tau > 0$ we calculate the difference quotient
\begin{align*}
  &\lefteqn{\frac{\Jsed(v+\tau \eta) - \Jsed(v)}{\tau}} \qquad 
  \\
  &= \dfrac{\kappa_{d,s}}{2} \lim_{R \to \infty} \hspace{-1mm}\iint_{B_R(0)\times B_R(0)}\hspace{-3mm} \!\!\!\!\!\kse(x,y)  \big( 2 (v(x)\!-\!v(y))(\eta(x)\!-\!\eta(y)) + \tau \abs{\eta(y)\!-\!\eta(x)}^2\big) \,dy\,dx 
  \\
  &= \kappa_{d,s} \lim_{R \to \infty} \int_{B_R(0)} \int_{B_R(0)} \!\!\kse(x,y) (v(x)-v(y))(\eta(x)-\eta(y)) \,dy\,dx  + \tau \mathcal{J}^s(\eta).
\end{align*}
This proves that the Gateaux derivative of $\mathcal{J}_D^s$ exists and
\begin{align}\label{eq:EL.eq-aux1}
  (\Jsed)'(v)(\eta) 
  &= \kappa_{d,s}\lim_{R \to \infty} \int\limits_{B_R(0)} \int\limits_{B_R(0)} \!\!\!\kse(x,y) (v(x)\!-\!v(y))(\eta(x)\!-\!\eta(y)) \,dy\,dx
\end{align}
for all $v \in u_\delta + W^{s,2}_0(B)$ and $\eta \in W^{s,2}_0(B)$.

Similar to the calculation in Proposition~\ref{prop:functional-nice} one can show that
\begin{align}
  \label{eq:Jprime-estimate}
  (\Jsed)'(v)(\eta)  &\lesssim \big(1 + \abs{v-u_\delta}_{W^{s,2}(\RRd)}\big) \norm{\eta}_{W^{s,2}(\RRd)}
\end{align}
uniformly for $\delta \in [0,\frac 12]$.  Thus, the Gateaux derivative is linear and bounded on~$W^{s,2}_0(B)$. The density of~$ C^\infty_0(B)$ to $W^{s,2}_0(B)$ allows us to restrict ourselves sometimes to~$\eta \in C^\infty_0(B)$.
Overall,  we have derived the Euler-Lagrange equation:
\begin{align}\label{eq:EL.eq}
  \begin{split}
    (\Jsed)'(\bar{u}_{\epsilon,\delta})(\eta) &= 0
  \end{split}
\end{align}
for all $\eta \in W^{s,2}_0(B)$. Moreover, $v \in u_\delta + W^{s,2}_0(B)$ is the minimizer of $\Jsed$ over $u_\delta + W^{s,2}_0(B)$ if and only if $(\Jsed)'(v)(\eta) =0$ for all $\eta \in W^{s,2}_0(B)$.

\section{Strong formulation}
\label{sec:strong-formulation}

Later for our Meyers-type example we want to find $\epsilon=\epsilon(\delta,s)$ such that $u_\delta = \bar{u}_{\epsilon,\delta}$, i.e. that $u_\delta$ minimizes $\Jsed$ over $u_\delta + W^{s,2}_0(B)$. For this we have to check that $(\Jsed)'(u_\delta)(\eta)=0$ for all $\eta \in W^{s,2}_0(\Omega)$. Due to the boundedness of $(\Jsed)'(u_\delta)$ (see \eqref{eq:Jprime-estimate}) and the density of $C^\infty_0(B)$ in $W^{s,2}_0(\Omega)$ it suffices to verify that
\begin{align*}
  (\Jsed)'(u_\delta)(\eta) &= 0
\end{align*}
for all $\eta \in C^\infty_0(\Omega)$ for suitable $\epsilon=\epsilon(\delta,d,s)$.

In the following we will derive a version of $(\Jsed)'(u_\delta)(\eta)$. In particular, we will introduce an object $(-\Delta_{\mathbb{A}_{s,\epsilon}})^s(u_\delta)$ such that
\begin{align}
  \label{eq:strong-version}
  (\Jsed)'(u_\delta)(\eta) &= \skp{ (-\Delta_{\mathbb{A}_{s,\epsilon}})^s u_\delta}{\eta}
\end{align}
for all $\eta \in C^\infty_0(\Omega)$.
We start by introducing~$(-\Delta_{\mathbb{A}_{s,\epsilon}})^s u_\delta$.
\begin{lemma}
  \label{eq:estimate-Deltas-ud}
  Let $\delta \in [0,\frac 12]$. Then for all $x \in B_1(0) \setminus \set{0}$ the limit
  \begin{align*}
    ((-\Delta_{\mathbb{A}_{s,\epsilon}})^s u_\delta)(x) \coloneqq 2\lim_{R \to \infty} \int_{\frac 1R \leq \abs{h} \leq R}  \kse(x,x+h) (u_\delta(x)-u_\delta(x+h))  \,dh
  \end{align*}
  exists. Moreover, $(-\Delta_{\mathbb{A}_{s,\epsilon}})^s u_\delta \in L^1(B_1(0))$ and all $x\in B_1(0)\setminus\{0\}$, we have
  \begin{align*}
 \bigabs{((-\Delta_{\mathbb{A}_{s,\epsilon}})^s u_\delta)(x)}\!\leq\! 2\sup_{R \geq 1} \biggabs{\int_{\frac 1R \leq \abs{h} \leq R}\hspace{-4mm}\kse(x,x\!+\!h) (u_\delta(x)\!-\!u_\delta(x\!+\!h))dh}\!\lesssim\! \abs{x}^{1-\delta-2s}.
  \end{align*}
\end{lemma}
\begin{proof}
  Let $x \in B_1(0) \setminus \set{0}$ and $R\geq 2$. We define
  \begin{align*}
    U_R(x) &\coloneqq \int_{\frac 1R \leq \abs{h} \leq R}  \kse(x,x+h) (u_\delta(x)-u_\delta(x+h))  \,dh.
  \end{align*}
  Now, let $R > r \geq 2$. Then
  \begin{align*}
    U_R(x) - U_r(x)
    &= \int_{\frac 1R \leq \abs{h} < \frac 1r} \kse(x,x+h) (u_\delta(x)-u_\delta(x+h))  \,dh
    \\
    &\quad + \int_{r < \abs{h} \leq R} \kse(x,x+h) u_\delta(x)  \,dh
    \\
    &\quad + \int_{r < \abs{h} \leq R} \kse(x,x+h) (-u_\delta(x+h))  \,dh\coloneqq \mathrm{I} + \mathrm{II} + \mathrm{III}.
  \end{align*}
  Note that $\mathrm{II} \lesssim r^{-2s} \abs{x}^{1-\delta}$.
  For $\mathrm{III}$, since $u_\delta(h)$ is odd and $\bbAse(h)$ is even, we have
  \begin{align*}
    \mathrm{III}_2 &\coloneqq\frac{1}{2}\int_{r < \abs{h} \leq R} \frac{\skp{(\bbAse(x)+\bbAse(h))\widehat{h}}{\widehat{h}}}{\abs{h}^{d+2s}} (-u_\delta(h))  \,dh =0.
  \end{align*}
On the other hand with $\abs{x} \leq 1$, using \eqref{eq:diff.ker} and $\abs{u_{\delta}(h)-u_{\delta}(x+h)}\lesssim \abs{h}^{-\delta}\abs{x}$,
  \begin{align*}
\bigabs{\mathrm{III} - \mathrm{III}_2}&\lesssim \int_{r < \abs{h} \leq R} \abs{h}^{-d-2s} \frac{\abs{x}}{\abs{h}} \big(1 + \abs{h}^{1-\delta}\big)  \,dh\lesssim r^{-2s} \abs{x}^{1-\delta}
  \end{align*}
  is obtained. We turn to $\mathrm{I}$. Let
  \begin{align*}
    \mathrm{I}_2 \coloneqq \int_{\frac 1R \leq \abs{h} < \min \set{\frac 12\abs{x},\frac 1r}} \indicatorset{\abs{h} < \abs{x}}\kse(x,x+h) (u_\delta(x)-u_\delta(x+h))  \,dh.
  \end{align*}
  Then
  \begin{align*}
    \bigabs{\mathrm{I} -\mathrm{I}_2} &\leq\int_{\frac 12\abs{x} \leq \abs{h} < \frac 1r} \kse(x,x+h) \abs{u_\delta(x)-u_\delta(x+h)}  \,dh
    \\
    &\lesssim \indicatorset{\frac 1r \geq \frac 12 \abs{x}} \int_{\abs{x} \leq \abs{h}}  \abs{h}^{-d-2s} \abs{x}^{1-\delta}  \,dh\lesssim \indicatorset{\frac 1r \geq \frac 12\abs{x}} \abs{x}^{1-\delta-2s}.
  \end{align*}
  For $\abs{h} \leq \frac 12 \abs{x}$ we have
  \begin{align*}
    \abs{u_{\delta}(x+h) - u_{\delta}(x) - \nabla u_{\delta}(x)h}
    &\lesssim \abs{x}^{-1-\delta} \abs{h}^2.
  \end{align*}
  Let
  \begin{align*}
    \mathrm{I}_3 \coloneqq \int_{\frac 1R \leq \abs{h} < \min \set{\frac 12\abs{x},\frac 1r}} \indicatorset{\abs{h} < \abs{x}}\kse(x,x+h) \nabla u_{\delta}(x)h  \,dh.
  \end{align*}
  Then
  \begin{align*}
    \abs{\mathrm{I}_2 -\mathrm{I}_3}
    &\lesssim \int_{\frac 1R \leq \abs{h} < \min \set{\frac 12\abs{x},\frac 1r}} \indicatorset{\abs{h} < \abs{x}} \abs{h}^{-d-2s+2} \abs{x}^{-1-\delta}  \,dh
    \\
    &\lesssim \big( \min \set{\abs{x},\tfrac 1r} \big)^{2(1-s)} \abs{x}^{-1-\delta}.
  \end{align*}
  Note that
  \begin{align*}
    \mathrm{I}_4 &\coloneqq
    \int_{\frac 1R \leq \abs{h} < \min \set{\frac 12\abs{x},\frac 1r}} \indicatorset{\abs{h} < \abs{x}} \frac{2 \skp{\bbAse(x) \widehat{h}}{\widehat{h}}}{\abs{h}^{d+2s}} \nabla u_{\delta}(x)h  \,dh = 0,
  \end{align*}
  since the integrand is odd with respect to $h$. Thus,
  \begin{align*}
    \abs{\mathrm{I}_3} &= \abs{\mathrm{I}_3 - \mathrm{I}_4}
    \lesssim
    \int_{\frac 1R \leq \abs{h} < \min \set{\frac 12\abs{x},\frac 1r}} \abs{h}^{-d-2s} \frac{\abs{h}}{\abs{x}} \abs{x}^{-\delta} \abs{h} \,dh
    \\
    &\lesssim \big( \min \set{\abs{x},\tfrac 1r} \big)^{2(1-s)} \abs{x}^{-1-\delta}.
  \end{align*}
  Overall, we have for $\abs{x} \leq 1$ and $R \geq r \geq 1$ that
  \begin{align*}
    \abs{U_R(x) - U_r(x)} &\lesssim r^{-2s} \abs{x}^{1-\delta} + \indicatorset{\frac 1r \geq \frac 12 \abs{x}} \abs{x}^{1-\delta-2s} + \big( \min \set{\abs{x},\tfrac 1r} \big)^{2(1-s)} \abs{x}^{-1-\delta}.
  \end{align*}
  This proves that for every~$x \in B_1(0) \setminus \set{0}$, $U_r(x)$ is a Cauchy sequence for $r \to \infty$. Thus $((-\Delta_{\mathbb{A}_{s,\epsilon}})^s u_\delta)(x) = \lim_{r \to \infty} U_r(x)$ exists for all $x \in B_1(0)\setminus\{0\}$. Moreover, for each~$x \in B_1(0) \setminus \set{0}$ we estimate
  \begin{align*}
    \abs{((-\Delta_{\mathbb{A}_{s,\epsilon}})^s u_\delta)(x)} &\leq \abs{U_1(x)} + \sup_{r \geq 1} \abs{U_R(x)-U_2(x)}
    \lesssim 0 + \abs{x}^{1-\delta-2s}.
  \end{align*}
  This proves the lemma.
\end{proof}

The next step is to verify~\eqref{eq:strong-version}.
\begin{lemma}
  \label{lem:strong}
  Let $\delta, \epsilon \in [0,\frac 12]$. Then for all $\eta \in C^\infty_0(B)$, we have
  \begin{align}
    \label{eq:strong}
    (\Jsed)'(u_\delta)(\eta) &= \int_B ((-\Delta_{\mathbb{A}_{s,\epsilon}})^s u_\delta)(x) \eta(x)\,dx.
  \end{align}
\end{lemma}
\begin{proof}
  Let $\eta \in C^\infty_0(B)$. Then by~\eqref{eq:EL.eq-aux1} we have
  \begin{align*}
    (\Jsed)'(u_\delta)(\eta) 
    &= \kappa_{d,s}\lim_{R \to \infty} \int_{B_R(0)} \int_{B_R(0)} \!\!\!\kse(x,y) (u_\delta(x)-u_\delta(y))(\eta(x)-\eta(y)) \,dy\,dx.
  \end{align*}
  In the following let~$R \geq 2$. We define
  \begin{align*}
    \mathrm{I} &\coloneqq
    \int_{B_R(0)} \int_{B_R(0)} \kse(x,y) (u_\delta(x)-u_\delta(y))(\eta(x)-\eta(y)) \,dy\,dx,
    \\
    \mathrm{II} &\coloneqq
    \int_{B_R(0)} \int_{B_R(0)} \indicatorset{\abs{x-y} \geq \frac 1R} \kse(x,y) (u_\delta(x)-u_\delta(y))(\eta(x)-\eta(y)) \,dy\,dx.
  \end{align*}
  Then since $\delta\in[0,\frac{1}{2}]$, by using $\abs{\abs{a}^{-\delta}a-\abs{b}^{-\delta}b}\lesssim (\abs{a}+\abs{b})^{-\delta}\abs{a-b}$ for $a,b\in\setR$,
  \begin{align*}
    \abs{\mathrm{I} - \mathrm{II}}
    &\lesssim \norm{\nabla \eta}_\infty\int_{B_2(0)} \int_{B_2(0)} \indicatorset{\abs{x-y}< \frac 1R} \abs{x-y}^{-d-2s+2}   \big(\abs{x}+\abs{y}\big)^{-\delta}\,dx
    \\
    &\lesssim \norm{\nabla \eta}_\infty\, R^{2s-2}
  \end{align*}
  holds. This proves that
  \begin{align*}
    (\Jsed)'(u_\delta)(\eta) &= \lim_{R \to \infty} \mathrm{I} =  \lim_{R \to \infty} \mathrm{II}.
  \end{align*}
  Moreover,
  \begin{align*}
    \mathrm{II} &= 2
    \int_{B_R(0)} \int_{B_R(0)} \indicatorset{\abs{x-y} \geq \frac 1R} \kse(x,y) (u_\delta(x)-u_\delta(y)) \,dy\,\eta(x)\,dx
    \\
    &= 2
    \int_{B_1(0)} \int_{B_R(-x)} \indicatorset{\abs{h} \geq \frac 1R} \kse(x,x+h) (u_\delta(x)-u_\delta(x+h)) \,dh\,\eta(x)\,dx.
  \end{align*}
  Let
  \begin{align*}
    \mathrm{III} &\coloneqq
    2\int_{B_1(0)} \int_{B_R(0)} \indicatorset{\abs{h} \geq \frac 1R} \kse(x,x+h) (u_\delta(x)-u_\delta(x+h)) \,dh\,\eta(x)\,dx.
  \end{align*}
  Then for $R \geq 2$
  \begin{align*}
    \abs{\mathrm{II} - \mathrm{III}}
    &\leq
    \int_{B_1(0)} \int_{\RRd} \indicatorset{R-1 \leq \abs{h} \leq R+1} \abs{h}^{-d-2s} \abs{u_\delta(x)-u_\delta(x+h)} \,dh\,\eta(x)\,dx
    \\
    &\lesssim R^{d-1} R^{-d-2s} R^{1-\delta} \lesssim R^{-2s}.
  \end{align*}
  Overall, we get that
  \begin{align*}
    &(\Jsed)'(u_\delta)(\eta) = \lim_{R \to \infty} \mathrm{I} =  \lim_{R \to \infty} \mathrm{III}
    \\
    &\quad\quad=2\lim_{R\rightarrow\infty}\int_{B_1(0)} \int_{\frac 1R \leq \abs{h} \leq R} \kse(x,x+h) (u_\delta(x)-u_\delta(x+h)) \,dh\,\eta(x)\,dx.
  \end{align*}
  Now, by Lemma~\ref{eq:estimate-Deltas-ud} and the theorem of dominated convergence, we conclude
  \begin{align*}
    (\Jsed)'(u_\delta)(\eta) &=     \int_{B_1(0)} ((-\Delta_{\mathbb{A}_{s,\epsilon}})^s u_\delta)(x)\eta(x)\,dx.
  \end{align*}
  This proves the claim.
\end{proof}

\section{Nonlocal Meyers' example}
\label{sec:nonl-meyers-example}

We are now ready to present our nonlocal Meyers' example in~$\RRd$ with $d \geq 2$ and to prove our main result mentioned in the introduction. Recall that for $\delta\in [0,\frac 12]$ and $\epsilon\in[0,\frac{1}{2}]$ we defined in \eqref{eq:As} that
\begin{align}\label{eq:def-nonlocalmeyer-Au}
  \bbAse(x)= \big(1-\tfrac{1+2s}{2}\epsilon\big)\identity+\tfrac{d+2s}{2}\epsilon\widehat{x}\otimes\widehat{x},
  \quad\text{and}\quad u_\delta(x)=\abs{x}^{1-\delta} \widehat{x}_1.
\end{align}

The main goal of this section is to show the following theorem.

\begin{theorem}[Nonlocal Meyers' example]\label{thm:Meyers}
  Let $d\geq 2$, $d\in\setN$, and $s\in(0,1)$.
  \begin{enumerate}
	\item With $\delta\in[0,\frac{1}{2}]$, the function $u_{\delta}:\setR^d\rightarrow\setR$ with $u_\delta(x)=\abs{x}^{1-\delta} \widehat{x}_1$ satisfies
    \begin{align*}
      u_{\delta}\in W_{\loc}^{t,q}(B_1(0))\quad\text{if and only if}\quad 1-\delta> t-\frac{d}{q},
    \end{align*}
    where $t \in (0,1)$ and $q \in (1,\infty)$.
  \item With $\epsilon\in[0,\frac{1}{2}]$, the matrix-valued coefficient $\bbAse:\setR^d\rightarrow\setR^{d\times d}$ is given by
    \begin{align*}
      \bbAse(x)= \left(1-\frac{1+2s}{2}\epsilon\right)\identity+\frac{d+2s}{2}\epsilon\widehat{x}\otimes\widehat{x},
    \end{align*}
    has the following properties:
    \begin{enumerate}[label={(\roman{*})}]
    \item It is uniformly elliptic with eigenvalues in the range
      \begin{align*}
        \big[1- \tfrac{1+2s}{2}\epsilon,1+\tfrac{d-1}{2} \epsilon\big] \subset \big[\tfrac 14,1+\tfrac {d-1}{4}\big].
      \end{align*}
    \item It satisfies the smallness assumption
      \begin{align}\label{eq:logAse}
        \norm{\log \bbAse}_{L^{\infty}(\RR^d)}\leq \tfrac{1+d+4s}{2} \epsilon,
      \end{align}
      where the matrix is measured in the spectral norm.
    \end{enumerate}
  \item There exists  $\delta_0=\delta_0(d,s)\in[0,\frac{1}{2}]$ such that the following holds: For any $\epsilon \in [0, \frac{1}{2}]$, there exists $\delta \in [0,\delta_0]$ with $\epsilon \leq C(d,s) \delta$ such that $(-\Delta_{\bbAse})^su_{\delta}(x)=0$ as a function on $B_1(0)$. The relation of $\epsilon$ and $\delta$ is given in \eqref{eq:b}.
  \end{enumerate}
\end{theorem}
The proof of this theorem is given at the end of this section.
Also, recalling $\bbAe(x) = (1-\epsilon) \identity + \epsilon \widehat{x} \otimes \widehat{x}$, $\bbAse(x)$ and $u_{\delta}(x)$ as in \eqref{eq:def-nonlocalmeyer-Au}, we will prove the following.

\begin{theorem}[Robustness as $s\nearrow 1$]\label{thm:s1}
Let $d\geq 2$, $d\in\setN$, and $s\in(0,1)$. We have
\begin{align*}
\lim_{s\nearrow 1}(-\Delta_{\bbAse})^s u_{\delta}(x)=(-\Delta_{\bbAe}) u_{\delta}(x)
\end{align*}
for a.e. $x\in\setR^d$, any $\epsilon\in[0,\frac{1}{2}]$ and $\delta\in[0,\frac{1}{2}]$.
\end{theorem}
We postpone the proof of this theorem to the end of this section.

\begin{remark}[Vectorial case]
\label{rem:vectorial}
Meyers' example becomes more elegant, when using the vector-valued function~$u_\delta(x) = \abs{x}^{1-\delta} \widehat{x}$. Our approach works similarly for the vectorial case. The restriction to the first component~$\widehat{x}_1$ is only to emphasize that the scalar case is already critical. In this linear context the vectorial case just decouples into several scalar problems.
\end{remark}

The case $d=2$ in Theorem~\ref{thm:Meyers} of special interest, since it provides the example with the worst regularity. Indeed, the homogeneity of our solution~$u_\delta$ is $1-\delta \in [\frac 12,1]$. However, the index of~$W^{s,2}_{\loc}$ is $s- \frac{d}{2} \in (-\frac d2, 1-\frac d2)$ decreases fast in~$d$. Therefore, the gap of additional regularity becomes larger for~$d$ increasing. This makes the case~$d=2$ the most important one. Therefore, we state in the following some simplifications for the case~$d=2$. 
\begin{proposition}[Simplification for $d=2$]\label{prop:up.lo.est}
  Suppose that $d=2$. Then we have $(-\Delta_{\bbAse})^su_{\delta}(x)=0$ as a function on $B_1(0)$, whenever
  \begin{align}\label{eq:ds}
    \epsilon&=\dfrac{2(2-s-\frac \delta2)\delta}{s(2-s)-\frac{\Gamma(s+1)\Gamma\left(3-s-\frac{\delta}{2}\right)\Gamma(1+\frac{\delta}{2})}{\Gamma(1-s)\Gamma(2-\frac{\delta}{2})\Gamma(s+\frac{\delta}{2})}}.
  \end{align}
  Moreover, 
  for any $s\in(0,1)$ and $\delta\in[0,\frac{2s^2}{1-s})$, we have
  \begin{align}\label{eq:up.lo.est}
    \dfrac{2(2-s-\frac \delta2)\delta}{s(2-s)}\leq\epsilon\leq\dfrac{2(2-s-\frac\delta2)\delta}{(2-s)(s-(1-s)(s+\frac\delta2))}.
  \end{align}
  Note that $\epsilon \to (2-\delta)\delta$ as $s\nearrow 1$ in agreement with \eqref{eq:meyer-epsilon}.
\end{proposition}

\begin{proof}
  The formula~\eqref{eq:ds} is a special case of Corollary~\ref{cor:frac.lap}, which we prove below.
  We can write \eqref{eq:ds} as
  \begin{align}\label{eq:epsilon}
    \epsilon=\dfrac{2(2-s-\frac\delta2)\delta}{s(2-s)-(1-s)(2-s)\left(2-\frac{\delta}{2}\right)\left(s+\frac{\delta}{2}\right)\underbrace{\frac{\Gamma(s+1)\Gamma\left(3-s-\frac{\delta}{2}\right)\Gamma(1+\frac{\delta}{2})}{\Gamma(3-s)\Gamma(3-\frac{\delta}{2})\Gamma(s+1+\frac{\delta}{2})}}_{=:G(s,\delta)}}.
  \end{align}
  The first inequality in \eqref{eq:up.lo.est} is clear since $G(s,\delta)>0$. We claim  that $G(s,\delta) \leq \frac 12$. For this we need \cite[Theorem~10]{Alz97}, where it has been shown that $0 \leq a_1 \leq \dots a_m$ and $0 \leq b_1 \leq \dots b_m$ with $\sum_{i=1}^k a_i \leq \sum_{i=1}^k b_i$ for $k=1,\dots, m$ implies
  \begin{align}
    \label{eq:alzer}
    x \mapsto \prod_{i=1}^m \frac{\Gamma(x+a_i)}{\Gamma(x+b_i)}.
  \end{align}
  is decreasing on $(0,\infty)$.   Applying this to $a_1=1$, $a_2 = 3- \frac\delta 2$, $b_1=3-s-\frac \delta2$, $b_2 = 1+s$ and $x=\frac \delta2$ yields
  \begin{align*}
    \frac{\Gamma(1) \Gamma(3-\frac \delta2)}{\Gamma(3-s-\frac \delta2)\Gamma(1+s)} &\geq
    \frac{\Gamma(1+\frac \delta2) \Gamma(3)}{\Gamma(3-s)\Gamma(1+s+\frac{\delta}{2})}.
  \end{align*}
  This implies that $ G(s,\delta) \leq \frac{\Gamma(1)}{\Gamma(3)} = \frac{1}{2}$ as desired.  Therefore, for \eqref{eq:epsilon} we estimate
  \begin{align*}
    \epsilon
    &\leq\dfrac{2(2-s-\frac \delta2)\delta}{s(2-s)-\frac{1}{2}(1-s)(2-s)\left(2-\frac{\delta}{2}\right)\left(s+\frac{\delta}{2}\right)}
    \\
    &\leq\dfrac{2(2-s-\frac \delta2)\delta}{s(2-s)-(1-s)(2-s)\left(s+\frac{\delta}{2}\right)}
    =
     \dfrac{2(2-s-\frac\delta 2)\delta}{(2-s)(s-(1-s)(s+\frac\delta2))},
  \end{align*}
  provided $s>(1\!-\!s)(s\!+\!\frac\delta2)$ which is equivalent to $\delta<\frac{2s^2}{1-s}$. This condition ensures that the denominator does not change its sign. This proves the claim.
\end{proof}
\begin{remark}[Rotation and scaling properties for $(-\Delta_{\bbAse})^s u$]\label{rem:prop}
We display the following properties of $(-\Delta_{\bbAse})^s u_{\delta}(y)$.
\begin{itemize}
\item Rotation property: For any $Q\in\setR^{d\times d}$ with $QQ^T=\identity$, we have
\begin{align*}
(-\Delta_{\bbAse})^s u_{\delta}(Qy)=(Q(-\Delta_{\bbAse})^s U_{\delta}(y))_1,
\end{align*}
where $U_{\delta}(y):=\abs{y}^{1-\delta}\widehat{y}$. 
\item Scaling property: For any $\lambda\in\setR$, 
\begin{center}
$(-\Delta_{\bbAse})^s u_{\delta}(\lambda y)=\sgn{\lambda}|\lambda|^{1-2s-\delta}(-\Delta_{\bbAse})^s u_{\delta}(y)$.
\end{center}
\end{itemize}
Thus, we see that $(-\Delta_{\bbAse})^s u_{\delta}(x)=0$ is equivalent to $(-\Delta_{\bbAse})^s u_{\delta}(e_1)=0$. 
\end{remark}

To show Theorems \ref{thm:Meyers} and \ref{thm:s1}, we first compute $(-\Delta_{\bbAse})^s u_{\delta}(e_1)$ and then use Remark \ref{rem:prop} to recover $(-\Delta_{\bbAse})^s u_{\delta}(x)$. We write
\begin{align}\label{eq:frac.lap}
(-\Delta_{\bbAse})^s u_{\delta}(e_1)=-2\kappa_{d,s}\left[\left(1-\frac{1+2s}{2}\epsilon\right)f_1(d,s,\delta)+\dfrac{d+2s}{2}\epsilon f_2(d,s,\delta)\right],
\end{align}
where
\begin{align}\label{eq:f1}
f_1(d,s,\delta)\coloneqq\pvint_{\setR^d}\dfrac{|e_1+h|^{1-\delta}\widehat{1+h_1}-1}{|h|^{d+2s}}\,dh=\dfrac{1}{2\kappa_{d,s}}(-\Delta)^s u_{\delta}(e_1)
\end{align}
and
\begin{align}\label{eq:f2}
\begin{split}
f_2(d,s,\delta)&\coloneqq\pvint_{\setR^d}\hspace{-1mm}\left<\dfrac{\widehat{e_1\!+\!h}\otimes\widehat{e_1\!+\!h}\!+\!e_1\otimes e_1}{2}\widehat{h},\widehat{h}\right>\dfrac{|e_1\!+\!h|^{1-\delta}\widehat{1\!+\!h_1}\!-\!1}{|h|^{d+2s}}dh\\
&=\frac{1}{2}\pvint_{\setR^d}\left(\dfrac{\skp{e_1+h}{h}^2}{|e_1+h|^2|h|^2}+\dfrac{\skp{e_1}{h}^2}{|h|^2}\right)\dfrac{|e_1+h|^{1-\delta}\widehat{1+h_1}-1}{|h|^{d+2s}}\,dh.
\end{split}
\end{align}

Now, we need to compute $f_1$ and $f_2$ using the Fourier transform of distributions. 

\begin{lemma}[Computation of $f_1(d,s,\delta)$ and $f_2(d,s,\delta)$]\label{lem:f}
For $\delta\in[0,\frac{1}{2}]$, we have
\begin{align}\label{eq:comp.f1}
f_1(d,s,\delta)=-\frac{2^{2s-1}}{\kappa_{d,s}}\dfrac{\Gamma\left(1-\frac{\delta}{2}+\frac{d}{2}\right)\Gamma\left(s+\frac{\delta}{2}\right)}{\Gamma\left(1-\frac{\delta}{2}+\frac{d}{2}-s\right)\Gamma\left(1+\frac{\delta}{2}\right)}\dfrac{\delta}{2}.
\end{align}
On the other hand, we have
\begin{align}\label{eq:d3}
\begin{split}
f_2(d,s,\delta)&=\dfrac{\pi^{\frac{d}{2}}}{2}\dfrac{\Gamma(-s+1)}{\Gamma\left(\frac{d}{2}+s\right)}\dfrac{\Gamma\left(\frac{d}{2}-\frac{\delta}{2}\right)}{\Gamma\left(1+\frac{\delta}{2}\right)}\dfrac{\Gamma\left(s+\frac{\delta}{2}\right)}{\Gamma\left(\frac{d}{2}-s+1-\frac{\delta}{2}\right)}f_{2,1}(d,s,\delta)\\
&\quad-\dfrac{\pi^{\frac{d}{2}}}{2}\dfrac{\Gamma(-s+1)}{\Gamma\left(\frac{d}{2}+s\right)}\dfrac{\Gamma\left(\frac{d}{2}\right)\Gamma\left(s\right)}{\Gamma\left(\frac{d}{2}-s\right)}\dfrac{d-1}{d+2s},
\end{split}
\end{align}
where
\begin{align}\label{eq:f21}
\begin{split}
&f_{2,1}(d,s,\delta)\\
&=-\dfrac{2(1-s)(2s+\delta-3)(2s+\delta-1)}{(d+2s)(d-2s+2-\delta)}+\dfrac{(1-s)(2s+\delta-3)}{d-2s+2-\delta}\\
&\quad+\dfrac{(2s+\delta-1)(2s+\delta)}{d+2s}-\dfrac{(d+2s+3)(2s+\delta-1)}{d+2s}\\
&\quad+\dfrac{d+2s-1}{2}-\dfrac{(d-2s-\delta)(2s+\delta)}{2s(d+2s)}\\
&=\dfrac{(d - \delta) \left[(2s+1)\left(\delta-\frac{d-2s+2}{2}\right)^2+\frac{1}{4}(d-2s+2)(2ds-d+4s^2-6s-2)\right]}{2 s (d + 2 s)(d-2s+2-\delta)}.
\end{split}
\end{align}
\end{lemma}

\begin{proof}
First, we explain the strategy. For $f_2$, let us rewrite $f_2$ in \eqref{eq:f2} as
\begin{align}\label{eq:f2'}
\begin{split}
&f_2(d,s,\delta)\\
&=\pvint_{\setR^d}\dfrac{-2h_1^3\!+\!|h|^2h_1^2\!+\!2h_1^2\!-\!2|h|^2h_1\!+\!|h|^4}{2|h|^{d+2s+2}|e_1-h|^{2}}\left(|e_1\!-\!h|^{-\delta}(1\!-\!h_1)\!-\!1\right)dh,
\end{split}
\end{align}
which formally corresponds to $(g_3\ast g_4)(e_1)$, where
\begin{align*}
g_3(h):=|h|^{-\delta-2}h_1\!-\!|h|^{-2}\,\,\,\,\text{and}\,\,\,\, g_4(h):=\dfrac{1}{\abs{h}^{d+2s}}\left(\dfrac{-h_1^3}{|h|^{2}}\!+\!\frac{h_1^2}{2}\!+\!\dfrac{h_1^2}{|h|^{2}}\!-\!h_1\!+\!\frac{\abs{h}^2}{2}\right).
\end{align*}

In order to find an explicit formula for $g_3\ast g_4$ we will use the Fourier transform and establish $\mathcal{F}\left[\mathcal{F}[g_3\ast g_4]\right]=\mathcal{F}\left[\mathcal{F}[g_3]\mathcal{F}[g_4]\right]$ and
\begin{align*}
f_2(d,s,\delta)=\mathcal{F}\left[\mathcal{F}[g_3\ast g_4]\right](-e_1)=\mathcal{F}\left[\mathcal{F}[g_3]\mathcal{F}[g_4]\right](-e_1)\,.
\end{align*}
The second equality will follow from 
$\mathcal{F}\left[\mathcal{F}[g_3\ast g_4]\right](x)=\mathcal{F}\left[\mathcal{F}[g_3]\cdot\mathcal{F}[g_4]\right](x)$. The evaluation $x=e_1$ for $\mathcal{F}\left[\mathcal{F}[g_3]\cdot\mathcal{F}[g_4]\right]$ then yields the desired expression for $f_2(d,s,\delta)$.

\medskip

Let us provide details for all necessary steps in the proof.  
\begin{enumerate}
	\item[(1)] We define $g_3$ and $g_4$ in the distributional sense as in Definition \ref{def:f3f4}. Note that $g_3$ and $g_4$ do not belong to $L^1_{\loc}$ in general.
	\item[(2)] We will show that the convolution $g_3\ast g_4$ is well-defined in the tempered distributional sense as in  \cite[Section 6.5]{Vla02}. We provide the details in Lemma \ref{lem:welldef3} and Lemma \ref{lem:conv3}. Note that Appendix \ref{app2} and \cite{Vla02} denote such a convolution with the symbol $\divideontimes$. 
	\item[(3)] In Lemma \ref{lem:A2} we show that $f_2(d,s,\delta)=\lim_{\sigma\rightarrow 0}\skp{g_3\ast g_4}{\phi_{\sigma}}$, where $\phi_{\sigma}$ is an approximation to the identity at $e_1$ and defined as follows: for the standard mollifier $\phi \in\mathcal{S}(\setR^d)$ and $\sigma\in(0,1)$, $\phi_{\sigma}(x):=\frac{1}{\sigma^d}\phi\left(\frac{x-e_1}{\sigma}\right)$.
	\item[(4)] By (2), the relation $\mathcal{F}\left[\mathcal{F}[g_3\ast g_4]\right]=\mathcal{F}\left[\mathcal{F}[g_3]\cdot\mathcal{F}[g_4]\right]$ is available (see \cite[Section 6.5]{Vla02}), where the definition of product of distributions `$\cdot$' is from \cite[Section 6.5]{Vla02}. The precise computation of $\mathcal{F}[g_3]$ and $\mathcal{F}[g_4]$ is given in Lemma \ref{lem:Fourier'}.
	\item[(5)] Lemma \ref{lem:Fourier} shows that $\mathcal{F}[g_3]\cdot\mathcal{F}[g_4]=\mathcal{F}[g_3]\mathcal{F}[g_4]$ in the distributional sense, where $\mathcal{F}[g_3]\mathcal{F}[g_4]$ means the usual multiplication of two locally integrable functions $\mathcal{F}[g_3]$ and $\mathcal{F}[g_4]$.
	\item[(6)] We compute $g_3\ast g_4$ via Fourier transform (see Lemma \ref{lem:final} with Theorem \ref{thm:Fourierconv}). Also, in the proof of Lemma \ref{lem:final}, we observe that $\mathcal{F}[\mathcal{F}[g_3]\mathcal{F}[g_4]]$ coincides with a locally integrable function near $-e_1$. 
\end{enumerate} 	
Overall, with $\overline{\phi}_{\sigma}(x)=\phi_{\sigma}(-x)$, we have
\begin{align}\label{eq:appen}
\begin{split}
f_2(d,s,\delta)&\underset{\text{(3)}}{=}\lim_{\sigma\rightarrow 0}\skp{g_3\ast g_4}{\phi_{\sigma}}\\
&\underset{\text{(4)}}{=}\lim_{\sigma\rightarrow 0}\skp{\mathcal{F}[\mathcal{F}[g_3]\cdot\mathcal{F}[g_4]]}{\overline{\phi}_{\sigma}}\\
&\underset{\text{(5)}}{=}\lim_{\sigma\rightarrow 0}\skp{\mathcal{F}[\mathcal{F}[g_3]\mathcal{F}[g_4]]}{\overline{\phi}_{\sigma}}\\
&\underset{\text{(6)}}{=}\dfrac{\pi^{\frac{d}{2}}}{2}\dfrac{\Gamma(-s+1)}{\Gamma\left(\frac{d}{2}+s\right)}\dfrac{\Gamma\left(\frac{d}{2}-\frac{\delta}{2}\right)}{\Gamma\left(1+\frac{\delta}{2}\right)}\dfrac{\Gamma\left(s+\frac{\delta}{2}\right)}{\Gamma\left(\frac{d}{2}-s+1-\frac{\delta}{2}\right)}f_{2,1}(d,s,\delta)\\
&\quad-\dfrac{\pi^{\frac{d}{2}}}{2}\dfrac{\Gamma(-s+1)}{\Gamma\left(\frac{d}{2}+s\right)}\dfrac{\Gamma\left(\frac{d}{2}\right)\Gamma\left(s\right)}{\Gamma\left(\frac{d}{2}-s\right)}\dfrac{d-1}{d+2s}.
\end{split}
\end{align}
This completes the proof of \eqref{eq:d3}. 

\medskip

The proof in case of $f_1(d,s,\delta)$ is similar. For some suitable distributions $g_1$ and $g_2$ in $\setR^d$, we argue as follows.
In detail, when $\delta=0$, then one can see that $f_1(d,s,\delta)=0$. Thus we assume $\delta>0$.
\begin{enumerate}
	\item[(1)] We define $g_1$ and $g_2$ in the distributional sense as follows: for $\phi\in\mathcal{S}(\setR^d)$,
\begin{align*}
\skp{g_1}{\phi}:=\skp{\abs{x}^{1-\delta}\widehat{x}_1}{\phi}=\displaystyle\int_{\setR^d}\abs{x}^{-\delta}x_1\phi(x)\,dx,
\end{align*}
and
\begin{align*}
\skp{g_2}{\phi}:=\skp{\abs{x}^{-d-2s}}{\phi}=\displaystyle\int_{\setR^d}\dfrac{\abs{x}^{-d-2s+2}\Delta\phi(x)}{2s(d+2s-2)}\,dx.
\end{align*}
One can check the second identity by Gauss-Green formula and $2-2s>0$.
	\item[(2)] We will show that $g_1\ast g_2$ is well-defined in the tempered distributional sense of \cite[Section 6.5]{Vla02} (See Lemma \ref{lem:welldef} and Lemma \ref{lem:conv10}. Note that the convolution is denoted by $\divideontimes$ in there).
	\item[(3)] Moreover, in Lemma \ref{lem:A20} we show $f_1(d,s,\delta)=\lim_{\sigma\rightarrow 0}\skp{g_1\ast g_2}{\phi_{\sigma}}$.
	\item[(4)] By (2), the relation $\mathcal{F}\left[\mathcal{F}[g_1\ast g_2]\right]=\mathcal{F}\left[\mathcal{F}[g_1]\cdot\mathcal{F}[g_2]\right]$ is available (see \cite[Section 6.5]{Vla02}), where the definition of product of distributions `$\cdot$' is from \cite[Section 6.5]{Vla02}. The computation of $\mathcal{F}[g_1]$ and $\mathcal{F}[g_2]$ is as follows:
\begin{align*}
\mathcal{F}[g_1](\xi)=\dfrac{1}{2-\delta}\mathcal{F}[\partial_1\abs{x}^{2-\delta}]&=\dfrac{2\pi i\xi_1}{2-\delta}\mathcal{F}[\abs{x}^{2-\delta}]\\
&=\dfrac{2\pi i\xi_1}{2-\delta}\dfrac{(2\pi)^{\delta-2}\Gamma\left(\frac{d}{2}-\frac{\delta}{2}+1\right)}{\pi^{\frac{d}{2}}2^{\delta-2}\Gamma\left(\frac{\delta}{2}-1\right)}\abs{\xi}^{-d+\delta-2}\\
&=-\pi^{-\frac{d}{2}+\delta-1}i\dfrac{\Gamma\left(1+\frac{d}{2}-\frac{\delta}{2}\right)}{\Gamma\left(\frac{\delta}{2}\right)}\abs{\xi}^{-d+\delta-2}\xi_1
\end{align*}
and
\begin{align*}
\mathcal{F}[g_2](\xi)=\pi^{\frac{d}{2}+2s}\dfrac{\Gamma(-s)}{\Gamma\left(\frac{d}{2}+s\right)}\abs{\xi}^{2s}.
\end{align*}
\item[(5)] Furthermore, Lemma \ref{lem:Fourier000} implies that $\mathcal{F}[g_1]\cdot\mathcal{F}[g_2]=\mathcal{F}[g_1]\mathcal{F}[g_2]$ in the distributional sense, where $\mathcal{F}[g_1]\mathcal{F}[g_2]$ means the usual multiplication of functions $\mathcal{F}[g_1]$ and $\mathcal{F}[g_2]$.
\item[(6)] We compute $g_1\ast g_2$ via Fourier transform as follows:
\begin{align*}
&\mathcal{F}[\mathcal{F}[g_1]\mathcal{F}[g_2]]\\
&=-\pi^{\delta-1+2s}i\dfrac{\Gamma(-s)\Gamma\left(\frac{d}{2}-\frac{\delta}{2}+1\right)}{\Gamma\left(\frac{d}{2}+s\right)\Gamma\left(\frac{\delta}{2}\right)}\dfrac{1}{-d+\delta+2s}\mathcal{F}[\partial_1\abs{\xi}^{-d+\delta+2s}]\\
&=-\pi^{\delta-1+2s}i\dfrac{\Gamma(-s)\Gamma\left(\frac{d}{2}-\frac\delta2+1\right)}{\Gamma\left(\frac d2+s\right)\Gamma\left(\frac{\delta}{2}\right)}\dfrac{2\pi i x_1}{-d+\delta+2s}\dfrac{(2\pi)^{d-\delta-2s}\Gamma\left(\frac{\delta}{2}+s\right)}{\pi^{\frac{d}{2}}2^{d-\delta-2s}\Gamma\left(\frac d2-\frac\delta2-s\right)}\abs{x}^{-\delta-2s}\\
&=\frac{2^{2s-1}}{\kappa_{d,s}}\dfrac{\Gamma\left(1-\frac\delta2+\frac d2\right)\Gamma\left(\frac{\delta}{2}+s\right)}{\Gamma\left(1-\frac\delta2+\frac d2-s\right)\Gamma\left(1+\frac{\delta}{2}\right)}\dfrac{\delta}{2}\abs{x}^{-\delta-2s}x_1.
\end{align*}
\end{enumerate} 	
Also, by the above computation, we observe that $\mathcal{F}[\mathcal{F}[g_1]\mathcal{F}[g_2]]$ coincides with a locally integrable function in $\setR^d$. Overall, with $\overline{\phi}_{\sigma}(x)=\phi_{\sigma}(-x)$, we have
\begin{align}\label{eq:appen0}
\begin{split}
f_1(d,s,\delta)&\underset{\text{(3)}}{=}\lim_{\sigma\rightarrow 0}\skp{g_1\ast g_2}{\phi_{\sigma}}\underset{\text{(4)}}{=}\lim_{\sigma\rightarrow 0}\skp{\mathcal{F}[\mathcal{F}[g_1]\cdot\mathcal{F}[g_2]]}{\overline{\phi}_{\sigma}}\\
&\underset{\text{(5)}}{=}\lim_{\sigma\rightarrow 0}\skp{\mathcal{F}[\mathcal{F}[g_1]\mathcal{F}[g_2]]}{\overline{\phi}_{\sigma}}\\
&\underset{\text{(6)}}{=}-\frac{2^{2s-1}}{\kappa_{d,s}}\dfrac{\Gamma\left(1-\frac{\delta}{2}+\frac d2\right)\Gamma\left(\frac{\delta}{2}+s\right)}{\Gamma\left(1-\frac\delta2+\frac d2-s\right)\Gamma\left(1+\frac{\delta}{2}\right)}\dfrac{\delta}{2}.
\end{split}
\end{align}

The proof is complete.
\end{proof}

By Remark \ref{rem:prop} and Lemma \ref{lem:f} with \eqref{eq:frac.lap}, we have the following corollary.
\begin{corollary}[Precise formula for $(-\Delta_{\bbAse})^s u_{\delta}(x)$]\label{cor:frac.lap}
For each $x\in\setR^d$, there holds
\begin{align*}
&(-\Delta_{\bbAse})^s u_{\delta}(x)\\
&\quad=\dfrac{2^{2s}\widetilde{c}_{d,s,\delta}}{4}\bigg[(d-\delta)\delta\bigg(1-\frac{1+2s}{2}\epsilon\bigg)\\
&\quad\quad\quad\quad\quad\quad-\frac{d+2s}{2}\bigg(2sf_{2,1}(d,s,\delta)-\frac{\widetilde{c}_{d,s,0}}{\widetilde{c}_{d,s,\delta}}\dfrac{s(d-1)(d-2s)}{d+2s}\bigg)\epsilon\bigg]|x|^{1-2s-\delta}\widehat{x}_1,
\end{align*}
where 
\begin{align*}
\widetilde{c}_{d,s,\delta}=\frac{\Gamma\left(\frac d2-\frac{\delta}{2}\right)\Gamma\left(\frac{\delta}{2}+s\right)}{\Gamma\left(\frac d2-s+1-\frac{\delta}{2}\right)\Gamma\left(1+\frac{\delta}{2}\right)}
\end{align*}
and $f_{2,1}(d,s,\delta)$ is defined in \eqref{eq:f21}. Moreover, $(-\Delta_{\bbAse})^s u_{\delta}(x)=0$ if and only if
\begin{align*}
(d-\delta)\delta\left(1-\frac{1+2s}{2}\epsilon\right)=\frac{d+2s}{2}\left(2sf_{2,1}(d,s,\delta)-\frac{\widetilde{c}_{d,s,0}}{\widetilde{c}_{d,s,\delta}}\dfrac{s(d-1)(d-2s)}{d+2s}\right)\epsilon,
\end{align*}
which is equivalent to
\begin{align}\label{eq:b}
\begin{split}
\epsilon&=\dfrac{2(d-2s+2-\delta)\delta}{(d-1)s(d-2s+2)-2(d-1)\dfrac{\Gamma\left(\frac{d}{2}\right)\Gamma\left(s+1\right)\Gamma\left(\frac{d}{2}-s-\frac{\delta}{2}+2\right)\Gamma\left(1+\frac{\delta}{2}\right)}{\Gamma\left(\frac d2-s\right)\Gamma\left(\frac d2-\frac\delta2+1\right)\Gamma\left(\frac{\delta}{2}+s\right)}}\eqqcolon b(\delta).
\end{split}
\end{align}
For $s \nearrow 1$, we have $\epsilon \to \frac{d-\delta}{d-1}\delta$ in agreement with~\eqref{eq:meyer-epsilon}.
\end{corollary}

Now we can show Theorem \ref{thm:s1}.

\begin{proof}[Proof of Theorem \ref{thm:s1}]
Let $d\geq 2$, $s\in(0,1)$, $\epsilon\in[0,\frac{1}{2}]$, $\delta\in[0,\frac{1}{2}]$, and fix $x\in\setR^d$. Note that for any $\delta\in[0,\frac{1}{2}]$, we have $\lim_{s\nearrow1} \widetilde{c}_{d,s,\delta}=1$. Then by Corollary \ref{cor:frac.lap},
\begin{align*}
\lim_{s\nearrow 1}(-\Delta_{\bbAse})^s u_{\delta}(x)&=\left[(d-\delta)\delta\left(1-\frac{3}{2}\epsilon\right)\right.\\
&\left.\quad\quad\quad-\frac{d+2}{2}\left(2f_{2,1}(d,1,\delta)-\dfrac{(d-1)(d-2)}{d+2}\right)\epsilon\right]|x|^{-1-\delta}\widehat{x}_1
\end{align*}
holds. Here,
\begin{align*}
f_{2,1}(d,1,\delta)=\dfrac{3\delta^2-3d\delta+d^2-d}{2(d+2)}
\end{align*}
so that
\begin{align*}
2f_{2,1}(d,1,\delta)-\dfrac{(d-1)(d-2)}{d+2}=\dfrac{3\delta^2-3d\delta+2d-2}{d+2}.
\end{align*}
Thus $\lim_{s\nearrow 1}(-\Delta_{\bbAse})^s u_{\delta}(x)=\left[(d-\delta)\delta-(d-1)\epsilon\right]|x|^{-1-\delta}\widehat{x}_1=(-\Delta_{\bbA_{\epsilon}}) u_{\delta}(x)$.
\end{proof}

We now prove a lemma, which is crucial for the proof of Theorem \ref{thm:Meyers}. 
\begin{lemma}\label{lem:b}
For $\epsilon=b(\delta)$ in \eqref{eq:b}, we have the following:
\begin{itemize}	
\item[(a)] There exists a number $\delta_0=\delta_0(d,s)\in(0,\frac{1}{2}]$ such that the function $b(\delta)$ defined in \eqref{eq:b} is an one-to-one map from $[0,\delta_0]$ to $[0,\frac{1}{2}]$.
\item[(b)] Moreover, in the above range, $0< b(\delta)\leq c_1\delta$ with some $0<c_1=c_1(d,s)$.
\end{itemize} 
\end{lemma}
\begin{proof}
  \textbf{Proof of (a).} First, note that $b(0)=0$. Furthermore, observe that if we define the following denominator in \eqref{eq:b} as
  \begin{align*}
    b_1(\delta)=(d-1)s(d-2s+2)-2(d-1)\dfrac{\Gamma\left(\frac{d}{2}\right)\Gamma\left(s+1\right)\Gamma\left(2+\frac d2 - s - \frac\delta2\right)\Gamma\left(1+\frac{\delta}{2}\right)}{\Gamma\left(\frac d2 -s\right)\Gamma\left(1+\frac d2 - \frac \delta2\right)\Gamma\left(s+\frac{\delta}{2}\right)},
  \end{align*}
  then
  \begin{align}\label{eq:b1}
    \begin{split}
      b_1(0)&=(d-1)s(d-2s+2)-2(d-1)\dfrac{\Gamma\left(\frac{d}{2}\right)\Gamma\left(s+1\right)\Gamma\left(2+\frac d2 - s \right)}{\Gamma\left(\frac d2-s\right)\Gamma\left(1+\frac d2\right)\Gamma\left(s\right)}\\
      &=\frac{2s^2}{d}(d-1)(d-2s+2)>0.
    \end{split}
  \end{align}
  For $\delta \in [0,\frac 12]$ we define the continuous function
  \begin{align}\label{eq:b2}
    L(\delta):=\dfrac{\Gamma\left(2+\frac d2 -s - \frac \delta2\right)\Gamma\left(1+\frac{\delta}{2}\right)}{\Gamma\left(1+\frac d2 - \frac\delta 2\right)\Gamma\left(s+\frac{\delta}{2}\right)}.
  \end{align}
  We claim that~$L$ is strictly increasing in~$\delta$. For this we calculate
  \begin{align*}
    L'(\delta)=\tfrac{L(\delta)}{2}\left(-\psi\left(2+\tfrac d2 -s - \tfrac \delta2\right)+\psi\left(1+\tfrac{\delta}{2}\right)+\psi\left(1+\tfrac d2 -\tfrac\delta2\right)-\psi\left(s+\tfrac \delta2\right)\right)>0,
  \end{align*}
  where $\psi$ is the digamma function. Now, the strict concavity of~$\psi$ implies that $\psi(a)+\psi(d) < \psi(b)+\psi(c)$ for $a< b< c < d$ and $a+d=b+c$. This and the positivity of~$L$ implies that $L'(\delta)>0$. Hence, $L$ is strictly increasing.

  Now let $\delta_1\in(0,+\infty]$ be such that
  \begin{align}\label{eq:delta1}
    \begin{split}
      \delta_1=\inf\left\{\delta\in(0,+\infty]:(d-2s+2-\delta)\delta\geq 2b_1(\delta)\right\}.
    \end{split}
  \end{align}
  We consider two cases.

  \medskip

  \textit{Case 1. When $\delta_1\leq\frac{1}{2}$:} By \eqref{eq:b1}, \eqref{eq:b2}, and \eqref{eq:delta1}, we have $b(\delta_1)=\frac{1}{2}$, and $b(\delta)$ is an one-to-one map from $[0,\delta_1]$ to $[0,\frac{1}{2}]$. Thus we put $\delta_0=\delta_0(d,s)=\delta_1$.

  \medskip

  \textit{Case 2. When $\delta_1>\frac{1}{2}$:} We will show that this case cannot occur. First, note that since $s\mapsto\Gamma(s)\left[\Gamma\left(\frac{2s+\delta}{2}\right)\right]^{-1}$ is a decreasing function on $s\in(0,1)$, we observe
  \begin{align}\label{eq:gamma.s}
    R_0(s,\delta):=\dfrac{\Gamma(s)\Gamma\left(1+\frac{\delta}{2}\right)}{\Gamma\left(\frac{\delta}{2}+s\right)}\geq \dfrac{\Gamma(1)\Gamma\left(1+\frac{\delta}{2}\right)}{\Gamma\left(1+\frac{\delta}{2}\right)}=1.
  \end{align}
  On the other hand, by \cite[Theorem 10]{Alz97}, the following is a decreasing function:
  \begin{align*}
    d\mapsto R(d):=\dfrac{\Gamma\left(\frac{d}{2}\right)}{\Gamma\left(\frac d2-s+1\right)}\dfrac{\Gamma\left(\frac d2-s-\frac\delta2+2\right)}{\Gamma\left(\frac d2-\frac\delta2+1\right)}.
  \end{align*}
  Then by the asymptotic approximation $\Gamma(x+a)\sim \Gamma(x)x^{a}$ for $x\to\infty$, we obtain $R(d)\geq \lim_{d\to\infty} R(d)=1$. Then together with \eqref{eq:gamma.s}, we see that $R_0(s,\delta)R(d)\geq 1$, so that for any $\delta\in[0,\frac{1}{2}]$, there holds
  \begin{align}\label{eq:epsilon.b}
    \epsilon=b(\delta)\geq\frac{2(d-2s+2-\delta)\delta}{(d\!-\!1)s(d\!-\!2s+2)-(d\!-\!1)(d\!-\!2s)s}=\frac{(d-2s+2-\delta)\delta}{s(d-1)}\geq \delta.
  \end{align}
  Since \eqref{eq:delta1} implies $\epsilon=b(\delta_1)=\frac{1}{2}$, there is a contradiction between \eqref{eq:epsilon.b} and $\delta_1>\frac{1}{2}$. Thus $\delta_1>\frac{1}{2}$ is not possible.

  \textbf{Proof of (b).} Let $\overline{\delta}_0=\min\{\delta_0,\frac{1}{2}\}$. To prove (b), it remains to show that $b'(\delta)\leq c_1$ for any $\delta\in[0,\overline{\delta}_0]$ with some $c_1=c_1(d,s)$. To do this, we see that
  \begin{align}\label{eq:b6}
    b'(\delta)=\dfrac{2(d-2s+2-2\delta)b_1(\delta)-2\delta(d-2s+2-\delta) b'_1(\delta)}{(b_1(\delta))^2}
  \end{align}
  holds. Now for $\delta\in[0,\overline{\delta}_0]$, we claim that
  \begin{align}\label{eq:b7}
    b_1(\delta)\geq c(d,s)>0.
  \end{align}
  Indeed, when $\delta\in[0,\overline{\delta}_0]$, by \eqref{eq:b1}, \eqref{eq:b2} and \eqref{eq:delta1}, there holds
  \begin{align*}
    b_1(\delta)\geq b_1(\overline{\delta}_0)\geq 2(d-2s+2-\overline{\delta}_0)\overline{\delta}_0(1+2s)=c(d,s)>0.
  \end{align*}

  Also, since $\delta\mapsto 2(d-2s+2-2\delta)b_1(\delta)-2\delta(d-2s+2-\delta) b'_1(\delta)$ is continuous on $\delta\in[0,\overline{\delta}_0]$, it is bounded above, so that
  \begin{align}\label{eq:b8}
    \abs{2(d-2s+2-2\delta)b_1(\delta)-2\delta(d-2s+2-\delta) b'_1(\delta)}\leq c(d,s)
  \end{align}
  for any $\delta\in[0,\overline{\delta}_0]$. Considering \eqref{eq:b7} and \eqref{eq:b8} with \eqref{eq:b6}, we see that $b'(\delta)\leq c_1$ for some $c_1=c_1(d,s)$. Therefore, by $0< b'(\delta)\leq c_1$, (b) is proved.
\end{proof}

Now we provide the proof of one of our main results, Theorem \ref{thm:Meyers}.
\begin{proof}[Proof of Theorem \ref{thm:Meyers}]
  Let $d\geq 2$, $s\in(0,1)$, $\epsilon\in[0,\frac{1}{2}]$, and $\delta\in[0,\frac{1}{2}]$. Since $s-\frac{d}{2}<1-\delta$ for the given range of $d,s$ and $\delta$, we see that $u\in W^{s,2}(B_1)$. Moreover,
  \begin{align*}
    u\in W^{t,q}_{\loc}(B_1(0))\,\,\text{for}\,\, 1-\delta>t-\tfrac{d}{q}\quad\text{and}\quad u\not\in W^{t,q}_{\loc}(B_1(0))\,\,\text{for}\,\, 1-\delta\leq t-\tfrac{d}{q}.
  \end{align*}
  In conclusion, (a) follows.

  Note that $\bbAse(x)=\left(1-\frac{1+2s}{2}\epsilon\right)\identity+\frac{d+2s}{2}\epsilon\widehat{x}\otimes\widehat{x}$ has 
  \begin{itemize}
	\item eigenvalue $1+\frac{d-1}{2}\epsilon$ with eigenvector $\widehat{x}$
	\item and eigenvalue $1-\frac{1+2s}{2}\epsilon$ with eigenspace $(\text{span}\{\widehat{x}\})^{\perp}$.
  \end{itemize}
  Hence for any $\epsilon\in[0,\frac{1}{2}]$, $\bbAse(x)$ is uniformly elliptic. 

  Also, $\bbAse(x)^{-1}=\frac{2}{2-(1+2s)\epsilon}\identity-\frac{2(d+2s)\epsilon}{(2-(1+2s)\epsilon)(2+(d-1)\epsilon)}\widehat{x}\otimes \widehat{x}$ has 
  \begin{itemize}
	\item eigenvalue $\frac{2}{2+(d-1)\epsilon}$ with eigenvector $\widehat{x}$
	\item and eigenvalue $\frac{2}{2-(1+2s)\epsilon}$ with eigenspace $(\text{span}\{\widehat{x}\})^{\perp}$.
  \end{itemize}
  Then for any $\epsilon\in[0,\frac{1}{2}]$, we have $ |\bbAse(x)||\bbAse(x)^{-1}|\leq \left(1+\tfrac{d-1}{2}\epsilon\right)\tfrac{2}{2-3\epsilon}\leq d+3$, where for the second inequality we evaluate $\epsilon=\frac{1}{2}$. We calculate
  \begin{align*}
    \log \bbAse(x)
    = \log\big(1-\tfrac{1+2s}{2}\epsilon\big)\identity + \log\big(1-\tfrac{(d+2s)\epsilon}{2+(d-1)\epsilon}\big) \hat x \otimes \hat x
  \end{align*}
similar to \cite[Example 24]{BDGN22}. This implies
  \eqref{eq:logAse}, so (b) is proved. Now (c) is from Lemma \ref{lem:b} together with Lemma \ref{eq:estimate-Deltas-ud}. The proof is completed.
\end{proof}

\section{Nonlocal equations with Riesz fractional derivatives}\label{sec:Riesz}

In this section we provide a Meyers-type example for a different fractional model. It is based on the fractional derivative $\nabla^s u$ defined by the Riesz potential as introduced in~\cite{ShiSpe15}. This approach allows to define a fractional $p$-Laplacian as done in~\cite{SchShiSpe18}. Moreover, the use of~$\nabla^s u$ allows to consider even concepts like quasi-convexity, see \cite{KreisbeckSchoenberger2022}.

We proceed as in~\cite{ShiSpe15}. Let $0<s<1$ and $1<p<\infty$. For $0 < \alpha < d$ let $I_\alpha u$ denote the Riesz potential defined by $ I_\alpha u = I_\alpha * u$, where
\begin{align*}
I_{\alpha}(x):=2\kappa_{d,\frac{-\alpha}{2}}\dfrac{1}{\abs{x}^{d-\alpha}}=\dfrac{2^{-\alpha}\Gamma\left(\frac{d-\alpha}{2}\right)}{\pi^{\frac{d}{2}}\Gamma\left(\frac{\alpha}{2}\right)}\dfrac{1}{\abs{x}^{d-\alpha}}.
\end{align*}
This definition extends to tempered distributions. Then the $s$-fractional gradient is defined by
\begin{align}
  \label{eq:nablas-vs-I}
  \nabla^s u = I_{1-s} (Du) = D (I_{1-s} u)
\end{align}
in the sense of tempered distributions. Now, the fractional Sobolev space $L^{s,p}(\RRn)$ is defined as the space with $u, \nabla^s u \in L^p(\RRn)$. This definition corresponds to the Triebel-Lizorkin space $F^s_{p,2}(\RRn)$. Thus, $L^{s,2}(\RRn) = W^{s,2}(\RRn)$, but $L^{s,p}(\RRn) \neq W^{s,p}(\RRn)$ if $p \neq 2$. 

We define the fractional partial derivatives as $\frac{\partial^s u}{\partial x_j^s} := \frac{\partial}{\partial x_j} (I_{1-s}  u)$. Then the fractional divergence $\divergence^sw$ is defined as $\divergence^s w = \sum_j \frac{\partial^s w_j}{\partial x_j^s}$ for a vector field~$w$. For $\epsilon \in (0,1)$ we define as in \cite{Me1} 
\begin{align}\label{eq:M}
  \bbA_{\epsilon}(x)\coloneqq(1-\epsilon)\identity+\epsilon\widehat{x}\otimes\widehat{x}.
\end{align}
By some basic properties of $\bbA_{\epsilon}$ in \cite[Example 24]{BDGN22}, there hold
\begin{center}
  $|\bbA_{\epsilon}(x)||\bbA_{\epsilon}(x)^{-1}|\leq (1-\epsilon)^{-1}$\quad and \quad $\norm{\log\bbM_{\epsilon}}_{L^{\infty}}\leq2\epsilon$.
\end{center}
We will now construct functions~$u_{s,\delta}$ such that
\begin{align}
   \label{eq:meyers-riesz}
  -\divergence^s(\bbA_{\epsilon}(x)\nabla^s u_{s, \delta}(x))=0.
\end{align}
The idea is simple. Using~\eqref{eq:nablas-vs-I} we formally have
\begin{align}
  -I_{1-s} \divergence\big(\bbA_{\epsilon} \nabla (I_{1-s} u_{s, \delta})\big)=0.
\end{align}
If $v := I_{1-s} u_{s,\delta}$, then this simplifies formally to
\begin{align}
  \label{eq:meyer-aux1}
  -\divergence\big(\bbA_{\epsilon} \nabla v \big)=0.
\end{align}
This is exactly the equation used by Meyers~\cite{Me1}. From~\cite[Example~24]{BDGN22} we know that $v(x) = \abs{x}^{1-\delta} \hat{x}_1$ is a solution of ~\eqref{eq:meyer-aux1} if $\epsilon = \frac{d-\delta}{d-1} \delta$. Thus, we only need to find a function $u_{s,\delta}$ such that $I_{1-s} u_{s,\delta} = v$. This can be formally done by
\begin{align*}
  u_{s,\delta} = \Delta I_2 u_{s,\delta} = \Delta I_{1+s} I_{1-s} u_{s,\delta} = \Delta I_{1+s} v.
\end{align*}
Note that the homogeneity of~$v$ is $1-\delta$. Thus the one of $u_{s,\delta}$ is $s-\delta$. The following theorem makes this idea more rigorous.
\begin{theorem}\label{thm:Meyers.r}
  Let $d\geq 2$, $d\in\setN$, and $s\in(0,1)$. Let $\delta\in(0,\frac{d}{2})$ and $\epsilon\in(0,1)$. Then the following holds:
  \begin{enumerate}
	\item $u_{s,\delta}(x) = \abs{x}^{s-\delta} \hat{x}_1\in W^{s,2}_{\loc}(B_1(0))$ satisfies
    \begin{align*}
      u_{s, \delta}\in W_{\loc}^{t,q}(B_1(0))\,\,\text{if and only if}\,\, s-\delta> t-\tfrac{d}{q}.
    \end{align*}
  \item Let $\epsilon=\frac{d-\delta}{d-1}\delta$ and let $\bbA_\epsilon$ be as in~\eqref{eq:M}. Then $u_{s,\delta}$ solves~\eqref{eq:meyers-riesz} in the sense of tempered distributions. Moreover, for all $\xi \in \mathcal{S}(\RRd)$ (Schwartz space),
\begin{align*}
\int_{\setR^d}\bbA_\epsilon \nabla^s u\cdot\nabla^s \xi\,dx &=0.
\end{align*}
\end{enumerate}
\end{theorem}

The proof of this theorem is given at the end of this section. 
\begin{remark}
We mention that when $d=2$, as $\delta$ goes to $\frac d2=1$, $u$ only satisfies $u\in W^{s,2}_{\loc}(B_1(0))$ and $u\notin W^{s+\oldepsilon,2+\oldepsilon}_{\loc}(B_1(0))$ for any fixed $\oldepsilon\in(0,1)$. This feature also appeared in the example given by Meyers \cite{Me1}.
\end{remark}

Together with 
\begin{align*}
I_{1-s}(x):=2\kappa_{d,\frac{-1+s}{2}}\dfrac{1}{\abs{x}^{d-1+s}}=\dfrac{2^{-1+s}\Gamma\left(\frac{d-1+s}{2}\right)}{\pi^{\frac{d}{2}}\Gamma\left(\frac{1-s}{2}\right)}\dfrac{1}{\abs{x}^{d-1+s}},
\end{align*}
where $\kappa_{d,s}$ is defined in \eqref{eq:kappa}, we define
\begin{align}\label{eq:I1-s}
\begin{split}
(I_{1-s}\ast u_{s,\delta})(x)&:=2\kappa_{d,\frac{-1+s}{2}}\pvint_{\setR^d}\dfrac{\abs{x-h}^{s-\delta}\widehat{x_1-h_1}}{\abs{h}^{d-1+s}}\,dh\\
&=2\kappa_{d,\frac{-1+s}{2}}\abs{x}^{1-\delta}\pvint_{\setR^d}\dfrac{\abs{e_1-h}^{s-\delta}\widehat{1-h_1}}{\abs{h}^{d-1+s}}\,dh,
\end{split}
\end{align}
following the definition in \cite[Definition 1.1]{ShiSpe15}. Note that the last expression in \eqref{eq:I1-s} is well-defined from Lemma \ref{lem:welldef.f3}. Then we define the following.
\begin{definition}\label{def:distributional}
Let $d\geq 2$, $d\in\setN$, $s\in(0,1)$, $\epsilon\in(0,1)$, and $\delta\in(0,\frac{d}{2})$. Let $v\in L^1_{\loc}(\setR^d)$ with $I_{1-s}\ast v$ and $\frac{\partial^s v}{\partial x^s_j}$ being well-defined. Then $v$ is a distributional solution to
\begin{align*}
-\divergence^s(\bbA_{\epsilon}(x)\nabla^s v(x))\coloneqq-\sum^{d}_{i,j=1}\dfrac{\partial^s}{\partial x^s_j}(\bbA_{\epsilon}(x))_{ij}\dfrac{\partial^s}{\partial x^s_i}v=0\quad\text{in }B_1(0)
\end{align*}
if for any $\phi\in C_{c}^{\infty}(B_1(0))$, we have
\begin{align*}
\int_{\setR^d}\bbA_{\epsilon}(x)\nabla^s v(x)\cdot\nabla^s \phi(x)\,dx=0.
\end{align*}
\end{definition}
Note that since $d\geq 2$, $d\in\setN$, $s\in(0,1)$, and $\delta\in[0,\frac{1}{2}]$, $u_{s,\delta}\in L^{1}_{\loc}(\setR^d)$.

Now we give necessary computations via Fourier transform.
\begin{lemma}\label{lem:Fourier.riesz}
For $x\in \setR^d$ and $\delta\in(0,\frac{d}{2})$, we have
\begin{align}\label{eq:m2}
\begin{split}
\bbM_{\epsilon}(x)^2\nabla^su_{s,\delta}(x)
&=c^*\abs{x}^{-\delta}\left[(1-\epsilon)^2e_1+(1-\delta-(1-\epsilon)^2)\widehat{x}_1\widehat{x}\right],
\end{split}
\end{align}
\begin{align}\label{eq:div0}
\begin{split}
&\divergence(\bbM_{\epsilon}(x)^2\nabla^su_{s,\delta}(x))\!=\!c^*\abs{x}^{-\delta-1}\left[-\delta(1\!-\!\delta)+(1\!-\!\delta\!-\!(1\!-\!\epsilon)^2)(d\!-\!1)\right]\widehat{x}_1,
\end{split}
\end{align}
and
\begin{align}\label{eq:I.div}
\begin{split}
&(I_{1-s}\ast\divergence(\bbM_{\epsilon}^2\nabla^su_{s,\delta}))(x)\\
&\quad=c^{**}\left[-\delta(1-\delta)+(1-\delta-(1-\epsilon)^2)(d-1)\right]\abs{x}^{-s-\delta}\widehat{x}_1
\end{split}
\end{align}
a.e. sense on $\RRn$ and in the sense of tempered distributions on $\RRn$, with 
\begin{align*}
c^*=2^{-1+s}\dfrac{\Gamma\left(\frac{d+s-\delta+1}{2}\right)\Gamma\left(\frac{\delta}{2}\right)}{\Gamma\left(\frac{d-\delta+2}{2}\right)\Gamma\left(\frac{-s+\delta+1}{2}\right)}\quad\text{and}\quad c^{**}=c^*2^{-1+s}\dfrac{\Gamma\left(\frac{d-\delta}{2}\right)\Gamma\left(\frac{s+\delta+1}{2}\right)}{\Gamma\left(\frac{d-s-\delta+1}{2}\right)\Gamma\left(\frac{\delta+2}{2}\right)}.
\end{align*}
\end{lemma}
\begin{proof}
From Theorem \ref{thm:Fourierconv4}, we can compute
\begin{align*}
\mathcal{F}[I_{1-s}\ast u_{s,\delta}](\xi)&=2\kappa_{d,\frac{-1+s}{2}}\mathcal{F}[|x|^{-d-s+1}](\xi)\mathcal{F}\left[\dfrac{1}{s-\delta+1}\partial_1|x|^{s-\delta+1}\right](\xi)\\
&=-ic^*\pi^{-\frac{d}{2}+\delta-1}\dfrac{\Gamma\left(\frac{d-\delta+2}{2}\right)}{\Gamma\left(\frac{\delta}{2}\right)}\abs{\xi}^{-d+\delta-2}\xi_1.
\end{align*}
Then with the help of Corollary \ref{cor:Fourierconv40}, we calculate
\begin{align*}
&(I_{1-s}\ast u_{s,\delta})(x)=(\mathcal{F}\circ\mathcal{F})[I_{1-s}\ast u_{s,\delta}](-x)\\
&\quad=-ic^*\pi^{-\frac{d}{2}+\delta-1}\dfrac{\Gamma\left(\frac{d-\delta+2}{2}\right)}{\Gamma\left(\frac{\delta}{2}\right)}\left(\dfrac{-2\pi i x_1}{-d+\delta}\pi^{\frac{d}{2}-\delta}\dfrac{\Gamma\left(\frac{\delta}{2}\right)}{\Gamma\left(\frac{d-\delta}{2}\right)}\abs{x}^{-\delta}\right)=c^*\abs{x}^{-\delta+1}\widehat{x}_1.
\end{align*}
Then $\nabla^su_{s,\delta}(x)=\nabla(I_{1-s}\ast u_{s,\delta})(x)=c^*|x|^{-\delta}(e_1-\delta\widehat{x}\widehat{x}_1)$ so \eqref{eq:m2} is proved. Then now it is straightforward to obtain \eqref{eq:div0}. By Theorem \ref{thm:Fourierconv5}, we compute
\begin{align*}
&\mathcal{F}\left[I_{1-s}\ast\divergence(\bbM_{\epsilon}^2\nabla^su_{s,\delta})\right](\xi)\\
&=\underbrace{2c^*\kappa_{d,\frac{-1+s}{2}}\left[-\delta(1\!-\!\delta)\!+\!(1\!-\!\delta\!-\!(1\!-\!\epsilon)^2)(d\!-\!1)\right]}_{=:c_0} \mathcal{F}[|y|^{-d-s+1}](\xi)\mathcal{F}\left[\tfrac{\partial_1|x|^{-\delta}}{-\delta}\right](\xi)\\
&=-c_0i\pi^{s+\delta}\dfrac{\Gamma\left(\frac{-s+1}{2}\right)\Gamma(\frac{d-\delta}{2})}{\Gamma\left(\frac{d+s-1}{2}\right)\Gamma\left(\frac{\delta+2}{2}\right)}\abs{\xi}^{-d+s+\delta-1}\xi_1.
\end{align*}
Then together with Corollary \ref{cor:Fourierconv50}, we have
\begin{align*}
&(I_{1-s}\ast\divergence(\bbM_{\epsilon}^2\nabla^su_{s,\delta}))(x)\\
&=(\mathcal{F}\circ\mathcal{F})[I_{1-s}\ast\divergence(\bbM_{\epsilon}^2\nabla^su_{s,\delta})](-x)\\
&=-c_0i\pi^{s+\delta}\dfrac{\Gamma\left(\frac{-s+1}{2}\right)\Gamma(\frac{d-\delta}{2})}{\Gamma\left(\frac{d+s-1}{2}\right)\Gamma\left(\frac{\delta+2}{2}\right)}\pi^{\frac{d}{2}-s-\delta}\dfrac{\Gamma\left(\frac{s+\delta+1}{2}\right)}{\Gamma\left(\frac{d-s-\delta+1}{2}\right)}\abs{x}^{-s-\delta-1}ix_1\\
&=\left[-\delta(1-\delta)+(1-\delta-(1-\epsilon)^2)(d-1)\right]c^{**}\abs{x}^{-s-\delta}\widehat{x}_1
\end{align*}
so \eqref{eq:I.div} is proved. 
\end{proof}

Note that $\bbM_{\epsilon}(x)^2\nabla^s \widetilde{u}_{s,\delta}(x)$ is $-\delta$-homogeneous and $I_{1-s}$ is $(1-s)$-homogeneous, so $\abs{(I_{1-s}\ast(\bbM_{\epsilon}^2\nabla^s \widetilde{u}_{s,\delta}))(x)}=\infty$ is possible for $x\in\setR^d$ and so it may not be well-defined. Instead,  $(I_{1-s}\ast\divergence(\bbM_{\epsilon}^2\nabla^su_{s,\delta}))(x)$ is well-defined as a locally integrable function by Lemma \ref{lem:welldef.f4}. Thus we prove the following.

\begin{lemma}\label{eq:equiv}
  Let $\epsilon \in (0,1)$ and $\delta\in(0,\frac{d}{2})$. For any $\phi\in \mathcal{S}(\setR^d)$, we have
  \begin{align*}
  \int_{\setR^d}\bbM_{\epsilon}(x)^2\nabla^su_{s,\delta}(x)\cdot\nabla^s\phi(x)\,dx=\int_{\setR^d}(I_{1-s}\ast\divergence(\bbM_{\epsilon}^2\nabla^su_{s,\delta}))(x)\phi(x)\,dx.
  \end{align*}
\end{lemma}
\begin{proof}
Let $\sigma\in(0,1)$ be given. From $\phi\in \mathcal{S}(\setR^d)$, we know $\abs{\nabla^s\phi(x)}\lesssim 1$ when $x\in B_{\sigma}(0)$. Then together with \eqref{eq:m2} we show that
\begin{align*}
&\biggabs{\int_{\abs{x}\leq\sigma}\bbM_{\sigma}(x)^2\nabla^su_{s,\delta}(x)\cdot\nabla^s\phi(x)\,dx}\lesssim \int_{\abs{x}\leq\sigma}\abs{x}^{-\delta}\,dx\lesssim \sigma^{d/2}.
\end{align*}
Also, we have $\abs{(I_{1-s}\ast\phi)(x)}\lesssim 1$ when $x\in B_{\sigma}(0)$. Then together with \eqref{eq:div0},
\begin{align*}
&\biggabs{\int_{\abs{x}\leq\sigma}\divergence(\bbM_{\sigma}^2\nabla^su_{s,\delta})(x)(I_{1-s}\ast\phi)(x)\,dx}\lesssim \int_{\abs{x}\leq\sigma}\abs{x}^{-1-\delta}\,dx\lesssim \sigma^{d-1-\delta}
\end{align*}
holds. Furthermore, from $\phi\in \mathcal{S}(\setR^d)$, we know $\abs{\nabla^s\phi(x)}\lesssim \abs{x}^{-d-s}$ when $x\in B_{\frac{1}{\sigma}}(0)$. Then together with \eqref{eq:m2} we observe that
\begin{align*}
&\biggabs{\int_{\abs{x}\geq\frac{1}{\sigma}}\bbM_{\sigma}(x)^2\nabla^su_{s,\delta}(x)\cdot\nabla^s\phi(x)\,dx}\lesssim \int_{\abs{x}\geq\frac{1}{\sigma}}\abs{x}^{-d-s}\,dx\lesssim \sigma^{s}.
\end{align*}
We have $\abs{(I_{1-s}\ast\phi)(x)}\lesssim \abs{x}^{-d-s+1}$ when $x\in B_{\frac{1}{\sigma}}(0)$. Then with \eqref{eq:div0},
\begin{align*}
&\biggabs{\int_{\abs{x}\geq\frac{1}{\sigma}}\divergence(\bbM_{\sigma}^2\nabla^su_{s,\delta})(x)(I_{1-s}\ast\phi)(x)\,dx}\lesssim \int_{\abs{x}\geq\frac{1}{\sigma}}\abs{x}^{-1}\abs{x}^{-d-s+1}\,dx\lesssim \sigma^{s}
\end{align*}
holds. Now by divergence theorem, we obtain
 \begin{align*}
  &\int_{\sigma\leq \abs{x}\leq\frac{1}{\sigma}}\bbM_{\sigma}(x)^2\nabla^su_{s,\delta}(x)\cdot\nabla^s\phi(x)\,dx\\
&=-\int_{\sigma\leq \abs{x}\leq\frac{1}{\sigma}}\divergence\left(\bbM_{\sigma}(x)^2\nabla^su_{s,\delta}(x)\right)(I_{1-s}\ast\phi)(x)\,dx\\
&\quad-\int_{\{\abs{x}=\sigma\}\cup\{\abs{x}=\frac{1}{\sigma}\}}\bbM_{\sigma}(x)^2\nabla^su_{s,\delta}(x)\cdot\frac{x}{\abs{x}}(I_{1-s}\ast\phi)(x)\,dx.
  \end{align*}
Here, using \eqref{eq:m2}, note that
\begin{align*}
\int_{\abs{x}=\sigma}\bbM_{\sigma}(x)^2\nabla^su_{s,\delta}(x)\cdot\frac{x}{\abs{x}}(I_{1-s}\ast\phi)(x)\,dx\lesssim \int_{\abs{x}=\sigma}\abs{x}^{-\delta}\,dx\lesssim\sigma^{d-1-\delta}
\end{align*}
which is from $\abs{(I_{1-s}\ast\phi)(x)}\lesssim 1$ when $x\in B_{\sigma}(0)$, and
\begin{align*}
\int_{\abs{x}=\frac{1}{\sigma}}\bbM_{\sigma}(x)^2\nabla^su_{s,\delta}(x)\cdot\frac{x}{\abs{x}}(I_{1-s}\ast\phi)(x)\,dx\lesssim \int_{\abs{x}=\frac{1}{\sigma}}\abs{x}^{-d-s+1}\,dx\lesssim\sigma^{s},
\end{align*}
which is from $\abs{(I_{1-s}\ast\phi)(x)}\lesssim \abs{x}^{-d-s+1}$ when $x\in B_{\frac{1}{\sigma}}(0)$.

Considering the above seven observations, sending $\sigma\rightarrow 0$, we get
\begin{align*}
  \int_{\setR^d}\bbM_{\sigma}(x)^2\nabla^su_{s,\delta}(x)\cdot\nabla^s\phi(x)\,dx&=-\int_{\setR^d}\divergence\left(\bbM_{\sigma}(x)^2\nabla^su_{s,\delta}(x)\right)(I_{1-s}\ast\phi)(x)\,dx.
  \end{align*}
Now using the fact that $I_{1-s}$ is self-adjoint (or simply using change of variables),
\begin{align*}
\int_{\setR^d}\hspace{-1mm}\divergence\left(\bbM_{\sigma}(x)^2\nabla^su_{s,\delta}(x)\right)(I_{1-s}\ast\phi)(x)dx\!=\!\int_{\setR^d}\hspace{-1mm}\left(I_{1-s}\ast\divergence\left(\bbM_{\sigma}^2\nabla^su_{s,\delta}\right)\right)(x)\phi(x)dx
\end{align*}
is obtained. The conclusion holds.
\end{proof}

Now we prove the main goal of this section.

\begin{proof}[Proof of Theorem \ref{thm:Meyers.r}]
Similar to the proof of Theorem \ref{thm:Meyers}, with  \cite[Example 24]{BDGN22}, (a) can be proved. For (b), by Lemma \ref{eq:equiv}, Lemma \ref{lem:Fourier} together with Lemma \ref{lem:welldef.f4}, it suffices to find the relation between $\epsilon$ and $\delta$ which make $u_{s,\delta}$ to solve $I_{1-s}\ast\divergence (\bbM_{\epsilon}(x)^2\nabla^s u_{s,\delta}(x))=0$ pointwisely. Then to get $I_{1-s}\ast\divergence (\bbM_{\epsilon}(x)^2\nabla^s u_{s,\delta}(x))=0$ in a.e. $x\in B_{1}(0)$ we need
\begin{align*}
-\delta(1-\delta)+(1-\delta-(1-\epsilon)^2)(d-1)=0\quad\iff\quad \epsilon=1-\sqrt{1-\delta-\frac{\delta(1-\delta)}{d-1}}
\end{align*}
from \eqref{eq:I.div}. Then we have $\epsilon\leq C(d,s)\delta$ for some $C(d)$. We complete the proof.
\end{proof}

We can also obtain the following robustness result as $s\nearrow 1$.

\begin{theorem}\label{thm:s2}
Let $d\geq 2$, $d\in\setN$, and $s\in(0,1)$. For $x\in\setR^d\setminus\{0\}$, we have
\begin{align*}
\lim_{s\nearrow 1}I_{1-s}\ast\divergence(\bbM_{\epsilon}^2\nabla(I_{1-s}\ast u_{s,\delta}))(x)=\divergence(\bbM_{\epsilon}(x)^2\nabla u_{\delta}(x))
\end{align*}
for any $\epsilon\in(0,1)$ and $\delta\in(0,\frac{d}{2})$.
\end{theorem}

\begin{proof}
For each $x\in\setR^d\setminus\{0\}$, using \eqref{eq:I.div} in Lemma \ref{lem:Fourier} and \cite[Example 24]{BDGN22}, the above equality is equivalent to saying that $\lim_{s\nearrow 1}c^{**}\abs{x}^{-s-\delta}\widehat{x}_1=\abs{x}^{-1-\delta}\widehat{x}_1$, which clearly holds since $\lim_{s\nearrow 1}c^{**}=1$. The proof is complete.
\end{proof}

\appendix
\section{Computations via Fourier transform}\label{app}
In this appendix, we give details in the proof of Lemma \ref{lem:f} which provides the computation of $f_2(d,s,\delta)$ given in \eqref{eq:f2'}. First, we define the distributions $g_3$ and $g_4$ in the proof of Lemma \ref{lem:f}, where the (formal) evaluation $x=e_1$ to their convolution $g_3\ast g_4$ will coincide to $f_2(d,s,\delta)$.
\begin{definition}\label{def:f3f4}
For $\phi\in \mathcal{S}(\setR^d)$, we define the distribution $g_3$ and $g_4$ in $\setR^d$ as
\begin{align*}
g_3:=j_1-j_2\quad\text{and}\quad g_4:=-j_3+\tfrac{j_4}{2}+j_5-j_6+\tfrac{j_7}{2},
\end{align*}
where the distributions $j_1, j_2, \dots, j_7$ in $\setR^d$ are defined as
\begin{align*}
\skp{j_1}{\phi}:=\skp{\abs{h}^{-\delta-2}h_1}{\phi}=\displaystyle\int_{\setR^d}\abs{x}^{-\delta-2}x_1\phi(x)\,dx,
\end{align*}
\begin{align*}
\skp{j_2}{\phi}:=\skp{\abs{h}^{-2}}{\phi}=
\begin{cases}
-\displaystyle\int_{\setR^d}\dfrac{x\log\abs{x}}{\abs{x}^2}\nabla\phi(x)\,dx&\,\text{if }d=2,\\
\displaystyle\int_{\setR^d}\dfrac{\phi(x)}{\abs{x}^{2}}\,dx&\,\text{if }d\geq 3,
\end{cases}
\end{align*}
\begin{align*}
\skp{j_3}{\phi}&:=\skp{\abs{h}^{-d-2s-2}h_1^3}{\phi}=\int_{\setR^d}\left(\dfrac{-(-d-2s+2)x_1^2+2\abs{x}^{2}}{\abs{x}^{d+2s}(-d\!-\!2s\!+\!2)(-d\!-\!2s)}\right)\partial_1\phi(x)dx,
\end{align*}
\begin{align*}
\skp{j_4}{\phi}:=\skp{\abs{h}^{-d-2s}h_1^2}{\phi}=\int_{\setR^d}\abs{x}^{-d-2s}x_1^2\phi(x)\,dx,
\end{align*}
\begin{align*}
\skp{j_5}{\phi}:=\skp{\abs{h}^{-d-2s-2}h_1^2}{\phi}=\int_{\setR^d}\dfrac{\abs{x}^{-d-2s+2}(2s\partial_1^2\phi(x)-\Delta\phi(x))}{2s(-d-2s+2)(-d-2s)}\,dx,
\end{align*}
\begin{align*}
\skp{j_6}{\phi}:=\skp{\abs{h}^{-d-2s}h_1}{\phi}=\int_{\setR^d}\dfrac{-\abs{x}^{-d-2s+2}\partial_1\phi(x)}{-d-2s+2}\,dx,
\end{align*}
and 
\begin{align*}
\skp{j_7}{\phi}&:=\skp{\abs{x}^{-d-2s+2}}{\phi}=\int_{\setR^d}\abs{x}^{-d-2s+2}\phi(x)\,dx.
\end{align*}	
\end{definition}

One can easily check that the above distributions are well-defined since $\phi\in \mathcal{S}(\setR^d)$ implies $\abs{\partial^k_i\phi(x)}\lesssim \min\{1,\frac{1}{\abs{x}^3}\}$ for $k=0,1,2,3$ and $i=0,1,\dots,d$.

Now we obtain the precise computation of $\mathcal{F}[g_3]$ and $\mathcal{F}[g_4]$.
\begin{lemma}\label{lem:Fourier0}
We have 
\begin{align}\label{eq:Fourier}
\mathcal{F}[g_3](\xi)=\mathcal{G}_3(\xi)\quad \text{and} \quad \mathcal{F}[g_4](\xi)=\pi^{\frac{d}{2}+2s}\dfrac{\Gamma(-s+1)}{\Gamma\left(\frac{d+2s}{2}\right)}\mathcal{G}_4(\xi)	
\end{align}
in the distributional sense, where 
\begin{align*}
\mathcal{G}_3(\xi):=
\begin{cases}
-i\pi^{\delta+1-\frac{d}{2}}\dfrac{\Gamma\left(\frac{d-\delta}{2}\right)}{\Gamma\left(\frac{\delta+2}{2}\right)}|\xi|^{-d+\delta}\xi_1-2\pi\log|\xi|+2\pi(\log 2-\gamma)\,\,&\text{when }d=2\\
-i\pi^{\delta+1-\frac{d}{2}}\dfrac{\Gamma\left(\frac{d-\delta}{2}\right)}{\Gamma\left(\frac{\delta+2}{2}\right)}|\xi|^{-d+\delta}\xi_1+\pi^{2-\frac{d}{2}}\Gamma\left(\frac{d-2}{2}\right)|\xi|^{-d+2}\,\,&\text{when }d\neq 2
\end{cases}
\end{align*}
with $\gamma=-\Gamma'(1)$ the Euler-Mascheroni constant, and
\begin{align*}
\mathcal{G}_4(\xi)&:=\dfrac{-2i(1-s)}{\pi(d+2s)}|\xi|^{2s-4}\xi_1^3+\dfrac{s-1}{2\pi^2}|\xi|^{2s-4}\xi_1^2-\dfrac{2}{d+2s}|\xi|^{2s-2}\xi_1^2\\
&\quad+\dfrac{i(d+2s+3)}{\pi(d+2s)}|\xi|^{2s-2}\xi_1+\dfrac{d+2s-1}{4\pi^2}|\xi|^{2s-2}-\dfrac{1}{s(d+2s)}|\xi|^{2s}
\end{align*}
are locally integrable distributions.
\end{lemma}
\begin{proof}
For $g_3$, we compute
\begin{align*}
\mathcal{F}\left[|h|^{-\delta-2}h_1\right](\xi)&=\left(\dfrac{i}{2\pi}\right)\partial_1\mathcal{F}\left[|h|^{-\delta-2}\right](\xi)=-i\pi^{\delta+1-\frac{d}{2}}\dfrac{\Gamma\left(\frac{d-\delta}{2}\right)}{\Gamma\left(\frac{\delta+2}{2}\right)}|\xi|^{-d+\delta}\xi_1,
\end{align*}
and with the help of \cite[Remark 1.2]{CheWet19},
\begin{align*}
\mathcal{F}\left[|h|^{-2}\right]=
\begin{cases}
\pi^{2-\frac{d}{2}}\Gamma\left(\frac{d-2}{2}\right)|\xi|^{-d+2}&\quad d\neq 2\\
-2\pi\log|\xi|+2\pi(\log 2-\gamma)&\quad d=2
\end{cases}
\end{align*}
in the distributional sense. For $g_4$, we get
\begin{align*}
\mathcal{F}\left[\dfrac{-h_1^3}{|h|^{d+2s+2}}\right]&=-\left(\dfrac{i}{2\pi}\right)^3\partial_1^3\mathcal{F}\left[|h|^{-d-2s-2}\right]\\
&=i\pi^{\frac{d}{2}+2s-1}\dfrac{\Gamma\left(-s+1\right)}{\Gamma\left(\frac{d+2s}{2}\right)}\dfrac{1}{d+2s}\left(2(s-1)|\xi|^{2s-4}\xi_1^3+3|\xi|^{2s-2}\xi_1\right),
\end{align*}
\begin{align*}
\mathcal{F}\left[\dfrac{h_1^2}{2|h|^{d+2s}}\right]=\dfrac{1}{4}\pi^{\frac{d}{2}+2s-2}\dfrac{\Gamma\left(-s+1\right)}{\Gamma\left(\frac{d+2s}{2}\right)}\left(2(s-1)|\xi|^{2s-4}\xi_1^2+|\xi|^{2s-2}\right),
\end{align*}
\begin{align*}
\mathcal{F}\left[\dfrac{h_1^2}{|h|^{d+2s+2}}\right]=-\pi^{\frac{d}{2}+2s}\dfrac{\Gamma\left(-s+1\right)}{\Gamma\left(\frac{d+2s}{2}\right)}\frac{1}{d+2s}\left(2|\xi|^{2s-2}\xi_1^2+\frac{1}{s}|\xi|^{2s}\right),
\end{align*}
\begin{align*}
\mathcal{F}\left[\dfrac{-h_1}{|h|^{d+2s}}\right]=i\pi^{\frac{d}{2}+2s-1}\dfrac{\Gamma\left(-s+1\right)}{\Gamma\left(\frac{d+2s}{2}\right)}|\xi|^{2s-2}\xi_1,
\end{align*}
and
\begin{align*}
\mathcal{F}\left[\dfrac{1}{2|h|^{d+2s-2}}\right]=\pi^{\frac{d}{2}+2s-2}\dfrac{\Gamma(-s+1)}{\Gamma\left(\frac{d+2s}{2}\right)}\dfrac{d+2s-2}{4}|\xi|^{2s-2}.
\end{align*}
Merging the terms calculated above, we obtain \eqref{eq:Fourier}.
\end{proof}

Now we provide the formal evaluation $x=e_1$ of $g_3\divideontimes g_4$ using approximation to identity, which is precisely computed by Fourier transform. We denote by $\phi_{\sigma}$ an approximation to the identity at $e_1$, i.e., for $\sigma\in(0,1)$  $\phi_{\sigma}(x):=\frac{1}{\sigma^d}\phi\left(\frac{x-e_1}{\sigma}\right)$ for some standard mollifier $\phi \in\mathcal{S}(\setR^d)$. 

\begin{lemma}\label{lem:final}
We have
\begin{align*}
&f_{2}(d,s,\delta)= \lim_{\epsilon\rightarrow 0}\skp{g_3\divideontimes g_4}{\phi_{\epsilon}}=\lim_{\epsilon\rightarrow 0}\skp{\mathcal{F}[\mathcal{F}[g_3]\mathcal{F}[g_4]]}{\overline{\phi}_{\epsilon}}\\
&\!=\!\dfrac{\pi^{\frac{d}{2}}}{2}\dfrac{\Gamma(-s\!+\!1)}{\Gamma\left(\frac{d+2s}{2}\right)}\dfrac{\Gamma\left(\frac{d-\delta}{2}\right)}{\Gamma\left(\frac{\delta+2}{2}\right)}\dfrac{\Gamma\left(\frac{2s+\delta}{2}\right)}{\Gamma\left(\frac{d-2s+2-\delta}{2}\right)}f_{2,1}(d,s,\delta)\!-\!\dfrac{\pi^{\frac{d}{2}}}{2}\dfrac{\Gamma(-s\!+\!1)}{\Gamma\left(\frac{d+2s}{2}\right)}\dfrac{\Gamma\left(\frac{d}{2}\right)\Gamma\left(s\right)}{\Gamma\left(\frac{d-2s}{2}\right)}\dfrac{d\!-\!1}{d\!+\!2s}.
\end{align*}
\end{lemma}
\begin{proof}
The first equality is given in Corollary \ref{cor:Fourierconv2}, so we only need to prove the second and the last inequality. Since $\mathcal{F}\left[|h|^{-2}\right]$ in $\mathcal{F}\left[g_3\right]$ is computed in different ways when $d\geq 3$ and $d=2$, we divide the computation in two cases.

	\textbf{In case of $d\geq 3$:} By Theorem \ref{thm:Fourierconv}, we have
	\begin{align*}
		\mathcal{F}[g_3\divideontimes g_4](\xi)=\mathcal{F}[g_3](\xi)\mathcal{F}[g_4](\xi)=\frac{\pi^{2s}}{2}\dfrac{\Gamma(-s+1)}{\Gamma\left(\frac{d+2s}{2}\right)}\sum_{j=1}^{12}q_j(\xi),
	\end{align*}
	where denoting $c_2=\pi^{\delta}\Gamma\left(\frac{d-\delta}{2}\right)/\Gamma\left(\frac{\delta+2}{2}\right)$,
	\begin{align*}
		&q_{1}(\xi)\coloneqq \dfrac{-4(1-s)}{d+2s}c_2|\xi|^{-d+2s-4+\delta}\xi^4_1,\quad\quad\,\,
		q_{2}(\xi)\coloneqq-i\dfrac{s-1}{\pi}c_2|\xi|^{-d+2s-4+\delta}\xi^3_1,\\
		&q_{3}(\xi)\coloneqq i\dfrac{4\pi}{d+2s}c_2|\xi|^{-d+2s-2+\delta}\xi^3_1,\quad\quad\quad\,\,\,
		q_{4}(\xi)\coloneqq\dfrac{2d+4s+6}{d+2s}c_2|\xi|^{-d+2s-2+\delta}\xi^2_1,\\
		&q_{5}(\xi)\coloneqq-i\dfrac{d+2s-1}{2\pi}c_2|\xi|^{-d+2s-2+\delta}\xi_1,\quad
		q_{6}(\xi)\coloneqq i\dfrac{2\pi}{s(d+2s)}c_2|\xi|^{-d+2s+\delta}\xi_1,
	\end{align*}
	and denoting $c_3=\Gamma\left(\frac{d-2}{2}\right)$,
	\begin{align*}
		&q_{7}(\xi)\coloneqq-i\dfrac{4(s-1)}{d+2s}\pi c_3|\xi|^{-d+2s-2}\xi^3_1,\,\,\,\,\quad\quad
		q_{8}(\xi)\coloneqq-(s-1)c_3|\xi|^{-d+2s-2}\xi^2_1,\\
		&q_{9}(\xi)\coloneqq\dfrac{4\pi^2}{d+2s}c_3|\xi|^{-d+2s}\xi^2_1,\quad\quad\quad\,\,\,\,\,\quad\quad
		q_{10}(\xi)\coloneqq-i\dfrac{2d+4s+6}{d+2s}\pi c_3|\xi|^{-d+2s}\xi_1,\\
		&q_{11}(\xi)\coloneqq-\dfrac{d+2s-1}{2}c_3|\xi|^{-d+2s},\quad\quad\quad\quad
		q_{12}(\xi)\coloneqq\dfrac{2\pi^2}{s(d+2s)}c_3|\xi|^{-d+2s+2}.
	\end{align*}

Here, one can see that
\begin{align*}
&\mathcal{F}\left[|\xi|^{-d+2s-4+\delta}\xi^4_1\right](x)=\left(\frac{i}{2\pi}\right)^4\partial^4_1\left(\mathcal{F}\left[|\xi|^{-d+2s-4+\delta}\right](x)\right)\\
&=\dfrac{\pi^{\frac{d}{2}-2s+4-\delta}}{(2\pi)^4}\dfrac{\Gamma\left(\frac{2s-4+\delta}{2}\right)}{\Gamma\left(\frac{d-2s+4-\delta}{2}\right)}\partial^4_1(|x|^{-2s+4-\delta})\\
&=\dfrac{\pi^{\frac{d}{2}-2s-\delta}}{4}\dfrac{\Gamma\left(\frac{2s+\delta}{2}\right)}{\Gamma\left(\frac{d-2s+4-\delta}{2}\right)}\\
&\quad\quad\cdot\left[(2s+\delta)(2s+\delta+2)|x|^{-2s-\delta-4}x^4_1-6(2s+\delta)|x|^{-2s-\delta-2}x^2_1+3|x|^{-2s-\delta}\right],
\end{align*}
so observing that the last one is a locally integrable function in $\setR^d\setminus\{0\}$, we have
\begin{align*}
\mathcal{F}[q_1](-e_1)=-\pi^{\frac{d}{2}-2s}\dfrac{\Gamma\left(\frac{d-\delta}{2}\right)}{\Gamma\left(\frac{\delta+2}{2}\right)}\dfrac{\Gamma\left(\frac{2s+\delta}{2}\right)}{\Gamma\left(\frac{d-2s+4-\delta}{2}\right)}\dfrac{1-s}{d+2s}(2s+\delta-3)(2s+\delta-1).
\end{align*}
We also calculate $\mathcal{F}[q_2](-e_1)$, $\mathcal{F}[q_3](-e_1)$, and $\mathcal{F}[q_7](-e_1)$ as follows, after observing that the corresponding expression is a locally integrable function in $\setR^d\setminus\{0\}$:
	\begin{align*}
		&\mathcal{F}\left[|\xi|^{-d+2s-4+\delta}\xi^3_1\right](x)\\
		&\quad=-i\dfrac{\pi^{\frac{d}{2}-2s+1-\delta}}{2}\dfrac{\Gamma\left(\frac{2s+\delta}{2}\right)}{\Gamma\left(\frac{d-2s+4-\delta}{2}\right)}\left((-2s-\delta)|x|^{-2s-\delta-2}x^3_1+3|x|^{-2s-\delta}x_1\right),
	\end{align*}
	and so
	\begin{align*}
		\mathcal{F}[q_2](-e_1)=\pi^{\frac{d}{2}-2s}\dfrac{\Gamma\left(\frac{d-\delta}{2}\right)}{\Gamma\left(\frac{\delta+2}{2}\right)}\dfrac{\Gamma\left(\frac{2s+\delta}{2}\right)}{\Gamma\left(\frac{d-2s+4-\delta}{2}\right)}\dfrac{1-s}{2}\left(2s+\delta-3\right)
	\end{align*}
	holds. Similarly,
	\begin{align*}
		&\mathcal{F}\left[|\xi|^{-d+2s-2+\delta}\xi^3_1\right](x)\\
		&\quad=-i\dfrac{\pi^{\frac{d}{2}-2s-1-\delta}}{2}\dfrac{\Gamma\left(\frac{2s+2+\delta}{2}\right)}{\Gamma\left(\frac{d-2s+2-\delta}{2}\right)}\left((-2s-\delta-2)|x|^{-2s-\delta-4}x^3_1+3|x|^{-2s-\delta-2}x_1\right)
	\end{align*}
	and
\begin{align*}
\mathcal{F}[|\xi|^{-d+2s-2}\xi^3_1](x)\!=\!-i\dfrac{\pi^{\frac{d}{2}-2s-1}}{2}\dfrac{\Gamma\left(s+1\right)}{\Gamma\left(\frac{d-2s+2}{2}\right)}\left((-2s\!-\!2)|x|^{-2s-4}x^3_1\!+\!3|x|^{-2s-2}x_1\right)
\end{align*}
	so that we find
	\begin{align*}
		\mathcal{F}[q_3](-e_1)=\pi^{\frac{d}{2}-2s}\dfrac{\Gamma\left(\frac{d-\delta}{2}\right)}{\Gamma\left(\frac{\delta+2}{2}\right)}\dfrac{}{}\dfrac{\Gamma\left(\frac{2s+2+\delta}{2}\right)}{\Gamma\left(\frac{d-2s+2-\delta}{2}\right)}\dfrac{2\left(2s+\delta-1\right)}{d+2s}
	\end{align*}
	and
	\begin{align*}
		\mathcal{F}[q_7](-e_1)=\pi^{\frac{d}{2}-2s}\dfrac{\Gamma\left(\frac{d-2}{2}\right)\Gamma\left(s+1\right)}{\Gamma\left(\frac{d-2s+2}{2}\right)}\dfrac{2(1-s)(2s-1)}{d+2s}.
	\end{align*}
	Now we start to compute $\mathcal{F}[q_4](-e_1)$, $\mathcal{F}[q_8](-e_1)$ and $\mathcal{F}[q_9](-e_1)$. Since
\begin{align*}
\mathcal{F}[|\xi|^{-d+2s-2+\delta}\xi_1^2]\!=\!\frac{1}{2}\pi^{\frac{d}{2}-2s-\delta}\dfrac{\Gamma\left(\frac{2s+\delta}{2}\right)}{\Gamma\left(\frac{d-2s+2-\delta}{2}\right)}\left((-2s\!-\!\delta)|x|^{-2s-\delta-2}x_1^2\!+\!|x|^{-2s-\delta}\right),
\end{align*}
	\begin{align*}
		\mathcal{F}\left[|\xi|^{-d+2s-2}\xi_1^2\right]=\frac{1}{2}\pi^{\frac{d}{2}-2s}\dfrac{\Gamma\left(s\right)}{\Gamma\left(\frac{d-2s+2}{2}\right)}\left((-2s)|x|^{-2s-2}x_1^2+|x|^{-2s}\right),
	\end{align*}
	and
	\begin{align*}
		\mathcal{F}\left[|\xi|^{-d+2s}\xi_1^2\right]=\frac{1}{2}\pi^{\frac{d}{2}-2s-2}\dfrac{\Gamma\left(s+1\right)}{\Gamma\left(\frac{d-2s}{2}\right)}\left((-2s-2)|x|^{-2s-4}x_1^2+|x|^{-2s-2}\right)
	\end{align*}
	holds, we have
	\begin{align*}
		\mathcal{F}[q_4](-e_1)=-\pi^{\frac{d}{2}-2s}\dfrac{\Gamma\left(\frac{d-\delta}{2}\right)}{\Gamma\left(\frac{\delta+2}{2}\right)}\dfrac{\Gamma\left(\frac{2s+\delta}{2}\right)}{\Gamma\left(\frac{d-2s+2-\delta}{2}\right)}\dfrac{\left(2s+\delta-1\right)(d+2s+3)}{d+2s},
	\end{align*}
	\begin{align*}
		\mathcal{F}[q_8](-e_1)=-\pi^{\frac{d}{2}-2s}\dfrac{\Gamma\left(\frac{d-2}{2}\right)\Gamma\left(s\right)}{\Gamma\left(\frac{d-2s+2}{2}\right)}\frac{(s-1)(1-2s)}{2},
	\end{align*}
	and
	\begin{align*}
		\mathcal{F}[q_9](-e_1)=-\pi^{\frac{d}{2}-2s}\dfrac{\Gamma\left(\frac{d-2}{2}\right)\Gamma\left(s+1\right)}{\Gamma\left(\frac{d-2s}{2}\right)}\dfrac{2\left(2s+1\right)}{d+2s}.
	\end{align*}
	Calculations of $\mathcal{F}[q_5](-e_1)$, $\mathcal{F}[q_6](-e_1)$ and $\mathcal{F}[q_{10}](-e_1)$ are given as below. We see
	\begin{align*}
		\mathcal{F}\left[|\xi|^{-d+2s-2+\delta}\xi_1\right]=-i\pi^{\frac{d}{2}-2s+1-\delta}\dfrac{\Gamma\left(\frac{2s+\delta}{2}\right)}{\Gamma\left(\frac{d-2s+2-\delta}{2}\right)}\left(|x|^{-2s-\delta}x_1\right),
	\end{align*}
	\begin{align*}
		\mathcal{F}\left[|\xi|^{-d+2s+\delta}\xi_1\right]=-i\pi^{\frac{d}{2}-2s-\delta-1}\dfrac{\Gamma\left(\frac{2s+\delta+2}{2}\right)}{\Gamma\left(\frac{d-2s-\delta}{2}\right)}\left(|x|^{-2s-\delta-2}x_1\right),
	\end{align*}
	and
	\begin{align*}
		\mathcal{F}\left[|\xi|^{-d+2s}\xi_1\right]=-i\pi^{\frac{d}{2}-2s-1}\dfrac{\Gamma\left(s+1\right)}{\Gamma\left(\frac{d-2s}{2}\right)}\left(|x|^{-2s-2}x_1\right),
	\end{align*}
	so that we have
	\begin{align*}
		\mathcal{F}[q_5](-e_1)=\dfrac{d+2s-1}{2}\pi^{\frac{d}{2}-2s}\dfrac{\Gamma\left(\frac{d-\delta}{2}\right)}{\Gamma\left(\frac{\delta+2}{2}\right)}\dfrac{\Gamma\left(\frac{2s+\delta}{2}\right)}{\Gamma\left(\frac{d-2s+2-\delta}{2}\right)},
	\end{align*}
	\begin{align*}
		\mathcal{F}[q_6](-e_1)=-\dfrac{2}{s(d+2s)}\pi^{\frac{d}{2}-2s}\dfrac{\Gamma\left(\frac{d-\delta}{2}\right)}{\Gamma\left(\frac{\delta+2}{2}\right)}\dfrac{\Gamma\left(\frac{2s+\delta+2}{2}\right)}{\Gamma\left(\frac{d-2s-\delta}{2}\right)},
	\end{align*}
	and
	\begin{align*}
		\mathcal{F}[q_{10}](-e_1)=\pi^{\frac{d}{2}-2s}\dfrac{\Gamma\left(\frac{d-2}{2}\right)\Gamma\left(s+1\right)}{\Gamma\left(\frac{d-2s}{2}\right)}\dfrac{2d+4s+6}{d+2s}.
	\end{align*}
	Finally, we compute $\mathcal{F}[q_{11}](-e_1)$ and $\mathcal{F}[q_{12}](-e_1)$ in the following way. Since
	\begin{align*}
		\mathcal{F}\left[|\xi|^{-d+2s}\right]=\pi^{\frac{d}{2}-2s}\dfrac{\Gamma\left(s\right)}{\Gamma\left(\frac{d-2s}{2}\right)}|x|^{-2s}
	\end{align*}
	and
	\begin{align*}	\mathcal{F}\left[|\xi|^{-d+2s+2}\right]=\pi^{\frac{d}{2}-2s-2}\dfrac{\Gamma\left(s+1\right)}{\Gamma\left(\frac{d-2s-2}{2}\right)}|x|^{-2s-2},
	\end{align*}
	we have
	\begin{align*}
		\mathcal{F}[q_{11}](-e_1)=-\pi^{\frac{d}{2}-2s}\dfrac{\Gamma\left(\frac{d-2}{2}\right)\Gamma\left(s\right)}{\Gamma\left(\frac{d-2s}{2}\right)}\dfrac{d+2s-1}{2}
	\end{align*}
	and
	\begin{align*}
		\mathcal{F}[q_{12}](-e_1)=\pi^{\frac{d}{2}-2s}\dfrac{\Gamma\left(\frac{d-2}{2}\right)\Gamma\left(s+1\right)}{\Gamma\left(\frac{d-2s-2}{2}\right)}\dfrac{2}{s(d+2s)}.
	\end{align*}
Note that $\mathcal{F}[q_4](x), \mathcal{F}[q_5](x), \dots, \mathcal{F}[q_{12}](x)$ are locally integrable functions in $\setR^d\setminus\{0\}$ and so the above evaluations at $x=e_1$ are justified.

Taking for the standard mollifier $\overline{\phi}(x)$ and $\phi_{\epsilon}(x)=\frac{1}{\epsilon^d}\overline{\phi}\left(\frac{x-e_1}{\epsilon}\right)$ for $\epsilon\in(0,1)$,	
\begin{align*}
\lim_{\epsilon\rightarrow 0}\skp{f\divideontimes g}{\phi_{\epsilon}}=\lim_{\epsilon\rightarrow 0}\skp{\mathcal{F}[\mathcal{F}[q_3]\mathcal{F}[q_4]]}{\phi_{\epsilon}}=\mathcal{F}[\mathcal{F}[q_3]\mathcal{F}[q_4]](e_1),
\end{align*}
where for the first equality we have used Corollary \ref{cor:Fourierconv2} and for the last equality we see that $\mathcal{F}[\mathcal{F}[q_3]\mathcal{F}[q_4]]$ coincide to a function when $\abs{x}>0$.

Therefore, we obtain $f_2(d,s,\delta)$ in case of $d\geq 3$ as follows.
\begin{align*}
&f_2(d,s,\delta)=(\mathcal{F}\circ\mathcal{F})[q_3\divideontimes q_4](-e_1)\\
&=\dfrac{\pi^{\frac{d}{2}}}{2}\dfrac{\Gamma(-s+1)}{\Gamma\left(\frac{d+2s}{2}\right)}\left(\dfrac{\Gamma\left(\frac{d-\delta}{2}\right)}{\Gamma\left(\frac{\delta+2}{2}\right)}\dfrac{\Gamma\left(\frac{2s+\delta}{2}\right)}{\Gamma\left(\frac{d-2s+2-\delta}{2}\right)}f_{2,1}(d,s,\delta)-\dfrac{\Gamma\left(\frac{d-2}{2}\right)\Gamma\left(s\right)}{\Gamma\left(\frac{d-2s+2}{2}\right)}f_{2,2}(d,s)\right),
\end{align*}
where $f_{2,1}$ is given in \eqref{eq:f21}, and
\begin{align*}
f_{2,2}(d,s)&\!=\!-\dfrac{2s(1-s)(2s-1)}{d+2s}+\dfrac{(1-s)(2s-1)}{2}\\
&\quad\!+\!\dfrac{s(2s+1)(d-2s)}{d+2s}-\dfrac{s(d+2s+3)(d-2s)}{d+2s}\\
&\quad\!+\!\dfrac{(d+2s-1)(d-2s)}{4}-\dfrac{(d-2s-2)(d-2s)}{2(d+2s)}\!=\!\dfrac{(d-2)(d-1)(d-2s)}{4(d+2s)}.
\end{align*}
Thus using $t\Gamma(t)=\Gamma(t+1)$ for $t\in\setR$, we obtain \eqref{eq:d3} in case of $d\geq 3$.
	
\medskip
	\textbf{In case of $d=2$:} In this case, the computation of the term $f_{2,1}$ is same, but $f_{2,2}$ is different because of the log term. We will use the following: from
	\begin{align*}
	\mathcal{F}\left[|\xi|^t\xi_1^{k}\right]=\pi^{-t-\frac{d}{2}}\dfrac{\Gamma\left(\frac{d+t}{2}\right)}{\Gamma\left(-\frac{t}{2}\right)}\left(\frac{i}{2\pi}\right)^k\partial_1^k\left(|x|^{-d-t}\right)
	\end{align*}
	for $k\in\mathbb{N}\cup\{0\}$ and $t\notin \{0\}\cup 2\mathbb{N}$, by differentiating with respect to $t$ in the distributional sense and using $\psi=\frac{\Gamma'}{\Gamma}$, we have
	\begin{align}\label{eq:log}
	\begin{split}
	&\mathcal{F}\left[|\xi|^t(\log|\xi|)\xi^k_1\right]\\
	&\quad=-\pi^{-t-\frac{d}{2}}(\log\pi)\dfrac{\Gamma\left(\frac{d+t}{2}\right)}{\Gamma\left(-\frac{t}{2}\right)}\left(\frac{i}{2\pi}\right)^k\partial_1^k\left(|x|^{-d-t}\right)\\
	&\quad\quad+\dfrac{1}{2}\pi^{-t-\frac{d}{2}}\dfrac{\Gamma\left(\frac{d+t}{2}\right)}{\Gamma\left(-\frac{t}{2}\right)}\left(\psi\left(\frac{d+t}{2}\right)+\psi\left(-\frac{t}{2}\right)\right)\left(\frac{i}{2\pi}\right)^k\partial_1^k\left(|x|^{-d-t}\right)\\
	&\quad\quad-\pi^{-t-\frac{d}{2}}\dfrac{\Gamma\left(\frac{d+t}{2}\right)}{\Gamma\left(-\frac{t}{2}\right)}\left(\frac{i}{2\pi}\right)^k\partial_1^k\left(|x|^{-d-t}\log|x|\right).
	\end{split}
	\end{align}

	Using the above equality together with Theorem \ref{thm:Fourierconv}, we will compute
	\begin{align*}
	&2\pi^{2+2s}\dfrac{\Gamma(1-s)}{\Gamma\left(1+s\right)}\left(\log|\xi|+\gamma-\log 2\right)\\
	&\,\,\,\,\cdot \left(\dfrac{-2i(1-s)}{\pi(1+s)}|\xi|^{2s-4}\xi_1^3+\dfrac{s-1}{\pi^2}|\xi|^{2s-4}\xi_1^2-\dfrac{2}{1+s}|\xi|^{2s-2}\xi_1^2\right.\\
	&\quad\quad\left.+\dfrac{(2s+5)i}{\pi(1+s)}|\xi|^{2s-2}\xi_1\!+\!\dfrac{1\!+\!2s}{2\pi^2}|\xi|^{2s-2}\!-\!\dfrac{1}{s(1\!+\!s)}|\xi|^{2s}\right)\eqqcolon 2\pi^{2s}\dfrac{\Gamma(1\!-\!s)}{\Gamma\left(1\!+\!s\right)}\sum^{12}_{j=1}p_j(\xi),
	\end{align*}
	where
	\begin{align*}
		&p_1(\xi)\coloneqq\dfrac{-2\pi i(1-s)}{1+s}|\xi|^{2s-4}(\log|\xi|)\xi_1^3,\quad
		p_2(\xi)\coloneqq(s-1)|\xi|^{2s-4}(\log|\xi|)\xi_1^2\\
		&p_3(\xi)\coloneqq-\dfrac{2\pi^2}{1+s}|\xi|^{2s-2}(\log|\xi|)\xi_1^2,\quad\quad\quad\,\,
		p_4(\xi)\coloneqq\dfrac{(2s+5)\pi i}{1+s}|\xi|^{2s-2}(\log|\xi|)\xi_1\\
		&p_5(\xi)\coloneqq\dfrac{1+2s}{2}|\xi|^{2s-2}(\log|\xi|),\quad\quad\quad\quad\,\,\,
		p_6(\xi)\coloneqq-\dfrac{\pi^2}{s(1+s)}|\xi|^{2s}(\log|\xi|)
	\end{align*}
	and with $c_4=\gamma-\log 2$,
	\begin{align*}
		&p_7(\xi)\coloneqq-\pi c_4\dfrac{2i(1-s)}{1+s}|\xi|^{2s-4}\xi_1^3,\quad\quad
		p_8(\xi)\coloneqq c_4(s-1)|\xi|^{2s-4}\xi_1^2\\
		&p_9(\xi)\coloneqq-\pi^2 c_4\dfrac{2}{1+s}|\xi|^{2s-2}\xi_1^2,\quad\,\,\,\,\,\quad
		p_{10}(\xi)\coloneqq\pi c_4\dfrac{(2s+5)i}{1+s}|\xi|^{2s-2}\xi_1\\
		&p_{11}(\xi)\coloneqq c_4\dfrac{1+2s}{2}|\xi|^{2s-2},\quad\quad\quad\,\,\,\,\quad
		p_{12}(\xi)\coloneqq-\pi^2 c_4\dfrac{1}{s(1+s)}|\xi|^{2s}.
	\end{align*}

	For $p_1(\xi)$, use \eqref{eq:log} to obtain
	\begin{align*}
		\mathcal{F}[|\xi|^{2s-4}(\log|\xi|)\xi_1^3]&=\frac{i}{8}\pi^{-2s}(\log\pi)\dfrac{\Gamma\left(s-1\right)}{\Gamma\left(2-s\right)}\partial_1^3\left(|x|^{-2s+2}\right)\\
		&\quad-\frac{i}{16}\pi^{-2s}\dfrac{\Gamma(s-1)}{\Gamma(2-s)}(\psi(s-1)+\psi(2-s))\partial_1^3\left(|x|^{-2s+2}\right)\\
		&\quad+\frac{i}{8}\pi^{-2s}\dfrac{\Gamma\left(s-1\right)}{\Gamma\left(2-s\right)}\partial_1^3\left(|x|^{-2s+2}\log|x|\right).
	\end{align*}
	Here, since $\partial_1^3\left(|x|^{-2s+2}\right)=-8(s-1)s(s+1)|x|^{-2s-4}x_1^3+12(s-1)s|x|^{-2s-2}x_1$ and
\begin{align*}
\partial^3_1\left(|x|^{-2s+2}\log|x|\right)&=8(1-s^2)s|x|^{-2s-4}(\log|x|)x_1^3+4(3s^2-1)|x|^{-2s-4}x_1^3\\
&\quad-12(1-s)s|x|^{-2s-2}(\log|x|)x_1+6(1-2s)|x|^{-2s-2}x_1,
\end{align*}
	observing that each expression is a locally integrable function in $\setR^d\setminus\{0\}$, we have
	\begin{align*}
	\mathcal{F}[p_1](-e_1)&\!=\!-\dfrac{\pi^{-2s+1}}{2(1+s)}\dfrac{\Gamma(s)}{\Gamma(2\!-\!s)}\left\{(\log\pi)(s\!-\!1)2s(2s\!-\!1)\right.\\
	&\left.\quad\quad\quad\quad\quad\quad\quad\quad\quad-(\psi(s\!-\!1)+\psi(2\!-\!s))(s\!-\!1)s(2s\!-\!1)\!-\!6s^2\!+\!6s\!-\!1\right\}\\
	&=-\dfrac{\pi^{-2s+1}}{2(1\!-\!s^2)}\dfrac{\Gamma(s)}{\Gamma(1\!-\!s)}\left\{(\log\pi)(s\!-\!1)2s(2s\!-\!1)\right.\\
	&\left.\quad\quad\quad\quad\quad\quad\quad\quad\quad-(\psi(s)\!+\!\psi(1\!-\!s))(s\!-\!1)s(2s\!-\!1)\!-\!2s^2\!+\!4s\!-\!1\right\},
	\end{align*}
	where for the last inequality we have used $\psi(s-1)+\psi(2-s)=\psi(s)+\psi(1-s)+\frac{2}{1-s}$.
	For $p_2(\xi)$, use \eqref{eq:log} to find
	\begin{align*}
		\mathcal{F}[|\xi|^{2s-4}(\log|\xi|)\xi_1^2]&=\frac{1}{4}\pi^{-2s+1}(\log\pi)\dfrac{\Gamma\left(s-1\right)}{\Gamma\left(2-s\right)}\partial_1^2\left(|x|^{-2s+2}\right)\\
		&\quad-\frac{1}{8}\pi^{-2s+1}\dfrac{\Gamma(s-1)}{\Gamma(2-s)}(\psi(s-1)+\psi(2-s))\partial_1^2\left(|x|^{-2s+2}\right)\\
		&\quad+\frac{1}{4}\pi^{-2s+1}\dfrac{\Gamma\left(s-1\right)}{\Gamma\left(2-s\right)}\partial_1^2\left(|x|^{-2s+2}\log|x|\right).
	\end{align*}
	Here, one can use $\partial_1^2\left(|x|^{-2s+2}\right)=(-2s+2)(-2s)|x|^{-2s-2}x_1^2+(-2s+2)|x|^{-2s}$ and
	\begin{align*}
		\partial^2_1\left(|x|^{-2s+2}\log|x|\right)&=-2s(2-2s)|x|^{-2s-2}x_1^2\log|x|\\
		&\quad+2(1-s)|x|^{-2s}\log|x|+2(-2s+1)|x|^{-2s-2}x_1^2+|x|^{-2s}
	\end{align*}
	to yield
	\begin{align*}
		\mathcal{F}[p_2](-e_1)&=\dfrac{1}{4}\pi^{-2s+1}\dfrac{\Gamma(s)}{\Gamma(2-s)}\left\{(\log\pi)(-2s+2)(-2s+1)\right.\\
		&\left.\quad\quad\quad\quad\quad\quad\quad\quad\quad-(\psi(s-1)+\psi(2-s))(-2s+1)(1-s)-4s+3\right\}\\
		&=\dfrac{\pi^{-2s+1}}{4(1-s)}\dfrac{\Gamma(s)}{\Gamma(1-s)}\left\{(\log\pi)(-2s+2)(-2s+1)\right.\\
		&\quad\quad\quad\quad\quad\quad\quad\quad\quad\left.-(\psi(s)+\psi(1-s))(-2s+1)(1-s)+1\right\},
	\end{align*}
	where we have used $\psi(s-1)+\psi(2-s)=\psi(s)+\psi(1-s)+\frac{2}{1-s}$. Similarly,
	\begin{align*}
\mathcal{F}[p_3](-e_1)=\dfrac{-\pi^{-2s+1}}{2(1+s)}\dfrac{\Gamma(s)}{\Gamma(1\!-\!s)}\left\{2(\log\pi)s(2s\!+\!1)\!-\!(\psi(s)\!+\!\psi(1\!-\!s))s(2s\!+\!1)\!-\!4s\!-\!1\right\}
	\end{align*}
	holds. For $p_4(\xi)$, use \eqref{eq:log} to obtain
	\begin{align*}
		\mathcal{F}[|\xi|^{2s-2}(\log|\xi|)\xi_1]&=-\frac{i}{2}\pi^{-2s}(\log\pi)\dfrac{\Gamma\left(s\right)}{\Gamma\left(1-s\right)}\partial_1\left(|x|^{-2s}\right)\\
		&\quad+\frac{i}{4}\pi^{-2s}\dfrac{\Gamma(s)}{\Gamma(1-s)}(\psi(s)+\psi(1-s))\partial_1\left(|x|^{-2s}\right)\\
		&\quad-\frac{i}{2}\pi^{-2s}\dfrac{\Gamma\left(s\right)}{\Gamma\left(1-s\right)}\partial_1\left(|x|^{-2s}\log|x|\right).
	\end{align*}
	Here, since $\partial_1\left(|x|^{-2s}\right)=(-2s)|x|^{-2s-2}x_1$ and
	\begin{align*}
		\partial_1\left(|x|^{-2s}\log|x|\right)=(-2s)|x|^{-2s-2}x_1\log|x|+|x|^{-2s-2}x_1
	\end{align*}
	holds, so we have
	\begin{align*}
		\mathcal{F}[p_4](-e_1)=\dfrac{2s+5}{2(1+s)}\pi^{-2s+1}\dfrac{\Gamma(s)}{\Gamma(1-s)}\left((\log\pi)2s-(\psi(s)+\psi(1-s))s-1\right).
	\end{align*}
	For the cases of $p_5(\xi)$ and $p_6(\xi)$, the resulting calculations are as follows:
	\begin{align*}
		\mathcal{F}[p_5](-e_1)=-\dfrac{1+2s}{2}\pi^{-2s+1}\dfrac{\Gamma(s)}{\Gamma(1-s)}\left(\log\pi-\frac{1}{2}(\psi(s)+\psi(1-s))\right)
	\end{align*}
	and
	\begin{align*}
		\mathcal{F}[p_6](-e_1)&=\dfrac{1}{s(1+s)}\pi^{-2s+1}\dfrac{\Gamma(1+s)}{\Gamma(-s)}\left(\log\pi-\frac{1}{2}(\psi(1+s)+\psi(-s))\right)\\
		&=-\dfrac{\pi^{-2s+1}}{1+s}\dfrac{\Gamma(1+s)}{\Gamma(1-s)}\left(\log\pi-\frac{1}{2}(\psi(s)+\psi(1-s))-\frac{1}{s}\right).
	\end{align*}
	Note that in the last equality we have used $\psi(1+s)+\psi(-s)=\psi(s)+\psi(1-s)+\frac{2}{s}$.
	Now we start to compute $p_7(\xi)$. Observe that
	\begin{align*}
		\mathcal{F}\left[|\xi|^{2s-4}\xi^3_1\right](x)&=-\dfrac{i}{8}\pi^{-2s}\dfrac{\Gamma\left(s-1\right)}{\Gamma\left(-s+2\right)}\partial^3_1(|x|^{-2s+2})\\
		&=-\dfrac{i}{2}\pi^{-2s}\dfrac{\Gamma\left(s+1\right)}{\Gamma\left(-s+2\right)}\left((-2-2s)|x|^{-2s-4}x_1^3+3|x|^{-2-2s}x_1\right),
	\end{align*}
	and so
	\begin{align*}
		\mathcal{F}[p_7](-e_1)=-(\gamma-\log 2)\pi^{-2s+1}\dfrac{\Gamma\left(1+s\right)}{\Gamma\left(-s+2\right)}\dfrac{(1-s)\left(2s-1\right)}{1+s}.
	\end{align*}
	To obtain $p_8(\xi)$, we see that
	\begin{align*}
\mathcal{F}[|\xi|^{2s-4}\xi^2_1](x)\!=\!-\!\dfrac{\pi^{-2s+1}}{4}\dfrac{\Gamma\left(s\!-\!1\right)}{\Gamma\left(2\!-\!s\right)}\left((-2s\!+\!2)(-2s)|x|^{-2s-2}x_1^2\!+\!(-2s\!+\!2)|x|^{-2s}\right)
	\end{align*}
	holds, thus it follows that
	\begin{align*}
		\mathcal{F}[p_8](-e_1)=-\frac{1}{2}(\gamma-\log 2)\pi^{-2s+1}\dfrac{\Gamma(s)}{\Gamma(2-s)}(s-1)(2s-1).
	\end{align*}
	For $p_9(\xi)$, we see that
	\begin{align*}
\mathcal{F}[|\xi|^{2s-2}\xi^2_1](x)\!=\!-\dfrac{\pi^{-2s-1}}{4}\dfrac{\Gamma\left(s\right)}{\Gamma\left(1\!-\!s\right)}\left((-2s)(-2s\!-\!2)|x|^{-2s-4}x_1^2\!+\!(-2s)|x|^{-2s-2}\right)
	\end{align*}
	holds, so that
	\begin{align*}
		\mathcal{F}[p_9](-e_1)=(\gamma-\log 2)\pi^{-2s+1}\dfrac{\Gamma\left(s\right)}{\Gamma\left(1-s\right)}\dfrac{s(2s+1)}{1+s}.
	\end{align*}

	The corresponding computation for the term $p_{10}(\xi)$ is as follows. We use
	\begin{align*}
		&\mathcal{F}\left[|\xi|^{2s-2}\xi_1\right](x)=\left(\dfrac{i}{2\pi}\right)\partial_1\left(\mathcal{F}\left[|\xi|^{2s-2} \right](x)\right)=-i\pi^{-2s}\dfrac{s\Gamma(s)}{\Gamma(-s+1)}|x|^{-2s-2}x_1
	\end{align*}
	to get
	\begin{align*}
		\mathcal{F}[p_{10}](-e_1)=-(\gamma-\log 2)\pi^{-2s+1}\dfrac{s\Gamma(s)}{\Gamma(1-s)}\dfrac{2s+5}{1+s}.
	\end{align*}
	For $p_{11}(\xi)$, $\mathcal{F}\left[|\xi|^{2s-2}\right](x)=\pi^{-2s+1}\frac{\Gamma(s)}{\Gamma(1-s)}|x|^{-2s}$ and so
	\begin{align*}
		\mathcal{F}[p_{11}](-e_1)=(\gamma-\log 2)\pi^{-2s+1}\frac{\Gamma(s)}{\Gamma(1-s)}\dfrac{1+2s}{2}.
	\end{align*}
	For $p_{12}(\xi)$, $\mathcal{F}\left[|\xi|^{2s}\right](x)=\pi^{-2s-1}\frac{\Gamma(1+s)}{\Gamma(-s)}|x|^{-2-2s}$ so that
	\begin{align*}
		\mathcal{F}[p_{12}](-e_1)=-(\gamma-\log 2)\pi^{-2s+1}\frac{\Gamma(1+s)}{\Gamma(-s)}\dfrac{1}{s(1+s)}.
	\end{align*}
Note that $\mathcal{F}[p_2](x), \mathcal{F}[p_3](x), \dots, \mathcal{F}[p_{12}](x)$ are locally integrable functions in $\setR^d\setminus\{0\}$ and so the above evaluations at $x=e_1$ are justified.

Taking for the standard mollifier $\overline{\phi}(x)$ and $\phi_{\epsilon}(x)=\frac{1}{\epsilon^d}\overline{\phi}\left(\frac{x-e_1}{\epsilon}\right)$ for $\epsilon\in(0,1)$,	
\begin{align*}
\lim_{\epsilon\rightarrow 0}\skp{f\divideontimes g}{\phi_{\epsilon}}=\lim_{\epsilon\rightarrow 0}\skp{\mathcal{F}[\mathcal{F}[g_3]\mathcal{F}[g_4]]}{\phi_{\epsilon}}=\mathcal{F}[\mathcal{F}[g_3]\mathcal{F}[g_4]](e_1),
\end{align*}
where for the first equality we have used Corollary \ref{cor:Fourierconv2} and for the last equality we see that $\mathcal{F}[\mathcal{F}[g_3]\mathcal{F}[g_4]]$ coincide to a function when $\abs{x}>0$.
	
	Summing up, we obtain $f_2(d,s,\delta)$ in case of $d=2$ as follows:
	\begin{align}\label{eq:d2}
		\begin{split}
		f_2(d,s,\delta)=\dfrac{\pi}{2}\dfrac{\Gamma(1-s)}{\Gamma\left(1+s\right)}\dfrac{\Gamma\left(\frac{2-\delta}{2}\right)}{\Gamma\left(\frac{\delta+2}{2}\right)}\dfrac{\Gamma\left(\frac{2s+\delta}{2}\right)}{\Gamma\left(\frac{4-2s-\delta}{2}\right)}\tilde{f}_{2,1}(s,\delta)+\dfrac{\pi}{2s(1-s^2)}\tilde{f}_{2,2}(s),
		\end{split}
	\end{align}
	where $\tilde{f}_{2,1}(s,\delta)\coloneqq  f_{2,1}(2,s,\delta)$ and
	\begin{align*}
		\tilde{f}_{2,2}(s)&=\left[-(\log\pi)(s-1)2s(2s-1)\right.\\
		&\quad\quad\quad\quad\quad\quad\left.+(\psi(s)+\psi(1-s))(s-1)s(2s-1)+2s^2-4s+1\right]\\
		&\quad+\dfrac{1}{2}(1+s)\left[(\log\pi)(-2s+2)(-2s+1)\right.\\
		&\quad\quad\quad\quad\quad\quad\left.-(\psi(s)+\psi(1-s))(-2s+1)(1-s)+1\right]\\
		&\quad-(1-s)\left[2(\log\pi)s(2s+1)-(\psi(s)+\psi(1-s))s(2s+1)-4s-1\right]\\
		&\quad+(2s+5)(1-s)\left[2(\log\pi)s-(\psi(s)+\psi(1-s))s-1\right]\\
		&\quad-(1-s^2)(1+2s)\left[(\log\pi)-\frac{1}{2}(\psi(s)+\psi(1-s))\right]\\
		&\quad-2(1-s)\left[(\log\pi)s-\dfrac{s}{2}(\psi(s)+\psi(1-s))-1\right]\\
		&=\tfrac{1}{2}s-\tfrac{1}{2}.
	\end{align*}
	Note that by the above computation together with the argument in the case $d\geq 3$, one can see that the expression \eqref{eq:d2} is the same as \eqref{eq:d3}.
\end{proof}

\section{Convolution theorems}\label{app2}

This section provides the rigorous distributional framework and detailed Fourier analysis calculations justifying the results stated in Lemma \ref{lem:f} and Lemma \ref{lem:Fourier}, where we compute the following terms $f_1$, $f_2$, $f_3$, and $f_4$ using Fourier transform. Throughout the section, we consider
\begin{align*}
d\geq 2\quad\text{with}\quad d\in\setN,\quad s\in(0,1),
\end{align*}
and $\delta\in [0,\tfrac{1}{2}]$ for Subsection \ref{sec:f1} and Subsection \ref{sec:f2}, and $\delta\in [0,\tfrac{d}{2})$ for Subsection \ref{sec:f3} and Subsection \ref{sec:f4}. We define
{\allowdisplaybreaks
\begin{align}\label{eq:f1}
&f_1(d,s,\delta) \coloneqq\pvint_{\setR^d}\dfrac{|e_1+h|^{1-\delta}\widehat{1+h_1}-1}{|h|^{d+2s}}\,dh, \\
&f_2(d,s,\delta) \label{eq:f2} \\
& \quad = \pvint_{\setR^d}\left<\dfrac{\widehat{e_1+h}\otimes\widehat{e_1+h}+e_1\otimes e_1}{2}\widehat{h},\widehat{h}\right>\dfrac{|e_1+h|^{1-\delta}\widehat{1+h_1}-1}{|h|^{d+2s}}\,dh \nonumber \\
\nonumber
& \quad = \pvint_{\setR^d}\dfrac{-2h_1^3\!+\!|h|^2h_1^2\!+\!2h_1^2\!-\!2|h|^2h_1\!+\!|h|^4}{2|h|^{d+2s+2}|e_1-h|^{2}}\left(|e_1\!-\!h|^{-\delta}(1\!-\!h_1)\!-\!1\right)dh, \\
\label{eq:f3}
&f_3(d,s,\delta):=\text{p.v.}\int_{\setR^d}\dfrac{\abs{e_1-h}^{s-\delta}\widehat{1-h_1}}{\abs{h}^{d-1+s}}\,dh, \\
\label{eq:f4}
&f_4(d,s,\delta):=\text{p.v.}\int_{\setR^d}\dfrac{\abs{e_1-h}^{-1-\delta}\widehat{1-h_1}}{\abs{h}^{d-1+s}}\,dh.
\end{align}
}
Let us explain the notation. By p.v. we denote the \emph{principal value} of an integral, where we define the principal value as the limit of the integrals over $B_{\frac{1}{r}}(0)\setminus B_r(0)$ for $r\rightarrow 0$. The unit vector $(1,0,\dots,0) \in\setR^{d}$ is denoted by $e_1$. We use $\widehat{x}=\frac{x}{\abs{x}}$ and $\widehat{x_1}=\frac{x_1}{\abs{x}}$ for vectors $x=(x_1,\dots,x_d)\in\setR^d$. $x\otimes x$ denotes the $d\times d$ matrix given by $(x\otimes x)_{i,j}=x_ix_j$ for any $1\leq i,j\leq d$.

\medskip

Let us mention that for the last equality of \eqref{eq:f2} uses elementary manipulations only. The integrals in the definition of  \eqref{eq:f1} and \eqref{eq:f2} are finite due to Lemma \ref{eq:estimate-Deltas-ud}. The integrals in the definition of \eqref{eq:f3} and \eqref{eq:f4} are finite due to Lemma \ref{lem:welldef.f3} and \ref{lem:welldef.f4}, respectively.

\medskip

The main results of this note are Theorem \ref{thm:Fourierconv0}, Theorem \ref{thm:Fourierconv},  Theorem \ref{thm:Fourierconv4}, and  Theorem \ref{thm:Fourierconv5} together with corresponding corollaries. The main results represent the values $f_i$ as certain convolutions of two distributions evaluated at $x=e_1$. Let us provide details for the case of $f_1$. For $\phi\in \mathcal{S}(\setR^d)$, we define distributions $g_1$ and $g_2$ by 
\begin{align*}
\skp{g_1}{\phi}&:=\skp{\abs{x}^{1-\delta}\widehat{x}_1}{\phi}=\displaystyle\int_{\setR^d}\abs{x}^{-\delta}x_1\phi(x)\,dx, \\
\skp{g_2}{\phi}&:=\skp{\abs{x}^{-d-2s}}{\phi}=\displaystyle\int_{\setR^d}\dfrac{\abs{x}^{-d-2s+2}\Delta\phi(x)}{2s(d+2s-2)}\,dx.
\end{align*}
We will show in Lemma \ref{lem:A20} $f_1(d,s,\delta)=\lim_{\epsilon\rightarrow 0}\skp{g_1\divideontimes g_2}{\phi_{\epsilon}}$, where $\phi_\epsilon$ is an approximation of the identity at $e_1$ and the convolution $g_1\divideontimes g_2$ is defined as in \cite{Vla02}. The evaluation then hinges on Theorem \ref{thm:Fourierconv0} where, if $\delta>0$, we show  
\begin{align*}
	\mathcal{F}[g_1\divideontimes g_2]=\mathcal{F}[g_1]\mathcal{F}[g_2]
\end{align*}
in the distributional sense. Here, $\mathcal{F}[g_1]\mathcal{F}[g_2]$ denotes the usual multiplication of $\mathcal{F}[g_1]$ and $\mathcal{F}[g_2]$. 

\medskip

The paper is organized as follows. Each of the subsequent sections provides  detailed calculations  $f_1(d,s,\delta)$, $f_2(d,s,\delta)$, $f_3(d,s,\delta)$ and $f_4(d,s,\delta)$ respectively. The strategy that we employ each subsection is as follows:
\begin{enumerate}
\item First we give definitions of distributions associated to $f_1$, $f_2$, $f_3$, and $f_4$.\\ 
(See Definition \ref{def:g1g2}, Definition \ref{def:g3g4}, Definition \ref{def:g5g6}, and Definition \ref{def:g7}.)
\item Then we prove the well-definedness of their convolutions in a distributional sense.\\ 
(See Lemma \ref{lem:welldef} with Lemma \ref{lem:conv10}, Lemma \ref{lem:welldef3} with Lemma \ref{lem:conv3}, Lemma \ref{lem:conv4} with Lemma \ref{lem:conv40}, and Lemma \ref{lem:conv5} with Lemma \ref{lem:conv50}.)
\item Furthermore, we show that these convolutions recover $f_1$, $f_2$, $f_3$, and $f_4$ when tested against an approximation to identity centered at $x=e_1$.\\
(See Lemma \ref{lem:A20}, Lemma \ref{lem:A2}, Lemma \ref{lem:A40}, and Lemma \ref{lem:A50}.)
\item By (b), we can apply convolution theorem as in \cite[Page 96]{Vla02}. Moreover, with our specific distributions in Definition \ref{def:g1g2}, Definition \ref{def:g3g4}, Definition \ref{def:g5g6}, and Definition \ref{def:g7}, the resulting product of the Fourier transforms of the two distributions can be understood in a certain way.\\
(See Theorem \ref{thm:Fourierconv0}, Theorem \ref{thm:Fourierconv},  Theorem \ref{thm:Fourierconv4}, and  Theorem \ref{thm:Fourierconv5}.)
\end{enumerate}

\subsection{Convolution theorem for $f_1(d,s,\delta)$}\label{sec:f1}

We define the distributions $g_1$ and $g_2$, where (formal) evaluation $x=e_1$ to their convolution will recover $f_1(d,s,\delta)$.
\begin{definition}\label{def:g1g2}
For $\phi\in \mathcal{S}(\setR^d)$, we define distributions $g_1$ and $g_2$ in $\setR^d$ as
\begin{align*}
\skp{g_1}{\phi}:=\skp{\abs{x}^{1-\delta}\widehat{x}_1}{\phi}=\displaystyle\int_{\setR^d}\abs{x}^{-\delta}x_1\phi(x)\,dx,
\end{align*}
and
\begin{align*}
\skp{g_2}{\phi}:=\skp{\abs{x}^{-d-2s}}{\phi}=\displaystyle\int_{\setR^d}\dfrac{\abs{x}^{-d-2s+2}\Delta\phi(x)}{2s(d+2s-2)}\,dx.
\end{align*}
\end{definition}

We observe the following:
\begin{itemize}
	\item One can easily check that the above distributions are well-defined since $\phi\in \mathcal{S}(\setR^d)$ implies $\abs{\partial^k_i\phi(x)}\lesssim \min\{1,\frac{1}{\abs{x}^{d+2}}\}$ for $k=0,1,2$ and $i=0,1,\dots,d$.
	\item The distribution $g_1$ can be understood as a locally integrable function in $\setR^d$ since $\delta\in\left[0,\tfrac{1}{2}\right]$.
	\item When $\abs{x}>0$, $g_2(x)\equiv\abs{x}^{-d-2s}$ holds. 
\end{itemize}

The main goal of this subsection is to show the following relation (see Theorem \ref{thm:Fourierconv0}): For $\delta>0$,
\begin{align}\label{eq:g1g2}
\mathcal{F}[g_1\divideontimes g_2]=\mathcal{F}[g_1]\mathcal{F}[g_2]
\end{align}
in the distributional sense, where $\divideontimes$ is a convolution of two distributions in $\mathcal{S}'(\setR^d)$ introduced in \cite[Page 96]{Vla02}, and $\mathcal{F}[g_1]\mathcal{F}[g_2]$ means the usual multiplication between $\mathcal{F}[g_1]$ and $\mathcal{F}[g_2]$. The well-definedness of the left-hand side term of \eqref{eq:g1g2} will be proved in Lemma \ref{lem:welldef} and Lemma \ref{lem:conv10}, and the well-definedness of the right-hand side term of \eqref{eq:g1g2} will be proved in Lemma \ref{lem:Fourier00}.

The case $\delta=0$ is excluded, because \eqref{eq:g1g2} does not hold in this case. An additional term involving distributions would appear. However, for the computation of $f_1(d,s,0)$, the case $\delta=0$ is not needed as $f_1(d,s,0)=0$ from \eqref{eq:f1}. Note that in Lemma \ref{lem:welldef}--Lemma \ref{lem:A20} the case of $\delta=0$ is included, but in Lemma \ref{lem:Fourier00} and the remaining part of this subsection we assume $\delta>0$.

From now on, let $c$ be a generic constant, and we use brackets to clarify the dependence of $c$, e.g., $c=c(d,s,\delta)$ means that $c$ depends only on $d,s$, and $\delta$. For numbers or functions $a$ and $b$, we occasionally use $a\lesssim b$, if there exists an implicit constant $c>0$ such that $a\leq cb$.

First, we prove the following lemma. 

\begin{lemma}\label{lem:welldef}
For any $\phi\in \mathcal{S}(\setR^d)$, $\skp{g_2(y)}{\skp{g_1(x)}{\phi(x+y)}_x}_y$ is well-defined, i.e.,
\begin{align*}
&\skp{g_2(y)}{\skp{g_1(x)}{\phi(x+y)}_x}_y\\
&\quad=\dfrac{1}{2s(d+2s-2)}\int_{\setR^d}\abs{y}^{-d-2s+2}\left(\int_{\setR^d}\abs{x}^{-\delta}x_1\Delta_y\phi(x+y)\,dx\right)\,dy
\end{align*}
exists for any $\phi\in\mathcal{S}(\setR^d)$.
\end{lemma}
\begin{proof}
Let us choose $\phi\in \mathcal{S}(\setR^d)$ and $R\geq2$. Before starting the proof, let us explain the basic strategy. In the above integral, the term $\abs{y}^{-d-2s+2}$ is integrable near 0, which means 
\begin{align*}
\int_{B_R(0)}\abs{y}^{-d-2s+2}\,dy<\infty,
\end{align*}
and it is not integrable near infinity, which means
\begin{align*}
\int_{\setR^d\setminus B_R(0)}\abs{y}^{-d-2s+2}\,dy=\infty.
\end{align*}
Thus together with the fact that $y\mapsto \abs{y}^{-d-2s+2}$ is even, we want to estimate as follows to estimate $\abs{\skp{g_2(y)}{\skp{g_1(x)}{\phi(x+y)}_x}_y}<\infty$:
\begin{align*}
\begin{split}
\int_{\setR^d}\abs{x}^{-\delta}x_1\Delta_y(x+y)\,dx\leq c\,R^{d}\cdot\indicator_{\{\abs{y}\leq R\}}+\bigg(c\abs{y}^{-\delta-2}y_1+\frac{c}{\abs{y}^{2}}\bigg)\cdot \indicator_{\{\abs{y}\geq R\}}.
\end{split}
\end{align*}
(See \eqref{eq:y<Ry>R} below.)

Let us first estimate the following term:
\begin{align*}
\int_{\setR^d}\abs{x}^{-\delta}x_1\Delta_y\phi(x+y)\,dx.
\end{align*}
Note that for each $y\in\setR^d$, $\abs{x}^{-\delta}x_1\Delta_y\phi(x+y)\in L^1(\setR^d)$ since $\phi\in\mathcal{S}(\setR^d)$. When $\abs{y}\leq R$, $\delta\in[0,\frac{1}{2}]$ implies $-\delta>-d$, so we obtain
\begin{align*}
&\Biggabs{\int_{\setR^d}\abs{x}^{-\delta}x_1\Delta_y\phi(x+y)\,dx}\leq\int_{\setR^d}\abs{x}^{-\delta+1}\abs{\Delta_y\phi(x+y)}\,dx\\
&\quad\quad\quad\leq\int_{B_R(0)}\abs{x}^{-\delta+1}\abs{\Delta_y\phi(x+y)}\,dx+\int_{\setR^d\setminus B_R(0)}\abs{x}^{-\delta+1}\abs{\Delta_y\phi(x+y)}\,dx\\
&\quad\quad\quad\lesssim R^{d-\delta+1}+\int_{\setR^d\setminus B_R(0)}\abs{x}^{-\delta+1}\abs{\Delta_y\phi(x+y)}\,dx,
\end{align*}
where for the last inequality we have used $\abs{\Delta_y\phi(x+y)}\lesssim 1$ since $\phi\in \mathcal{S}(\setR^d)$. Here and throughout the proof of this lemma, implicit generic constants $c$ in `$\lesssim$' depend only on
\begin{align*}
d,s,\delta,\quad\text{and}\quad \norm{\nabla^i\phi}_{L^{\infty}(\setR^d)}\quad\text{for}\quad i=0,1,2.
\end{align*}
Then if $-y\in B_{\frac{3}{4}R}(0)$, then using $\abs{\Delta_y\phi(x+y)}\lesssim\abs{x+y}^{-d-2}$ (since $\phi\in \mathcal{S}(\setR^d)$), we obtain
\begin{align*}
\int_{\setR^d\setminus B_R(0)}\abs{x}^{-\delta+1}\abs{\Delta_y\phi(x+y)}\,dx\lesssim\int_{\setR^d\setminus B_R(0)}\abs{x}^{-\delta+1}\abs{x+y}^{-d-2}\,dx\lesssim R^{-\delta-1}.
\end{align*}
On the other hand, if $-y\not\in B_{\frac{3}{4}R}(0)$ so that $B_{\frac{1}{4}R}(-y)\cap B_{\frac{1}{2}R}(0)=\emptyset$, then employing $\abs{\Delta_y\phi(x+y)}\lesssim \min\{1,\abs{x+y}^{-d-2}\}$,
\begin{align*}
&\int_{\setR^d\setminus B_R(0)}\abs{x}^{-\delta+1}\abs{\Delta_y\phi(x+y)}\,dx\\
&\quad\lesssim  \int_{B_{\frac{1}{4}R}(-y)}\abs{x}^{-\delta+1}\,dx+\int_{\setR^d\setminus (B_{\frac{1}{2}R}(0)\cup B_{\frac{1}{4}R}(-y))}\abs{x}^{-\delta+1}\abs{x+y}^{-d-2}\,dx\\
&\quad \lesssim R^{-\delta+1}\cdot R^d+R^{-\delta+1}R^{-2}\lesssim R^{d+1}
\end{align*}
holds, since $R\geq 2$. Then if $\abs{y}\leq R$, we obtain
\begin{align}\label{eq:y<R}
\Biggabs{\int_{\setR^d}\abs{x}^{-\delta}x_1\Delta_y\phi(x+y)\,dx}\lesssim R^{d+1}.
\end{align}

Now we consider the case of $\abs{y}\geq R$. In this case we will use the fact that the integral of the odd function on $\setR^d$ is zero as well as the divergence theorem. Using $\Delta_y\phi(x+y)=\Delta_x\phi(x+y)$, observe that 
\begin{align}\label{eq:x.delta}
\begin{split}
&\int_{\setR^d}\abs{x}^{-\delta}x_1\Delta_y\phi(x+y)\,dx\\
&\quad=\int_{\setR^d}\abs{x-y}^{-\delta}(x_1-y_1)\Delta\phi(x)\,dx\\
&\quad=\int_{B_\frac{\abs{y}}{2}(0)}\abs{x-y}^{-\delta}(x_1-y_1)\Delta\phi(x)\,dx+\int_{\setR^d\setminus B_\frac{\abs{y}}{2}(0)}\abs{x-y}^{-\delta}(x_1-y_1)\Delta\phi(x)\,dx
\end{split}
\end{align}
holds. Here, using divergence theorem, we obtain
\begin{align}\label{eq:div}
\begin{split}
&\int_{B_\frac{\abs{y}}{2}(0)}\abs{x-y}^{-\delta}(x_1-y_1)\Delta\phi(x)\,dx\\
&\quad=\int_{B_\frac{\abs{y}}{2}(0)}\Delta_x(\abs{x-y}^{-\delta}(x_1-y_1))\phi(x)\,dx\\
&\quad\quad+\int_{\partial B_{\frac{\abs{y}}{2}}(0)}\abs{x-y}^{-\delta}(x_1-y_1)\frac{x}{\abs{x}}\cdot\nabla\phi(x)\,dx\\
&\quad\quad-\int_{\partial B_{\frac{\abs{y}}{2}}(0)}\phi(x)\nabla_x(\abs{x-y}^{-\delta}(x_1-y_1))\cdot\frac{x}{\abs{x}}\,dx.
\end{split}
\end{align}
Using $\Delta_x(\abs{x-y}^{-\delta}(x_1-y_1))=-(d-\delta)\delta\abs{x-y}^{-\delta-2}(x_1-y_1)$, we estimate
\begin{align*}
&\int_{B_\frac{\abs{y}}{2}(0)}\Delta_x(\abs{x-y}^{-\delta}(x_1-y_1))\phi(x)\,dx\\
&\quad=-(d-\delta)\delta\int_{B_\frac{\abs{y}}{2}(0)}\abs{x-y}^{-\delta-2}(x_1-y_1)\phi(x)\,dx.
\end{align*}
Moreover, 
\begin{align*}
\abs{\abs{x-y}^{-\delta-2}(x_1-y_1)+\abs{y}^{-\delta-2}y_1}\lesssim\abs{y}^{-\delta-2}\abs{x}\quad\text{for}\quad x\in B_{\frac{\abs{y}}{2}}(0)
\end{align*}
from Lemma \ref{lem:basic} and $\abs{\phi(x)}\lesssim\min\{1,\abs{x}^{-2d}\}$, we obtain
\begin{align*}
&\biggabs{\int_{B_\frac{\abs{y}}{2}(0)}(\abs{x-y}^{-\delta-2}(x_1-y_1)+\abs{y}^{-\delta-2}y_1)\phi(x)\,dx}\\
&\quad\quad\lesssim \abs{y}^{-\delta-2}\int_{B_\frac{\abs{y}}{2}(0)}\abs{x}\phi(x)\,dx\lesssim \abs{y}^{-\delta-2}\lesssim \abs{y}^{-2}.
\end{align*}
Thus, merging the above two estimates, we find that
\begin{align*}
&\int_{B_\frac{\abs{y}}{2}(0)}\Delta_x(\abs{x-y}^{-\delta}(x_1-y_1))\phi(x)\,dx\\
&\quad\leq (d-\delta)\delta \abs{y}^{-\delta-2}y_1\int_{B_{\frac{\abs{y}}{2}}(0)}\phi(x)\,dx+c\abs{y}^{-2}.
\end{align*}

Also, by $\abs{\nabla\phi(x)}\lesssim\min\{1,\abs{x}^{-d-2}\}\lesssim\abs{y}^{-d-2}$ on $x\in\partial B_{\frac{\abs{y}}{2}}(0)$, we have
\begin{align*}
\biggabs{\int_{\partial B_{\frac{\abs{y}}{2}}(0)}\abs{x-y}^{-\delta}(x_1-y_1)\frac{x}{\abs{x}}\cdot\nabla\phi(x)\,dx}\lesssim \abs{y}^{-\delta+1-d-2}\abs{y}^{d-1}\lesssim \abs{y}^{-2}.
\end{align*}
Finally, using $\abs{\nabla_x(\abs{x-y}^{-\delta}(x_1-y_1))}\lesssim \abs{x-y}^{-\delta}\lesssim\abs{y}^{-\delta}$ when $x\in \partial B_{\frac{\abs{y}}{2}}(0)$ and $\abs{\phi(x)}\leq\min\{1,\abs{x}^{-d-1}\}$, 
\begin{align*}
&\biggabs{\int_{\partial B_{\frac{\abs{y}}{2}}(0)}\phi(x)\nabla_x(\abs{x-y}^{-\delta}(x_1-y_1))\cdot\nu\,dx}\lesssim \abs{y}^{-\delta}\abs{y}^{-d-1}\abs{y}^{d-1}\lesssim \abs{y}^{-2}
\end{align*}
holds. Then for \eqref{eq:div}, we obtain
\begin{align*}
\int_{B_\frac{\abs{y}}{2}(0)}\abs{x-y}^{-\delta}(x_1-y_1)\Delta\phi(x)\,dx\leq (d-\delta)\delta \abs{y}^{-\delta-2}y_1\int_{B_{\frac{\abs{y}}{2}}(0)}\phi(x)\,dx+c\abs{y}^{-2}.
\end{align*}
Summing up, with \eqref{eq:y<R} we see that
\begin{align}\label{eq:y<Ry>R}
\begin{split}
&\int_{\setR^d}\abs{x}^{-\delta}x_1\Delta_y(x+y)\,dx\\
&\quad\leq c\,R^{d+1}\indicator_{\{\abs{y}\leq R\}}+\bigg((d-\delta)\delta\abs{y}^{-\delta-2}y_1\int_{B_\frac{\abs{y}}{2}(0)}\phi(x)\,dx+\frac{c}{\abs{y}^{2}}\bigg)\indicator_{\{\abs{y}\geq R\}}.
\end{split}
\end{align}
Note that the integral of $\phi(\cdot)$ over $B_\frac{\abs{y}}{2}(0)$ is bounded since $\phi\in\mathcal{S}(\setR^d)$.

For $L\geq 2R$, using \eqref{eq:y<Ry>R}, $\Delta_y\phi(x+y)=\Delta_x\phi(x+y)$, and $\abs{y}^{-d-2s+2}>0$, we estimate as follows: 
\begin{align*}
&\Biggabs{\dfrac{1}{2s(d+2s-2)}\int_{\abs{y}\leq L}\abs{y}^{-d-2s+2}\left(\int_{\setR^d}\abs{x}^{-\delta}x_1\Delta_y\phi(x+y)\,dx\right)\,dy}\\
&\leq c\int_{\abs{y}\leq R}R^{d+1}\abs{y}^{-d-2s+2}\,dy\\
&\quad+\int_{R\leq\abs{y}\leq L}\bigg((d-\delta)\delta\abs{y}^{-\delta-2}y_1\int_{B_\frac{\abs{y}}{2}(0)}\phi(x)\,dx+\frac{c}{\abs{y}^{2}}\bigg)\abs{y}^{-d-2s+2}dy.
\end{align*}
Here, we estimate
\begin{align*}
\int_{\abs{y}\leq R}R^{d+1}\abs{y}^{-d-2s+2}\,dy\lesssim R^{d-2s+3}.
\end{align*}
Moreover, using Fubini's theorem and observing that $y\mapsto\abs{y}^{-\delta-2}\abs{y}^{-d-2s+2}y_1$ is odd, we find
\begin{align}\label{eq;y>R0}
\begin{split}
&\int_{R\leq \abs{y}\leq L}\bigg(\int_{B_\frac{\abs{y}}{2}(0)}\phi(x)\,dx\bigg)\abs{y}^{-\delta-2}y_1\cdot\abs{y}^{-d-2s+2}dy\\
&\quad=\int_{\setR^d}\bigg(\int_{\max\{2\abs{x},R\}\leq \abs{y}\leq L}\abs{y}^{-d-2s+2-\delta-2}y_1\,dy\bigg)\phi(x)dx=0.
\end{split}
\end{align}
Also, we see that
\begin{align}\label{eq;y>R2s}
\int_{R\leq\abs{y}\leq L}\frac{c}{\abs{y}^2}\abs{y}^{-d-2s+2}dy\lesssim R^{-2s}.
\end{align}
Merging the above three estimates, it follows that
\begin{align*}
&\Biggabs{\dfrac{1}{2s(d+2s-2)}\int_{\abs{y}\leq L}\abs{y}^{-d-2s+2}\left(\int_{\setR^d}\abs{x}^{-\delta}x_1\Delta_y\phi(x+y)\,dx\right)\,dy}\\
&\quad\lesssim R^{d-2s+3}+R^{-2s}\lesssim R^{d+3}.
\end{align*}

Using the above argument, we can obtain that for $\widetilde{L}>L$, we have
\begin{align*}
&\Biggabs{\dfrac{1}{2s(d+2s-2)}\int_{L\leq\abs{y}\leq\widetilde{L}}\abs{y}^{-d-2s+2}\left(\int_{\setR^d}\abs{x}^{-\delta}x_1\Delta_y\phi(x+y)\,dx\right)\,dy}\\
&\leq c\biggabs{\int_{L\leq \abs{y}\leq \widetilde{L}}\bigg((d-\delta)\delta\abs{y}^{-\delta-2}y_1\int_{B_\frac{\abs{y}}{2}(0)}\phi(x)\,dx+\frac{c}{\abs{y}^{2}}\bigg)\abs{y}^{-d-2s+2}dy}.
\end{align*}
By Fubini's theorem and observing that $y\mapsto\abs{y}^{-\delta-2}\abs{y}^{-d-2s+2}y_1$ is odd, we find
\begin{align}\label{eq;y>R0'}
\begin{split}
&\int_{L\leq \abs{y}\leq \widetilde{L}}\bigg(\int_{B_\frac{\abs{y}}{2}(0)}\phi(x)\,dx\bigg)\abs{y}^{-\delta-2}y_1\cdot\abs{y}^{-d-2s+2}dy\\
&\quad=\int_{\setR^d}\bigg(\int_{\max\{2\abs{x},L\}\leq\abs{y}\leq\widetilde{L}}\abs{y}^{-d-2s+2-\delta-2}y_1\,dy\bigg)\phi(x)dx=0.
\end{split}
\end{align}
Also, we see that
\begin{align}\label{eq;y>R2s'}
\int_{L\leq\abs{y}\leq \widetilde{L}}\frac{c}{\abs{y}^2}\abs{y}^{-d-2s+2}dy\lesssim L^{-2s}.
\end{align}
Merging the above three estimates, it follows that
\begin{align*}
\Biggabs{\dfrac{1}{2s(d+2s-2)}\int_{L\leq\abs{y}\leq\widetilde{L}}\abs{y}^{-d-2s+2}\left(\int_{\setR^d}\abs{x}^{-\delta}x_1\Delta_y\phi(x+y)\,dx\right)\,dy}\lesssim L^{-2s}.
\end{align*}
This proves that 
\begin{align*}
\dfrac{1}{2s(d+2s-2)}\int_{\abs{y}\leq L}\abs{y}^{-d-2s+2}\left(\int_{\setR^d}\abs{x}^{-\delta}x_1\Delta_y\phi(x+y)\,dx\right)\,dy
\end{align*}
is a Cauchy sequence for $L$. Then we see that 
\begin{align*}
&\skp{g_2(y)}{\skp{g_1(x)}{\phi(x+y)}_x}_y\\
&\quad=\dfrac{1}{2s(d+2s-2)}\int_{\setR^d}\abs{y}^{-d-2s+2}\left(\int_{\setR^d}\abs{x}^{-\delta}x_1\Delta_y\phi(x+y)\,dx\right)\,dy
\end{align*}
exists. Therefore, we get the conclusion. 
\end{proof}

Following \cite[Page 96]{Vla02}, let $\eta\in\mathcal{S}(\setR^d)$ be such that 

\medskip

\begin{align}\label{eq:eta}
\text{$\eta(x)=1$ in $x\in B_1(0)$, $\support\eta\subset B_2(0)$, and $\abs{\nabla^i\eta}\leq c(d)$ for $i=0,1,2,3$.}	
\end{align}

\medskip

Define $\eta_k(x)=\eta(x/k)$ as an 1-sequence in $\mathcal{S}(\setR^d)$. Also, with \cite[Page 52]{Vla02}, let $\xi\in\mathcal{S}(\setR^{2d})$ be such that 

\medskip

\begin{align}\label{eq:xi}
\begin{split}
&\text{$\xi(x,y)=1$ in $(x,y)\in B_1(0)\times B_1(0)$, $\support\xi\subset B_2(0)\times B_2(0)$,}\\
&\text{and $\abs{\nabla^i\xi}\leq c(d)$ for any $i=0,1,2,3$.}
\end{split}
\end{align}

\medskip

Define $\xi_j(x,y)=\xi(x/j,y/j)$ as an 1-sequence in $\mathcal{S}(\setR^{2d})$. Now we prove the following lemma.

\begin{lemma}\label{lem:conv10}
The convolution of distributions $g_1\divideontimes g_2$ is well-defined in $x\in\setR^d$ in the sense that for any $\phi\in \mathcal{S}(\setR^d)$, 
\begin{align}\label{eq:kj0}
\begin{split}
\skp{g_1\divideontimes g_2}{\phi}&=\skp{g_2\divideontimes g_1}{\phi}\\
&:=\lim_{k\rightarrow\infty}\lim_{j\rightarrow\infty}\bigskp{(\eta_kg_2)(y)}{\skp{g_1(x)}{\xi_j(x,y)\phi(x+y)}_x}_y.
\end{split}		
\end{align}
Moreover, we have
\begin{align*}
\skp{g_1\divideontimes g_2}{\phi}=\bigskp{g_2(y)}{\skp{g_1(x)}{\phi(x+y)}_x}_y.
\end{align*}
\end{lemma}
\begin{proof}
Note that once the double limit in the right-hand side of \eqref{eq:kj0} exists, then
\begin{align*}
\skp{g_1\divideontimes g_2}{\phi}&:=\lim_{k\rightarrow\infty}\lim_{j\rightarrow\infty}\skp{(\eta_kg_2)(\cdot)\times g_1(\bigcdot)}{\xi_j(\cdot,\bigcdot)\phi(\cdot+\bigcdot)}\\
&:=\lim_{k\rightarrow\infty}\lim_{j\rightarrow\infty}\bigskp{(\eta_kg_2)(y)}{\skp{g_1(x)}{\xi_j(x,y)\phi(x+y)}_x}_y
\end{align*}	
holds. Here, for the first definition since $g_1,g_2\in\mathcal{S}'(\setR^d)\subset (C^{\infty}_c)'(\setR^d)$ and $\xi_j\in C^{\infty}_c(\setR^{2d})$, \cite[Page 96 and Page 52]{Vla02} is used, and for the second definition since $\xi_j(x,y)\phi(x+y)\in C^{\infty}_c(\setR^{2d})$, \cite[Page 41]{Vla02} is used. Moreover, by \cite[Page 96]{Vla02}, we have $\skp{g_1\divideontimes g_2}{\phi}=\skp{g_2\divideontimes g_1}{\phi}$.

Now it suffices to show that
\begin{align}\label{eq:doublelimit0}
\begin{split}
&\lim_{k\rightarrow\infty}\lim_{j\rightarrow\infty}\bigskp{(\eta_kg_2)(y)}{\skp{g_1(x)}{\xi_j(x,y)\phi(x+y)}_x}_y\\
&\quad=\bigskp{g_2(y)}{\skp{g_1(x)}{\phi(x+y)}_x}_y.
\end{split}
\end{align}
To do this, for $k,j>1$, we have
\begin{align*}
&\bigabs{\bigskp{(\eta_kg_2)(y)}{\skp{g_1(x)}{\xi_j(x,y)\phi(x+y)}_x}_y-\bigskp{g_2(y)}{\skp{g_1(x)}{\phi(x+y)}_x}_y}\\
&\quad\leq\bigabs{\bigskp{(\eta_kg_2)(y)}{\skp{g_1(x)}{(1-\xi_j(x,y))\phi(x+y)}_x}_y}\\
&\quad\quad+\bigabs{\bigskp{(1-\eta_k)g_2(y)}{\skp{g_1(x)}{\phi(x+y)}_x}_y}\\
&\quad=:I_1+I_2.
\end{align*}
Using $\sum^{3}_{i=0}\abs{\nabla^i[(1-\xi^{j}(x,y))\phi(x+y)]}\lesssim\abs{x+y}^{-d}\lesssim\abs{y}^{-d}$ from $\phi\in\mathcal{S}(\setR^d)$ and $\abs{\nabla^i\xi}\leq c(d)$ for $i=0,1,2,3$, $\abs{x}\leq 2k$ and $\abs{y}>j$ with $j\geq 4k$, we obtain
\begin{align*}
I_1&\lesssim\int_{\abs{x}\leq 2k}\int_{\abs{y}>j}\abs{\eta_k(x)}\abs{y}^{-d-2s+2}(\abs{x}^{1-\delta})\abs{y}^{-d}\,dx\,dy\\
&\lesssim k^{d+1-\delta}j^{-d-2s+2}\underset{j\rightarrow \infty}{\longrightarrow} 0.
\end{align*}
For $I_2$, we write the integrand $J(x,y):\setR^{2d}\rightarrow\setR$ as follows:
\begin{align*}
\bigskp{(1-\eta_k)g_2(y)}{\skp{g_1(x)}{\phi(x+y)}_x}_y=:\int_{\setR^d}\int_{\setR^d}J(x,y)\,dy\,dx=\int_{\abs{y}> k}\int_{\setR^d}J(x,y)\,dx\,dy,
\end{align*}
where for the last equality we have used $\eta_k(\cdot)\equiv 1$ when $\abs{y}\leq k$. Then following the proof of Lemma \ref{lem:welldef}, we obtain \eqref{eq:y<Ry>R} with the choice of $x$ replaced by $y$ and $R=k/2$. Then imitating the remaining
proof of Lemma \ref{lem:welldef} but only considering the case of $\abs{y}\geq R$ in its proof, we get the analogous estimate of \eqref{eq;y>R2s} as follows:
\begin{align*}
\biggabs{\int_{\abs{y}> k}\int_{\setR^d}J(x,y)\,dx\,dy}\lesssim k^{-2s}\underset{k\rightarrow \infty}{\longrightarrow} 0.
\end{align*}
Then \eqref{eq:doublelimit0} holds and we conclude the proof.
\end{proof}

Now we define the standard mollifier $\overline{\phi}(x)\in\mathcal{S}(\setR^d)$ as
\begin{align}\label{eq:mol}
\overline{\phi}(x)=
\begin{cases}
\overline{c}\exp\left(-\frac{1}{1-\abs{x}^2}\right)&\quad\text{if }\abs{x}< 1\\
0&\quad\text{if }\abs{x}\geq 1,
\end{cases}
\end{align}
where $\overline{c}$ is defined with the following property $\int_{\setR^d}\overline{\phi}(x)\,dx=1$.
Also, we denote 
\begin{align*}
\phi_{\epsilon}(x)=\frac{1}{\epsilon^{d}}\overline{\phi}\left(\frac{x-e_1}{\epsilon}\right)
\end{align*} 
for $\epsilon\in(0,1)$. Then we prove the following lemma, which rigorously establishes the connection between $f_1(d,s,\delta)$ in \eqref{eq:f2} and $g_1\divideontimes g_2$.
\begin{lemma}\label{lem:A20}
We have
$f_1(d,s,\delta)=\lim_{\epsilon\rightarrow 0}\skp{g_1\divideontimes g_2}{\phi_{\epsilon}}$.
\end{lemma}
\begin{proof}
With $\epsilon\in(0,\frac{1}{10})$, we observe that
\begin{align*}
\skp{g_1\divideontimes g_2}{\phi_{\epsilon}}&=\skp{g_1(x)}{\skp{g_2(y)}{\phi(x+y)}_y}_x\\
&=\dfrac{1}{2s(d+2s-2)}\int_{\setR^d}\int_{\setR^d}\abs{y}^{-d-2s+2}\abs{x}^{-\delta}x_1\Delta_y\phi_{\epsilon}(x+y)\,dx\,dy.
\end{align*}

Before starting the proof, we explain the strategy. We want to apply the argument in the proof of \cite[Page 714, Theorem 7]{Eva10}. But since $g_1$ and $g_2$ are distributions, and moreover, for $g_2$ the integrand $\abs{x}^{-d-2s+2}$ in Definition \ref{def:g1g2} is different from the expression of $g_2\equiv \abs{x}^{-d-2s}$ when $\abs{x}>0$. To overcome this issue, we first use divergence theorem to send the Laplacian $\Delta$ applied for the test function $\phi(x)$ to $\abs{x}^{-d-2s+2}$, to recover the expression $\abs{x}^{-d-2s}$ (See \textbf{Step 1}). Then we use the argument of \cite[Page 714, Theorem 7]{Eva10} (See \textbf{Step 2}).

We divide the proof into two steps.

\textbf{Step 1: Modifying the formula of $\skp{g_1\divideontimes g_2}{\phi_{\epsilon}}$ using divergence theorem.}
In this step we will show that
\begin{align*}
&\dfrac{1}{2s(d+2s-2)}\int_{\setR^d}\int_{\setR^d}\abs{y}^{-d-2s+2}\abs{x}^{-\delta}x_1\Delta_y\phi_{\epsilon}(x+y)\,dx\,dy\\
&\quad=\int_{\setR^d}\int_{\setR^d}\abs{y}^{-d-2s}\abs{x}^{-\delta}x_1(\phi_{\epsilon}(x+y)-\phi_{\epsilon}(x)-y\cdot\nabla\phi_{\epsilon}(x))\,dx\,dy
\end{align*}
holds. For $\sigma\in(0,\frac{1}{10})$ with $\sigma\leq\epsilon$, and a cross-shaped area
\begin{align*}
\mathcal{C}_{\sigma}:=(\{x\in\setR^d:\abs{x}\leq\sigma\}\times\setR^d)\cup(\setR^d\times\{y\in\setR^d:\abs{y}\leq \sigma\}),
\end{align*}
we will first show
\begin{align}\label{eq:c.sigma}
\lim_{\sigma\rightarrow 0}\iint_{\mathcal{C}_{\sigma}}\abs{y}^{-d-2s+2}\abs{x}^{-\delta}x_1\Delta_y\phi_{\epsilon}(x+y)\,dx\,dy=0.
\end{align}
To this end, we observe that
\begin{align}\label{eq:c.sigma.d}
\begin{split}
&\Bigabs{\iint_{\mathcal{C}_{\sigma}}\abs{y}^{-d-2s+2}\abs{x}^{-\delta}x_1\Delta_y\phi_{\epsilon}(x+y)\,dx\,dy}\\
&\quad\leq\iint_{\mathcal{D}_{\sigma,\epsilon}}\abs{y}^{-d-2s+2}\abs{x}^{1-\delta}\abs{\Delta_y\phi_{\epsilon}(x+y)}\,dx\,dy
\end{split}
\end{align}
with
\begin{align*}
\mathcal{D}_{\sigma,\epsilon}&:=(\{\abs{x}\leq\sigma\}\times\{1-\sigma-\epsilon\leq \abs{y}\leq 1+\sigma+\epsilon\})\\
&\quad\quad\quad\quad\quad\quad\cup(\{\abs{y}\leq\sigma\}\times\{1-\sigma-\epsilon\leq \abs{x}\leq 1+\sigma+\epsilon\}),
\end{align*}
where for the last inequality we have used the fact that since $\support\phi_{\epsilon}\subset B_{\epsilon}(e_1)$, on the set $(x,y)\in\mathcal{C}_{\sigma}$ we obtain $\support \Delta_y\phi_{\epsilon}(x+y)\subset \mathcal{D}_{\sigma,\epsilon}$.
Then we argue as follows: 
\begin{align*}
&\iint_{\mathcal{D}_{\sigma,\epsilon}}\abs{y}^{-d-2s+2}\abs{x}^{1-\delta}\abs{\Delta_y\phi_{\epsilon}(x+y)}\,dy\,dx\\
&\quad\leq c(d,s,\epsilon)\iint_{\{\abs{x}\leq\sigma\}\times\{1-\sigma-\epsilon\leq\abs{y}\leq 1+\sigma+\epsilon\}}\abs{y}^{-d-2s+2}\abs{x}^{1-\delta}\abs{\Delta_y\phi_{\epsilon}(x+y)}\,dy\,dx\\
&\quad\quad\quad+c(d,s,\epsilon)\iint_{\{\abs{y}\leq\sigma\}\times\{1-\sigma-\epsilon\leq\abs{x}\leq 1+\sigma+\epsilon\}}\abs{y}^{-d-2s+2}\abs{x}^{1-\delta}\abs{\Delta_y\phi_{\epsilon}(x+y)}\,dy\,dx\\
&\quad\leq c(d,s,\epsilon)(\sigma^{d-\delta+1}+\sigma^{2-2s})\leq c(d,s,\epsilon)\sigma^{2-2s}.
\end{align*}
Here, the $\epsilon$ dependence to the universal constant is from the term $\abs{\Delta_y\phi_{\epsilon}(x+y)}$. Thus for each fixed $\epsilon\in(0,\frac{1}{10})$, for any $\sigma\in(0,\epsilon]$ we have  found the above estimate. Now sending $\sigma\rightarrow 0$, together with \eqref{eq:c.sigma.d}, we obtain \eqref{eq:c.sigma}.

Now for $\mathcal{E}_\sigma=(\{\abs{y}\geq\sigma\}\times\{\abs{x}\geq1/\sigma\})\cup(\{\abs{x}\geq\sigma\}\times\{\abs{y}\geq 2/\sigma\})$, we have to show
\begin{align}\label{eq:e.sigma}
\lim_{\sigma\rightarrow 0}\iint_{\mathcal{E}_{\sigma}}\abs{y}^{-d-2s+2}\abs{x}^{-\delta}x_1\Delta_y\phi_{\epsilon}(x+y)\,dx\,dy=0.
\end{align}
To do this, note that on the set $\{\abs{y}\leq\sigma\}\times\{\abs{x}\geq1/\sigma\}$, we have $\phi_{\epsilon}(x+y)=0$ from $\support\phi_{\epsilon}\subset B_{\epsilon}(e_1)$. Then we have
\begin{align*}
&\iint_{\{\abs{y}\geq\sigma\}\times\{\abs{x}\geq1/\sigma\}}\abs{y}^{-d-2s+2}\abs{x}^{-\delta}x_1\Delta_y\phi_{\epsilon}(x+y)\,dx\,dy\\
&\quad=\int_{\setR^d}\int_{\abs{x}>\frac{1}{\sigma}}\abs{y}^{-d-2s+2}\abs{x}^{-\delta}x_1\Delta_y\phi_{\epsilon}(x+y)\,dy\,dx.
\end{align*}
Now by following the proof of Lemma \ref{lem:welldef}, as in \eqref{eq:y<Ry>R} with the choice of $R=\frac{1}{\sigma}$, when $\abs{x}\geq\frac{1}{\sigma}$
\begin{align*}
\begin{split}
\int_{\setR^d}\abs{x}^{-\delta}x_1\Delta_{y}\phi_{\epsilon}(x+y)\,dx\leq (d-\delta)\delta\abs{y}^{-\delta-2}y_1\int_{B_\frac{\abs{y}}{2}(0)}\Delta\phi_{\epsilon}(x)\,dx+\frac{c(d,s,\epsilon)}{\abs{y}^2}
\end{split}
\end{align*}
holds. Then following the remaining proof of Lemma \ref{lem:welldef}, especially from \eqref{eq;y>R0} and \eqref{eq;y>R2s} we obtain
\begin{align*}
\Bigabs{\iint_{\mathcal{E}_{\sigma}}\abs{y}^{-d-2s+2}\abs{x}^{-\delta}x_1\Delta_y\phi_{\epsilon}(x+y)\,dx\,dy}\leq c(d,s,\epsilon)\sigma^{2s}.
\end{align*}

On the set $\{x:\abs{x}\leq\sigma\}\times\{y:\abs{y}>2/\sigma\}$, we have $\phi_{\epsilon}(x+y)=0$ from $\support\phi_{\epsilon}\subset B_{\epsilon}(e_1)$. Then we have
\begin{align*}
&\iint_{\{\abs{x}\geq\sigma\}\times\{\abs{y}\geq 2/\sigma\}}\abs{y}^{-d-2s+2}\abs{x}^{-\delta}x_1\Delta_y\phi_{\epsilon}(x+y)\,dx\,dy\\
&\quad=\int_{\setR^d}\int_{\abs{y}>\frac{1}{\sigma}}\abs{y}^{-d-2s+2}\abs{x}^{-\delta}x_1\Delta_y\phi_{\epsilon}(x+y)\,dy\,dx.
\end{align*}
Now by following the proof of Lemma \ref{lem:welldef}, as in \eqref{eq:y<Ry>R} with the choice of $R=\frac{2}{\sigma}$, when $\abs{y}\geq\frac{2}{\sigma}$
\begin{align*}
\begin{split}
\int_{\setR^d}\abs{x}^{-\delta}x_1\Delta_{y}\phi_{\epsilon}(x+y)\,dx\leq (d-\delta)\delta\abs{y}^{-\delta-2}y_1\int_{B_\frac{\abs{y}}{2}(0)}\phi_{\epsilon}(x)\,dx+\frac{c(d,s,\epsilon)}{\abs{y}^2}
\end{split}
\end{align*}
holds. Then following the remaining proof of Lemma \ref{lem:welldef}, especially from \eqref{eq;y>R0} and \eqref{eq;y>R2s} we obtain
\begin{align*}
\Bigabs{\int_{\{\abs{y}>2/\sigma\}}\int_{\{\abs{x}\leq\sigma\}}\abs{y}^{-d-2s+2}\abs{x}^{-\delta}x_1\Delta_y\phi_{\epsilon}(x+y)\,dx\,dy}\leq c(d,s,\epsilon)\sigma^{2s}.
\end{align*}
For the integral on the set $\{y:\abs{y}\leq\sigma\}\times\{x:\abs{x}>1/\sigma\}$ is similar so we get
\begin{align*}
\Bigabs{\int_{\{\abs{y}\leq\sigma\}}\int_{\{\abs{x}\geq1/\sigma\}}\abs{y}^{-d-2s+2}\abs{x}^{-\delta}x_1\Delta_y\phi_{\epsilon}(x+y)\,dx\,dy}\leq c(d,s,\epsilon)\sigma^{2s}.
\end{align*}
Thus for each fixed $\epsilon\in(0,\frac{1}{10})$, for any $\sigma\in(0,\epsilon]$ we have found the above estimate. Now sending $\sigma\rightarrow 0$ establishes \eqref{eq:e.sigma}.

We are now ready to apply divergence theorem. Using
\begin{align*}
\Delta_y\phi_{\epsilon}(x+y)=\Delta_y(\phi_{\epsilon}(x+y)-\phi_{\epsilon}(x)-y\cdot\nabla\phi_{\epsilon}(x)),
\end{align*}
and $\nabla_y\phi_{\epsilon}(x+y)=\nabla_x\phi_{\epsilon}(x+y)$, for  
\begin{align*} 
\mathcal{F}_{\sigma}:=\{x\in\setR^{d}:\sigma\leq\abs{x}\leq 1/\sigma\}\times \{y\in\setR^{d}:\sigma\leq\abs{y}\leq 2/\sigma\},
\end{align*}
we have
\begin{align*}
&\dfrac{1}{2s(d+2s-2)}\iint_{\mathcal{F}_{\sigma}}\abs{y}^{-d-2s+2}\abs{x}^{-\delta}x_1\Delta_y\phi_{\epsilon}(x+y)\,dx\,dy\\
&\,\,=\iint_{\mathcal{F}_{\sigma}}\abs{y}^{-d-2s}\abs{x}^{-\delta}x_1\left(\phi_{\epsilon}(x+y)-\phi_{\epsilon}(x)-y\cdot\nabla\phi_{\epsilon}(x)\right)\,dx\,dy\\
&\quad-\frac{1}{2s(d+2s-2)}\int_{\partial\mathcal{F}_{\sigma}}\abs{y}^{-d-2s+2}\abs{x}^{-\delta}x_1\nu(x,y)\cdot\nabla_x(\phi_{\epsilon}(x+y)-\phi_{\epsilon}(x))\,dS\\
&\quad-\frac{1}{2s}\int_{\partial\mathcal{F}_{\sigma}}\left(\phi_{\epsilon}(x+y)-\phi_{\epsilon}(x)-y\cdot\nabla\phi_{\epsilon}(x)\right)\abs{x}^{-\delta}x_1\nu(x,y)\cdot\abs{y}^{-d-2s}y\,dS\\
&=:J_1+J_2(\partial\mathcal{F}_{\sigma})+J_3(\partial\mathcal{F}_{\sigma}).
\end{align*}
Here, $\nu(x,y)$ is the outward pointing unit normal vector, and
\begin{align*}
\partial\mathcal{F}_{\sigma}&=\{(x,y)\in\setR^{2d}:\abs{x}=\sigma,\sigma\leq\abs{y}\leq 2/\sigma\}\\
&\quad\cup\{(x,y)\in\setR^{2d}:\abs{x}=1/\sigma,\sigma\leq\abs{y}\leq 2/\sigma\}\\
&\quad\cup\{(x,y)\in\setR^{2d}:\abs{y}=\sigma,\sigma\leq\abs{x}\leq 1/\sigma\}\\
&\quad\cup\{(x,y)\in\setR^{2d}:\abs{y}=2/\sigma,\sigma\leq\abs{x}\leq 1/\sigma\}\\
&:=I_1+I_2+I_3+I_4.
\end{align*}
On $I_1$, using $\phi_{\epsilon}(x)\equiv 0$ on $\abs{x}=\sigma$, $\abs{\nu(x,y)}=1$, and $\support\phi_{\epsilon}\subset B_{\epsilon}(e_1)$, we have
\begin{align*}
\abs{J_2(I_1)}&\leq\biggabs{\int_{I_1}\abs{y}^{-d-2s+2}\abs{x}^{-\delta}x_1\nu(x,y)\nabla_x\phi_{\epsilon}(x+y)\,dS}\\
&\leq \int_{I_1}\abs{y}^{-d-2s+2}\abs{x}^{-\delta+1}\abs{\nabla_x \phi_{\epsilon}(x+y)}\,dS\\
&\leq \int_{\abs{x}=\sigma}\int_{1-\sigma-\epsilon\leq\abs{y}\leq 1+\sigma+\epsilon}\abs{y}^{-d-2s+2}\abs{x}^{-\delta+1}\abs{\nabla_x \phi_{\epsilon}(x+y)}\,dS\\
&\leq c(d,s,\epsilon)\sigma^{d-\delta}
\end{align*}
and
\begin{align*}
\abs{J_3(I_1)}&\lesssim \biggabs{\int_{I_1}\phi_{\epsilon}(x+y)\abs{x}^{-\delta}x_1\nu(x,y)\cdot\abs{y}^{-d-2s}y\,dS}\\
&\lesssim \int_{\abs{x}=\sigma}\int_{1-\sigma-\epsilon\leq\abs{y}\leq 1+\sigma+\epsilon}\abs{x}^{1-\delta}\abs{y}^{-d-2s+1}\,dS\\
&\leq c(d,s,\epsilon)\sigma^{d-\delta}.
\end{align*}
For the integral on the set $I_2$, using divergence theorem again, together with $\nu(x,y)=\frac{x}{\abs{x}}$ on $I_2$,
\begin{align*}
\divergence\left(\abs{x}^{-\delta}x_1\frac{x}{\abs{x}}\right)=(d-\delta)\abs{x}^{-\delta-1}x_1,
\end{align*}
and $\phi_{\epsilon}(x)\equiv 0$ on $\abs{x}=\frac{1}{\sigma}$, we write
\begin{align*}
-2s(d+2s-2)&J_2(I_2)=\int_{I_2}\abs{y}^{-d-2s+2}\abs{x}^{-\delta}x_1\nu(x,y)\cdot\nabla_x\phi_{\epsilon}(x+y)\,dS\\
&=-(-d-2s+2)(d-\delta)\int_{I_2}\abs{y}^{-d-2s+2}\abs{x}^{-\delta-1}x_1\phi_{\epsilon}(x+y)\,dS.
\end{align*}
Here, we have used the fact that on 
\begin{align*}
\partial(I_2)=\{\abs{x}=1/\sigma,\abs{y}=\sigma\}\cup\{\abs{x}=1/\sigma,\abs{y}=2/\sigma\},
\end{align*}
since $\support\phi_{\epsilon}\subset B_{\epsilon}(e_1)$, $\phi_{\epsilon}(x+y)\equiv 0$ on $\partial(I_2)$. Then we have
\begin{align}\label{eq:I_2}
\begin{split}
&\int_{I_2}\abs{y}^{-d-2s+2}\abs{x}^{-\delta-1}x_1\phi_{\epsilon}(x+y)\,dS\\
&=\int_{\abs{x}=\frac{1}{\sigma}}\int_{\sigma\leq\abs{y-x-e_1}\leq\frac{2}{\sigma}}\abs{y-x-e_1}^{-d-2s+2}\abs{x}^{-\delta-1}x_1\phi_{\epsilon}(y-e_1)\,dS.
\end{split}
\end{align}

Now, since $x\mapsto \abs{x}^{-d-2s+2}\abs{x}^{-\delta-1}x_1$ is odd and $\phi_{\epsilon}(y-e_1)=\phi_{\epsilon}(-y+e_1)$, note that
\begin{align}\label{eq:x=1/sigma}
\int_{\abs{x}=\frac{1}{\sigma}}\int_{\sigma\leq\abs{y-x-e_1}\leq\frac{2}{\sigma}}\abs{x}^{-d-2s+2}\abs{x}^{-\delta-1}x_1\phi_{\epsilon}(y-e_1)\,dS=0.
\end{align}
Then for $\abs{y-e_1}\leq\epsilon\leq\frac{1}{2\sigma}\leq 2\abs{x}$, using
\begin{align*}
\abs{\abs{y-x-e_1}^{-d-2s+2}+\abs{x}^{-d-2s+2}}\lesssim \abs{x}^{-d-2s+1}\abs{y-e_1}
\end{align*}
by Lemma \ref{lem:basic} and using $\support\phi_{\epsilon}(y-e_1)\subset B_{1}(e_1)$, we obtain that
\begin{align*}
&\bigg|\int_{\abs{x}=\frac{1}{\sigma}}\int_{\sigma\leq\abs{y-x-e_1}\leq\frac{2}{\sigma}}\abs{y-x-e_1}^{-d-2s+2}\abs{x}^{-\delta-1}x_1\phi_{\epsilon}(y-e_1)\,dS\\
&\quad +\int_{\abs{x}=\frac{1}{\sigma}}\int_{\sigma\leq\abs{y-x-e_1}\leq\frac{2}{\sigma}}\abs{x}^{-d-2s+2}\abs{x}^{-\delta-1}x_1\phi_{\epsilon}(y-e_1)\,dS\bigg|\\
&\quad\lesssim \int_{\abs{x}=\frac{1}{\sigma}}\int_{\sigma\leq\abs{y-x-e_1}\leq\frac{2}{\sigma}}\abs{x}^{-d-2s+1}\abs{y-e_1}\abs{x}^{-\delta}\phi_{\epsilon}(y-e_1)\,dS\\
&\quad\lesssim \sigma^{-d+1}\sigma^{d+2s-1}\sigma^{\delta}\lesssim \sigma^{2s+\delta}.
\end{align*}
Thus, we have found
\begin{align*}
J_2(I_2)\leq c(d,s,\epsilon) \sigma^{2s+\delta}.
\end{align*}
Also, for $J_3$ on $I_2$, similar to \eqref{eq:I_2}, using \eqref{eq:x=1/sigma} and
\begin{align*}
\abs{\abs{y-x-e_1}^{-d-2s}(y-x-e_1)+\abs{x}^{-d-2s}x}\lesssim \abs{x}^{-d-2s}\abs{y-e_1},
\end{align*}
together with $\support\phi_{\epsilon}(\cdot)\subset B_{\epsilon}(0)$, we obtain
\begin{align*}
\abs{J_3(I_2)}\lesssim \biggabs{\int_{I_2}\phi_{\epsilon}(x+y)\abs{x}^{-\delta}x_1\nu(x,y)\cdot\abs{y}^{-d-2s}y\,dS}\leq c(d,s,\epsilon)\sigma^{2s+\delta}.
\end{align*}

For the integral on the set $I_3$, using $\support\phi_{\epsilon}\subset B_{\epsilon}(e_1)$, $\sigma\leq\epsilon$, and the fact that $\abs{\nabla_x(\phi_{\epsilon}(x+y)-\phi_{\epsilon}(x))}\lesssim \abs{y}$, we obtain
\begin{align*}
\abs{J_2(I_3)}&=\biggabs{\int_{\abs{y}=\sigma}\int_{\sigma\leq\abs{x}\leq\frac{1}{\sigma}}\abs{y}^{-d-2s+2}\abs{x}^{-\delta}x_1\nu(x,y)\nabla_x(\phi_{\epsilon}(x+y)-\phi_{\epsilon}(x))\,dS}\\
&\lesssim c(d,s,\epsilon)\int_{\abs{y}=\sigma}\int_{1-\epsilon-\sigma\leq\abs{x}\leq1+\epsilon+\delta}\abs{y}^{-d-2s+3}\,dS\lesssim \sigma^{d-1}\sigma^{-d-2s+3}\lesssim \sigma^{2-2s}.
\end{align*}
Also, for $J_3$, using $\support\phi_{\epsilon}\subset B_{\epsilon}(e_1)$, and then $\abs{\phi_{\epsilon}(x+y)-\phi_{\epsilon}(x)-y\cdot\nabla\phi_{\epsilon}(x)}\lesssim\abs{y}^2$ and $\nu(y)=-\frac{y}{\abs{y}}$, we estimate
\begin{align*}
&\abs{J_3(I_3)}\lesssim\biggabs{\int_{\abs{y}=\sigma}\int_{\sigma\leq\abs{x}\leq\frac{1}{\sigma}}\abs{\phi_{\epsilon}(x+y)-\phi_{\epsilon}(x)-y\cdot\nabla\phi_{\epsilon}(x)}\abs{x}^{-\delta+1}\abs{y}^{-d-2s+1}\,dS}\\
&\quad\lesssim\int_{\abs{y}=\sigma}\int_{1-\epsilon-\sigma\leq\abs{x}\leq1+\epsilon+\sigma}\abs{\phi_{\epsilon}(x+y)-\phi_{\epsilon}(x)-y\cdot\nabla\phi_{\epsilon}(x)}\abs{x}^{-\delta+1}\abs{y}^{-d-2s+1}\,dS\\
&\quad\lesssim c(d,s,\epsilon)\sigma^{-d-2s+3}\sigma^{d-1}\lesssim \sigma^{2-2s}.
\end{align*}

Finally, we consider the integral on $I_4$. Using $\nu(x,y)=\frac{y}{\abs{y}}$ and divergence theorem, together with 
\begin{align}\label{eq:nabla.x}
\nabla(\abs{x}^{-\delta}x_1)=-\delta\abs{x}^{-\delta-2}x_1x+\abs{x}^{-\delta}e_1,
\end{align}
and $\phi_{\epsilon}(x+y)=0$ on $I_4$, we have
\begin{align*}
&-2s(d+2s-2)J_2(I_4)\\
&\quad=\int_{\abs{y}=\frac{2}{\sigma}}\int_{\sigma\leq\abs{x}\leq\frac{1}{\sigma}}\abs{y}^{-d-2s+2}\abs{x}^{-\delta}x_1\nu(x,y)\cdot\nabla_x(\phi_{\epsilon}(x+y)-\phi_{\epsilon}(x))\,dS\\
&\quad=-\int_{\abs{y}=\frac{2}{\sigma}}\int_{\sigma\leq\abs{x}\leq\frac{1}{\sigma}}\abs{y}^{-d-2s+1}y(-\delta\abs{x}^{-\delta-2}x_1x+\abs{x}^{-\delta}e_1)\phi_{\epsilon}(x)\,dS.
\end{align*}
Here, we have used the fact that on 
\begin{align*}
\partial(I_4)=\{\abs{y}=2/\sigma,\abs{x}=\sigma\}\cup\{\abs{y}=2/\sigma,\abs{x}=1/\sigma\},
\end{align*}
since $\support\phi_{\epsilon}\subset B_{\epsilon}(e_1)$, $\phi_{\epsilon}(x+y)\equiv 0$ on $\partial(I_4)$. Then we have
\begin{align}\label{eq:I_4}
\int_{\abs{y}=\frac{2}{\sigma}}\int_{\sigma\leq\abs{x}\leq\frac{1}{\sigma}}\abs{y}^{-d-2s+1}y(-\delta\abs{x}^{-\delta-2}x_1x+\abs{x}^{-\delta}e_1)\phi_{\epsilon}(x)\,dS=0,
\end{align}
since 
\begin{align}\label{eq:y=2/sigma}
\int_{\abs{y}=\frac{2}{\sigma}}\abs{y}^{-d-2s+1}y\,dS=0.
\end{align}
Also, for $J_3$, recalling $\nu(x,y)=\frac{y}{\abs{y}}$, and using \eqref{eq:y=2/sigma}, $\phi_{\epsilon}(x+y)=0$ on $I_4$, and $\support\phi_{\epsilon}\subset B_{\epsilon}(e_1)$, we observe
\begin{align*}
&\abs{J_3(I_4)}\\
&=\biggabs{\frac{1}{2s}\int_{\abs{y}=\frac{2}{\sigma}}\int_{\sigma\leq\abs{x}\leq\frac{1}{\sigma}}(\phi_{\epsilon}(x+y)-\phi_{\epsilon}(x)-y\cdot\nabla\phi_{\epsilon}(x))\abs{x}^{-\delta}x_1\nu(x,y)\cdot\abs{y}^{-d-2s}y\,dS}\\
&=\frac{1}{2s}\int_{\abs{y}=\frac{2}{\sigma}}\int_{1-\epsilon\leq\abs{x}\leq1+\epsilon}\phi_{\epsilon}(x)\abs{x}^{-\delta+1}\cdot\abs{y}^{-d-2s+1}\,dS\lesssim \sigma^{d+2s-1}\sigma^{-d+1}\lesssim\sigma^{2s}.
\end{align*}
Then similar to \eqref{eq:I_4}, we obtain $J_3\lesssim \sigma^{2s}$ in this case.

Overall, recalling $\mathcal{C}_\sigma$, $\mathcal{E}_\sigma$, and $\mathcal{F}_\sigma$, we have
\begin{align*}
&\dfrac{1}{2s(d+2s-2)}\int_{\setR^d}\int_{\setR^d}G_{\epsilon}(x,y)\,dx\,dy\\
&\quad=\lim_{\sigma\rightarrow 0}\left(\int_{\mathcal{C}_\sigma}G_{\epsilon}(x,y)\,dx\,dy+\int_{\mathcal{E}_\sigma}G_{\epsilon}(x,y)\,dx\,dy\right)\\
&\quad\quad+\lim_{\sigma\rightarrow 0}\left(\int_{\mathcal{F}_\sigma}\abs{y}^{-d-2s}\abs{x}^{-\delta}x_1(\phi_{\epsilon}(x+y)-\phi_{\epsilon}(x)-y\cdot\nabla\phi_{\epsilon}(x))\,dx\,dy\right)\\
&\quad=\int_{\setR^d}\int_{\setR^d}\abs{y}^{-d-2s}\abs{x}^{-\delta}x_1(\phi_{\epsilon}(x+y)-\phi_{\epsilon}(x)-y\cdot\nabla\phi_{\epsilon}(x))\,dx\,dy,
\end{align*}
where $G_{\epsilon}(x,y)=\abs{y}^{-d-2s+2}\abs{x}^{-\delta}x_1\Delta_y\phi_{\epsilon}(x+y)$.

\textbf{Step 2: Proof of \eqref{eq:approx.identity}.} We need to show 
\begin{align}\label{eq:Z0}
\lim_{\epsilon\rightarrow 0}\int_{\setR^d}\int_{\setR^d}\abs{y}^{-d-2s}\abs{x}^{-\delta}x_1(\phi_{\epsilon}(x+y)-\phi_{\epsilon}(x)-y\cdot\nabla\phi_{\epsilon}(x))\,dx\,dy=f_1(d,s,\delta).
\end{align}
To do this, for fixed $\kappa\in(0,1/100)$, consider $\epsilon\in(0,1/100)$ with $2\epsilon\leq\kappa$ and define
\begin{align*}
&\mathcal{G}_{\kappa}=\{(x,y)\in\setR^{2d}:\abs{y}\leq\kappa\},\\
&\mathcal{H}_{\kappa}=\{(x,y)\in\setR^{2d}:\abs{e_1-y}\leq\kappa\},\\
&\mathcal{I}_{\kappa}=\setR^{2d}\setminus(\mathcal{G}_{\kappa}\cup\mathcal{H}_{\kappa}).
\end{align*}
We first estimate the integration on $\mathcal{G}_{\kappa}$. Using $\support\phi\subset B_{\epsilon}(e_1)$ and the fact that $y\mapsto\abs{y}^{-d-2s}y$ is odd, and for $q(y)=\abs{x+e_1-y}^{-\delta}(x_1+1-y_1)$ when $x\in B_{\epsilon}(0)$ and $y\in B_{\kappa}(0)$,
\begin{align*}
\abs{q(y)-q(0)-y\cdot\nabla q(0)}\lesssim\abs{x-e_1}^{-\delta-1}\abs{y}^2
\end{align*}
holds, so using change of variables, we have
\begin{align*}
&\iint_{\mathcal{G}_{\kappa}}\abs{y}^{-d-2s}\abs{x}^{-\delta}x_1(\phi_{\epsilon}(x+y)-\phi_{\epsilon}(x)-y\cdot\nabla\phi_{\epsilon}(x))\,dx\,dy\\
&\,\,=\iint_{\mathcal{G}_{\kappa}}\abs{y}^{-d-2s}\abs{x}^{-\delta}x_1(\phi_{\epsilon}(x+y)-\phi_{\epsilon}(x))\,dx\,dy\\
&\,\,=\frac{\overline{c}}{\epsilon^d}\iint_{\mathcal{G}_{\kappa}}\abs{y}^{-d-2s}\left(q(y)-q(0)-y\cdot\nabla q(0)\right)\overline{\phi}\left(\frac{x}{\epsilon}\right)\,dx\,dy\\
&\,\,\lesssim \int_{\abs{y}\leq\kappa}\abs{y}^{-d-2s+2}\dashint_{\abs{x}\leq\epsilon}\overline{\phi}\left(\frac{x}{\epsilon}\right)\,dx\,dy\lesssim \kappa^{2-2s}.
\end{align*}

Now we consider the integral on the set $\mathcal{H}_{\kappa}$. Let us write
\begin{align*}
&\iint_{\mathcal{H}_{\kappa}}\abs{y}^{-d-2s}\abs{x}^{-\delta}x_1(\phi_{\epsilon}(x+y)-\phi_{\epsilon}(x)-y\cdot\nabla\phi_{\epsilon}(x))\,dx\,dy\\
&\,\,=\iint_{\mathcal{H}_{\kappa}}\abs{y}^{-d-2s}\abs{x}^{-\delta}x_1\phi_{\epsilon}(x+y)\,dx\,dy\\
&\quad-\iint_{\mathcal{H}_{\kappa}}\abs{y}^{-d-2s}\abs{x}^{-\delta}x_1\phi_{\epsilon}(x)\,dx\,dy-\iint_{\mathcal{H}_{\kappa}}\abs{y}^{-d-2s}\abs{x}^{-\delta}x_1y\cdot\nabla\phi_{\epsilon}(x)\,dx\,dy\\
&\,\,=:K_1+K_2+K_3.
\end{align*}
For $K_1$, we compute
\begin{align*}
&K_1=\int_{\abs{e_1-y}\leq\kappa}\int_{\abs{x+y-e_1}\leq\epsilon}\abs{y}^{-d-2s}\abs{x}^{-\delta}x_1\phi_{\epsilon}(x+y)\,dx\,dy\\
&=\dfrac{\overline{c}}{\epsilon^d}\int_{\abs{e_1-y}\leq\kappa}\int_{\abs{x}\leq\epsilon}\abs{y}^{-d-2s}\abs{x-y+e_1}^{-\delta}(x_1-y_1+1)\overline{\phi}\left(\frac{x}{\epsilon}\right)\,dx\,dy\\
&=\dfrac{\overline{c}}{\epsilon^d}\int_{\abs{e_1-y}\leq\kappa}\abs{y}^{-d-2s}\hspace{-6mm}\int_{\frac{1}{2}\abs{e_1-y}\leq\abs{x}\leq\epsilon}\hspace{-6mm}\left(\abs{x-y+e_1}^{-\delta}(x_1-y_1+1)-\abs{x}^{-\delta}x_1\right)\overline{\phi}\left(\frac{x}{\epsilon}\right)\,dx\,dy\\
&\,\,+\dfrac{\overline{c}}{\epsilon^d}\int_{\abs{e_1-y}\leq\kappa}\int_{\abs{x}\leq\min\{\frac{1}{2}\abs{e_1-y},\epsilon\}}\abs{y}^{-d-2s}\abs{x-y+e_1}^{-\delta}(x_1-y_1+1)\overline{\phi}\left(\frac{x}{\epsilon}\right)\,dx\,dy\\
&\lesssim\int_{\abs{e_1-y}\leq\kappa}\abs{y-e_1}^{1-\delta}\dashint_{\abs{x}\leq\epsilon}\overline{\phi}\left(\frac{x}{\epsilon}\right)\,dx\,dy\lesssim\kappa^{d+1-\delta}\lesssim\kappa^d,
\end{align*}
where for the first inequality we have used that 
\begin{align*}
\dashint_{\frac{1}{2}\abs{e_1-y}\leq\abs{x}\leq\epsilon}\abs{x}^{-\delta}x_1\overline{\phi}\left(\frac{x}{\epsilon}\right)\,dx=0
\end{align*}
as well as that on the set $\frac{1}{2}\abs{e_1-y}\leq\abs{x}\leq\epsilon$, by Lemma \ref{lem:basic}, we have
\begin{align*}
\abs{\abs{x-y+e_1}^{-\delta}(x_1-y_1+1)-\abs{x}^{-\delta}x_1}\lesssim \abs{x}^{-\delta}\abs{y-e_1},
\end{align*}
and on the set $\abs{x}\leq\min\{\frac{1}{2}\abs{e_1-y},\epsilon\}$, $\abs{\abs{x-y+e_1}^{-\delta}(x_1-y_1+1)}\lesssim \abs{y-e_1}^{1-\delta}$ holds.

Also, for $K_2$, we estimate
\begin{align*}
&\int_{\abs{e_1-y}\leq\kappa}\int_{\abs{x-e_1}\leq\epsilon}\abs{y}^{-d-2s}\abs{x}^{-\delta}x_1\phi_{\epsilon}(x)\,dx\,dy\\
&=\dfrac{\overline{c}}{\epsilon^d}\int_{\abs{e_1-y}\leq\kappa}\int_{\abs{x}\leq\epsilon}\abs{y}^{-d-2s}\abs{x+e_1}^{-\delta}(x_1+1)\overline{\phi}\left(\frac{x}{\epsilon}\right)\,dx\,dy\\
&\lesssim \int_{\abs{e_1-y}\leq\kappa}\dashint_{\abs{x}\leq\epsilon}\overline{\phi}\left(\frac{x}{\epsilon}\right)\,dx\,dy\leq\kappa^d,
\end{align*}
and for $K_3$, by divergence theorem, using \eqref{eq:nabla.x} we have
\begin{align*}
K_3&=\int_{\abs{e_1-y}\leq\kappa}\int_{\setR^d}\abs{y}^{-d-2s}y\left(-\delta\abs{x}^{-\delta-2}x_1x+\abs{x}^{-\delta}e_1\right)\phi_{\epsilon}(x)\,dx\,dy\\
&=\int_{\abs{e_1-y}\leq\kappa}\int_{\abs{x-e_1}\leq\epsilon}\abs{y}^{-d-2s}y\left(-\delta\abs{x}^{-\delta-2}x_1x+\abs{x}^{-\delta}e_1\right)\phi_{\epsilon}(x)\,dx\,dy\\
&\lesssim \int_{\abs{e_1-y}\leq\kappa}\dashint_{\abs{x}\leq\epsilon}\overline{\phi}\left(\frac{x}{\epsilon}\right)\,dx\,dy\lesssim\kappa^d.
\end{align*}
Hence we obtain
\begin{align*}
&\biggabs{\pvint_{\substack{\abs{y}\leq\kappa\\\cup\abs{e_1-y}\leq\kappa}}\abs{y}^{-d-2s}\abs{x}^{-\delta}x_1(\phi_{\epsilon}(x+y)-\phi_{\epsilon}(x)-y\cdot\nabla\phi_{\epsilon}(x))dy}\lesssim \kappa^d.
\end{align*}

Now on the set $\mathcal{I}_{\kappa}$, the integrand $\abs{y}^{-d-2s}\abs{x}^{-\delta}x_1(\phi_{\epsilon}(x+y)-\phi_{\epsilon}(x)-y\cdot\nabla\phi_{\epsilon}(x))$ is no longer singular, so we can apply the argument of the proof of \cite[Page 714, Theorem 7]{Eva10}. In detail, let us write
\begin{align*}
&\iint_{\mathcal{I}_{\kappa}}\abs{y}^{-d-2s}\abs{x}^{-\delta}x_1(\phi_{\epsilon}(x+y)-\phi_{\epsilon}(x)-y\cdot\nabla\phi_{\epsilon}(x))\,dx\,dy\\
&\,\,=\iint_{\mathcal{I}_{\kappa}}\abs{y}^{-d-2s}\abs{x}^{-\delta}x_1\phi_{\epsilon}(x+y)\,dx\,dy\\
&\quad-\iint_{\mathcal{I}_{\kappa}}\abs{y}^{-d-2s}\abs{x}^{-\delta}x_1\phi_{\epsilon}(x)\,dx\,dy-\iint_{\mathcal{I}_{\kappa}}\abs{y}^{-d-2s}\abs{x}^{-\delta}x_1y\cdot\nabla\phi_{\epsilon}(x)\,dx\,dy\\
&\,\,=:L_1+L_2+L_3.
\end{align*}
For $L_1$, recalling $\kappa\geq 2\epsilon$ and $\support\overline{\phi}\subset B_{1}(0)$, with \cite[Page 715]{Eva10} we obtain
\begin{align*}
L_1&=\int_{\substack{\abs{e_1-y}\geq\kappa\\\cap\abs{y}\geq\kappa}}\int_{\setR^d}\abs{y}^{-d-2s}\abs{x}^{-\delta}x_1\phi_{\epsilon}(x+y)\,dx\,dy\\
&=\int_{\substack{\abs{e_1-y}\geq\kappa\\\cap\abs{y}\geq\kappa}}\int_{\setR^d}\abs{y}^{-d-2s}\abs{x-y-e_1}^{-\delta}(x_1-y_1-1)\overline{\phi}\left(\frac{x}{\epsilon}\right)\,dx\,dy\\
&\quad\,\,\underset{\epsilon\rightarrow 0}{\longrightarrow}\,\,-\int_{\substack{\abs{e_1-y}\geq\kappa\\\cap\abs{y}\geq\kappa}}\abs{y}^{-d-2s}\abs{y+e_1}^{-\delta}(y_1+1)\,dy.
\end{align*}
For $L_2$, we estimate
\begin{align*}
L_2&=-\int_{\substack{\abs{e_1-y}\geq\kappa\\\cap\abs{y}\geq\kappa}}\int_{\setR^d}\abs{y}^{-d-2s}\abs{x}^{-\delta}x_1\phi_{\epsilon}(x)\,dx\,dy\\
&=-\int_{\substack{\abs{e_1-y}\geq\kappa\\\cap\abs{y}\geq\kappa}}\int_{\setR^d}\abs{y}^{-d-2s}\abs{x-e_1}^{-\delta}(x_1-1)\overline{\phi}\left(\frac{x}{\epsilon}\right)dx\,dy\underset{\epsilon\rightarrow 0}{\longrightarrow}\int_{\substack{\abs{e_1-y}\geq\kappa\\\cap\abs{y}\geq\kappa}}\abs{y}^{-d-2s}\,dy.
\end{align*}
Finally, for $L_3$, by divergence theorem,
\begin{align*}
L_3&=-\iint_{\mathcal{I}_{\kappa}}\abs{y}^{-d-2s}\abs{x}^{-\delta}x_1y\cdot\nabla\phi_{\epsilon}(x)\,dx\,dy\\
&=-\iint_{\mathcal{I}_{\kappa}}\abs{y}^{-d-2s}y\cdot\left(-\delta\abs{x}^{-\delta-2}x_1x+\abs{x}^{-\delta}e_1\right)\phi_{\epsilon}(x)\,dx\,dy\\
&=-\frac{\overline{c}}{\epsilon^d}\int_{\substack{\abs{e_1-y}\geq\kappa\\\cap\abs{y}\geq\kappa}}\int_{\setR^d}\abs{y}^{-d-2s}y\\
&\quad\quad\quad\cdot\left(-\delta\abs{x-e_1}^{-\delta-2}(x_1-1)(x-e_1)+\abs{x-e_1}^{-\delta}e_1\right)\overline{\phi}\left(\frac{x}{\epsilon}\right)\,dx\,dy\\
&\,\,\underset{\epsilon\rightarrow 0}{\longrightarrow}\,\,(\delta-1)\int_{\substack{\abs{e_1-y}\geq\kappa\\\cap\abs{y}\geq\kappa}}\abs{y}^{-d-2s}y\,dy.
\end{align*}
Then note that
\begin{align*}
\lim_{\epsilon\rightarrow 0}L_3=(\delta-1)\int_{\substack{\abs{e_1-y}\geq\kappa\\\cap\abs{y}\geq\kappa\\\cap\abs{e_1+y}\geq\kappa}}\abs{y}^{-d-2s}y\,dy+(\delta-1)\int_{\abs{e_1+y}\leq\kappa}\abs{y}^{-d-2s}y\,dy\lesssim \kappa^{d},
\end{align*}
since the first term in the right-hand side is zero by symmetry and $\abs{y}^{-d-2s}y$ is odd.

Therefore, we conclude
\begin{align*}
&\biggabs{\pvint_{\substack{\abs{y}\geq\kappa\\\cap\abs{e_1-y}\geq\kappa}}\hspace{-5mm}\abs{y}^{-d-2s}\abs{x}^{-\delta}x_1(\phi_{\epsilon}(x+y)-\phi_{\epsilon}(x)-y\cdot\nabla\phi_{\epsilon}(x))dy-L_1-L_2}\lesssim \kappa^{2-2s}
\end{align*}
together with
\begin{align*}
\lim_{\kappa\rightarrow 0}L_1+L_2=f_1(d,s,\delta).
\end{align*}
Since $\kappa\in(0,1/100)$ is arbitrary chosen and $\epsilon$ is already sent to zero. The proof is complete.
\end{proof}

For the remaining part of this subsection, we assume $\delta>0$. We have the following.

\begin{lemma}\label{lem:Fourier00}
For $\delta>0$, we have 
\begin{align}\label{eq:Fourier0}
\begin{split}
&\mathcal{F}[g_1](\xi)=-\pi^{-\frac{d}{2}+\delta-1}i\dfrac{\Gamma\left(\frac{d-\delta+2}{2}\right)}{\Gamma\left(\frac{\delta}{2}\right)}\abs{\xi}^{-d+\delta-2}\xi_1:=\mathcal{G}_1(\xi)\\&\quad \text{and} \quad \mathcal{F}[g_2](\xi)=\pi^{\frac{d}{2}+2s}\dfrac{\Gamma(-s)}{\Gamma\left(\frac{d+2s}{2}\right)}\abs{\xi}^{2s}:=\mathcal{G}_2(\xi)
\end{split}
\end{align}
in the distributional sense.
\end{lemma}
\begin{proof}
We compute
\begin{align*}
\mathcal{F}[g_1](\xi)=\dfrac{1}{2-\delta}\mathcal{F}[\partial_1\abs{x}^{2-\delta}]&=\dfrac{2\pi i\xi_1}{2-\delta}\mathcal{F}[\abs{x}^{2-\delta}]\\
&=\dfrac{2\pi i\xi_1}{2-\delta}\dfrac{(2\pi)^{\delta-2}\Gamma\left(\frac{d-\delta+2}{2}\right)}{\pi^{\frac{d}{2}}2^{\delta-2}\Gamma\left(\frac{\delta-2}{2}\right)}\abs{\xi}^{-d+\delta-2}\\
&=-\pi^{-\frac{d}{2}+\delta-1}i\dfrac{\Gamma\left(\frac{d-\delta+2}{2}\right)}{\Gamma\left(\frac{\delta}{2}\right)}\abs{\xi}^{-d+\delta-2}\xi_1
\end{align*}
and
\begin{align*}
\mathcal{F}[g_2](\xi)=\pi^{\frac{d}{2}+2s}\dfrac{\Gamma(-s)}{\Gamma\left(\frac{d+2s}{2}\right)}\abs{\xi}^{2s}.
\end{align*}
This completes the proof.
\end{proof}

Now with $\oldphi_\epsilon(x):=\frac{1}{\epsilon^d}\overline{\phi}\left(\frac{x}{\epsilon}\right)$ for $\epsilon\in(0,1)$ where $\overline{\phi}$ is a standard mollifier in $\setR^d$, we obtain the following lemma.

\begin{lemma}\label{lem:Fourier000}
For $\delta>0$, the multiplication of distributions $\mathcal{G}_1\cdot \mathcal{G}_2$ is well-defined in $x\in\setR^d$ in the sense that for any $\phi\in \mathcal{S}(\setR^d)$,
\begin{align*}
\skp{\mathcal{G}_1\cdot \mathcal{G}_2}{\phi}=\lim_{\epsilon\rightarrow 0}\bigskp{\skp{\mathcal{G}_1(y)}{\oldphi_{\epsilon}(\cdot-y)}_y\mathcal{G}_2(\cdot)}{\phi(\cdot)}
\end{align*}
exists. Moreover, we have
\begin{align}\label{eq:G1G2}
\lim_{\epsilon\rightarrow 0}\bigskp{\skp{\mathcal{G}_1(y)}{\oldphi_{\epsilon}(\cdot-y)}_y\mathcal{G}_2(\cdot)}{\phi(\cdot)}=\skp{\mathcal{G}_1\mathcal{G}_2}{\phi},
\end{align}
where 
\begin{align*}
\mathcal{G}_1\mathcal{G}_2=\pi^{\delta-1+2s}i\dfrac{\Gamma(-s)\Gamma\left(\frac{d-\delta+2}{2}\right)}{\Gamma\left(\frac{d+2s}{2}\right)\Gamma\left(\frac{\delta}{2}\right)}\abs{\xi}^{-d+\delta-2+2s}\xi_1
\end{align*}
means the usual multiplication.
\end{lemma}
\begin{proof}
Let $\phi\in\mathcal{S}(\setR^d)$ be given. We will show that
\begin{align*}
\lim_{\epsilon\rightarrow 0}\bigabs{\bigskp{\skp{\mathcal{G}_1(y)}{\oldphi_{\epsilon}(\cdot-y)}_y\mathcal{G}_2(\cdot)}{\phi(\cdot)}-\skp{\mathcal{G}_1\mathcal{G}_2(\cdot)}{\phi(\cdot)}}=0.
\end{align*}
For the proof, we use Lebesgue differentiation theorem. For $\kappa\geq 2\epsilon>0$, consider
\begin{align*}
&\biggabs{\int_{\abs{x}\leq\kappa}\left(\int_{\setR^d}\left[\abs{y}^{-d+\delta-2}y_1-\abs{x}^{-d+\delta-2}x_1\right]\oldphi_{\epsilon}(x-y)\,dy\right)\abs{x}^{2s}\phi(x)\,dx}\\
&\quad=\biggabs{\int_{\abs{x}\leq\kappa}\abs{x}^{-d+\delta-2+2s}x_1\phi(x)\,dx}\\
&\quad\quad+\biggabs{\int_{\abs{x}\leq\kappa}\left(\int_{\setR^d}\abs{y}^{-d+\delta-2}y_1\oldphi_{\epsilon}(x-y)\,dy\right)\abs{x}^{2s}\phi(x)\,dx}\\
&\quad=:\abs{I_1}+\abs{I_2}.
\end{align*}
Using the fact that $x\mapsto \abs{x}^{-d+\delta-2+2s}x_1\frac{\phi(x)-\phi(-x)}{2}$ is odd, and $\abs{\phi(x)-\phi(-x)}\lesssim \abs{x}$, note that
\begin{align}\label{eq:I12}
\begin{split}
\abs{I_1}&=\biggabs{\int_{\abs{x}\leq\kappa}\abs{x}^{-d+\delta-2+2s}x_1\frac{\phi(x)-\phi(-x)}{2}\,dx}\\
&=\int_{\abs{x}\leq\kappa}\abs{x}^{-d+\delta+2s}\,dx\lesssim\kappa^{\delta+2s}.
\end{split}
\end{align}

For $I_2$, using the fact that the map
\begin{align*}
x\mapsto \abs{x}^{2s}\dfrac{\phi(x)+\phi(-x)}{2}\left(\int_{\setR^d}\abs{y}^{-d+\delta-2}y_1\oldphi_{\epsilon}(x-y)\,dy\right)
\end{align*}
is odd since $\oldphi_{\epsilon}$ is even, we write
\begin{align*}
I_2&=\int_{\abs{x}\leq\kappa}\left(\int_{\setR^d}\abs{y}^{-d+\delta-2}y_1\oldphi_{\epsilon}(x-y)\,dy\right)\abs{x}^{2s}\frac{\phi(x)-\phi(-x)}{2}\,dx\\
&=\int_{\abs{x}\leq 2\epsilon}\left(\int_{\setR^d}\abs{y}^{-d+\delta-2}y_1\oldphi_{\epsilon}(x-y)\,dy\right)\abs{x}^{2s}\frac{\phi(x)-\phi(-x)}{2}\,dx\\
&\quad+\int_{2\epsilon\leq \abs{x}\leq\kappa}\left(\int_{\setR^d}\abs{y}^{-d+\delta-2}y_1\oldphi_{\epsilon}(x-y)\,dy\right)\abs{x}^{2s}\frac{\phi(x)-\phi(-x)}{2}\,dx.
\end{align*}

Let us estimate the first term in the right-hand side. We have
\begin{align*}
&\int_{\abs{x}\leq 2\epsilon}\left(\int_{\setR^d}\abs{y}^{-d+\delta-2}y_1\oldphi_{\epsilon}(x-y)\,dy\right)\abs{x}^{2s}\frac{\phi(x)-\phi(-x)}{2}\,dx\\
&=\int_{\abs{x}\leq 2\epsilon}\left(\int_{\abs{y}\leq\frac{1}{2}\abs{x}}\abs{y}^{-d+\delta-2}y_1\left(\oldphi_{\epsilon}(x-y)-\oldphi_{\epsilon}(x)\right)\,dy\right)\abs{x}^{2s}\frac{\phi(x)-\phi(-x)}{2}\,dx\\
&\quad+\int_{\abs{x}\leq 2\epsilon}\left(\int_{\abs{y}\geq\frac{1}{2}\abs{x}}\abs{y}^{-d+\delta-2}y_1\oldphi_{\epsilon}(x-y)\,dy\right)\abs{x}^{2s}\frac{\phi(x)-\phi(-x)}{2}\,dx,
\end{align*}
where we have used the fact that
\begin{align*}
\oldphi_{\epsilon}(x)\int_{\abs{y}\leq\frac{1}{2}\abs{x}}\abs{y}^{-d+\delta-2}y_1\,dy=0.
\end{align*}

Using $\abs{\oldphi_{\epsilon}(x-y)-\oldphi_{\epsilon}(x)}\lesssim \epsilon^{-1-d}\abs{y\cdot\nabla\overline{\phi}}$, we further compute
\begin{align*}
&\int_{\abs{x}\leq 2\epsilon}\left(\int_{\abs{y}\leq\frac{1}{2}\abs{x}}\abs{y}^{-d+\delta-2}y_1\left(\oldphi_{\epsilon}(x-y)-\oldphi_{\epsilon}(x)\right)\,dy\right)\abs{x}^{2s}\frac{\phi(x)-\phi(-x)}{2}\,dx\\
&\quad\lesssim\frac{1}{\epsilon}\dashint_{\abs{x}\leq 2\epsilon}\left(\int_{\abs{y}\leq\frac{1}{2}\abs{x}}\abs{y}^{-d+\delta}\,dy\right)\abs{x}^{2s+1}\,dx\lesssim \epsilon^{2s+\delta}\lesssim \kappa^{2s+\delta},
\end{align*}
where we have used $\abs{\phi(x)-\phi(-x)}\lesssim \abs{x}$. For the other term, using $\support\oldphi_{\epsilon}\subset B_{\epsilon}(0)$, we have
\begin{align*}
&\int_{\abs{x}\leq 2\epsilon}\left(\int_{\abs{y}\geq\frac{1}{2}\abs{x}}\abs{y}^{-d+\delta-2}y_1\oldphi_{\epsilon}(x-y)\,dy\right)\abs{x}^{2s}\frac{\phi(x)-\phi(-x)}{2}\,dx\\
&\quad\lesssim \int_{\abs{x}\leq 2\epsilon}\abs{x}^{2s+1}\epsilon^{\delta-1}\left(\int_{\abs{y}\leq\epsilon}\oldphi_{\epsilon}(y)\,dy\right)\,dx\lesssim \epsilon^{2s+\delta}\lesssim \kappa^{2s+\delta}.
\end{align*}
Note that the implicit constants in the above estimates do not depend on $\epsilon$. Also, using $\support\oldphi_{\epsilon}\subset B_{\epsilon}(0)$ and $\abs{\phi(x)-\phi(-x)}\lesssim \abs{x}$, we can estimate
\begin{align*}
&\int_{2\epsilon\leq \abs{x}\leq\kappa}\left(\int_{\setR^d}\abs{y}^{-d+\delta-2}y_1\oldphi_{\epsilon}(x-y)\,dy\right)\abs{x}^{2s}\frac{\phi(x)-\phi(-x)}{2}\,dx\\
&\quad=\int_{2\epsilon\leq \abs{x}\leq\kappa}\left(\int_{\abs{x-y}\leq\epsilon}\abs{y}^{-d+\delta-2}y_1\oldphi_{\epsilon}(x-y)\,dy\right)\abs{x}^{2s}\frac{\phi(x)-\phi(-x)}{2}\,dx\\
&\quad=\overline{c}\int_{2\epsilon\leq \abs{x}\leq\kappa}\left(\dashint_{\abs{y}\leq\epsilon}\abs{x-y}^{-d+\delta-2}(x_1-y_1)\overline{\phi}\left(\frac{y}{\epsilon}\right)\,dy\right)\abs{x}^{2s}\frac{\phi(x)-\phi(-x)}{2}\,dx\\
&\underset{\epsilon\rightarrow 0}{\longrightarrow} \overline{c}\int_{\abs{x}\leq\kappa}\abs{x}^{-d+\delta-1}\abs{x}^{2s}\frac{\phi(x)-\phi(-x)}{2}\,dx\\
&\quad\leq \int_{\abs{x}\leq\kappa}\abs{x}^{-d+\delta+2s}\,dx\lesssim \kappa^{\delta+2s}.
\end{align*}

Then merging the estimates, we have
\begin{align*}
\biggabs{\int_{\abs{x}\leq\kappa}\left(\int_{\setR^d}(\mathcal{G}_1(y)-\mathcal{G}_1(x))\oldphi_{\epsilon}(x-y)\,dy\right)\mathcal{G}_2(x)\phi(x)\,dx}\lesssim\kappa^{2s+\delta}.
\end{align*}

Now similar to the proof of Lemma \ref{lem:A20}, since on the set $\abs{x}\geq\kappa$ with $\kappa\geq 2\epsilon$, both $\mathcal{G}_1(x)$ and $\mathcal{G}_2(x)$ are no longer singular, by Lebesgue differentiation theorem we have
\begin{align*}
&\lim_{\epsilon\rightarrow 0}\biggabs{\int_{\abs{x}\geq\kappa}\left(\int_{\setR^d}(\mathcal{G}_1(y)-\mathcal{G}_1(x))\oldphi_{\epsilon}(x-y)\,dy\right)\mathcal{G}_2(x)\phi(x)\,dx}=0.
\end{align*}
Therefore, we obtain
\begin{align*}
&\lim_{\epsilon\rightarrow 0}\bigabs{\bigskp{\skp{\mathcal{G}_1(y)}{\oldphi_{\epsilon}(\cdot-y)}_y\mathcal{G}_2(\cdot)}{\phi(\cdot)}-\skp{\mathcal{G}_1\mathcal{G}_2(\cdot)}{\phi(\cdot)}}\\
&\quad\leq \lim_{\epsilon\rightarrow 0}\biggabs{\int_{\abs{x}\leq\kappa}\left(\int_{\setR^d}(\mathcal{G}_1(y)-\mathcal{G}_1(x))\oldphi_{\epsilon}(x-y)\,dy\right)\mathcal{G}_2(x)\phi(x)\,dx}\\
&\quad\quad +\lim_{\epsilon\rightarrow 0}\biggabs{\int_{\abs{x}\geq\kappa}\left(\int_{\setR^d}(\mathcal{G}_1(y)-\mathcal{G}_1(x))\oldphi_{\epsilon}(x-y)\,dy\right)\mathcal{G}_2(x)\phi(x)\,dx}\lesssim\kappa^{2s+\delta}.
\end{align*}
Since $\kappa$ was arbitrary, we obtain \eqref{eq:G1G2}.
\end{proof}

With the help of Lemma \ref{lem:Fourier00} and \ref{lem:Fourier000}, we have the following theorem.
\begin{theorem}\label{thm:Fourierconv0}
For $\delta>0$, we have 
\begin{align*}
\mathcal{F}[g_1\divideontimes g_2]=\mathcal{F}[g_1]\mathcal{F}[g_2]
\end{align*}
in the distributional sense, where $\mathcal{F}[g_1]\mathcal{F}[g_2]$ means the usual multiplication between $\mathcal{F}[g_1]$ and $\mathcal{F}[g_2]$.
\end{theorem}
\begin{proof}
For any $\phi\in\mathcal{S}(\setR^d)$, note that
\begin{align*}
\skp{\mathcal{F}[g_1\divideontimes g_2]}{\phi}&=\skp{\mathcal{F}[g_1]\cdot\mathcal{F}[g_2]}{\phi}\\
&=\lim_{\epsilon\rightarrow 0}\skp{(\mathcal{F}[g_1]\divideontimes\oldphi_{\epsilon})\mathcal{F}[g_2]}{\phi}\\
&=\lim_{\epsilon\rightarrow 0}\bigskp{\skp{\mathcal{F}[g_1](y)}{\oldphi_{\epsilon}(x-y)}_y\mathcal{F}[g_2](x)}{\phi(x)}_x\\
&=\skp{\mathcal{F}[g_1]\mathcal{F}[g_2]}{\phi},
\end{align*}
where for the first equality since $g_1,g_2\in\mathcal{S}'(\setR^d)$, we have used \cite[Page 96]{Vla02} together with Lemma \ref{lem:welldef} and Lemma \ref{lem:conv10}, for the second equality since $\mathcal{F}[g_1]\divideontimes\oldphi_{\epsilon}\in\theta_M$ as in \cite[Page 84]{Vla02}, where with multi index $\alpha=(\alpha_1,\alpha_2,\dots,\alpha_d)$,
\begin{align}\label{eq:thetaM}
\theta_{M}:=\{f\in C^{\infty}(\setR^d):\abs{\partial^{\alpha}f(x)}\leq c_{\alpha}(1+\abs{x})^{m_{\alpha}}\quad\text{for some }c_{\alpha},m_{\alpha}>0\}
\end{align}
(defined in \cite[Page 76]{Vla02}) and $\mathcal{F}[g_2]$ in $\mathcal{S}'(\setR^d)$, $(\mathcal{F}[g_1]\divideontimes\oldphi_{\epsilon})\mathcal{F}[g_2]\in\mathcal{S}'(\setR^d)$ holds as in \cite[Page 79]{Vla02}, for the third equality again since $\mathcal{F}[g_1]\divideontimes\oldphi_{\epsilon}\in\theta_M$, \cite[Page 84]{Vla02} is used, and for the last equality Lemma \ref{lem:Fourier000} together with Lemma \ref{lem:Fourier00} is used.
\end{proof}

Then it is worth mentioning the following corollary.
\begin{corollary}\label{cor:Fourierconv20}
For $\delta>0$,	we have $(\mathcal{F}\circ\mathcal{F})[g_1\divideontimes g_2]=\mathcal{F}\left[\mathcal{F}[g_1]\mathcal{F}[g_2]\right]$.
\end{corollary}
\begin{proof}
Since $\phi\in\mathcal{S}(\setR^d)$ implies $\mathcal{F}[\phi]\in\mathcal{S}(\setR^d)$, by Theorem \ref{thm:Fourierconv0} we have
\begin{align*}
\skp{\mathcal{F}[g_1\divideontimes g_2]}{\mathcal{F}[\phi]}=\skp{\mathcal{F}[g_1]\cdot\mathcal{F}[g_2]}{\mathcal{F}[\phi]}=\skp{\mathcal{F}[g_1]\mathcal{F}[g_2]}{\mathcal{F}[\phi]},
\end{align*}
which proves the conclusion.
\end{proof}

\subsection{Convolution theorem for $f_2(d,s,\delta)$}\label{sec:f2}

We define the distributions $g_3$ and $g_4$, where the (formal) evaluation $x=e_1$ to their convolution will recover $f_2(d,s,\delta)$. The definition is the same as Definition \ref{def:f3f4}.
\begin{definition}\label{def:g3g4}
For $\phi\in \mathcal{S}(\setR^d)$, we define the distributions $g_3$ and $g_4$ in $\setR^d$ as
\begin{align*}
g_3:=j_1-j_2\quad\text{and}\quad g_4:=-j_3+\tfrac{j_4}{2}+j_5-j_6+\tfrac{j_7}{2},
\end{align*}
where the distributions $j_1, j_2, \dots, j_7$ in $\setR^d$ are defined as
\begin{align*}
\skp{j_1}{\phi}:=\skp{\abs{h}^{-\delta-2}h_1}{\phi}=\displaystyle\int_{\setR^d}\abs{x}^{-\delta-2}x_1\phi(x)\,dx,
\end{align*}
\begin{align*}
\skp{j_2}{\phi}:=\skp{\abs{h}^{-2}}{\phi}=
\begin{cases}
-\displaystyle\int_{\setR^d}\dfrac{x\log\abs{x}}{\abs{x}^2}\nabla\phi(x)\,dx&\,\text{if }d=2,\\
\displaystyle\int_{\setR^d}\dfrac{\phi(x)}{\abs{x}^{2}}\,dx&\,\text{if }d\geq 3,
\end{cases}
\end{align*}
\begin{align*}
\skp{j_3}{\phi}&:=\skp{\abs{h}^{-d-2s-2}h_1^3}{\phi}\\
&=\int_{\setR^d}\left(\dfrac{-(-d-2s+2)x_1^2+2\abs{x}^{2}}{\abs{x}^{d+2s}(-d-2s+2)(-d-2s)}\right)\partial_1\phi(x)\,dx,
\end{align*}
\begin{align*}
\skp{j_4}{\phi}:=\skp{\abs{h}^{-d-2s}h_1^2}{\phi}=\int_{\setR^d}\abs{x}^{-d-2s}x_1^2\phi(x)\,dx,
\end{align*}
\begin{align*}
\skp{j_5}{\phi}&:=\skp{\abs{h}^{-d-2s-2}h_1^2}{\phi}\\
&=\int_{\setR^d}\dfrac{2s\abs{x}^{-d-2s+2}\partial_1^2\phi(x)-\abs{x}^{-d-2s+2}\Delta\phi(x)}{2s(-d-2s+2)(-d-2s)}\,dx,
\end{align*}
\begin{align*}
\skp{j_6}{\phi}:=\skp{\abs{h}^{-d-2s}h_1}{\phi}=\int_{\setR^d}\dfrac{-\abs{x}^{-d-2s+2}\partial_1\phi(x)}{-d-2s+2}\,dx,
\end{align*}
and 
\begin{align*}
\skp{j_7}{\phi}&:=\skp{\abs{x}^{-d-2s+2}}{\phi}=\int_{\setR^d}\abs{x}^{-d-2s+2}\phi(x)\,dx.
\end{align*}	

\end{definition}

Note that
\begin{itemize}
	\item One can easily check that the above distributions are well-defined since $\phi\in \mathcal{S}(\setR^d)$ implies $\abs{\partial^k_i\phi(x)}\lesssim \min\{1,\frac{1}{\abs{x}^3}\}$ for $k=0,1,2,3$ and $i=0,1,\dots,d$.
	\item Distributions $j_1$, $j_2$ when $d\geq 3$, $j_4$, and $j_7$ can be understood as locally integrable functions in $\setR^d$, since (especially for $j_1$) $\delta\in\left[0,\tfrac{1}{2}\right]$.
	\item Also, when $\abs{x}>0$, $j_2(x)\equiv -\divergence\left(\frac{x\log\abs{x}}{\abs{x}^2}\right)=\abs{x}^{-2}$ holds in case of $d=2$, $j_3(x)\equiv \abs{x}^{-d-2s-2}x_1^3$, $j_5(x)\equiv\abs{x}^{-d-2s-2}x_1^2$, and $j_6(x)\equiv \abs{x}^{-d-2s}x_1$. 
	\item Note that when $d=2$, for any $\phi\in \mathcal{S}(\setR^d)$ and any $\kappa\in(0,\infty)$, by divergence theorem we have
	\begin{align}\label{eq:j2}
	\begin{split}
	\skp{j_2}{\phi}&=-\displaystyle\int_{\abs{x}\leq \kappa}\dfrac{x\log\abs{x}}{\abs{x}^2}\nabla(\phi(x)-\phi(0))\,dx-\displaystyle\int_{\abs{x}\geq \kappa}\dfrac{x\log\abs{x}}{\abs{x}^2}\nabla\phi(x)\,dx\\
	&=\displaystyle\int_{\abs{x}\leq \kappa}\dfrac{\phi(x)-\phi(0)}{\abs{x}^2}\,dx+\displaystyle\int_{\abs{x}\geq \kappa}\dfrac{\phi(x)}{\abs{x}^2}\,dx+\abs{S^{1}}\log(\kappa)\phi(0),
	\end{split}
	\end{align}
	where $\abs{S^{1}}$ means the circumference of the $2$-dimensional unit sphere.
\end{itemize}

The main goal of this subsection is to show the following relation (see Theorem \ref{thm:Fourierconv}):
\begin{align}\label{eq:g3g4}
\mathcal{F}[g_3\divideontimes g_4]=\mathcal{F}[g_3]\mathcal{F}[g_4]
\end{align}
in the distributional sense, where $\divideontimes$ is a convolution of two distributions in $\mathcal{S}'(\setR^d)$ introduced in \cite[Page 96]{Vla02}, and $\mathcal{F}[g_3]\mathcal{F}[g_4]$ means the usual multiplication between $\mathcal{F}[g_3]$ and $\mathcal{F}[g_4]$. Note that from Lemma \ref{lem:Fourier0}$, \mathcal{F}[g_3]$ and $\mathcal{F}[g_4]$ are locally integrable functions in $\setR^d$. The well-definedness of the left-hand side term of \eqref{eq:g3g4} will be proved in Lemma \ref{lem:welldef3} and Lemma \ref{lem:conv3}, and the well-definedness of the right-hand side term of \eqref{eq:g3g4} will be proved in Lemma \ref{lem:Fourier.riesz}.

To show \eqref{eq:g3g4}, first we prove the following lemma.
\begin{lemma}\label{lem:welldef3}
For any $\phi\in \mathcal{S}(\setR^d)$, $\skp{g_4(y)}{\skp{g_3(x)}{\phi(x+y)}_x}_y$ is well-defined.
\end{lemma}
\begin{proof}
Let us choose $R\geq2$ and $\phi\in \mathcal{S}(\setR^d)$, and then divide the proof into three steps. Recall that $g_3:=j_1-j_2$ and $g_4:=-j_3+\tfrac{j_4}{2}+j_5-j_6+\tfrac{j_7}{2}$. Here and throughout the proof of this lemma, implicit generic constants $c$ in `$\lesssim$' will depend only on
\begin{align*}
d,s,\delta,\quad\text{and}\quad \norm{\nabla^i\phi}_{L^{\infty}(\setR^d)}\quad\text{for}\quad i=0,1,2,3.
\end{align*}
Note that for each $y\in\setR^d$, the integrand in defining $\skp{g_3(x)}{\phi(x+y)}_x$ is in $L^{1}(\setR^d)$ since $\phi\in\mathcal{S}(\setR^d)$.

\textbf{Step 1: Estimate for $j_2$ in the case of $d=2$. }When $\abs{y}\leq R$, we obtain
\begin{align*}
&\Biggabs{\int_{\setR^d}\dfrac{x\log\abs{x}}{\abs{x}^2}\nabla_x(\phi(x+y))\,dx}\\
&\leq\Biggabs{\int_{B_R(0)}\dfrac{x\log\abs{x}}{\abs{x}^2}\nabla_x(\phi(x+y))\,dx}+\Biggabs{\int_{\setR^d\setminus B_R(0)}\dfrac{x\log\abs{x}}{\abs{x}^2}\nabla_x(\phi(x+y))\,dx}\\
&\lesssim R\log R+\Biggabs{\int_{\setR^d\setminus B_R(0)}\dfrac{x\log\abs{x}}{\abs{x}^2}\nabla_x(\phi(x+y))\,dx},
\end{align*}
where for the last inequality we have used $\abs{\nabla_x(\phi(x+y))}\lesssim 1$ since $\phi\in \mathcal{S}(\setR^d)$. Here, if $-y\in B_{\frac{3}{4}R}(0)$, then employing $\abs{\nabla_x(\phi(x+y))}\lesssim \abs{x+y}^{-3}$ since $\phi\in \mathcal{S}(\setR^d)$,
\begin{align*}
\Biggabs{\int_{\setR^d\setminus B_R(0)}\dfrac{x\log\abs{x}}{\abs{x}^2}\nabla_x(\phi(x+y))\,dx}\lesssim\Biggabs{\int_{\setR^d\setminus B_R(0)}\dfrac{\log\abs{x}}{\abs{x}}\frac{1}{\abs{x+y}^{3}}\,dx}\lesssim\frac{\log R}{R^{2}}
\end{align*}
holds. On the other hand, if $-y\not\in B_{\frac{3}{4}R}(0)$ so that $B_{\frac{1}{4}R}(-y)\cap B_{\frac{1}{2}R}(0)=\emptyset$, then using $\abs{\nabla_x(\phi(x+y))}\lesssim \min\{1,\abs{x+y}^{-3}\}$,
\begin{align*}
&\Biggabs{\int_{\setR^d\setminus B_R(0)}\dfrac{x\log\abs{x}}{\abs{x}^2}\nabla_x(\phi(x+y))\,dx}\\
&\quad\leq \int_{B_{\frac{1}{4}R}(-y)}\dfrac{\log\abs{x}}{\abs{x}}\,dx+\int_{\setR^d\setminus (B_{\frac{1}{2}R}(0)\cup B_{\frac{1}{4}R}(-y))}\dfrac{\log\abs{x}}{\abs{x}\abs{x+y}^{3}}\,dx\\
&\quad \lesssim R\log R+\dfrac{\log R}{R}\cdot \frac{1}{R}\lesssim R\log R
\end{align*}
holds, since $R\geq 2$. Then if $\abs{y}\leq R$, we obtain
\begin{align}\label{eq:y<R3}
\Biggabs{\int_{\setR^d}\dfrac{x\log\abs{x}}{\abs{x}^2}\nabla_x(\phi(x+y))\,dx}\lesssim R\log R.
\end{align}

Now we consider the case of $\abs{y}\geq R$. Observe that
\begin{align*}
&\int_{\setR^d}\dfrac{x\log\abs{x}}{\abs{x}^2}\nabla_x(\phi(x+y))\,dx\\
&\quad=\int_{B_\frac{\abs{y}}{2}(-y)}\dfrac{x\log\abs{x}}{\abs{x}^2}\nabla_x(\phi(x+y))\,dx+\int_{\setR^d\setminus B_\frac{\abs{y}}{2}(-y)}\dfrac{x\log\abs{x}}{\abs{x}^2}\nabla_x(\phi(x+y))\,dx\\
&\quad=\int_{B_\frac{\abs{y}}{2}(-y)}\dfrac{1}{\abs{x}^2}\phi(x+y)\,dx+\int_{\partial B_\frac{\abs{y}}{2}(-y)}\dfrac{(x\cdot\nu(x,y))\log\abs{x}}{\abs{x}^2}\phi(x+y)\,dS\\
&\quad\quad\quad+\int_{\setR^d\setminus B_\frac{\abs{y}}{2}(-y)}\dfrac{x\log\abs{x}}{\abs{x}^2}\nabla_x(\phi(x+y))\,dx\\
&\quad=\int_{B_\frac{\abs{y}}{2}(0)}\dfrac{1}{\abs{x-y}^2}\phi(x)\,dx+\int_{\partial B_\frac{\abs{y}}{2}(0)}\dfrac{((x-y)\cdot x)\log\abs{x-y}}{\abs{x}\abs{x-y}^2}\phi(x)\,dS\\
&\quad\quad\quad+\int_{\setR^d\setminus B_\frac{\abs{y}}{2}(0)}\dfrac{(x-y)\log\abs{x-y}}{\abs{x-y}^2}\nabla\phi(x)\,dx,
\end{align*}
where for the second inequality we have used divergence theorem with the term $-\divergence\left(\frac{x\log\abs{x}}{\abs{x}^2}\right)=\abs{x}^{-2}$ for $x\neq 0$, and $\nu(x,y)=\widehat{x+y}$ means the outward pointing unit normal vector along $\partial B_{\frac{\abs{y}}{2}}(-y)$. Here, using $\abs{\phi(x)}\lesssim 1/\abs{x}^{2-s}$ from $\phi \in\mathcal{S}(\setR^d)$, and recalling $\abs{y}\geq R$, it holds that
\begin{align*}
\Biggabs{\int_{B_\frac{\abs{y}}{2}(0)}\dfrac{1}{\abs{x-y}^2}\phi(x)\,dx}\lesssim\int_{B_\frac{\abs{y}}{2}(0)}\dfrac{1}{\abs{x-y}^2}\frac{1}{\abs{x}^{2-s}}\,dx\lesssim \dfrac{1}{\abs{y}^2}\abs{y}^{s}\lesssim\frac{1}{\abs{y}^{2-s}}.
\end{align*}
Moreover, using $\abs{\phi(x)}\lesssim 1/\abs{x}^{2}$ from $\phi \in\mathcal{S}(\setR^d)$, and employing $\log\abs{y}\lesssim\abs{y}^{s}$,
\begin{align*}
&\Biggabs{\int_{\partial B_\frac{\abs{y}}{2}(0)}\dfrac{((x-y)\cdot x)\log\abs{x-y}}{\abs{x}\abs{x-y}^2}\phi(x)\,dS}\\
&\quad\lesssim \int_{\partial B_\frac{\abs{y}}{2}(0)}\dfrac{\log\abs{x-y}}{\abs{x-y}}\frac{1}{\abs{x}^2}\,dS\lesssim \abs{y}\frac{\log\abs{y}}{\abs{y}}\cdot\frac{1}{\abs{y}^2}\lesssim \frac{\log\abs{y}}{\abs{y}^2}\lesssim \frac{1}{\abs{y}^{2-s}}
\end{align*}
holds. Also, with $\abs{y}\geq R$, and $\abs{\nabla\phi(x)}\lesssim 1/\abs{x}^{3}$ from $\phi \in\mathcal{S}(\setR^d)$, we obtain
\begin{align*}
&\Biggabs{\int_{\setR^d\setminus B_{\frac{\abs{y}}{2}}(0)}\dfrac{(x-y)\log\abs{x-y}}{\abs{x-y}^2}\nabla\phi(x)\,dx}\\
&\quad\leq\int_{B_{\frac{\abs{y}}{2}}(y)}\dfrac{\log\abs{x-y}}{\abs{x-y}}\frac{1}{\abs{x}^3}\,dx+\int_{\setR^d\setminus (B_{\frac{\abs{y}}{2}}(0)\cup B_{\frac{\abs{y}}{2}}(y))}\dfrac{\log\abs{x-y}}{\abs{x-y}}\frac{1}{\abs{x}^3}\,dx\\
&\quad\leq \abs{y}\log \abs{y}\cdot \frac{1}{\abs{y}^3}+\frac{\log\abs{y}}{\abs{y}}\cdot\frac{1}{\abs{y}}\lesssim \frac{\log\abs{y}}{\abs{y}^2}\lesssim \frac{1}{\abs{y}^{2-s}}.
\end{align*}
Summing up, with \eqref{eq:y<R3} we see that
\begin{align}\label{eq:y<Ry>R3}
\Biggabs{\int_{\setR^d}\dfrac{x\log\abs{x}}{\abs{x}^2}\nabla_x(\phi(x+y))\,dx}\lesssim R\log R\cdot\indicator_{\{\abs{y}\leq R\}}+\frac{1}{\abs{y}^{2-s}}\cdot \indicator_{\{\abs{y}\geq R\}}.
\end{align}
Using this, $\partial_{1}\phi\in\mathcal{S}(\setR^d)$ from $\phi\in\mathcal{S}(\setR^d)$, and $\abs{y}^{-d-2s}y_1^2,\abs{y}^{-d-2s+2}>0$, we estimate as follows for $L\geq 2R$: 
\begin{align*}
&\Biggabs{\int_{\abs{y}\leq L}\left(\int_{\setR^d}\dfrac{x\log\abs{x}}{\abs{x}^2}\nabla_x\partial_{y_1}\phi(x+y)\,dx\right)\left(\dfrac{-(-d-2s+2)y_1^2+2\abs{y}^{2}}{\abs{y}^{d+2s}(-d-2s+2)(-d-2s)}\right)\,dy}\\
&\quad\lesssim \Biggabs{\int_{\abs{y}\leq L}\abs{y}^{-d-2s}y_1^2\left(\int_{\setR^d}\dfrac{x\log\abs{x}}{\abs{x}^2}\nabla_x\partial_{y_1}\phi(x+y)\,dx\right)dy}\\
&\quad\quad+\Biggabs{\int_{\abs{y}\leq L}\abs{y}^{-d-2s+2}\left(\int_{\setR^d}\dfrac{x\log\abs{x}}{\abs{x}^2}\nabla_x\partial_{y_1}\phi(x+y)\,dx\right)dy}\\
&\quad\lesssim\int_{\abs{y}\leq L}\left(R\log R\cdot\indicator_{\{\abs{y}\leq R\}}+\frac{1}{\abs{y}^{2-s}}\cdot \indicator_{\{\abs{y}\geq R\}}\right)\abs{y}^{-d-2s+2}\,dy\\
&\quad\lesssim R^{3-2s}\log R+R^{-s}\lesssim R^3.
\end{align*}
Using the above argument, we can obtain that for $\widetilde{L}>L$, we have
\begin{align*}
&\Biggabs{\int_{L\leq\abs{y}\leq\widetilde{L}}\left(\int_{\setR^d}\dfrac{x\log\abs{x}}{\abs{x}^2}\nabla_x\partial_{y_1}\phi(x+y)\,dx\right)\left(\dfrac{-(-d-2s+2)y_1^2+2\abs{y}^{2}}{\abs{y}^{d+2s}(-d-2s+2)(-d-2s)}\right)\,dy}\\
&\quad\lesssim \Biggabs{\int_{L\leq\abs{y}\leq\widetilde{L}}\abs{y}^{-d-2s}y_1^2\left(\int_{\setR^d}\dfrac{x\log\abs{x}}{\abs{x}^2}\nabla_x\partial_{y_1}\phi(x+y)\,dx\right)dy}\\
&\quad\quad+\Biggabs{\int_{L\leq\abs{y}\leq\widetilde{L}}\abs{y}^{-d-2s+2}\left(\int_{\setR^d}\dfrac{x\log\abs{x}}{\abs{x}^2}\nabla_x\partial_{y_1}\phi(x+y)\,dx\right)dy}\\
&\quad\lesssim\int_{L\leq\abs{y}\leq\widetilde{L}}\left(\frac{1}{\abs{y}^{2-s}}\cdot \indicator_{\{\abs{y}\geq R\}}\right)\abs{y}^{-d-2s+2}\,dy\lesssim L^{-s}.
\end{align*}
This proves that 
\begin{align*}
\begin{split}
\int_{\abs{y}\leq L}\left(\int_{\setR^d}\dfrac{x\log\abs{x}}{\abs{x}^2}\nabla_x\partial_{y_1}\phi(x+y)\,dx\right)\left(\dfrac{-(-d-2s+2)y_1^2+2\abs{y}^{2}}{\abs{y}^{d+2s}(-d-2s+2)(-d-2s)}\right)\,dy
\end{split}
\end{align*}
is a Cauchy sequence for $L$. Then we see that
\begin{align*}
\skp{j_3(y)}{\skp{j_2(x)}{\phi(x+y)}_x}_y
\end{align*}
exists. Similarly, using $\partial^k_{i}\phi\in\mathcal{S}(\setR^d)$ for $k=0,1,2,3$ and $i=1,2,\dots,d$, we conclude that
\begin{align}
\skp{(-j_3+\tfrac{j_4}{2}+j_5-j_6+\tfrac{j_7}{2})(y)}{\skp{j_2(x)}{\phi(x+y)}_x}_y
\end{align}
exists.

\textbf{Step 2: Estimate for $j_2$ in the case of $d\geq 3$. }When $\abs{y}\leq R$, we obtain
\begin{align*}
\Biggabs{\int_{\setR^d}\dfrac{\phi(x+y)}{\abs{x}^2}\,dx}&\leq\Biggabs{\int_{B_R(0)}\dfrac{\phi(x+y)}{\abs{x}^2}\,dx}+\Biggabs{\int_{\setR^d\setminus B_R(0)}\dfrac{\phi(x+y)}{\abs{x}^2}\,dx}\\
&\lesssim R^{d-2}+\Biggabs{\int_{\setR^d\setminus B_R(0)}\dfrac{\phi(x+y)}{\abs{x}^2}\,dx},
\end{align*}
where we have used $\abs{\nabla_x(\phi(x+y))}\lesssim 1$ since $\phi\in \mathcal{S}(\setR^d)$. Here, if $-y\in B_{\frac{3}{4}R}(0)$, then employing $\abs{\nabla_x(\phi(x+y))}\lesssim \abs{x+y}^{-d-1}$, we obtain
\begin{align*}
\Biggabs{\int_{\setR^d\setminus B_R(0)}\dfrac{\phi(x+y)}{\abs{x}^2}\,dx}\lesssim \int_{\setR^d\setminus B_R(0)}\abs{x}^{-2}\abs{x+y}^{-d-1}\,dx\lesssim R^{-3}.
\end{align*}
On the other hand, if $-y\not\in B_{\frac{3}{4}R}(0)$ so that $B_{\frac{1}{4}R}(-y)\cap B_{\frac{1}{2}R}(0)=\emptyset$, then using $\abs{\nabla_x(\phi(x+y))}\lesssim \min\{1,\abs{x+y}^{-d-1}\}$ we have
\begin{align*}
&\Biggabs{\int_{\setR^d\setminus B_R(0)}\dfrac{\phi(x+y)}{\abs{x}^2}\,dx}\\
&\quad\leq \int_{B_{\frac{1}{4}R}(-y)}\abs{x}^{-2}\,dx+\int_{\setR^d\setminus (B_{\frac{1}{2}R}(0)\cup B_{\frac{1}{4}R}(-y))}\abs{x}^{-2}\abs{x+y}^{-d-1}\,dx\lesssim R^{d-2}.
\end{align*}
Then if $\abs{y}\leq R$, we obtain
\begin{align}\label{eq:y<R3.1}
\Biggabs{\int_{\setR^d}\dfrac{\phi(x+y)}{\abs{x}^2}\,dx}\lesssim R^{d-2}.
\end{align}

Now we consider the case of $\abs{y}\geq R$. Observe that
\begin{align*}
&\int_{\setR^d}\dfrac{\phi(x+y)}{\abs{x}^2}\,dx\\
&=\int_{B_\frac{\abs{y}}{2}(-y)}\dfrac{\phi(x+y)}{\abs{x}^2}\,dx+\int_{B_\frac{\abs{y}}{2}(0)}\dfrac{\phi(x+y)}{\abs{x}^2}\,dx+\int_{\setR^d\setminus (B_\frac{\abs{y}}{2}(-y)\cup B_\frac{\abs{y}}{2}(0))}\dfrac{\phi(x+y)}{\abs{x}^2}\,dx.
\end{align*}
Here, using $\abs{\phi(x+y)}\lesssim 1/\abs{x+y}^{d-s}$ from $\phi \in\mathcal{S}(\setR^d)$, and recalling $\abs{y}\geq R$,
\begin{align*}
\Biggabs{\int_{B_\frac{\abs{y}}{2}(-y)}\dfrac{\phi(x+y)}{\abs{x}^2}\,dx}\lesssim\int_{B_\frac{\abs{y}}{2}(-y)}\frac{1}{\abs{x}^2}\dfrac{1}{\abs{x+y}^{d-s}}\,dx\lesssim \dfrac{1}{\abs{y}^2}\abs{y}^{s}\lesssim\frac{1}{\abs{y}^{2-s}}
\end{align*}
holds. Moreover, employing $\abs{\phi(x+y)}\lesssim 1/\abs{x+y}^{d+1}$ from $\phi \in\mathcal{S}(\setR^d)$ and using $\abs{y}\geq R\geq 2$, we obtain
\begin{align*}
&\Biggabs{\int_{B_\frac{\abs{y}}{2}(0)}\dfrac{\phi(x+y)}{\abs{x}^2}\,dx+\int_{\setR^d\setminus (B_\frac{\abs{y}}{2}(-y)\cup B_\frac{\abs{y}}{2}(0))}\dfrac{\phi(x+y)}{\abs{x}^2}\,dx}\\
&\quad\lesssim\int_{B_{\frac{\abs{y}}{2}}(0)}\dfrac{1}{\abs{x}^2}\frac{1}{\abs{y}^{d+1}}\,dx+\int_{\setR^d\setminus (B_{\frac{\abs{y}}{2}}(-y)\cup B_{\frac{\abs{y}}{2}}(0))}\dfrac{1}{\abs{y}^2}\frac{1}{\abs{x+y}^{d+1}}\,dx\\
&\quad\lesssim \abs{y}^{d-2}\dfrac{1}{\abs{y}^{d+1}}+\dfrac{1}{\abs{y}^2}\cdot\dfrac{1}{\abs{y}}\lesssim\frac{1}{\abs{y}^{2-s}}.
\end{align*}
Summing up, with \eqref{eq:y<R3.1} we see that
\begin{align}\label{eq:y<Ry>R3.2}
\Biggabs{\int_{\setR^d}\dfrac{\phi(x+y)}{\abs{x}^2}\,dx}\lesssim R^{d-2}\cdot\indicator_{\{\abs{y}\leq R\}}+\frac{1}{\abs{y}^{2-s}}\cdot \indicator_{\{\abs{y}\geq R\}}.
\end{align}
Using this and $\partial^k_{i}\phi\in\mathcal{S}(\setR^d)$ for $k=0,1,2$ and $i=1,2,\dots,d$, we compute for $L\geq 2R$ as follows:
\begin{align}\label{eq:y<Ry>R3.3}
\begin{split}
&\int_{\abs{y}\leq L}\left(\int_{\setR^d}\dfrac{\phi(x+y)}{\abs{x}^2}\,dx\right)\left(\dfrac{-(-d-2s+2)y_1^2+2\abs{y}^{2}}{\abs{y}^{d+2s}(-d-2s+2)(-d-2s)}\right)\,dy\\
&\quad\lesssim R^{d-2s}+R^{-s}\lesssim R^d.
\end{split}
\end{align}
Moreover, we can obtain that for $\widetilde{L}>L$, we have
\begin{align}\label{eq:y<Ry>R3.3'}
\begin{split}
\int_{L\leq\abs{y}\leq\widetilde{L}}\left(\int_{\setR^d}\dfrac{\phi(x+y)}{\abs{x}^2}\,dx\right)\left(\dfrac{-(-d-2s+2)y_1^2+2\abs{y}^{2}}{\abs{y}^{d+2s}(-d-2s+2)(-d-2s)}\right)\,dy\lesssim L^{-s}.
\end{split}
\end{align}
This proves that 
\begin{align*}
\begin{split}
\int_{\abs{y}\leq L}\left(\int_{\setR^d}\dfrac{\phi(x+y)}{\abs{x}^2}\,dx\right)\left(\dfrac{-(-d-2s+2)y_1^2+2\abs{y}^{2}}{\abs{y}^{d+2s}(-d-2s+2)(-d-2s)}\right)\,dy
\end{split}
\end{align*}
is a Cauchy sequence for $L$. Then we see that
\begin{align*}
\begin{split}
\skp{j_3(y)}{\skp{j_2(x)}{\phi(x+y)}_x}_y
\end{split}
\end{align*}
exists. Similarly, we conclude that
\begin{align*}
\begin{split}
\skp{(-j_3+\tfrac{j_4}{2}+j_5-j_6+\tfrac{j_7}{2})(y)}{\skp{j_2(x)}{\phi(x+y)}_x}_y
\end{split}
\end{align*}
exists.

\textbf{Step 3: Estimate for $j_1$. }When $\abs{y}\leq R$, using $\delta\in[0,\frac{1}{2}]$ implies $-\delta-1>-2\geq-d$, so we obtain
\begin{align*}
&\Biggabs{\int_{\setR^d}\abs{x}^{-\delta-2}x_1\phi(x+y)\,dx}\leq\int_{\setR^d}\abs{x}^{-\delta-1}\abs{\phi(x+y)}\,dx\\
&\quad\quad\quad\leq\int_{B_R(0)}\abs{x}^{-\delta-1}\abs{\phi(x+y)}\,dx+\int_{\setR^d\setminus B_R(0)}\abs{x}^{-\delta-1}\abs{\phi(x+y)}\,dx\\
&\quad\quad\quad\lesssim R^{d-\delta-1}+\int_{\setR^d\setminus B_R(0)}\abs{x}^{-\delta-1}\abs{\phi(x+y)}\,dx,
\end{align*}
where for the last inequality we have used $\abs{\phi(x+y)}\lesssim 1$ since $\phi\in \mathcal{S}(\setR^d)$. Here, if $-y\in B_{\frac{3}{4}R}(0)$, then using $\abs{\phi(x+y)}\lesssim\abs{x+y}^{-d-1}$, we obtain
\begin{align*}
\int_{\setR^d\setminus B_R(0)}\abs{x}^{-\delta-1}\abs{\phi(x+y)}\,dx\lesssim\int_{\setR^d\setminus B_R(0)}\abs{x}^{-\delta-1}\abs{x+y}^{-d-1}\,dx\lesssim R^{-\delta-2}.
\end{align*}
On the other hand, if $-y\not\in B_{\frac{3}{4}R}(0)$ so that $B_{\frac{1}{4}R}(-y)\cap B_{\frac{1}{2}R}(0)=\emptyset$, then employing $\abs{\phi(x+y)}\lesssim \min\{1,\abs{x+y}^{-d-1}\}$,
\begin{align*}
&\int_{\setR^d\setminus B_R(0)}\abs{x}^{-\delta-1}\abs{\phi(x+y)}\,dx\\
&\quad\leq \int_{B_{\frac{1}{4}R}(-y)}\abs{x}^{-\delta-1}\,dx+\int_{\setR^d\setminus (B_{\frac{1}{2}R}(0)\cup B_{\frac{1}{4}R}(-y))}\abs{x}^{-\delta-1}\abs{x+y}^{-d-1}\,dx\\
&\quad \lesssim R^{-\delta-1}\cdot R^d+R^{-\delta-1}R^{-1}\lesssim R^{d}
\end{align*}
holds, since $R\geq 2$. Then if $\abs{y}\leq R$, we obtain
\begin{align}\label{eq:y<R3.2}
\Biggabs{\int_{\setR^d}\abs{x}^{-\delta-2}x_1\phi(x+y)\,dx}\lesssim R^{d}.
\end{align}

Now we consider the case of $\abs{y}\geq R$. In this case we will use the fact that the integral of the odd function on $\setR^d$ is zero instead of using divergence theorem. Observe that 
\begin{align*}
&\int_{\setR^d}\abs{x}^{-\delta-2}x_1\phi(x+y)\,dx\\
&\quad=\int_{\setR^d}\abs{x-y}^{-\delta-2}(x_1-y_1)\phi(x)\,dx\\
&\quad=\int_{B_\frac{\abs{y}}{2}(0)}\abs{x-y}^{-\delta-2}(x_1-y_1)\phi(x)\,dx+\int_{B_\frac{\abs{y}}{2}(y)}\abs{x-y}^{-\delta-2}(x_1-y_1)\phi(x)\,dx\\
&\quad\quad\quad+\int_{\setR^d\setminus (B_\frac{\abs{y}}{2}(0)\cup B_\frac{\abs{y}}{2}(y))}\abs{x-y}^{-\delta-2}(x_1-y_1)\phi(x)\,dx.
\end{align*}
Here, when $x\in B_\frac{\abs{y}}{2}(0)$, by Lemma \ref{lem:basic} we observe that
\begin{align*}
\abs{\abs{x-y}^{-\delta-2}(x_1-y_1)-\abs{y}^{-\delta-2}y_1}\lesssim \abs{y}^{-\delta-2}\abs{x_1}
\end{align*}
and so together with $\abs{x_1\phi(x)}\lesssim\min\{1,\abs{x}^{-d-1}\}$ from $\phi\in\mathcal{S}(\setR^d)$,
\begin{align*}
\begin{split}
&\int_{B_\frac{\abs{y}}{2}(0)}\abs{x-y}^{-\delta-2}(x_1-y_1)\phi(x)\,dx\\
&\quad\leq \int_{B_\frac{\abs{y}}{2}(0)}\abs{y}^{-\delta-2}y_1\phi(x)\,dx+c\int_{B_\frac{\abs{y}}{2}(0)}\abs{y}^{-\delta-2}\abs{x_1\phi(x)}\,dx\\
&\quad\leq \int_{B_\frac{\abs{y}}{2}(0)}\abs{y}^{-\delta-2}y_1\phi(x)\,dx+\frac{c}{\abs{y}^2}
\end{split}
\end{align*}
for some $c\geq 1$. On the other hand, using $\abs{\phi(x)}\lesssim 1/\abs{x}^{d+1}$ from $\phi \in\mathcal{S}(\setR^d)$, and recalling $\abs{y}\geq R\geq 2$, it holds that
\begin{align*}
&\Bigabs{\int_{B_\frac{\abs{y}}{2}(y)}\abs{x-y}^{-\delta-2}(x_1-y_1)\phi(x)\,dx+\hspace{-3mm}\int_{\setR^d\setminus (B_\frac{\abs{y}}{2}(0)\cup B_\frac{\abs{y}}{2}(y))}\hspace{-5mm}\abs{x-y}^{-\delta-2}(x_1-y_1)\phi(x)\,dx}\\
&\quad\lesssim\int_{B_\frac{\abs{y}}{2}(y)}\frac{\abs{x-y}^{-\delta-1}}{\abs{x}^{d+1}}\,dx+\int_{\setR^d\setminus (B_\frac{\abs{y}}{2}(0)\cup B_\frac{\abs{y}}{2}(y))}\frac{\abs{x-y}^{-\delta-1}}{\abs{x}^{d+1}}\,dx\\
&\quad\lesssim \abs{y}^{d-\delta-1}\cdot\frac{1}{\abs{y}^{d+1}}+\abs{y}^{-\delta-1}\cdot\frac{1}{\abs{y}}\lesssim\frac{1}{\abs{y}^{2}}.
\end{align*}
Summing up, with \eqref{eq:y<R3.2} we see that
\begin{align}\label{eq:y<Ry>R3.4}
\begin{split}
&\int_{\setR^d}\abs{x}^{-\delta-2}x_1\phi(x+y)\,dx\\
&\quad\lesssim R^{d}\cdot\indicator_{\{\abs{y}\leq R\}}+\bigg(\int_{B_\frac{\abs{y}}{2}(0)}\abs{y}^{-\delta-2}y_1\phi(x)\,dx+\frac{c}{\abs{y}^2}\bigg)\cdot \indicator_{\{\abs{y}\geq R\}}.
\end{split}
\end{align}

Using this and $\partial_{1}\phi\in\mathcal{S}(\setR^d)$ from $\phi\in\mathcal{S}(\setR^d)$, we estimate as follows for $L\geq 2R$: 
\begin{align*}
&\Biggabs{\int_{\abs{y}\leq L}\left(\int_{\setR^d}\abs{x}^{-\delta-2}x_1\partial_{y_1}\phi(x+y)\,dx\right)\left(\dfrac{-(-d-2s+2)y_1^2+2\abs{y}^{2}}{\abs{y}^{d+2s}(-d-2s+2)(-d-2s)}\right)\,dy}\\
&\leq c\int_{\abs{y}\leq R}R^{d}\abs{y}^{-d-2s+2}\,dy\\
&\,\,+\biggabs{\int_{R\leq \abs{y}\leq L}\bigg(\int_{B_\frac{\abs{y}}{2}(0)}\dfrac{y_1\partial_1\phi(x)}{\abs{y}^{\delta+2}}\,dx+\frac{c}{\abs{y}^2}\bigg)\left(\dfrac{-(-d-2s+2)y_1^2+2\abs{y}^{2}}{\abs{y}^{d+2s}(-d-2s+2)(-d-2s)}\right)dy}.
\end{align*}
Here, we estimate
\begin{align*}
\int_{\abs{y}\leq R}R^{d}\abs{y}^{-d-2s+2}\,dy\lesssim R^{d-2s+2}.
\end{align*}
Moreover, using Fubini's theorem and observing that the integrand is odd, we find
\begin{align}\label{eq:y>R3.1}
\begin{split}
&\int_{R\leq\abs{y}\leq L}\bigg(\int_{B_\frac{\abs{y}}{2}(0)}\abs{y}^{-\delta-2}y_1\partial_1\phi(x)\,dx\bigg)\left(\dfrac{-(-d-2s+2)y_1^2+2\abs{y}^{2}}{\abs{y}^{d+2s}(-d-2s+2)(-d-2s)}\right)dy\\
&=\int_{\setR^d}\bigg(\int_{\max\{2\abs{x},R\}\leq\abs{y}\leq L}\hspace{-3mm}\left(\abs{y}^{-\delta-2}y_1\right)\cdot\dfrac{-(-d-2s+2)y_1^2+2\abs{y}^{2}}{\abs{y}^{d+2s}(-d-2s+2)(-d-2s)}\,dy\bigg)\partial_1\phi(x)dx\\
&=0.
\end{split}
\end{align}
Also, we see that
\begin{align}\label{eq:y>R3.2}
\int_{R\leq\abs{y}\leq L}\frac{c}{\abs{y}^2}\left(\dfrac{-(-d-2s+2)y_1^2+2\abs{y}^{2}}{\abs{y}^{d+2s}(-d-2s+2)(-d-2s)}\right)dy\lesssim R^{-2s}.
\end{align}
Using the above argument, we can obtain that for $\widetilde{L}>L$, we have
\begin{align*}
&\Biggabs{\int_{L\leq\abs{y}\leq\widetilde{L}}\left(\int_{\setR^d}\abs{x}^{-\delta-2}x_1\partial_{y_1}\phi(x+y)\,dx\right)\left(\dfrac{-(-d-2s+2)y_1^2+2\abs{y}^{2}}{\abs{y}^{d+2s}(-d-2s+2)(-d-2s)}\right)\,dy}\\
&\quad\leq\biggabs{\int_{L\leq\abs{y}\leq \widetilde{L}}\bigg(\int_{B_\frac{\abs{y}}{2}(0)}\dfrac{y_1\partial_1\phi(x)}{\abs{y}^{\delta+2}}\,dx+\frac{c}{\abs{y}^2}\bigg)\left(\dfrac{-(-d-2s+2)y_1^2+2\abs{y}^{2}}{\abs{y}^{d+2s}(-d-2s+2)(-d-2s)}\right)dy}\\
&\quad\leq L^{-2s}.
\end{align*}
This proves that
\begin{align*}
\int_{\abs{y}\leq L}\left(\int_{\setR^d}\abs{x}^{-\delta-2}x_1\partial_{y_1}\phi(x+y)\,dx\right)\left(\dfrac{-(-d-2s+2)y_1^2+2\abs{y}^{2}}{\abs{y}^{d+2s}(-d-2s+2)(-d-2s)}\right)\,dy
\end{align*}
is a Cauchy sequence for $L$. Then we see that
\begin{align*}
\skp{j_3(y)}{\skp{j_1(x)}{\phi(x+y)}_x}_y
\end{align*}
exists. Similarly, for each $i=1,2,\dots,d$, using $\partial^k_{i}\phi\in\mathcal{S}(\setR^d)$ for $k=0,1,2$ together with the fact that $h_i\mapsto \abs{h}^{-d-2s}h_1^2$ and $h_i\mapsto \abs{h}^{-d-2s+2}$ are even, we see that
\begin{align*}
\skp{(-j_3+\tfrac{j_4}{2}+j_5-j_6+\tfrac{j_7}{2})(y)}{\skp{j_1(x)}{\phi(x+y)}_x}_y
\end{align*}
exists.

Considering \textbf{Step 1}--\textbf{Step 3}, we prove the conclusion. 
\end{proof}

Recalling $\eta_k(\cdot)$ as in \eqref{eq:eta}, and $\xi_j(\cdot,\cdot)$ as in \eqref{eq:xi}, we prove the following lemma.

\begin{lemma}\label{lem:conv3}
	The convolution of distributions $g_3\divideontimes g_4$ is well-defined in $x\in\setR^d$ in the sense that for any $\phi\in \mathcal{S}(\setR^d)$, 
	\begin{align}\label{eq:kj}
	\skp{g_3\divideontimes g_4}{\phi}:=\lim_{k\rightarrow\infty}\lim_{j\rightarrow\infty}\bigskp{(\eta_kg_3)(x)}{\skp{g_4(y)}{\xi_j(x,y)\phi(x+y)}_y}_x.
	\end{align}
	Moreover, we have
	\begin{align*}
	\skp{g_3\divideontimes g_4}{\phi}=\bigskp{g_3(x)}{\skp{g_4(y)}{\phi(x+y)}_y}_x.
	\end{align*}
\end{lemma}
\begin{proof}
Note that once the double limit in the right-hand side of \eqref{eq:kj} exists, then
\begin{align*}
\skp{g_3\divideontimes g_4}{\phi}
&:=\lim_{k\rightarrow\infty}\lim_{j\rightarrow\infty}\skp{(\eta_kg_3)(\cdot)\times g_4(\bigcdot)}{\xi_j(\cdot,\bigcdot)\phi(\cdot+\bigcdot)}\\
&:=\lim_{k\rightarrow\infty}\lim_{j\rightarrow\infty}\bigskp{(\eta_kg_3)(x)}{\skp{g_4(y)}{\xi_j(x,y)\phi(x+y)}_y}_x
\end{align*}	
holds. Here, for the first definition since $g_3,g_4\in\mathcal{S}'(\setR^d)\subset (C^{\infty}_c)'(\setR^d)$ and $\xi_j\in C^{\infty}_c(\setR^{2d})$, \cite[Page 96 and Page 52]{Vla02} is used, and for the second definition since $\xi_j(x,y)\phi(x+y)\in C^{\infty}_c(\setR^{2d})$, \cite[Page 41]{Vla02} is used. Moreover, by \cite[Page 96]{Vla02}, we have $\skp{g_3\divideontimes g_4}{\phi}=\skp{g_4\divideontimes g_3}{\phi}$.

Now it suffices to show that
\begin{align}\label{eq:doublelimit}
\lim_{k\rightarrow\infty}\lim_{j\rightarrow\infty}\bigskp{(\eta_kg_4)(x)}{\skp{g_3(y)}{\xi_j(x,y)\phi(x+y)}_y}_x=\bigskp{g_4(x)}{\skp{g_3(y)}{\phi(x+y)}_y}_x.
\end{align}
To do this, for $k,j>0$, we have
\begin{align*}
&\bigabs{\bigskp{(\eta_kg_4)(x)}{\skp{g_3(y)}{\xi_j(x,y)\phi(x+y)}_y}_x-\bigskp{g_4(x)}{\skp{g_3(y)}{\phi(x+y)}_y}_x}\\
&\quad\leq\bigabs{\bigskp{(\eta_kg_4)(x)}{\skp{g_3(y)}{(1-\xi_j(x,y))\phi(x+y)}_y}_x}\\
&\quad\quad+\bigabs{\bigskp{(1-\eta_k)g_4(x)}{\skp{g_3(y)}{\phi(x+y)}_y}_x}\\
&\quad=:I_1+I_2.
\end{align*}
Using $\sum^{3}_{i=0}\abs{\nabla^i[(1-\xi^{j}(x,y))\phi(x+y)]}\lesssim\abs{x+y}^{-d}\lesssim\abs{y}^{-d}$ from $\phi\in\mathcal{S}(\setR^d)$, $\abs{x}\leq 2k$ and $\abs{y}>j$ with $4k\leq j$, we obtain
\begin{align*}
I_1&\lesssim\int_{\abs{x}\leq 2k}\int_{\abs{y}>j}\abs{\eta_k(x)}\abs{x}^{-d-2s+2}(\abs{y}^{-1-\delta}+(\abs{\log\abs{y}}\cdot\abs{y}^{-1}))\abs{y}^{-d}\,dx\,dy\\
&\lesssim k^{2-2s}j^{-1-\delta}\underset{j\rightarrow \infty}{\longrightarrow} 0.
\end{align*}
For $I_2$, we write the following integrand as $J(x,y):\setR^{2d}\rightarrow\setR$:
\begin{align*}
\bigskp{(1-\eta_k)g_4(x)}{\skp{g_3(y)}{\phi(x+y)}_y}_x=:\int_{\setR^d}\int_{\setR^d}J(x,y)\,dy\,dx=\int_{\abs{x}> k}\int_{\setR^d}J(x,y)\,dy\,dx,
\end{align*}
where for the last equality we have used $\eta_k\equiv 1$ when $\abs{x}\leq k$. Then following the proof of Lemma \ref{lem:welldef3}, we obtain \eqref{eq:y<Ry>R3}, \eqref{eq:y<Ry>R3.2}, and \eqref{eq:y<Ry>R3.4} with the choice of $y$ replaced by $x$ and $R=k/2$. Then following the remaining
proof of Lemma \ref{lem:welldef3} but only considering the case of $\abs{y}\geq R$ in its proof we get the analogous estimate of \eqref{eq:y<Ry>R3}, \eqref{eq:y<Ry>R3.3}, and \eqref{eq:y>R3.2} as follows:
\begin{align*}
\biggabs{\int_{\abs{x}> k}\int_{\setR^d}J(x,y)\,dy\,dx}\lesssim k^{-2s}\underset{k\rightarrow \infty}{\longrightarrow} 0.
\end{align*}
Then \eqref{eq:doublelimit} holds and we conclude the proof.
\end{proof}

Now for the standard mollifier $\overline{\phi}(x)\in\mathcal{S}(\setR^d)$ defined in \eqref{eq:mol} and $\phi_{\epsilon}(x)=\frac{1}{\epsilon^{d}}\overline{\phi}\left(\frac{x-e_1}{\epsilon}\right)$ for $\epsilon\in(0,1)$, we prove the following lemma, which rigorously establishes the connection between $f_2(d,s,\delta)$ in \eqref{eq:f2} and $g_3\divideontimes g_4$.
\begin{lemma}\label{lem:A2}
We have
$f_2(d,s,\delta)=\lim_{\epsilon\rightarrow 0}\skp{g_3\divideontimes g_4}{\phi_{\epsilon}}$.
\end{lemma}
\begin{proof}
We need to show
\begin{align}\label{eq:approx.identity}
f_{2}(d,s,\delta)=\lim_{\epsilon\rightarrow 0}\skp{(j_1-j_2)(x)}{\skp{(-j_3+\tfrac{j_4}{2}+j_5-j_6+\tfrac{j_7}{2})(y)}{\phi_{\epsilon}(x+y)}_y}_x.
\end{align}
To do this, with $\epsilon\in(0,\frac{1}{10})$, let us write the integrand used for defining the distribution $g_3\divideontimes g_4$ tested with $\phi_{\epsilon}(x)$ as $J_{\epsilon}(x,y):\setR^{2d}\rightarrow\setR$ in the following:
\begin{align*}
\skp{g_3\divideontimes g_4}{\phi_{\epsilon}}=\int_{\setR^d}\int_{\setR^d}J_{\epsilon}(x,y)\,dx\,dy.
\end{align*}
We divide the proof into two steps.

\textbf{Step 1: Modifying the formula of $\skp{g_3\divideontimes g_4}{\phi_{\epsilon}}$ using divergence theorem.}
In this step we will show that
\begin{align*}
\int_{\setR^d}\int_{\setR^d}J_{\epsilon}(x,y)\,dx\,dy=\mathcal{Z}_{\epsilon}
\end{align*}
holds, where with 
\begin{align}\label{eq:g}
g(y)=\dfrac{-2y_1^3+|y|^2y_1^2-2|y|^2y_1+|y|^4}{2|y|^{d+2s+2}},
\end{align}
$\mathcal{Z}_{\epsilon}$ is of the form
\begin{align}\label{eq:Z2}
\begin{split}
&\iint_{\setR^{2d}}g(y)|x|^{-2-\delta}x_1\phi_{\epsilon}(x+y)\,dx\,dy\\
&\,\,+\iint_{\setR^{2d}}\dfrac{y_1^2}{|y|^{d+2s+2}}|x|^{-2-\delta}x_1(\phi_{\epsilon}(x+y)-\phi_{\epsilon}(x))\,dx\,dy\\
&\,\,-\iint_{\setR^{2d}}g(y)\abs{x}^{-2}(\phi_{\epsilon}(x+y)-\indicator_{\{\abs{x}\leq 1\}}\phi_{\epsilon}(y))\,dx\,dy\\
&\,\,-\iint_{\setR^{2d}}\dfrac{y_1^2}{|y|^{d+2s+2}}\abs{x}^{-2}[(\phi_{\epsilon}(x+y)-\phi_{\epsilon}(x))-\indicator_{\{\abs{x}\leq 1\}}(\phi_{\epsilon}(y)-\phi_{\epsilon}(0))]\,dx\,dy
\end{split}
\end{align}
when $d=2$, and 
\begin{align}\label{eq:Z3}
\begin{split}
&\iint_{\setR^{2d}}g(y)(|x|^{-2-\delta}x_1-\abs{x}^{-2})\phi_{\epsilon}(x+y)\,dx\,dy\\
&\quad+\iint_{\setR^{2d}}\dfrac{y_1^2}{|y|^{d+2s+2}}(|x|^{-2-\delta}x_1-\abs{x}^{-2})(\phi_{\epsilon}(x+y)-\phi_{\epsilon}(x))\,dx\,dy
\end{split}
\end{align}
when $d\geq 3$.

For $\sigma\in(0,\frac{1}{10})$ with $\sigma\leq\epsilon$, and a cross-shaped region
\begin{align*}
\mathcal{C}_{\sigma}:=(\{x\in\setR^d:\abs{x}\leq\sigma\}\times\setR^d)\cup(\setR^d\times\{y\in\setR^d:\abs{y}\leq \sigma\})
\end{align*}
we will first show
\begin{align}\label{eq:sigma0}
\lim_{\sigma\rightarrow 0}\iint_{\mathcal{C}_{\sigma}}J_{\epsilon}(x,y)\,dx\,dy=0.
\end{align}
To this end, we observe that
\begin{align}\label{eq:C.sigma}
\Bigabs{\iint_{\mathcal{C}_{\sigma}}J_{\epsilon}(x,y)\,dx\,dy}\leq \iint_{\mathcal{C}_{\sigma}}\abs{J_{\epsilon}(x,y)}\,dx\,dy\leq\iint_{\mathcal{D}_{\sigma,\epsilon}}\abs{J_{\epsilon}(x,y)}\,dx\,dy
\end{align}
with
\begin{align*}
\mathcal{D}_{\sigma,\epsilon}&:=(\{\abs{x}\leq\sigma\}\times\{1-\sigma-\epsilon\leq \abs{y}\leq 1+\sigma+\epsilon\})\\
&\quad\quad\quad\quad\quad\quad\cup(\{\abs{y}\leq\sigma\}\times\{1-\sigma-\epsilon\leq \abs{x}\leq 1+\sigma+\epsilon\}),
\end{align*}
where for the last inequality we have used the fact that since $\support\phi_{\epsilon}\subset B_{\epsilon}(e_1)$, on the set $(x,y)\in\mathcal{C}_{\sigma}$ we obtain $\support J(x,y)\subset \mathcal{D}_{\sigma,\epsilon}$.
We first consider the term originating from $\skp{j_1\divideontimes j_3}{\phi_{\epsilon}}$ in the following:
\begin{align*}
&\iint_{\mathcal{D}_{\sigma,\epsilon}}\biggabs{\abs{x}^{-\delta-2}x_1\partial_{y_1}\phi_{\epsilon}(x+y)\left(\dfrac{-(-d-2s+2)y_1^2+2\abs{y}^{2}}{\abs{y}^{d+2s}(-d-2s+2)(-d-2s)}\right)}\,dy\,dx\\
&\quad\lesssim\iint_{\{\abs{x}\leq\sigma\}\times\{1-\sigma-\epsilon\leq\abs{y}\leq 1+\sigma+\epsilon\}}\abs{x}^{-\delta-1}\abs{\partial_{y_1}\phi_{\epsilon}(x+y)}\abs{y}^{-d-2s+2}\,dy\,dx\\
&\quad\quad\quad+\iint_{\{\abs{y}\leq\sigma\}\times\{1-\sigma-\epsilon\leq\abs{x}\leq 1+\sigma+\epsilon\}}\abs{x}^{-\delta-1}\abs{\partial_{y_1}\phi_{\epsilon}(x+y)}\abs{y}^{-d-2s+2}\,dy\,dx\\
&\quad\lesssim\sigma^{d-\delta-1}+\sigma^{2-2s}\lesssim\sigma^{2-2s}.
\end{align*}
Similarly, the term originated from $\skp{j_1\divideontimes(j_4+j_5-j_6+j_7)}{\phi_{\epsilon}}$ can be bounded by $\sigma^{2-2s}$. Moreover, when $d=2$, the term from $\skp{j_2\divideontimes j_3}{\phi_{\epsilon}}$ is estimated as
\begin{align*}
&\iint_{\mathcal{D}_{\sigma,\epsilon}}\biggabs{\dfrac{x\log\abs{x}}{\abs{x}^2}\nabla_{x}\partial_{y_1}\phi_{\epsilon}(x+y)\left(\dfrac{-(-d-2s+2)y_1^2+2\abs{y}^{2}}{\abs{y}^{d+2s}(-d-2s+2)(-d-2s)}\right)}\,dy\,dx\\
&\quad\lesssim\iint_{\{\abs{x}\leq\sigma\}\times\{1-\sigma-\epsilon\leq\abs{y}\leq 1+\sigma+\epsilon\}}\dfrac{\abs{\log\abs{x}}}{\abs{x}}\abs{\nabla_{x}\partial_{y_1}\phi_{\epsilon}(x+y)}\abs{y}^{-d-2s+2}\,dy\,dx\\
&\quad\quad\quad+\iint_{\{\abs{y}\leq\sigma\}\times\{1-\sigma-\epsilon\leq\abs{x}\leq 1+\sigma+\epsilon\}}\dfrac{\abs{\log\abs{x}}}{\abs{x}}\abs{\nabla_{x}\partial_{y_1}\phi_{\epsilon}(x+y)}\abs{y}^{-d-2s+2}\,dy\,dx\\
&\quad\lesssim\sigma^{1}\abs{\log\sigma}+\sigma^{2-2s}\lesssim\sigma^{1-s}.
\end{align*}
The term originated from $\skp{j_2\divideontimes(j_4+j_5-j_6+j_7)}{\phi_{\epsilon}}$ can be similarly bounded by $\sigma^{1-s}$.
When $d\geq 3$, the term from $\skp{j_2\divideontimes j_3}{\phi_{\epsilon}}$ is bounded as
\begin{align*}
&\iint_{\mathcal{D}_{\sigma,\epsilon}}\biggabs{\dfrac{1}{\abs{x}^2}\partial_{y_1}\phi_{\epsilon}(x+y)\left(\dfrac{-(-d-2s+2)y_1^2+2\abs{y}^{2}}{\abs{y}^{d+2s}(-d-2s+2)(-d-2s)}\right)}\,dy\,dx\\
&\quad\lesssim\iint_{\{\abs{x}\leq\sigma\}\times\{1-\sigma-\epsilon\leq\abs{y}\leq 1+\sigma+\epsilon\}}\abs{x}^{-2}\abs{\partial_{y_1}\phi_{\epsilon}(x+y)}\abs{y}^{-d-2s+2}\,dy\,dx\\
&\quad\quad\quad+\iint_{\{\abs{y}\leq\sigma\}\times\{1-\sigma-\epsilon\leq\abs{x}\leq 1+\sigma+\epsilon\}}\abs{x}^{-2}\abs{\partial_{y_1}\phi_{\epsilon}(x+y)}\abs{y}^{-d-2s+2}\,dy\,dx\\
&\quad\lesssim\sigma^{d-2}+\sigma^{2-2s}\lesssim\sigma^{1-s}.
\end{align*}
Note that the implicit constants in the above estimate do not depend on $\epsilon$. Likewise, the term from $\skp{j_2\divideontimes(j_4+j_5-j_6+j_7)}{\phi_{\epsilon}}$ can be bounded by $\sigma^{1-s}$. Taking the limit $\sigma \to 0$ proves \eqref{eq:sigma0} for fixed $\epsilon$.

Now for $\mathcal{E}_\sigma=(\{\abs{y}\geq\sigma\}\times\{\abs{x}\geq1/\sigma\})\cup(\{\abs{x}\geq\sigma\}\times\{\abs{y}\geq1/\sigma\})$ we have to show
\begin{align*}
\lim_{\sigma\rightarrow 0}\iint_{\mathcal{E}_{\sigma}}J_{\epsilon}(x,y)\,dx\,dy=0.
\end{align*}
To do this, note that on the set $\{x:\abs{x}\leq\sigma\}\times\{y:\abs{y}>\frac{1}{\sigma}\}$, we have $\phi_{\epsilon}(x+y)=0$ from $\support\phi_{\epsilon}\subset B_{\epsilon}(e_1)$. Then we have
\begin{align*}
\iint_{\{\abs{y}>\frac{1}{\sigma}\}\times\{\abs{x}\leq\sigma\}}J_{\epsilon}(x,y)\,dx\,dy=\int_{\setR^d}\int_{\abs{y}>\frac{1}{\sigma}}J_{\epsilon}(x,y)\,dy\,dx.
\end{align*}
Now by following the proof of Lemma \ref{lem:welldef3}, as in \eqref{eq:y<Ry>R3.4} with the choice of $R=\frac{2}{\sigma}$
\begin{align*}
\begin{split}
\int_{\setR^d}\abs{x}^{-\delta-2}x_1\phi_{\epsilon}(x+y)\,dx\leq\int_{B_\frac{\abs{y}}{2}(0)}\abs{y}^{-\delta-2}y_1\phi_{\epsilon}(x)\,dx+\frac{c(d,s,\delta)}{\abs{y}^2}
\end{split}
\end{align*}
holds when $\abs{y}\geq\frac{2}{\sigma}$. Then following the remaining proof of Lemma \ref{lem:welldef3}, especially from \eqref{eq:y>R3.1} and \eqref{eq:y>R3.2} we obtain
\begin{align*}
\Bigabs{\iint_{\{\abs{y}>\frac{1}{\sigma}\}\times\{\abs{x}\leq\sigma\}}J_{\epsilon}(x,y)\,dx\,dy}\lesssim\sigma^{2s}.
\end{align*}
The integral on the set $\{\abs{x}>\frac{1}{\sigma}\}\times\{\abs{y}\leq\sigma\}$ is estimated similarly. Therefore, we have
\begin{align*}
\Bigabs{\iint_{\mathcal{E}_{\sigma}}J_{\epsilon}(x,y)\,dx\,dy}\lesssim\sigma^{2s}.
\end{align*}

Now we are ready to apply divergence theorem. Here we divide the case $d=2$ and $d\geq 3$. If $d=2$, by divergence theorem, for  
\begin{align*} 
\mathcal{F}_{\sigma}:=\{x\in\setR^{d}:\sigma\leq\abs{x}\leq 1/\sigma\}\times \{y\in\setR^{d}:\sigma\leq\abs{y}\leq 1/\sigma\},
\end{align*}
together with $j_1(x)=\abs{x}^{-\delta-2}x_1$, $j_2(x)=\abs{x}^{-2}$, $j_3(x)= \abs{x}^{-d-2s-2}x_1^3$, $j_4(x)=\abs{x}^{-d-2s}x_1^2$, $j_5(x)=\abs{x}^{-d-2s-2}x_1^2$, $j_6(x)= \abs{x}^{-d-2s}x_1$, and $j_7(x)=\abs{x}^{-d-2s+2}$ when $\abs{x}>0$, \eqref{eq:j2}, and using $\partial_{y_i}(\phi_{\epsilon}(x+y)-\phi_{\epsilon}(x))=\partial_{y_i}\phi_{\epsilon}(x+y)$ we have
\begin{align*}
&\iint_{\mathcal{F}_{\sigma}}J_{\epsilon}(x,y)\,dx\,dy\\
&\,\,=\iint_{\mathcal{F}_{\sigma}}g(y)|x|^{-2-\delta}x_1\phi_{\epsilon}(x+y)\,dx\,dy\\
&\quad+\iint_{\mathcal{F}_{\sigma}}\dfrac{y_1^2}{|y|^{d+2s+2}}|x|^{-2-\delta}x_1(\phi_{\epsilon}(x+y)-\phi_{\epsilon}(x))\,dx\,dy\\
&\quad-\iint_{\mathcal{F}_{\sigma}}g(y)\abs{x}^{-2}(\phi_{\epsilon}(x+y)-\indicator_{\{\abs{x}\leq 1\}}\phi_{\epsilon}(y))\,dx\,dy\\
&\quad-\iint_{\mathcal{F}_{\sigma}}\dfrac{y_1^2}{|y|^{d+2s+2}}\abs{x}^{-2}[(\phi_{\epsilon}(x+y)-\phi_{\epsilon}(x))-\indicator_{\{\abs{x}\leq 1\}}(\phi_{\epsilon}(y)-\phi_{\epsilon}(0))]\,dx\,dy\\
&\quad+\int_{\partial\mathcal{F}_{\sigma}}\widetilde{J}(x,y)\,dS
\end{align*}
for some $\widetilde{J}(x,y)$. Here, we need to estimate the term $\widetilde{J}(x,y)$ on the set
\begin{align*}
\partial\mathcal{F}_{\sigma}&=\{(x,y)\in\setR^{2d}:\abs{x}=\sigma,\sigma\leq\abs{y}\leq 1/\sigma\}\\
&\quad\cup\{(x,y)\in\setR^{2d}:\abs{x}=1/\sigma,\sigma\leq\abs{y}\leq 1/\sigma\}\\
&\quad\cup\{(x,y)\in\setR^{2d}:\abs{y}=\sigma,\sigma\leq\abs{x}\leq 1/\sigma\}\\
&\quad\cup\{(x,y)\in\setR^{2d}:\abs{y}=1/\sigma,\sigma\leq\abs{x}\leq 1/\sigma\}\\
&:=I_1+I_2+I_3+I_4.
\end{align*}
First of all, note that in the function $\widetilde{J}(x,y)$, there is no term corresponding $j_1\divideontimes j_4$ and $j_1\divideontimes j_7$, since in the definition of $j_1$, $j_4$ and $j_7$, there is no partial derivative of the test function $\phi$ in the integrand, thus in fact divergence theorem is not applied for terms $j_1\divideontimes j_4$ and $j_1\divideontimes j_7$.

For the term corresponding to $j_1\divideontimes j_3$, note that 
\begin{align*}
&\int_{I_1}\abs{x}^{-\delta-2}x_1\left(\dfrac{-(-d-2s+2)y_1^2+2\abs{y}^2}{\abs{y}^{d+2s}(-d-2s+2)(-d-2s)}\right)\phi_{\epsilon}(x+y)\nu(y)\,dS\\
&\quad=\int_{I_1}\abs{x}^{-\delta-2}x_1\left(\dfrac{-(-d-2s+2)y_1^2+2\abs{y}^2}{\abs{y}^{d+2s}(-d-2s+2)(-d-2s)}\right)(\phi_{\epsilon}(x+y)-\phi_{\epsilon}(y))\nu(y)\,dS\\
&\quad\quad+\int_{I_1}\abs{x}^{-\delta-2}x_1\left(\dfrac{-(-d-2s+2)y_1^2+2\abs{y}^2}{\abs{y}^{d+2s}(-d-2s+2)(-d-2s)}\right)\phi_{\epsilon}(y)\nu(y)\,dS\\
&\quad=:A_1+A_2.
\end{align*}
Here, using $\support\phi_{\epsilon}\subset B_{\epsilon}(e_1)$, and then $\abs{\phi_{\epsilon}(x+y)-\phi_{\epsilon}(y)}\lesssim\abs{x}$ for $\abs{x}=\sigma$ and $1-\epsilon-\sigma\leq\abs{y}\leq 1+\epsilon+\sigma$, and $\abs{\nu(y)}\leq 1$, we have
\begin{align*}
\abs{A_1}&\lesssim\int_{\abs{x}=\sigma}\int_{1-\epsilon-\sigma\leq\abs{y}\leq 1+\epsilon+\sigma }\abs{x}^{-\delta-1}\abs{y}^{-d-2s+2}\abs{\phi_{\epsilon}(x+y)-\phi_{\epsilon}(y)}\abs{\nu(y)}\,dS\\
&\lesssim\int_{\abs{x}=\sigma}\int_{1-\epsilon-\sigma\leq\abs{y}\leq 1+\epsilon+\sigma}\abs{x}^{-\delta}\abs{y}^{-d-2s+2}\,dS\\
&\lesssim \sigma^{d-1-\delta}.
\end{align*}
For $A_2$, since
\begin{align*}
\int_{\abs{x}=\sigma}\abs{x}^{-\delta-2}x_1\,dx=0,
\end{align*}
we have $A_2=0$. Then it follows that
\begin{align*}
\biggabs{\int_{I_1}\abs{x}^{-\delta-2}x_1\left(\dfrac{-(-d-2s+2)y_1^2+2\abs{y}^2}{\abs{y}^{d+2s}(-d-2s+2)(-d-2s)}\right)\phi_{\epsilon}(x+y)\nu(y)\,dS}\lesssim \sigma^{d-1-\delta}.
\end{align*}
Similarly, the term for $j_1\divideontimes j_5$ can be bound as
\begin{align*}
\int_{\abs{x}=\sigma}\int_{1-\epsilon-\sigma\leq \abs{y}\leq1+\epsilon+\sigma}\bigabs{\abs{x}^{-\delta-2}x_1}\left(\abs{y}^{-d-2s+2}+\abs{y}^{-d-2s+1}\right)\,dS\lesssim\sigma^{d-1-\delta}.
\end{align*}
Likewise, the term for $j_1\divideontimes j_6$ will be estimated by $\sigma^{d-1-\delta}$ up to a constant.

Now for the term corresponding to $j_2\divideontimes j_3$, we see the bound
\begin{align*}
&\int_{\abs{x}=\sigma}\int_{1-\epsilon-\sigma\leq\abs{y}\leq 1+\epsilon+\sigma}\dfrac{\abs{\log\abs{x}}}{\abs{x}}\left(\abs{y}^{-d-2s+2}+\abs{y}^{-d-2s+1}\right)\,dS\lesssim \sigma\abs{\log\sigma}.
\end{align*}
Similarly, for the term corresponding to $j_2\divideontimes (j_4+j_5-j_6+j_7)$, we have the bound $\sigma\abs{\log\sigma}$. Finally, 
\begin{align*}
\int_{I_1}\widetilde{J}(x,y)\,dS\lesssim \sigma^{d-1-\delta}+\sigma\abs{\log\sigma}\lesssim \sigma\abs{\log\sigma}.
\end{align*}
For the integral on the set $I_2$, for the term corresponding to $j_1\divideontimes j_3$, using $\support\phi_{\epsilon}\subset B_{\epsilon}(e_1)$ we estimate
\begin{align*}
&\biggabs{\int_{I_2}\abs{x}^{-\delta-2}x_1\left(\dfrac{-(-d-2s+2)y_1^2+2\abs{y}^2}{\abs{y}^{d+2s}(-d-2s+2)(-d-2s)}\right)\phi_{\epsilon}(x+y)\nu(x)\,dS}\\
&\quad\lesssim\int_{\abs{x}=\frac{1}{\sigma}}\int_{\abs{x+y-e_1}\leq\epsilon}\sigma^{\delta+1}\sigma^{d+2s-2}\,dS\\
&\quad\lesssim\left(\frac{1}{\sigma}\right)^{d-1}\sigma^{\delta+1}\sigma^{d+2s-2}\lesssim \sigma^{\delta+2s}.
\end{align*}
In the similar way, for the term corresponding to $j_1\divideontimes (j_5-j_6)$ can be bound by $\sigma^{\delta+2s}$. For the term corresponding to $j_2\divideontimes j_3$, using $\sigma^s\abs{\log\sigma}\lesssim 1$ we obtain the estimate
\begin{align*}
&\biggabs{\int_{\abs{x}=\frac{1}{\sigma}}\int_{\abs{x+y-e_1}\leq\epsilon}\frac{\abs{\log\abs{x}}}{\abs{x}}\abs{y}^{-d-2s+2}\,dS}\\
&\quad\lesssim \sigma\abs{\log\sigma}\left(\frac{1}{\sigma}\right)^{d-1}\sigma^{d+2s-2}\leq \sigma^{s}.
\end{align*}
Similarly, the term corresponding to $j_2\divideontimes (j_4+j_5-j_6+j_7)$, we have
\begin{align*}
&\biggabs{\int_{\abs{x}=\frac{1}{\sigma}}\int_{\abs{x+y-e_1}\leq\epsilon}\frac{\abs{\log\abs{x}}}{\abs{x}}(\abs{y}^{-d-2s+2}+\abs{y}^{-d-2s+1})\,dS}\lesssim\sigma^{s}.
\end{align*}

For the integral on the set $I_3$, for the term corresponding to $j_1\divideontimes j_3$, we estimate
\begin{align*}
&\int_{I_3}\abs{x}^{-\delta-2}x_1\left(\dfrac{-(-d-2s+2)y_1^2+2\abs{y}^2}{\abs{y}^{d+2s}(-d-2s+2)(-d-2s)}\right)\phi_{\epsilon}(x+y)\nu(y)\,dS\\
&\quad=\int_{I_3}\abs{x}^{-\delta-2}x_1\left(\dfrac{-(-d-2s+2)y_1^2+2\abs{y}^2}{\abs{y}^{d+2s}(-d-2s+2)(-d-2s)}\right)(\phi_{\epsilon}(x+y)-\phi_{\epsilon}(x))\nu(y)\,dS\\
&\quad\quad+\int_{I_3}\abs{x}^{-\delta-2}x_1\left(\dfrac{-(-d-2s+2)y_1^2+2\abs{y}^2}{\abs{y}^{d+2s}(-d-2s+2)(-d-2s)}\right)\phi_{\epsilon}(x)\nu(y)\,dS\\
&\quad=:A_3+A_4.
\end{align*}
For $A_3$,
using $\support\phi_{\epsilon}\subset B_{\epsilon}(e_1)$, and then $\abs{\phi_{\epsilon}(x+y)-\phi_{\epsilon}(x)}\lesssim\abs{y}$ and $\abs{\nu(y)}\leq 1$, we have
\begin{align*}
\abs{A_1}&\lesssim\int_{\abs{y}=\sigma}\int_{1-\epsilon-\sigma\leq\abs{x}\leq 1+\epsilon+\sigma }\abs{x}^{-\delta-1}\abs{y}^{-d-2s+2}\abs{\phi_{\epsilon}(x+y)-\phi_{\epsilon}(x)}\abs{\nu(y)}\,dS\\
&\lesssim\int_{\abs{y}=\sigma}\int_{1-\epsilon-\sigma\leq\abs{x}\leq 1+\epsilon+\sigma}\abs{x}^{-\delta}\abs{y}^{-d-2s+3}\,dS\\
&\lesssim \sigma^{2-2s}.
\end{align*}
On the other hand, $A_4=0$ holds, since $\nu(y)=\frac{y_1}{\abs{y}}$ implies
\begin{align*}
\int_{\abs{y}=\sigma}\left(\dfrac{-(-d-2s+2)y_1^2+2\abs{y}^2}{\abs{y}^{d+2s}(-d-2s+2)(-d-2s)}\right)\nu(y)\,dy=0.
\end{align*}
Then there holds
\begin{align*}
\biggabs{\int_{I_3}\abs{x}^{-\delta-2}x_1\left(\dfrac{-(-d-2s+2)y_1^2+2\abs{y}^2}{\abs{y}^{d+2s}(-d-2s+2)(-d-2s)}\right)\phi_{\epsilon}(x+y)\nu(y)\,dS}\lesssim \sigma^{2-2s}.
\end{align*}
In the similar way, the term corresponding to $j_1\divideontimes j_6$ can be bound by $\sigma^{2-2s}$.
Now we need to estimate the terms from $j_1\divideontimes j_5$ in $\widetilde{J}(x,y)$. Together with the observation $\partial_{y_i}\phi_{\epsilon}(x+y)=\partial_{y_i}(\phi_{\epsilon}(x+y)-\phi_{\epsilon}(x))$, the terms are of the forms
\begin{align*}
&\int_{I_3}\abs{x}^{-\delta-2}x_1\left[\abs{y}^{-d-2s+2}\partial_{y_i}\phi_{\epsilon}(x+y)-(\phi_{\epsilon}(x+y)-\phi_{\epsilon}(x))\partial_{y_i}(\abs{y}^{-d-2s+2})\right]\frac{y_i}{\abs{y}}\,dS\\
&=\int_{\abs{y}=\sigma}\int_{\sigma\leq\abs{x}\leq\frac{1}{\sigma}}\abs{x}^{-\delta-2}x_1\\
&\quad\quad\times\left[\abs{y}^{-d-2s+1}y_i(\partial_i\phi_{\epsilon})(x+y)\right.\\
&\quad\quad\quad\quad\quad\quad\left.-(-d-2s+2)(\phi_{\epsilon}(x+y)-\phi_{\epsilon}(x))\abs{y}^{-d-2s-1}y_i^2\right]\,dx\,dy
\end{align*}
for $i=1,2,\dots,n$. It suffices to show when $i=1$, and then the other terms estimated similarly. Note that
\begin{align*}
\int_{\abs{y}=\sigma}\abs{y}^{-d-2s+1}y_1\,dy=0
\end{align*}
and using $\abs{(\partial_{1}\phi_{\epsilon})(x+y)-(\partial_{1}\phi_{\epsilon})(x)}\lesssim\abs{y}$ for $\abs{y}=\sigma$ and $\sigma\leq\abs{x}\leq\frac{1}{\sigma}$, we have
\begin{align*}
\abs{y}^{-d-2s+1}[y_1(\partial_{1}\phi_{\epsilon})(x+y)-y_1(\partial_{1}\phi_{\epsilon})(x)]\lesssim \abs{y}^{-d-2s+3}.
\end{align*} 
Also, note that
\begin{align*}
\int_{\abs{y}=\sigma}\abs{y}^{-d-2s-1}y^2_iy\cdot\nabla\phi_{\epsilon}(x)\,dy=0.
\end{align*}
Then for $\abs{y}=\sigma$ and $\sigma\leq\abs{x}\leq\frac{1}{\sigma}$, using $\abs{\phi_{\epsilon}(x+y)-\phi_{\epsilon}(x)-y\cdot\nabla\phi_{\epsilon}(x)}\lesssim \abs{y}^2$ from $\phi\in\mathcal{S}(\setR^d)$, we have
\begin{align*}
\abs{y}^{-d-2s-1}y_i^2\abs{\phi_{\epsilon}(x+y)-\phi_{\epsilon}(x)-y\cdot\nabla\phi_{\epsilon}(x)}\lesssim \abs{y}^{-d-2s+3}.
\end{align*}
Hence, together with $\support\phi\subset B_{\epsilon}(e_1)$ we obtain
\begin{align*}
&\biggabs{\int_{I_3}\hspace{-1mm}\abs{x}^{-\delta-2}x_1\left[\abs{y}^{-d-2s+2}\partial_{y_i}\phi_{\epsilon}(x+y)\!-\!(\phi_{\epsilon}(x+y)\!-\!\phi_{\epsilon}(x))\partial_{y_i}(\abs{y}^{-d-2s+2})\right]\frac{y_i}{\abs{y}}\,dS}\\
&\quad\lesssim \int_{\abs{y}=\sigma}\int_{1-\epsilon-\sigma\leq\abs{x}\leq 1+\epsilon+\sigma}\abs{x}^{-\delta-1}\abs{y}^{-d-2s+3}\,dx\,dy\\
&\quad\lesssim \sigma^{d-1}\sigma^{-d-2s+3}\lesssim \sigma^{2-2s}.
\end{align*}

For the term corresponding to $j_2\divideontimes j_3$, using 
\begin{align*}
\abs{\nabla_x\phi_{\epsilon}(x+y)-\nabla_x\phi_{\epsilon}(x)}=\abs{\nabla\phi_{\epsilon}(x+y)-\nabla\phi_{\epsilon}(x)}\lesssim \abs{y}
\end{align*}
and
\begin{align*}
\int_{\abs{y}=\sigma}\left(\frac{-(-d-2s+2)y_1^2+2\abs{y}^2}{\abs{y}^{d+2s}(-d-2s+2)(-d-2s)}\right)\dfrac{y_1}{\abs{y}}\nabla\phi_{\epsilon}(x)\,dy=0,
\end{align*}
together with $\support\phi\subset B_{\epsilon}(e_1)$ we obtain the estimate
\begin{align*}
&\biggabs{\int_{\abs{y}=\sigma}\int_{\sigma\leq\abs{x}\leq\frac{1}{\sigma}}\frac{x\log\abs{x}}{\abs{x}^2}\left(\frac{-(-d-2s+2)y_1^2+2\abs{y}^2}{\abs{y}^{d+2s}(-d-2s+2)(-d-2s)}\right)\nabla_x\phi_{\epsilon}(x+y)\frac{y_1}{\abs{y}}\,dS}\\
&\quad\lesssim \biggabs{\int_{\abs{y}=\sigma}\int_{1-\epsilon-\sigma\leq\abs{x}\leq1+\epsilon+\sigma}\frac{\abs{\log\abs{x}}}{\abs{x}}\abs{y}^{-d-2s+2}\abs{y}\,dS}\lesssim \sigma^{2-2s}.
\end{align*}
Similarly, for the term corresponding to $j_2\divideontimes (j_4+j_5-j_6+j_7)$, we have
\begin{align*}
\biggabs{\int_{\abs{y}=\sigma}\int_{1-\epsilon-\sigma\leq\abs{x}\leq1+\epsilon+\sigma}\frac{\abs{\log\abs{x}}}{\abs{x}}\abs{y}^{-d-2s+3}\,dS}\lesssim \sigma^{2-2s}.
\end{align*}
Finally, we consider the integral on $I_4$. For the corresponding term to $j_1\divideontimes j_3$, using $\support\phi_{\epsilon}\subset B_{\epsilon}(e_1)$ we have
\begin{align*}
&\biggabs{\int_{I_4}\abs{x}^{-\delta-2}x_1\left(\dfrac{-(-d-2s+2)y_1^2+2\abs{y}^2}{\abs{y}^{d+2s}(-d-2s+2)(-d-2s)}\right)\phi_{\epsilon}(x+y)\nu(y)\,dS}\\
&\quad\lesssim\int_{\abs{y}=\frac{1}{\sigma}}\int_{\abs{x+y-e_1}\leq\epsilon}\sigma^{\delta+1}\sigma^{d+2s-2}\,dS\\
&\lesssim\left(\dfrac{1}{\sigma}\right)^{d-1}\sigma^{\delta+1}\sigma^{d+2s-2}\lesssim \sigma^{\delta+2s}.
\end{align*}
Similarly, for the term corresponding to $j_1\divideontimes (j_5-j_6)$ can be bound by $\sigma^{\delta+2s}$. For the term corresponding to $j_2\divideontimes j_3$, using $\sigma^s\abs{\log\sigma}\lesssim 1$, we obtain the estimate
\begin{align*}
&\biggabs{\int_{\abs{y}=\frac{1}{\sigma}}\int_{\frac{1}{\sigma}-1\leq\abs{x}\leq\frac{1}{\sigma}+1}\frac{\abs{\log\abs{x}}}{\abs{x}}\abs{y}^{-d-2s+2}\,dS}\\
&\quad\lesssim \sigma\abs{\log\sigma}\left(\frac{1}{\sigma}\right)^{d-1}\sigma^{d+2s-2}\leq \sigma^{s}.
\end{align*}
Similarly, the term corresponding to $j_2\divideontimes (j_4+j_5-j_6+j_7)$, we have
\begin{align*}
&\biggabs{\int_{\abs{y}=\frac{1}{\sigma}}\int_{\frac{1}{\sigma}-1\leq\abs{x}\leq\frac{1}{\sigma}+1}\frac{\abs{\log\abs{x}}}{\abs{x}}(\abs{y}^{-d-2s+2}+\abs{y}^{-d-2s+1})\,dS}\lesssim\sigma^{s}.
\end{align*}

When $d\geq 3$, again by divergence theorem,
\begin{align*}
&\iint_{\mathcal{F}_{\sigma}}J_{\epsilon}(x,y)\,dx\,dy\\
&\quad=\iint_{\mathcal{F}_{\sigma}}g(y)(|x|^{-2-\delta}x_1-\abs{x}^{-2})\phi_{\epsilon}(x+y)\,dx\,dy\\
&\quad\quad+\iint_{\mathcal{F}_{\sigma}}\dfrac{y_1^2}{|y|^{d+2s+2}}(|x|^{-2-\delta}x_1-\abs{x}^{-2})(\phi_{\epsilon}(x+y)-\phi_{\epsilon}(x))\,dx\,dy\\
&\quad\quad+\int_{\partial\mathcal{F}_{\sigma}}\overline{J}(x,y)\,dS
\end{align*}
for some $\overline{J}(x,y)$. Here, we need to estimate the term $\overline{J}(x,y)$ on the set $\partial\mathcal{F}_{\sigma}=I_1+I_2+I_3+I_4$. However, its proof is similar to the case of $d=2$ so we skip the proof. Overall, recalling $\mathcal{C}_\sigma$, $\mathcal{E}_\sigma$, and $\mathcal{F}_\sigma$, we have
\begin{align*}
&\int_{\setR^d}\int_{\setR^d}J_{\epsilon}(x,y)\,dx\,dy\\
&\quad=\lim_{\sigma\rightarrow 0}\left(\int_{\mathcal{C}_\sigma}J_{\epsilon}(x,y)\,dx\,dy+\int_{\mathcal{E}_\sigma}J_{\epsilon}(x,y)\,dx\,dy+\int_{\mathcal{F}_\sigma}J_{\epsilon}(x,y)\,dx\,dy\right)\\
&\quad=\mathcal{Z}_{\epsilon},
\end{align*}
where with $g(y)$ in \eqref{eq:g}, $\mathcal{Z}_{\epsilon}$ is of the form in \eqref{eq:Z2} when $d=2$ and \eqref{eq:Z3} when $d\geq 3$.

\textbf{Step 2: Proof of \eqref{eq:approx.identity}.} We need to show 
\begin{align}\label{eq:Z}
\lim_{\epsilon\rightarrow 0}\mathcal{Z}_{\epsilon}=f_2(d,s,\delta).
\end{align}
To do this, we only give the proof in the case of $d=2$ since when $d\geq 3$ the proof is similar. For fixed $\kappa\in(0,1/100)$, consider $\epsilon\in(0,1/100)$ with $2\epsilon\leq\kappa$ and define
\begin{align*}
&\mathcal{G}_{\kappa}=\{(x,y)\in\setR^{2d}:\abs{y}\leq\kappa\},\\
&\mathcal{H}_{\kappa}=\{(x,y)\in\setR^{2d}:\abs{e_1-y}\leq\kappa\},\\
&\mathcal{I}_{\kappa}=\setR^{2d}\setminus(\mathcal{G}_{\kappa}\cup\mathcal{H}_{\kappa}).
\end{align*}
We first estimate the integration on $\mathcal{G}_{\kappa}$. Using $\support\phi\subset B_{\epsilon}(e_1)$,
\begin{align*}
&\iint_{\mathcal{G}_{\kappa}}\dfrac{y_1^3}{\abs{y}^{d+2s+2}}|x|^{-2-\delta}x_1\phi_{\epsilon}(x+y)\,dx\,dy\\
&\,\,=\int_{\abs{y}\leq\kappa}\int_{\abs{x+y-e_1}\leq\epsilon}\dfrac{y_1^3}{\abs{y}^{d+2s+2}}|x|^{-2-\delta}x_1\phi_{\epsilon}(x+y)\,dx\,dy\\
&\,\,=\int_{\abs{y}\leq\kappa}\dashint_{\abs{x}\leq\epsilon}\dfrac{y_1^3}{\abs{y}^{d+2s+2}}|x-y+e_1|^{-2-\delta}(x_1-y_1+1)\overline{\phi}\left(\frac{x}{\epsilon}\right)\,dx\,dy.
\end{align*}
Note that for each $x\in B_{\epsilon}(0)$, there holds
\begin{align*}
\int_{\abs{y}\leq\kappa}\dfrac{y_1^3}{\abs{y}^{d+2s+2}}\abs{x+e_1}^{-2-\delta}(x_1+1)\,dy=0.
\end{align*}
Then with $f(z)=\abs{x-z+e_1}^{-2-\delta}(x_1-z_1+1)$, using
\begin{align*}
\abs{\abs{x-y+e_1}^{-2-\delta}(x_1-y_1+1)-\abs{x+e_1}^{-2-\delta}(x_1+1)}=\abs{f(y)-f(0)}\lesssim \abs{y}
\end{align*}
for $\abs{x}\leq\epsilon$ and $\abs{y}\leq\kappa$, we have
\begin{align*}
&\biggabs{\int_{\abs{y}\leq\kappa}\dashint_{\abs{x}\leq\epsilon}\dfrac{y_1^3}{\abs{y}^{d+2s+2}}|x-y+e_1|^{-2-\delta}(x_1-y_1+1)\overline{\phi}\left(\frac{x}{\epsilon}\right)\,dx\,dy}\\
&\quad\lesssim \int_{\abs{y}\leq\kappa}\dashint_{\abs{x}\leq\epsilon}\dfrac{\abs{y}^4}{\abs{y}^{d+2s+2}}\overline{\phi}\left(\frac{x}{\epsilon}\right)\,dx\,dy\lesssim \kappa^{2-2s}.
\end{align*}
Then
\begin{align}\label{eq:conv1}
\biggabs{\iint_{\mathcal{G}_{\kappa}}\dfrac{y_1^3}{\abs{y}^{d+2s+2}}|x|^{-2-\delta}x_1\phi_{\epsilon}(x+y)\,dx\,dy}\lesssim \kappa^{2-2s}.
\end{align}
Similarly, we obtain
\begin{align}\label{eq:conv2}
\biggabs{\iint_{\mathcal{G}_{\kappa}}g(y)|x|^{-2-\delta}x_1\phi_{\epsilon}(x+y)\,dx\,dy}\lesssim\kappa^{2-2s}.
\end{align}

On the other hand, using change of variables, 
\begin{align*}
&\abs{\abs{x-y+e_1}^{-2-\delta}(x_1-y_1+1)-\abs{x+e_1}^{-2-\delta}(x_1+1)-\nabla f(0)\cdot y}\\
&\quad=\abs{f(y)-f(0)-\nabla f(0)\cdot y}\lesssim \abs{y}^2,
\end{align*}
and
\begin{align*}
\int_{\abs{y}\leq\kappa}\dfrac{y_1^2}{\abs{y}^{d+2s+2}}\cdot y\,dy=0,
\end{align*}
we have
\begin{align}\label{eq:conv3}
\begin{split}
&\biggabs{\iint_{\mathcal{G}_{\kappa}}\dfrac{y_1^2}{|y|^{d+2s+2}}|x|^{-2-\delta}x_1(\phi_{\epsilon}(x+y)-\phi_{\epsilon}(x))\,dx\,dy}\\
&\,\,=\frac{c}{\epsilon^d}\biggabs{\iint_{\mathcal{G}_{\kappa}}\dfrac{y_1^2}{\abs{y}^{d+2s+2}}|x-y+e_1|^{-2-\delta}(x_1-y_1+1)\left(\overline{\phi}\left(\frac{x}{\epsilon}\right)-\overline{\phi}\left(\frac{x-y}{\epsilon}\right)\right)dxdy}\\
&\,\,=\frac{c}{\epsilon^d}\bigg|\iint_{\mathcal{G}_{\kappa}}\dfrac{y_1^2}{\abs{y}^{d+2s+2}}\\
&\quad\quad\quad\quad\times\left(|x-y+e_1|^{-2-\delta}(x_1-y_1+1)-\abs{x+e_1}^{-2-\delta}(x_1+1)\right)\overline{\phi}\left(\frac{x}{\epsilon}\right)dxdy\bigg|\\
&\,\,\lesssim \int_{\abs{y}\leq\kappa}\abs{y}^{-d-2s+2}\dashint_{\abs{x}\leq\epsilon}\phi\left(\frac{x}{\epsilon}\right)\,dx\,dy\lesssim \kappa^{2-2s}.
\end{split}
\end{align}
Similarly, observing that $\phi_{\epsilon}(y)\equiv 0$ when $\abs{y}\leq \kappa\leq 1/100$, we can obtain
\begin{align}\label{eq:conv4}
\begin{split}
&\biggabs{\iint_{\mathcal{G}_{\kappa}}g(y)\abs{x}^{-2}(\phi_{\epsilon}(x+y)-\indicator_{\{\abs{x}\leq 1\}}\phi_{\epsilon}(y))\,dx\,dy}\\
&\quad+\biggabs{\iint_{\mathcal{G}_{\kappa}}\dfrac{y_1^2}{|y|^{d+2s+2}}\abs{x}^{-2}[(\phi_{\epsilon}(x+y)-\phi_{\epsilon}(x))-\indicator_{\{\abs{x}\leq 1\}}(\phi_{\epsilon}(y)-\phi_{\epsilon}(0))]\,dy}\\
&=\biggabs{\iint_{\mathcal{G}_{\kappa}}g(y)\abs{x}^{-2}\phi_{\epsilon}(x+y)\,dx\,dy}+\biggabs{\iint_{\mathcal{G}_{\kappa}}\dfrac{y_1^2}{|y|^{d+2s+2}}\abs{x}^{-2}(\phi_{\epsilon}(x+y)-\phi_{\epsilon}(x))\,dy}\\
&\lesssim\kappa^{2-2s}.
\end{split}
\end{align}
Note that the implicit constant in the above inequalities do not depend on $\epsilon$.

Now we consider the integral on the set $\mathcal{H}_{\kappa}$. 
Using $\support\phi\subset B_{\epsilon}(e_1)$,
\begin{align*}
&\iint_{\mathcal{H}_{\kappa}}\dfrac{y_1^3}{\abs{y}^{d+2s+2}}|x|^{-2-\delta}x_1\phi_{\epsilon}(x+y)\,dx\,dy\\
&\,\,=\int_{\abs{e_1-y}\leq\kappa}\int_{\abs{x+y-e_1}\leq\epsilon}\dfrac{y_1^3}{\abs{y}^{d+2s+2}}|x|^{-2-\delta}x_1\phi_{\epsilon}(x+y)\,dx\,dy\\
&\,\,=\overline{c}\int_{\abs{e_1-y}\leq\kappa}\dashint_{\abs{x}\leq\epsilon}\dfrac{y_1^3}{\abs{y}^{d+2s+2}}|x-y+e_1|^{-2-\delta}(x_1-y_1+1)\overline{\phi}\left(\frac{x}{\epsilon}\right)\,dx\,dy.
\end{align*}
To estimate further, we have
\begin{align*}
&\overline{c}\int_{\abs{e_1-y}\leq\kappa}\dashint_{\abs{x}\leq\epsilon}\dfrac{y_1^3}{\abs{y}^{d+2s+2}}|x-y+e_1|^{-2-\delta}(x_1-y_1+1)\overline{\phi}\left(\frac{x}{\epsilon}\right)\,dx\,dy\\
&=\dfrac{\overline{c}}{\epsilon^d}\hspace{-3mm}\int_{\abs{e_1-y}\leq\kappa}\hspace{-3mm}\dfrac{y_1^3}{\abs{y}^{d+2s+2}}\hspace{-6mm}\int_{\frac{1}{2}\abs{e_1-y}\leq\abs{x}\leq\epsilon}\hspace{-6mm}\left(\abs{x-y+e_1}^{-2-\delta}(x_1-y_1+1)-\abs{x}^{-2-\delta}x_1\right)\overline{\phi}\left(\frac{x}{\epsilon}\right)\,dx\,dy\\
&\,\,+\dfrac{\overline{c}}{\epsilon^d}\int_{\abs{e_1-y}\leq\kappa}\int_{\abs{x}\leq\min\{\frac{1}{2}\abs{e_1-y},\epsilon\}}\dfrac{y_1^3}{\abs{y}^{d+2s+2}}\abs{x-y+e_1}^{-2-\delta}(x_1-y_1+1)\overline{\phi}\left(\frac{x}{\epsilon}\right)\,dx\,dy\\
&\lesssim\int_{\abs{e_1-y}\leq\kappa}\abs{y-e_1}^{-1-\delta}\dashint_{\abs{x}\leq\epsilon}\overline{\phi}\left(\frac{x}{\epsilon}\right)\,dx\,dy\lesssim\kappa^{d-1-\delta}\lesssim \kappa^{\frac{1}{2}},
\end{align*}
where for the first inequality we have used that 
\begin{align*}
\dashint_{\frac{1}{2}\abs{e_1-y}\leq\abs{x}\leq\epsilon}\abs{x}^{-2-\delta}x_1\overline{\phi}\left(\frac{x}{\epsilon}\right)\,dx=0
\end{align*}
as well as that on the set $\frac{1}{2}\abs{e_1-y}\leq\abs{x}\leq\epsilon$, by Lemma \ref{lem:basic}, we have
\begin{align*}
\abs{\abs{x-y+e_1}^{-2-\delta}(x_1-y_1+1)-\abs{x}^{-2-\delta}x_1}\lesssim \abs{x}^{-2-\delta}\abs{y-e_1},
\end{align*}
and on the set $\abs{x}\leq\min\{\frac{1}{2}\abs{e_1-y},\epsilon\}$, $\abs{\abs{x-y+e_1}^{-2-\delta}(x_1-y_1+1)}\lesssim \abs{y\!-\!e_1}^{-1-\delta}$ holds.
Hence we get
\begin{align}\label{eq:conv5}
\biggabs{\iint_{\mathcal{H}_{\kappa}}\dfrac{y_1^3}{\abs{y}^{d+2s+2}}|x|^{-2-\delta}x_1\phi_{\epsilon}(x+y)\,dx\,dy}\lesssim \kappa^{\min\{\frac{1}{2},2-2s\}}.
\end{align}
Similarly, we obtain
\begin{align}\label{eq:conv6}
\begin{split}
&\biggabs{\iint_{\mathcal{H}_{\kappa}}g(y)|x|^{-2-\delta}x_1\phi_{\epsilon}(x+y)\,dx\,dy}\\
&+\biggabs{\iint_{\mathcal{H}_{\kappa}}\dfrac{y_1^2}{|y|^{d+2s+2}}|x|^{-2-\delta}x_1\left(\phi_{\epsilon}(x+y)-\phi_{\epsilon}(x)\right)\,dx\,dy}\lesssim\kappa^{2-2s}.
\end{split}
\end{align}
On the other hand, we have
\begin{align*}
&\iint_{\mathcal{H}_{\kappa}}\dfrac{y_1^2}{|y|^{d+2s+2}}\abs{x}^{-2}[(\phi_{\epsilon}(x+y)-\phi_{\epsilon}(x))-\indicator_{\{\abs{x}\leq 1\}}(\phi_{\epsilon}(y)-\phi_{\epsilon}(0))]\,dx\,dy\\
&\quad=\int_{\abs{e_1-y}\leq\kappa}\frac{y_1^2}{\abs{y}^{d+2s+2}}\int_{\abs{x}\leq 1}\abs{x}^{-2}\left[(\phi_{\epsilon}(x+y)-\phi_{\epsilon}(x))-(\phi_{\epsilon}(y)-\phi_{\epsilon}(0))\right]\,dx\,dy\\
&\quad\quad+\int_{\abs{e_1-y}\leq\kappa}\frac{y_1^2}{\abs{y}^{d+2s+2}}\int_{\abs{x}\geq 1}\abs{x}^{-2}\phi_{\epsilon}(x)\,dx\,dy,
\end{align*}
where we have used the fact that when $\abs{e_1-y}\leq\kappa$ and $\abs{x}\geq 1$, $\phi_{\epsilon}(x+y)\equiv 0$, since $\abs{x+y-e_1}\geq\abs{x}-\abs{y-e_1}\geq 1-\kappa\geq\epsilon$. Then 
\begin{align*}
&\int_{\abs{e_1-y}\leq\kappa}\frac{y_1^2}{\abs{y}^{d+2s+2}}\int_{\abs{x}\geq 1}\abs{x}^{-2}\phi_{\epsilon}(x)\,dx\,dy\\
&\quad\quad\lesssim \int_{\abs{e_1-y}\leq\kappa}\dashint_{\abs{x-e_1}\leq\epsilon}\abs{x}^{-2}\overline{\phi}\left(\frac{x}{\epsilon}\right)\,dx\,dy\lesssim \kappa^2.
\end{align*}
Also, when $\abs{x}\leq 1$, $\phi_{\epsilon}(x)\equiv 0$ holds. Then together with $\phi_{\epsilon}(0)=0$ since $\epsilon\leq 1/100$, we have
\begin{align}\label{eq:epconv}
\begin{split}
&\int_{\abs{e_1-y}\leq\kappa}\frac{y_1^2}{\abs{y}^{d+2s+2}}\int_{\abs{x}\leq 1}\abs{x}^{-2}\left[(\phi_{\epsilon}(x+y)-\phi_{\epsilon}(x))-(\phi_{\epsilon}(y)-\phi_{\epsilon}(0))\right]\,dx\,dy\\
&\quad=\int_{\abs{e_1-y}\leq\kappa}\frac{y_1^2}{\abs{y}^{d+2s+2}}\int_{\abs{x}\leq 1}\abs{x}^{-2}\left(\phi_{\epsilon}(x+y)-\phi_{\epsilon}(y)\right)\,dx\,dy\\
&\quad=\int_{\abs{x}\leq 1}\abs{x}^{-2}\int_{\setR^d}\left(\indicator_{\abs{e_1-y+x}\leq\kappa}\dfrac{(y_1-x_1)^2}{\abs{y-x}^{d+2s+2}}-\indicator_{\abs{e_1-y}\leq\kappa}\dfrac{y_1^2}{\abs{y}^{d+2s+2}}\right)\phi_{\epsilon}(y)\,dy\,dx\\
&\quad\underset{\epsilon\rightarrow 0}{\longrightarrow}\int_{\abs{x}\leq 1}\abs{x}^{-2}\left(\indicator_{\abs{x}\leq\kappa}\dfrac{(1-x_1)^2}{\abs{e_1-x}^{d+2s+2}}-1\right)\,dx.
\end{split}
\end{align}
To check the convergence as $\epsilon\rightarrow 0$, note that when $\abs{x}\geq 1-2\kappa$ and $\abs{y-e_1}\leq\epsilon$, then $\abs{e_1-y+x}>\kappa$ since $\kappa\geq 2\epsilon$ and so
\begin{align*}
&\int_{\abs{x}\leq 1}\abs{x}^{-2}\int_{\setR^d}\left(\indicator_{\abs{e_1-y+x}\leq\kappa}\dfrac{(y_1-x_1)^2}{\abs{y-x}^{d+2s+2}}-\indicator_{\abs{e_1-y}\leq\kappa}\dfrac{y_1^2}{\abs{y}^{d+2s+2}}\right)\phi_{\epsilon}(y)\,dy\,dx\\
&\quad=-\int_{1-2\kappa\leq\abs{x}\leq 1}\abs{x}^{-2}\dashint_{B_{\epsilon}(e_1)}\indicator_{\abs{e_1-y}\leq\kappa}\dfrac{y_1^2}{\abs{y}^{d+2s+2}}\overline{\phi}\left(\frac{y}{\epsilon}\right)\,dy\,dx\\
&\quad\quad+\int_{\abs{x}\leq 1-2\kappa}\abs{x}^{-2}\dashint_{B_{\epsilon}(e_1)}(\indicator_{\abs{e_1-y+x}\leq\kappa}-\indicator_{\abs{e_1-y}\leq\kappa})\dfrac{(y_1-x_1)^2}{\abs{y-x}^{d+2s+2}}\overline{\phi}\left(\frac{y}{\epsilon}\right)\,dy\,dx\\
&\quad\quad+\int_{\abs{x}\leq 1-2\kappa}\abs{x}^{-2}\dashint_{B_{\epsilon}(e_1)}\indicator_{\abs{e_1-y}\leq\kappa}\left(\dfrac{(y_1-x_1)^2}{\abs{y-x}^{d+2s+2}}-\dfrac{y_1^2}{\abs{y}^{d+2s+2}}\right)\overline{\phi}\left(\frac{y}{\epsilon}\right)\,dy\,dx.
\end{align*}
Here, if $y\in B_{\epsilon}(e_1)$ then $\abs{e_1-y}\leq\kappa$ so that
\begin{align*}
&\int_{1-2\kappa\leq\abs{x}\leq 1}\abs{x}^{-2}\dashint_{B_{\epsilon}(e_1)}\indicator_{\abs{e_1-y}\leq\kappa}\dfrac{y_1^2}{\abs{y}^{d+2s+2}}\overline{\phi}\left(\frac{y}{\epsilon}\right)\,dy\,dx\\
&\quad\lesssim\int_{1-2\kappa\leq\abs{x}\leq 1}\abs{x}^{-2}\,dx\lesssim -\log(1-2\kappa)\lesssim\kappa.
\end{align*}
Also, if $\abs{x}\leq\kappa-\epsilon$ and $\abs{y-e_1}\leq\epsilon$, then $\abs{e_1-y+x}\leq\kappa$, thus
\begin{align*}
&\biggabs{\int_{\abs{x}\leq 1-2\kappa}\abs{x}^{-2}\dashint_{B_{\epsilon}(e_1)}(\indicator_{\abs{e_1-y+x}\leq\kappa}-\indicator_{\abs{e_1-y}\leq\kappa})\dfrac{(y_1-x_1)^2}{\abs{y-x}^{d+2s+2}}\overline{\phi}\left(\frac{y}{\epsilon}\right)\,dy\,dx}\\
&\quad\leq\int_{\kappa-\epsilon\leq\abs{x}\leq 1-2\kappa}\abs{x}^{-2}\dashint_{B_{\epsilon}(e_1)}\dfrac{(y_1-x_1)^2}{\abs{y-x}^{d+2s+2}}\overline{\phi}\left(\frac{y}{\epsilon}\right)\,dy\,dx\lesssim -\kappa^{-d-2s}\log\frac{1-2\kappa}{\kappa/2}.
\end{align*}
Moreover, we use
\begin{align*}
\biggabs{\dfrac{(y_1-x_1)^2}{\abs{y-x}^{d+2s+2}}-\dfrac{y_1^2}{\abs{y}^{d+2s+2}}}\lesssim \frac{\abs{x}}{\abs{y}^{d+2s+2}}
\end{align*}
when $\abs{x}\leq 1-2\kappa$ and $\abs{y-e_1}\leq\epsilon$ to obtain
\begin{align*}
&\biggabs{\int_{\abs{x}\leq 1-2\kappa}\abs{x}^{-2}\dashint_{B_{\epsilon}(e_1)}\indicator_{\abs{e_1-y}\leq\kappa}\left(\dfrac{(y_1-x_1)^2}{\abs{y-x}^{d+2s+2}}-\dfrac{y_1^2}{\abs{y}^{d+2s+2}}\right)\overline{\phi}\left(\frac{y}{\epsilon}\right)\,dy\,dx}\\
&\quad\lesssim\int_{\abs{x}\leq 1-2\kappa}\abs{x}^{-1}\dashint_{B_{\epsilon}(e_1)}\frac{1}{\abs{y}^{d+2s+2}}\overline{\phi}\left(\frac{y}{\epsilon}\right)\,dy\,dx\lesssim 1-2\kappa.
\end{align*}
Note that the implicit constant in the above inequalities again do not depend on $\epsilon$.
Thus for fixed $\kappa\in(0,1/100)$ with $\kappa\geq 2\epsilon$, by Lebesgue dominated convergence theorem, we obtain the convergence as $\epsilon\rightarrow 0$ in \eqref{eq:epconv}. Now,
\begin{align*}
&\int_{\abs{x}\leq 1}\abs{x}^{-2}\left(\indicator_{\abs{x}\leq\kappa}\dfrac{(1-x_1)^2}{\abs{e_1-x}^{d+2s+2}}-1\right)\,dx\\
&\quad=\int_{\abs{x}\leq \kappa}\abs{x}^{-2}\left(\dfrac{(1-x_1)^2}{\abs{e_1-x}^{d+2s+2}}-1\right)\,dx-\int_{\kappa\leq\abs{x}\leq 1}\abs{x}^{-2}\,dx.
\end{align*}
Here, by using $\abs{(1-x_1)^2\abs{e_1-x}^{-d-2s-2}-1}\lesssim\abs{x}$, observe that for $\abs{x}\leq\kappa$,
\begin{align}\label{eq:conv7}
\biggabs{\int_{\abs{x}\leq \kappa}\abs{x}^{-2}\left(\dfrac{(1-x_1)^2}{\abs{e_1-x}^{d+2s+2}}-1\right)\,dx}\lesssim\kappa^2.
\end{align}
Also, we write
\begin{align}\label{eq:conv8}
\int_{\kappa\leq\abs{x}\leq 1}\abs{x}^{-2}\,dx=:L_1(\kappa),
\end{align}
which is a large number since $\kappa\in(0,1/100)$ is small, but it will be canceled out to the other term in the following estimate \eqref{eq:conv11}. 

On the other hand, we have
\begin{align*}
&-\iint_{\mathcal{H}_{\kappa}}g(y)\abs{x}^{-2}[\phi_{\epsilon}(x+y)-\indicator_{\{\abs{x}\leq 1\}}\phi_{\epsilon}(y)]\,dx\,dy\\
&\quad=-\int_{\abs{e_1-y}\leq\kappa}g(y)\int_{\abs{x}\leq 1}\abs{x}^{-2}\left(\phi_{\epsilon}(x+y)-\phi_{\epsilon}(y)\right)\,dx\,dy\\
&\quad=-\int_{\abs{x}\leq 1}\abs{x}^{-2}\int_{\setR^d}\left(\indicator_{\abs{e_1-y+x}\leq\kappa}g(y-x)-\indicator_{\abs{e_1-y}\leq\kappa}g(y)\right)\phi_{\epsilon}(y)\,dy\,dx\\
&\quad\underset{\epsilon\rightarrow 0}{\longrightarrow}-\int_{\abs{x}\leq 1}\abs{x}^{-2}\left(\indicator_{\abs{x}\leq\kappa}g(e_1-x)-g(e_1)\right)\,dx\\
&\quad=-\int_{\abs{x}\leq 1}\abs{x}^{-2}\left(\indicator_{\abs{x}\leq\kappa}g(e_1-x)-1\right)\,dx,
\end{align*}
where the convergence $\epsilon\rightarrow 0$ is justified again by Lebesgue dominated convergence theorem and the similar argument as above. Now, we see that
\begin{align}\label{eq:conv9}
\begin{split}
&-\int_{\abs{x}\leq 1}\abs{x}^{-2}\left(\indicator_{\abs{x}\leq\kappa}g(e_1-x)-1\right)\,dx\\
&\quad=-\int_{\abs{x}\leq \kappa}\abs{x}^{-2}\left(g(e_1-x)-1\right)\,dx-L_1(\kappa).
\end{split}
\end{align}
Here, by using $\abs{g(e_1-x)-1}\lesssim\abs{x}$, observe that
\begin{align}\label{eq:conv10}
\biggabs{\int_{\abs{x}\leq \kappa}\abs{x}^{-2}\left(g(e_1-x)-1\right)\,dx}\lesssim\kappa^2.
\end{align}
Overall, merging the estimates for $\mathcal{H}_{\kappa}$ in \eqref{eq:conv5}--\eqref{eq:conv10}, we have
\begin{align}\label{eq:conv11}
\begin{split}
&\lim_{\epsilon\rightarrow 0}\left|\iint_{\mathcal{H}_{\kappa}}\frac{y_1^2}{\abs{y}^{d+2s+2}}\abs{x}^{-2}\left[(\phi_{\epsilon}(x+y)-\phi_{\epsilon}(x))-(\phi_{\epsilon}(y)-\phi_{\epsilon}(0))\right]\,dx\,dy-L_1(\kappa)\right.\\
&\quad \left.-\left(\iint_{\mathcal{H}_{\kappa}}g(y)\abs{x}^{-2}[\phi_{\epsilon}(x+y)-\indicator_{\{\abs{x}\leq 1\}}\phi_{\epsilon}(y)]\,dx\,dy-L_1(\kappa)\right)\right|\\
&\quad \lesssim\kappa^{2-2s}.
\end{split}
\end{align}

Now on the set $\mathcal{I}_{\kappa}$, the integrand including $g(y)$, $y_1^2\abs{y}^{-d-2s-2}$, $\abs{x}^{-2-\delta}x_1$, and $\abs{x}^{-2}$ are no longer singular, so we can apply the argument of the proof of \cite[Page 714, Theorem 7]{Eva10}. In detail, recalling $\kappa\geq 2\epsilon$ and $\support\phi\subset B_{\epsilon}(0)$, from
\begin{align*}
(x,y)\in\mathcal{I}_{\kappa}\,\,\implies\,\,\abs{e_1-y}\geq\kappa\,\,\implies\,\,\phi_{\epsilon}(y)\equiv 0,
\end{align*}
together with $\phi_{\epsilon}(0)\equiv0$ (also from $\support\phi\subset B_{\epsilon}(0)$) we have
\begin{align*}
&\iint_{\mathcal{I}_{\kappa}}\left(\dfrac{y_1^2}{|y|^{d+2s+2}}\abs{x}^{-2}[(\phi_{\epsilon}(x+y)-\phi_{\epsilon}(x))-\indicator_{\{\abs{x}\leq 1\}}(\phi_{\epsilon}(y)-\phi_{\epsilon}(0))]\right.\\
&\quad\quad\quad\quad\quad\quad\quad\quad\quad\left.-\dfrac{y_1^2}{|y|^{d+2s+2}}\abs{e_1-y}^{-2}\right)\,dx\,dy\\
&\quad=\iint_{\mathcal{I}_{\kappa}}\dfrac{y_1^2}{\abs{y}^{d+2s+2}}\left(\abs{x}^{-2}(\phi_{\epsilon}(x+y)-\phi_{\epsilon}(x))-\abs{e_1-y}^{-2}\right)\,dx\,dy\\
&\quad=\iint_{\mathcal{I}_{\kappa}}\dfrac{y_1^2}{\abs{y}^{d+2s+2}}\left(\abs{x}^{-2}-\abs{e_1-y}^{-2}\right)\phi_{\epsilon}(x+y)\,dx\,dy\\
&\quad\quad-
\iint_{\mathcal{I}_{\kappa}}\dfrac{y_1^2}{\abs{y}^{d+2s+2}}\abs{x}^{-2}\phi_{\epsilon}(x)\,dx\,dy.
\end{align*}
Then by the argument in \cite[Page 715]{Eva10}, we see that the first term in the rightmost term in the above converges to zero as $\epsilon\rightarrow 0$ by Lebesgue differentiation theorem. Moreover, the second term is computed as
\begin{align}\label{eq:conv12}
\mathcal{P}:=\iint_{\mathcal{I}_{\kappa}}\dfrac{y_1^2}{\abs{y}^{d+2s+2}}\abs{x}^{-2}\phi_{\epsilon}(x)\,dx\,dy\,\,\underset{\epsilon\rightarrow 0}{\longrightarrow}\,\,\int_{\substack{\abs{y}\geq\kappa\\\cap\abs{y-e_1}\geq\kappa}}\dfrac{y_1^2}{\abs{y}^{d+2s+2}}\,dy=:L_2(\kappa)
\end{align}
which is a large number since $\kappa$ is small, but it will be canceled out to the other term which appear soon. Similarly,
\begin{align}\label{eq:conv13}
\begin{split}
\mathcal{Q}:=&\iint_{\mathcal{I}_{\kappa}}\left(-\dfrac{y_1^2}{|y|^{d+2s+2}}|x|^{-2-\delta}x_1(\phi_{\epsilon}(x+y)-\phi_{\epsilon}(x))\right.\\
&\quad\quad\quad\quad\quad\quad\quad\left.+\dfrac{y_1^2}{|y|^{d+2s+2}}\abs{e_1-y}^{-2-\delta}(e_1-y)\right)\,dx\,dy\\
&\underset{\epsilon\rightarrow 0}{\longrightarrow}-\int_{\substack{\abs{y}\geq\kappa\\\cap\abs{y-e_1}\geq\kappa}}\dfrac{y_1^2}{\abs{y}^{d+2s+2}}\,dy=-L_2(\kappa).
\end{split}
\end{align}
Thus both $L_2(\kappa)$ is canceled out and so we have
\begin{align*}
\lim_{\epsilon\rightarrow 0}\mathcal{P}+\mathcal{Q}=0.
\end{align*}

The other terms estimated similarly so that the following term
\begin{align*}
&\bigg|\iint_{\mathcal{I}_{\kappa}}g(y)|x|^{-2-\delta}x_1\phi_{\epsilon}(x+y)\,dx\,dy\\
&\quad-\iint_{\mathcal{I}_{\kappa}}g(y)\abs{x}^{-2}(\phi_{\epsilon}(x+y)-\indicator_{\{\abs{x}\leq 1\}}\phi_{\epsilon}(y))\,dx\,dy\\
&\quad-\pvint_{\substack{\abs{y}\geq\kappa\\\cap\abs{e_1-y}\geq\kappa}}g(y)\left(|e_1-y|^{-2-\delta}(1-y_1)-\abs{e_1-y}^{-2}\right)dy\bigg|
\end{align*}
converges to 0 as $\epsilon\rightarrow 0$. Note that
\begin{align*}
f_2(d,s,\delta)&=\lim_{\kappa\rightarrow 0}\pvint_{\substack{\abs{y}\geq\kappa\\\cap\abs{e_1-y}\geq\kappa}}g(y)\left(|e_1-y|^{-2-\delta}(1-y_1)-\abs{e_1-y}^{-2}\right)dy,
\end{align*}
where $g(y)$ is defined in \eqref{eq:g}.

Therefore, recalling the proof of Lemma \ref{eq:estimate-Deltas-ud} so that
\begin{align*}
&\biggabs{\pvint_{\substack{\abs{y}\leq\kappa\\\cup\abs{e_1-y}\leq\kappa}}\dfrac{-2y_1^3+|y|^2y_1^2+2y_1^2-2|y|^2y_1+|y|^4}{2|y|^{d+2s+2}|e_1-y|^{2}}\left(|e_1-y|^{-\delta}(1-y_1)-1\right)dy}\\
&\quad\quad\lesssim o(\kappa),
\end{align*}
the estimates of the integral on $\mathcal{G}_{\kappa}$ and $\mathcal{H}_{\kappa}$ bounded by $\kappa^{2-2s}$ respectively by \eqref{eq:conv1}--\eqref{eq:conv4}, \eqref{eq:conv11}--\eqref{eq:conv13}, and also noticing that each $L_2(\kappa)$ appearing in \eqref{eq:conv12} and \eqref{eq:conv13} cancels out each other, we have obtained the desired estimate \eqref{eq:Z} since $\kappa\in(0,1/100)$ is arbitrary chosen and $\epsilon$ is already sent to zero. The proof is complete.
\end{proof}

Now a computation gives $\mathcal{F}[g_3](\xi)=\mathcal{G}_3(\xi)$ and $\mathcal{F}[g_4](\xi)=c\mathcal{G}_4(\xi)$ for some $c=c(d,s)$, where 
\begin{align*}
\mathcal{G}_3(\xi):=
\begin{cases}
-i\pi^{\delta+1-\frac{d}{2}}\dfrac{\Gamma\left(\frac{d-\delta}{2}\right)}{\Gamma\left(\frac{\delta+2}{2}\right)}|\xi|^{-d+\delta}\xi_1-2\pi\log|\xi|+2\pi(\log 2-\gamma)\,\,&\text{when }d=2\\
-i\pi^{\delta+1-\frac{d}{2}}\dfrac{\Gamma\left(\frac{d-\delta}{2}\right)}{\Gamma\left(\frac{\delta+2}{2}\right)}|\xi|^{-d+\delta}\xi_1+\pi^{2-\frac{d}{2}}\Gamma\left(\frac{d-2}{2}\right)|\xi|^{-d+2}\,\,&\text{when }d\neq 2
\end{cases}
\end{align*}
and
\begin{align*}
\mathcal{G}_4(\xi)&:=\dfrac{-4i(1-s)}{\pi(d+2s)}|x|^{2s-4}x_1^3+\dfrac{s-1}{\pi^2}|x|^{2s-4}x_1^2-\dfrac{4}{d+2s}|x|^{2s-2}x_1^2\\
&\quad+\dfrac{i(2d+4s+6)}{\pi(d+2s)}|x|^{2s-2}x_1+\dfrac{d+2s-1}{2\pi^2}|x|^{2s-2}-\dfrac{2}{s(d+2s)}|x|^{2s}
\end{align*}
are locally integrable functions in $\setR^d$. 

\begin{lemma}\label{lem:Fourier'}
We have 
\begin{align}\label{eq:Fourier'}
\mathcal{F}[g_3](\xi)=\mathcal{G}_3(\xi)\quad \text{and} \quad \mathcal{F}[g_4](\xi)=\pi^{\frac{d}{2}+2s}\dfrac{\Gamma(-s+1)}{\Gamma\left(\frac{d+2s}{2}\right)}\mathcal{G}_4(\xi)	
\end{align}
in the distributional sense.
\end{lemma}
\begin{proof}
The first term of the right-hand side of the above can be calculated as follows: when $\delta>0$,
\begin{align*}
\mathcal{F}\left[|h|^{-\delta-2}h_1\right](\xi)&=\left(\dfrac{i}{2\pi}\right)\partial_1\mathcal{F}\left[|h|^{-\delta-2}\right](\xi)\\
&=\dfrac{i}{2\pi}\pi^{\delta+2-\frac{d}{2}}\dfrac{\Gamma\left(\frac{d-\delta-2}{2}\right)}{\Gamma\left(\frac{\delta+2}{2}\right)}\partial_1(|\xi|^{-d+\delta+2})\\
&=-i\pi^{\delta+1-\frac{d}{2}}\dfrac{\Gamma\left(\frac{d-\delta}{2}\right)}{\Gamma\left(\frac{\delta+2}{2}\right)}|\xi|^{-d+\delta}\xi_1.
\end{align*}
When $\delta=0$, we still obtain the above equality using the following relation:
\begin{align*}
\mathcal{F}[\abs{h}^{-2}h_1](\xi)&=\mathcal{F}[\partial_1(\log\abs{h})]=2\pi i\xi_1\mathcal{F}[\log\abs{h}]\\
&=2\pi i\xi_1\left(-\frac{\Gamma\left(\frac{d}{2}\right)}{2\pi^{\frac{d}{2}}}\frac{1}{\abs{\xi}^d}+\left(\log2+\frac{\psi\left(\frac{d}{2}\right)-\gamma}{2}\right)\delta_0\right)\\
&=-i\pi^{1-\frac{d}{2}}\Gamma\left(\frac{d}{2}\right)\abs{\xi}^{-d}\xi_1,
\end{align*}
where $\psi(\cdot)$ is the digamma function, $\gamma$ is the Euler-Mascheroni constant, and $\delta_0$ is the Dirac delta distribution concentrated at 0, and for the third equality we have used \cite[Remark 1.2]{CheWet19}, and for the last line we have used $\xi_1\delta\equiv 0$.

We also see that
\begin{align*}
\mathcal{F}\left[|h|^{-2}\right]=
\begin{cases}
\pi^{2-\frac{d}{2}}\Gamma\left(\frac{d-2}{2}\right)|\xi|^{-d+2}&\quad d\neq 2\\
-2\pi\log|\xi|+2\pi(\log 2-\gamma)&\quad d=2.
\end{cases}
\end{align*}
Here, in case of $d=2$, from \cite[Remark 1.2]{CheWet19}, we have
\begin{align*}
\mathcal{F}\left[\frac{1}{|\xi|^d}\right](x)=\dfrac{2\pi^{\frac{d}{2}}}{\Gamma\left(\frac{d}{2}\right)}\left(-\log|x|+\log 2+\frac{\psi\left(\frac{d}{2}\right)-\gamma}{2}\right)
\end{align*}
in the distributional sense, where $\psi(\cdot)$ is the digamma function and $\gamma=-\Gamma'(1)$ is the Euler-Mascheroni constant.

Subsequent computations yield
\begin{align*}
\mathcal{F}\left[\dfrac{-2h_1^3}{|h|^{d+2s+2}}\right]&=-2\left(\dfrac{i}{2\pi}\right)^3\partial_1^3\mathcal{F}\left[|h|^{-d-2s-2}\right]\\
&=\dfrac{2i}{(2\pi)^3}\pi^{\frac{d}{2}+2s+2}\dfrac{\Gamma\left(-s-1\right)}{\Gamma\left(\frac{d+2s+2}{2}\right)}\partial_1^3(|\xi|^{2s+2})\\
&=i\pi^{\frac{d}{2}+2s-1}\dfrac{\Gamma\left(-s+1\right)}{\Gamma\left(\frac{d+2s}{2}\right)}\dfrac{2}{d+2s}\left(2(s-1)|\xi|^{2s-4}\xi_1^3+3|\xi|^{2s-2}\xi_1\right).
\end{align*}
Similarly, we have
\begin{align*}
\mathcal{F}\left[\dfrac{h_1^2}{|h|^{d+2s}}\right]=\dfrac{1}{2}\pi^{\frac{d}{2}+2s-2}\dfrac{\Gamma\left(-s+1\right)}{\Gamma\left(\frac{d+2s}{2}\right)}\left(2(s-1)|\xi|^{2s-4}\xi_1^2+|\xi|^{2s-2}\right),
\end{align*}
\begin{align*}
\mathcal{F}\left[\dfrac{2h_1^2}{|h|^{d+2s+2}}\right]=-\pi^{\frac{d}{2}+2s}\dfrac{\Gamma\left(-s+1\right)}{\Gamma\left(\frac{d+2s}{2}\right)}\frac{2}{d+2s}\left(2|\xi|^{2s-2}\xi_1^2+\frac{1}{s}|\xi|^{2s}\right),
\end{align*}
\begin{align*}
\mathcal{F}\left[\dfrac{-2h_1}{|h|^{d+2s}}\right]=2i\pi^{\frac{d}{2}+2s-1}\dfrac{\Gamma\left(-s+1\right)}{\Gamma\left(\frac{d+2s}{2}\right)}|\xi|^{2s-2}\xi_1,
\end{align*}
and
\begin{align*}
\mathcal{F}\left[\dfrac{1}{|h|^{d+2s-2}}\right]=\pi^{\frac{d}{2}+2s-2}\dfrac{\Gamma(-s+1)}{\Gamma\left(\frac{d+2s}{2}\right)}\dfrac{d+2s-2}{2}|\xi|^{2s-2}.
\end{align*}
Merging the terms calculated above, we obtain \eqref{eq:Fourier}.
\end{proof}

Now with $\oldphi_{\epsilon}(x):=\frac{1}{\epsilon^d}\overline{\phi}\left(\frac{x}{\epsilon}\right)$ with $\epsilon\in(0,1)$ where $\overline{\phi}$ is a standard mollifier in $\setR^d$, we obtain the following lemma.

\begin{lemma}\label{lem:Fourier}
The multiplication of distributions $\mathcal{G}_3\cdot \mathcal{G}_4$ is well-defined in $x\in\setR^d$ in the sense that for any $\phi\in \mathcal{S}(\setR^d)$,
\begin{align*}
\skp{\mathcal{G}_3\cdot \mathcal{G}_4}{\phi}=\lim_{\epsilon\rightarrow 0}\bigskp{\skp{\mathcal{G}_3(y)}{\oldphi_{\epsilon}(\cdot-y)}_y\mathcal{G}_4(\cdot)}{\phi(\cdot)}.
\end{align*}
Moreover, we have
\begin{align}\label{eq:f5f6}
\lim_{\epsilon\rightarrow 0}\bigskp{\skp{\mathcal{G}_3(y)}{\oldphi_{\epsilon}(\cdot-y)}_y\mathcal{G}_4(\cdot)}{\phi(\cdot)}=\skp{\mathcal{G}_3\mathcal{G}_4}{\phi},
\end{align}
where $\mathcal{G}_3\mathcal{G}_4$ means the usual multiplication.
\end{lemma}
\begin{proof}
Let us choose $\kappa\geq 2\epsilon$. We first prove for the case of $d=2$. Consider
\begin{align*}
&\biggabs{\hspace{-2mm}\int_{\abs{x}\leq\kappa}\hspace{-2mm}\left(\int_{\setR^d}\left[(\abs{y}^{-d+\delta}y_1+\log\abs{y})\!-\!(\abs{x}^{-d+\delta}x_1+\log\abs{x})\right]\oldphi_{\epsilon}(x-y)\,dy\right)\mathcal{G}_4(x)\phi(x)\,dx}\\
&\quad=\biggabs{\int_{\abs{x}\leq\kappa}(\abs{x}^{-d+\delta}x_1+\log\abs{x})\mathcal{G}_4(x)\phi(x)\,dx}\\
&\quad\quad+\biggabs{\overline{c}\dashint_{B_{\epsilon}(0)}\left(\int_{\abs{x}\leq\kappa}\mathcal{G}_4(x)\phi(x)(\abs{x-y}^{-d+\delta}(x_1-y_1)+\log\abs{x-y})\,dx\right)\overline{\phi}\left(\frac{y}{\epsilon}\right)\,dy}\\
&\quad=:\abs{I_1}+\abs{\overline{c}I_2}.
\end{align*}
For $I_1$, we estimate
\begin{align*}
\biggabs{\int_{\abs{x}\leq\kappa}\abs{x}^{-d+\delta}x_1\left(\abs{x}^{2s-4}x_1^3+\abs{x}^{2s-2}x_1+\abs{x}^{2s-2}x_1^2+\abs{x}^{2s}\right)\phi(x)\,dx}\lesssim\kappa^{2s+\delta},
\end{align*}	
and
\begin{align*}
\biggabs{\int_{\abs{x}\leq\kappa}\log\abs{x}\mathcal{G}_4(x)\phi(x)\,dx}\lesssim \int_{\abs{x}\leq\kappa}\abs{x}^{2s-2}\abs{\log\abs{x}}\abs{\phi(x)}\,dx\lesssim \kappa^s.
\end{align*}
Also, we see that 
\begin{align*}
&\biggabs{\int_{\abs{x}\leq\kappa}\abs{x}^{2s-4}(x_1^2+\abs{x}^2)\abs{x}^{-d+\delta}x\phi(x)\,dx}\\
&\quad=\biggabs{\int_{\abs{x}\leq\kappa}\abs{x}^{2s-4}(x_1^2+\abs{x}^2)\abs{x}^{-d+\delta}x\frac{\phi(x)-\phi(-x)}{2}\,dx}\\
&\quad\lesssim\int_{\abs{x}\lesssim\kappa}\abs{x}^{-d+\delta+2s}\,dx\lesssim \kappa^{\delta+2s}.
\end{align*}	
Here, for the last estimate we have used the fact that the map 
\begin{align*}
x\mapsto\abs{x}^{2s-4}(x_1^2+\abs{x}^2)\abs{x}^{-d+\delta}\frac{\phi(x)+\phi(-x)}{2}x
\end{align*}
is odd, and that since $\phi\in\mathcal{S}(\setR^d)$, $\abs{\phi(x)-\phi(-x)}\lesssim\abs{x}$ holds. Then we have
\begin{align*}
\biggabs{\int_{\abs{x}\leq\kappa}\abs{x}^{2s-4}(x_1^2+\abs{x}^2)\abs{x}^{-d+\delta}x_1\phi(x)\,dx}\lesssim \kappa^{\delta+2s}.
\end{align*}	
Summing up the above estimates, we obtain
\begin{align*}
\abs{I_1}\lesssim\kappa^{s}.
\end{align*}

Also for $I_2$, since $2s-2<0$ we estimate
\begin{align*}
&\biggabs{\dashint_{B_{\epsilon}(0)}\left(\int_{\abs{x}\leq\kappa}\mathcal{G}_4(x)\phi(x)\log\abs{x-y}\,dx\right)\overline{\phi}\left(\frac{y}{\epsilon}\right)dy}\\
&\quad\lesssim\dashint_{B_{\epsilon}(0)}\left(\int_{\abs{x}\leq\kappa}\abs{x}^{2s-2}\abs{\log\abs{x-y}}\,dx\right)\overline{\phi}\left(\frac{y}{\epsilon}\right)\,dy\\
&\quad\lesssim\dashint_{B_{\epsilon}(0)}\left(\int_{\abs{x}\leq\kappa}\abs{x}^{2s-2}\abs{\log\abs{x}}\,dx\right)\overline{\phi}\left(\frac{y}{\epsilon}\right)\,dy\lesssim\kappa^s.
\end{align*}
For the terms $\abs{x}^{2s-4}x_1^3+\abs{x}^{2s-2}x_1^2+\abs{x}^{2s-2}x_1+\abs{x}^{2s}$ in $\mathcal{G}_4(x)$, we see that
\begin{align*}
\biggabs{\dashint_{B_{\epsilon}(0)}\left(\int_{\abs{x}\leq\kappa}\abs{x}^{2s-1}\abs{\phi(x)}\abs{x-y}^{-d+\delta+1}\,dx\right)\overline{\phi}\left(\frac{y}{\epsilon}\right)dy}\lesssim\kappa^{2s+\delta}
\end{align*}
holds, and for the terms $\abs{x}^{2s-4}x_1^2$ and $\abs{x}^{2s-2}$ in $\mathcal{G}_4(x)$, we observe that
\begin{align*}
x\mapsto \dashint_{B_{\epsilon}(0)}\abs{x-y}^{-d+\delta}(x-y)\overline{\phi}\left(\frac{y}{\epsilon}\right)\,dy
\end{align*}
is odd, and so the map 
\begin{align*}
x\mapsto\abs{x}^{2s-4}x_1^2\frac{\phi(x)+\phi(-x)}{2}\left(\dashint_{B_{\epsilon}(0)}\abs{x-y}^{-d+\delta}(x-y)\overline{\phi}\left(\frac{y}{\epsilon}\right)\,dy\right)
\end{align*}
is odd. Moreover, since $\phi\in\mathcal{S}(\setR^d)$, $\abs{\phi(x)-\phi(-x)}\lesssim\abs{x}$ holds. Using this, we see that
\begin{align*}
&\biggabs{\int_{\abs{x}\leq\kappa}\abs{x}^{2s-4}x_1^2\phi(x)\left(\dashint_{B_{\epsilon}(0)}\abs{x-y}^{-d+\delta}(x-y)\overline{\phi}\left(\frac{y}{\epsilon}\right)\,dy\right)\,dx}\\
&\quad\lesssim\biggabs{\int_{\abs{x}\leq\kappa}\abs{x}^{2s-4}x_1^2\frac{\phi(x)-\phi(-x)}{2}\left(\dashint_{B_{\epsilon}(0)}\abs{x-y}^{-d+\delta}(x-y)\overline{\phi}\left(\frac{y}{\epsilon}\right)\,dy\right)\,dx}\\
&\quad\lesssim\biggabs{\dashint_{B_{\epsilon}(0)}\left(\int_{\abs{x}\leq\kappa}\abs{x}^{2s-1}\abs{x-y}^{-d+\delta+1}\overline{\phi}\left(\frac{y}{\epsilon}\right)\,dx\right)\,dy}\lesssim \kappa^{\delta+2s}.
\end{align*}
Then we obtain 
\begin{align*}
\biggabs{\int_{\abs{x}\leq\kappa}\abs{x}^{2s-4}x_1^2\phi(x)\left(\dashint_{B_{\epsilon}(0)}\abs{x-y}^{-d+\delta}(x_1-y_1)\overline{\phi}\left(\frac{y}{\epsilon}\right)\,dy\right)\,dx}\lesssim\kappa^{\delta+2s}.
\end{align*}
Similarly,
\begin{align*}
\biggabs{\int_{\abs{x}\leq\kappa}\abs{x}^{2s-2}\phi(x)\left(\dashint_{B_{\epsilon}(0)}\abs{x-y}^{-d+\delta}(x_1-y_1)\overline{\phi}\left(\frac{y}{\epsilon}\right)\,dy\right)\,dx}\lesssim\kappa^{\delta+2s}
\end{align*}
holds. Then merging the estimates, we have
\begin{align*}
\biggabs{\int_{\abs{x}\leq\kappa}\left(\int_{\setR^d}(\mathcal{G}_3(y)-\mathcal{G}_3(x))\oldphi_{\epsilon}(x-y)\,dy\right)\mathcal{G}_4(x)\phi(x)\,dx}\lesssim\kappa^s.
\end{align*}
Note that the implicit constants in the above estimates do not depend on $\epsilon$.

Now similar to the proof of Lemma \ref{lem:A2}, since on the set $\abs{x}\geq\kappa$ with $\kappa\geq 2\epsilon$, both $\mathcal{G}_3(x)$ and $\mathcal{G}_4(x)$ ar no longer singular, by Lebesgue differentiation theorem we have
\begin{align*}
&\lim_{\epsilon\rightarrow 0}\biggabs{\int_{\abs{x}\geq\kappa}\left(\int_{\setR^d}(\mathcal{G}_3(y)-\mathcal{G}_3(x))\oldphi_{\epsilon}(x-y)\,dy\right)\mathcal{G}_4(x)\phi(x)\,dx}=0.
\end{align*}
Hence, we have
\begin{align*}
&\lim_{\epsilon\rightarrow 0}\bigabs{\bigskp{\skp{\mathcal{G}_3(y)}{\oldphi_{\epsilon}(\cdot-y)}_y\mathcal{G}_4(\cdot)}{\phi(\cdot)}-\skp{\mathcal{G}_3\mathcal{G}_4}{\phi}}\\
&\leq\lim_{\epsilon\rightarrow 0}\biggabs{\int_{\abs{x}\leq\kappa}\left(\int_{\setR^d}(\mathcal{G}_3(y)-\mathcal{G}_3(x))\oldphi_{\epsilon}(x-y)\,dy\right)\mathcal{G}_4(x)\phi(x)\,dx}\\
&\quad+\lim_{\epsilon\rightarrow 0}\biggabs{\int_{\abs{x}\geq\kappa}\left(\int_{\setR^d}(\mathcal{G}_3(y)-\mathcal{G}_3(x))\oldphi_{\epsilon}(x-y)\,dy\right)\mathcal{G}_4(x)\phi(x)\,dx}\lesssim \kappa^s.
\end{align*}
Since $\kappa$ was arbitrary, we obtain \eqref{eq:f5f6} in case of $d=2$. In case of $d=3$ is similar so we skip the proof.
\end{proof}

With the help of Lemma \ref{lem:Fourier0} and \ref{lem:Fourier}, we have the following theorem.
\begin{theorem}\label{thm:Fourierconv}
We have 
\begin{align*}
\mathcal{F}[g_3\divideontimes g_4]=\mathcal{F}[g_3]\mathcal{F}[g_4]
\end{align*}
in the distributional sense, where $\mathcal{F}[g_3]\mathcal{F}[g_4]$ means the usual multiplication between $\mathcal{F}[g_3]$ and $\mathcal{F}[g_4]$.
\end{theorem}
\begin{proof}
For any $\phi\in\mathcal{S}(\setR^d)$, note that
\begin{align*}
\skp{\mathcal{F}[g_3\divideontimes g_4]}{\phi}&=\skp{\mathcal{F}[g_3]\cdot\mathcal{F}[g_4]}{\phi}\\
&=\lim_{\epsilon\rightarrow\infty 0}\skp{(\mathcal{F}[g_3]\divideontimes\oldphi_{\epsilon})\mathcal{F}[g_4]}{\phi}\\
&=\lim_{\epsilon\rightarrow\infty 0}\bigskp{\skp{\mathcal{F}[g_3](y)}{\oldphi_{\epsilon}(x-y)}_y\mathcal{F}[g_4](x)}{\phi(x)}_x\\
&=\skp{\mathcal{F}[g_3]\mathcal{F}[g_4]}{\phi},
\end{align*}
where for the first equality since $g_3,g_4\in\mathcal{S}'(\setR^d)$, we have used \cite[Page 96]{Vla02} together with Lemma \ref{lem:welldef3} and Lemma \ref{lem:conv3}, for the second equality since $\mathcal{F}[g_3]\divideontimes\oldphi_{\epsilon}\in\theta_M$ as in \cite[Page 84]{Vla02} ($\theta_M$ is defined in \eqref{eq:thetaM}) and $\mathcal{F}[g_4]$ in $\mathcal{S}'(\setR^d)$, $(\mathcal{F}[g_3]\divideontimes\oldphi_{\epsilon})\mathcal{F}[g_4]\in\mathcal{S}'(\setR^d)$ holds as in \cite[Page 79]{Vla02}, for the third equality again since $\mathcal{F}[g_3]\divideontimes\oldphi_{\epsilon}\in\theta_M$, \cite[Page 84]{Vla02} is used, and for the last equality Lemma \ref{lem:Fourier} together with Lemma \ref{lem:Fourier0} is used.
\end{proof}

Then it is worth mentioning the following corollary.
\begin{corollary}\label{cor:Fourierconv2}
	We have $(\mathcal{F}\circ\mathcal{F})[g_3\divideontimes g_4]=\mathcal{F}\left[\mathcal{F}[g_3]\mathcal{F}[g_4]\right]$.
\end{corollary}
\begin{proof}
Since $\phi\in\mathcal{S}(\setR^d)$ implies $\mathcal{F}[\phi]\in\mathcal{S}(\setR^d)$, by Theorem \ref{thm:Fourierconv} we have
\begin{align*}
\skp{\mathcal{F}[g_3\divideontimes g_4]}{\mathcal{F}[\phi]}=\skp{\mathcal{F}[g_3]\cdot\mathcal{F}[g_4]}{\mathcal{F}[\phi]}=\skp{\mathcal{F}[g_3]\mathcal{F}[g_4]}{\mathcal{F}[\phi]},
\end{align*}
which proves the conclusion.
\end{proof}

\subsection{Convolution theorem for $f_3(d,s,\delta)$}\label{sec:f3}

We define the distribution $g_7$, whose convolution with $g_6$ will recover $f_3(d,s,\delta)$. Note that from now on, we assume 
\begin{align*}
\delta\in[0,\tfrac{d}{2}).
\end{align*}
\begin{definition}\label{def:g5g6}
For $\phi\in \mathcal{S}(\setR^d)$, we define the distribution $g_5$ and $g_6$ in $\setR^d$ as
\begin{align*}
\skp{g_5}{\phi}:=\skp{\abs{x}^{s-\delta}\widehat{x_1}}{\phi}=\displaystyle\int_{\setR^d}\abs{x}^{s-\delta-1}x_1\phi(x)\,dx,
\end{align*}
and
\begin{align*}
\skp{g_6}{\phi}:=\skp{\abs{x}^{-d-1+s}}{\phi}=\displaystyle\int_{\setR^d}\abs{x}^{-d+1-s}\phi(x)\,dx.
\end{align*}
\end{definition}

Note that the above distributions are well-defined with $\phi\in \mathcal{S}(\setR^d)$, and $g_5$ and $g_6$ can be understood as locally integrable functions in $\setR^d$. The main goal of this subsection is to show the following relation (see Theorem \ref{thm:Fourierconv4}):
\begin{align}\label{eq:g5g6}
\mathcal{F}[g_5\divideontimes g_6]=\mathcal{F}[g_5]\mathcal{F}[g_6]-ic_1(d,s)(\partial_1\delta_0)\identity_{\{\delta=0\}}
\end{align}
in the distributional sense, where $\divideontimes$ is a convolution of two distributions in $\mathcal{S}'(\setR^d)$ introduced in \cite[Page 96]{Vla02}, and $\mathcal{F}[g_5]\mathcal{F}[g_6]$ means the usual multiplication between $\mathcal{F}[g_5]$ and $\mathcal{F}[g_6]$ and $c_1(d,s)$ is defined in Lemma \ref{lem:Fourier40}. The well-definedness of the left-hand side term of \eqref{eq:g5g6} will be proved in Lemma \ref{lem:conv4} and Lemma \ref{lem:conv40}, and the well-definedness of the right-hand side term of \eqref{eq:g5g6} will be proved in Lemma \ref{lem:Fourier40}. We emphasize that the additional term involving $(\partial_1\delta_0)$ in \eqref{eq:g5g6} is only necessary for $\delta=0$ case, see Lemma \ref{lem:Fourier40} and its proof below. This additional term with $(\partial_1\delta_0)$ is different from other results such as \eqref{eq:g3g4} and \eqref{eq:g7}, however this difference does not change much the application in Lemma \ref{lem:Fourier}.

To show \eqref{eq:g5g6}, first we prove the following lemma.
\begin{lemma}\label{lem:conv4}
For any $\phi\in \mathcal{S}(\setR^d)$, $\skp{g_5(y)}{\skp{g_6(x)}{\phi(x+y)}_x}_y$ is well-defined, i.e.,
\begin{align*}
\skp{g_5(y)}{\skp{g_6(x)}{\phi(x+y)}_x}_y=\int_{\setR^d}\abs{y}^{-d+1-s}\left(\int_{\setR^d}\abs{x}^{s-\delta-1}x_1\phi(x+y)\,dx\right)\,dy
\end{align*}
exists for any $\phi\in\mathcal{S}(\setR^d)$.
\end{lemma}
\begin{proof}
Let us choose $\phi\in \mathcal{S}(\setR^d)$. Note that for each $y\in\setR^d$, $\abs{x}^{s-\delta-1}x_1\phi(x+y)\in L^1(\setR^d)$ since $\phi\in\mathcal{S}(\setR^d)$. Since $y\mapsto\abs{y}^{-d+1-s}$ is even and $x\mapsto\abs{x}^{s-\delta-1}x_1$ is odd, observe that
\begin{align}\label{eq:even.odd}
\begin{split}
&\int_{\setR^d}\abs{y}^{-d+1-s}\left(\int_{\setR^d}\abs{x}^{s-\delta-1}x_1\phi(x+y)\,dx\right)\,dy\\
&\quad =\int_{\setR^d}\abs{y}^{-d+1-s}\left(\int_{\setR^d}\abs{x}^{s-\delta-1}x_1\dfrac{\phi(x+y)-\phi(-x-y)}{2}\,dx\right)\,dy.
\end{split}
\end{align}
Let $R\geq2$. When $\abs{y}\leq R$, we obtain

\begin{align*}
\Biggabs{\int_{\setR^d}\abs{x}^{s-\delta-1}x_1\phi(x+y)\,dx}\lesssim R^{d+s-\delta}+\int_{\setR^d\setminus B_R(0)}\abs{x}^{s-\delta}\abs{\phi(x+y)}\,dx.
\end{align*}
Here, if $-y\in B_{\frac{3}{4}R}(0)$, then using $\abs{\phi(x+y)}\lesssim\abs{x+y}^{-d-1}$, we obtain
\begin{align*}
\int_{\setR^d\setminus B_R(0)}\abs{x}^{s-\delta}\abs{\phi(x+y)}\,dx\lesssim\int_{\setR^d\setminus B_R(0)}\abs{x}^{s-\delta}\abs{x+y}^{-d-1}\,dx\lesssim R^{-1+s}.
\end{align*}
On the other hand, if $-y\not\in B_{\frac{3}{4}R}(0)$ so that $B_{\frac{1}{4}R}(-y)\cap B_{\frac{1}{2}R}(0)=\emptyset$, then employing $\abs{\phi(x+y)}\lesssim \min\{1,\abs{x+y}^{-d-s}\}$,
\begin{align*}
&\int_{\setR^d\setminus B_R(0)}\abs{x}^{s-\delta}\abs{\phi(x+y)}\,dx\\
&\quad\leq \int_{B_{\frac{1}{4}R}(-y)}\abs{x}^{s-\delta}\,dx+\int_{\setR^d\setminus (B_{\frac{1}{2}R}(0)\cup B_{\frac{1}{4}R}(-y))}\abs{x}^{s-\delta}\abs{x+y}^{-d-s}\,dx\lesssim R^{d+s}
\end{align*}
holds, since $R\geq 2$. The integral with the integrand $\abs{x}^{s-\delta-1}x_1\phi(-x-y)$ in \eqref{eq:even.odd} is estimated similarly. Hence if $\abs{y}\leq R$, we obtain
\begin{align}\label{eq:y<R4}
\Biggabs{\int_{\setR^d}\abs{x}^{s-\delta-1}x_1\dfrac{\phi(x+y)-\phi(-x-y)}{2}\,dx}\lesssim R^{d+s}.
\end{align}

Now we consider the case of $\abs{y}\geq R$. Using change of variables, let us write
\begin{align}\label{eq:x.delta4}
\begin{split}
&\int_{\setR^d}\abs{x}^{s-\delta-1}x_1\dfrac{\phi(x+y)-\phi(-x-y)}{2}\,dx\\
&=\frac{1}{2}\int_{\setR^d}\left(\abs{y+x}^{s-\delta-1}(y_1+x_1)-\abs{y-x}^{s-\delta-1}(y_1-x_1)\right)\phi(x)\,dx\,dy\\
&=\frac{1}{2}\int_{B_\frac{\abs{y}}{2}(0)}\left(\abs{y+x}^{s-\delta-1}(y_1+x_1)-\abs{y-x}^{s-\delta-1}(y_1-x_1)\right)\phi(x)\,dx\\
&\quad+\frac{1}{2}\int_{B_\frac{\abs{y}}{2}(y)}\left(\abs{y+x}^{s-\delta-1}(y_1+x_1)-\abs{y-x}^{s-\delta-1}(y_1-x_1)\right)\phi(x)\,dx\\
&\quad+\frac{1}{2}\hspace{-1mm}\int_{\setR^d\setminus (B_\frac{\abs{y}}{2}(0)\cup B_\frac{\abs{y}}{2}(y))}\hspace{-1mm}\left(\abs{y+x}^{s-\delta-1}(y_1+x_1)-\abs{y-x}^{s-\delta-1}(y_1-x_1)\right)\phi(x)dx.
\end{split}
\end{align}
Here, using
\begin{align*}
&\abs{\abs{y+x}^{s-\delta-1}(y_1+x_1)-\abs{y-x}^{s-\delta-1}(y_1-x_1)}\\
&\quad=\abs{\abs{y+x}^{s-\delta-1}(y_1+x_1)-\abs{y}^{s-\delta-1}y_1-\abs{y-x}^{s-\delta-1}(y_1-x_1)+\abs{y}^{s-\delta-1}y_1}\\
&\quad\lesssim \abs{y}^{s-\delta-2}\abs{x}^2,
\end{align*}
for $x\in B_{\frac{\abs{y}}{2}}(0)$, we obtain
\begin{align*}
&\int_{B_\frac{\abs{y}}{2}(0)}\left(\abs{y+x}^{s-\delta-1}(y_1+x_1)-\abs{y-x}^{s-\delta-1}(y_1-x_1)\right)\phi(x)\,dx\\
&\quad\lesssim \abs{y}^{s-\delta-2}\int_{B_\frac{\abs{y}}{2}(0)}\abs{x}^2\phi(x)\,dx\lesssim\abs{y}^{s-\delta-2}.
\end{align*}
For another term, using $\abs{\phi(x)}\lesssim \abs{x}^{-d-2}\eqsim \abs{y}^{-d-2}$ for $x\in B_\frac{\abs{y}}{2}(y)$, we estimate
\begin{align*}
&\int_{B_\frac{\abs{y}}{2}(y)}\left(\abs{y+x}^{s-\delta-1}(y_1+x_1)-\abs{y-x}^{s-\delta-1}(y_1-x_1)\right)\phi(x)\,dx\\
&\quad\lesssim \abs{y}^{d+s-\delta}\abs{y}^{-d-2}\lesssim \abs{y}^{s-\delta-2}.
\end{align*}
For the other term, using $\abs{\phi(x)}\lesssim \abs{x}^{-d-2}$, there holds
\begin{align*}
&\int_{\setR^d\setminus (B_\frac{\abs{y}}{2}(0)\cup B_\frac{\abs{y}}{2}(y))}\left(\abs{y+x}^{s-\delta-1}(y_1+x_1)-\abs{y-x}^{s-\delta-1}(y_1-x_1)\right)\phi(x)\,dx\\
&\quad\lesssim\abs{y}^{s-\delta}\abs{y}^{-2}\lesssim \abs{y}^{s-\delta-2}.
\end{align*}
Summing up, with \eqref{eq:y<R4} we see that
\begin{align}\label{eq:dist'}
\begin{split}
\int_{\setR^d}\abs{x}^{s-\delta-1}x_1\dfrac{\phi(x+y)-\phi(-x-y)}{2}\,dx\lesssim \,R^{d}\cdot\indicator_{\{\abs{y}\leq R\}}+\abs{y}^{s-\delta-2}\indicator_{\{\abs{y}\geq R\}}.
\end{split}
\end{align}

Then together with $\abs{y}^{-d+1-s}>0$, we estimate as follows for $L\geq 2R$: 
\begin{align*}
&\Biggabs{\int_{\abs{y}\leq L}\abs{y}^{-d+1-s}\left(\int_{\setR^d}\abs{x}^{s-\delta-1}x_1\dfrac{\phi(x+y)-\phi(-x-y)}{2}\,dx\right)\,dy}\\
&\lesssim \int_{\abs{y}\leq R}R^{d}\abs{y}^{-d+1-s}\,dy+\int_{R\leq\abs{y}\leq L}\abs{y}^{s-\delta-2}\abs{y}^{-d+1-s}\,dy\\
&\lesssim R^{d+1-s}+R^{-\delta-1}\lesssim R^{d+1}.
\end{align*}
Also, for $\widetilde{L}>L$, we have
\begin{align*}
&\Biggabs{\int_{L\leq\abs{y}\leq\widetilde{L}}\abs{y}^{-d+1-s}\left(\int_{\setR^d}\abs{x}^{s-\delta-1}x_1\dfrac{\phi(x+y)-\phi(-x-y)}{2}\,dx\right)\,dy}\\
&\lesssim\int_{L\leq\abs{y}\leq \widetilde{L}}\abs{y}^{s-\delta-2}\abs{y}^{-d+1-s}\,dy\lesssim L^{-\delta-1}.
\end{align*}
This proves that 
\begin{align*}
&\int_{\abs{y}\leq L}\abs{y}^{-d+1-s}\left(\int_{\setR^d}\abs{x}^{s-\delta-1}x_1\dfrac{\phi(x+y)-\phi(-x-y)}{2}\,dx\right)\,dy
\end{align*}
is a Cauchy sequence for $L$. Then we see that
\begin{align*}
&\skp{g_2(y)}{\skp{g_1(x)}{\phi(x+y)}_x}_y\\
&=\int_{\setR^d}\abs{y}^{-d+1-s}\left(\int_{\setR^d}\abs{x}^{s-\delta-1}x_1\dfrac{\phi(x+y)-\phi(-x-y)}{2}\,dx\right)\,dy
\end{align*}
exists, so we prove the conclusion. 
\end{proof}

Recalling $\eta_{k}(\cdot)$ in \eqref{eq:eta}, and $\xi_j(\cdot,\cdot)$ in \eqref{eq:xi}, we have the following lemma. Its proof is similar to Lemma \ref{lem:conv10} and \ref{lem:conv3} so we skip the proof.

\begin{lemma}\label{lem:conv40}
The convolution of distributions $g_1\divideontimes g_2$ is well-defined in $x\in\setR^d$ in the sense that for any $\phi\in \mathcal{S}(\setR^d)$, 
\begin{align}\label{eq:kj4}
\begin{split}
\skp{g_5\divideontimes g_6}{\phi}&=\skp{g_6\divideontimes g_5}{\phi}\\
&:=\lim_{k\rightarrow\infty}\lim_{j\rightarrow\infty}\bigskp{(\eta_kg_6)(y)}{\skp{g_5(x)}{\xi_j(x,y)\phi(x+y)}_x}_y.
\end{split}		
\end{align}
Moreover, we have
\begin{align*}
\skp{g_5\divideontimes g_6}{\phi}=\bigskp{g_6(y)}{\skp{g_5(x)}{\phi(x+y)}_x}_y.
\end{align*}
\end{lemma}

Here, we show that $f_3(d,s,\delta)$ exists.
\begin{lemma}\label{lem:welldef.f3}
The function $f_3(d,s,\delta)$ is well-defined in the sense that
\begin{align*}
\lim_{R\rightarrow \infty}\int_{\frac 1R \leq \abs{h} \leq R}\dfrac{\abs{e_1-h}^{s-\delta-1}(1-h_1)}{\abs{h}^{d-1+s}}\,dh=c(d,s,\delta)\in[0,\infty).
\end{align*}
\end{lemma}
\begin{proof}
We define
\begin{align*}
U_R&\coloneqq \int_{\frac 1R \leq \abs{h} \leq R}\dfrac{\abs{e_1-h}^{s-\delta-1}(1-h_1)}{\abs{h}^{d-1+s}}\,dh.
\end{align*}
Now, let $R > r \geq 2$. Then
\begin{align*}
U_R- U_r
&= \int_{\frac 1R \leq \abs{h} < \frac 1r} \dfrac{\abs{e_1-h}^{s-\delta-1}(1-h_1)}{\abs{h}^{d-1+s}}\,dh
\\
&\quad + \int_{r < \abs{h} \leq R}\dfrac{\abs{e_1-h}^{s-\delta-1}(1-h_1)}{\abs{h}^{d-1+s}}\,dh
\coloneqq \mathrm{I} + \mathrm{II}.
\end{align*}
We can see that
\begin{align*}
\mathrm{I} &\leq\int_{\frac 1R \leq \abs{h} < \frac 1r} \dfrac{\abs{e_1-h}^{s-\delta}}{\abs{h}^{d-1+s}}\,dh\lesssim r^{s-1}.
\end{align*}
Also, for $\mathrm{II}$, we estimate
\begin{align*}
&\int_{r < \abs{h} \leq R}\dfrac{\abs{e_1-h}^{s-\delta-1}(1-h_1)}{\abs{h}^{d-1+s}}\,dh\\
&\quad=\frac{1}{2}\int_{r < \abs{h} \leq R}\dfrac{\abs{e_1-h}^{s-\delta-1}(1-h_1)+\abs{e_1+h}^{s-\delta-1}(1+h_1)}{\abs{h}^{d-1+s}}\,dh\\
&\quad\lesssim\int_{r < \abs{h} \leq R}\dfrac{\abs{h}^{s-\delta-2}}{\abs{h}^{d-1+s}}\,dh\lesssim r^{-\delta-1},
\end{align*}
where we have used the fact that
\begin{align*}
&\abs{\abs{h+e_1}^{s-\delta-1}(h_1+1)-\abs{h-e_1}^{s-\delta-1}(h_1-1)}\\
&\quad=\abs{\abs{h+e_1}^{s-\delta-1}(h_1+1)-\abs{h}^{s-\delta-1}h_1-\abs{h-e_1}^{s-\delta-1}(h_1-1)+\abs{h}^{s-\delta-1}h_1}\\
&\quad\lesssim \abs{h}^{s-\delta-2}
\end{align*}
for $\abs{h}\geq r\geq 2$. This proves that $U_R$ is a Cauchy sequence for $R \to \infty$. Thus $f_3(d,s,\delta) = \lim_{R \to \infty} U_R$ exists. This proves the lemma.
\end{proof}

\begin{corollary}\label{cor:welldef.F3}
  For all $x \in \setR^d$ the limit
  \begin{align*}
  F_1(x):=\lim_{R \to \infty} \int_{\frac 1R \leq \abs{h} \leq R}\dfrac{\abs{x-h}^{s-\delta}\widehat{x_1-h_1}}{\abs{h}^{d-1+s}}\,dh
  \end{align*}
  exists. Moreover, $F_1(\cdot)\in L^1(B_1(0))$ and
  \begin{align*}
  \abs{F_1(x)} \lesssim \abs{x}^{1-\delta}.
  \end{align*}
\end{corollary}
\begin{proof}
For $x\in\setR^d\setminus\{0\}$, by change of variables, we can write
\begin{align*}
\int_{\frac 1R \leq \abs{h} \leq R}\dfrac{\abs{x-h}^{s-\delta}\widehat{x_1-h_1}}{\abs{h}^{d-1+s}}\,dh=\abs{x}^{1-\delta}\int_{\frac{\abs{x}}{R} \leq \abs{h} \leq \abs{x}R}\dfrac{\abs{e_1-h}^{s-\delta}\widehat{1-h_1}}{\abs{h}^{d-1+s}}\,dh.
\end{align*}
Applying a similar argument as in the proof of Lemma \ref{lem:welldef.f3} yields the conclusions. When $x=0$, then $F_1(0)=0$ since the map $h\mapsto \abs{h}^{-d-\delta}h_1$ is odd. This completes the proof.
\end{proof}

Now we prove the following lemma, which rigorously establishes the connection between $f_3(d,s,\delta)$ and $g_5\divideontimes g_6$.
\begin{lemma}\label{lem:A40}
We have
$f_3(d,s,\delta)=\lim_{\epsilon\rightarrow 0}\skp{g_5\divideontimes g_6}{\phi_{\epsilon}}$.
\end{lemma}

\begin{proof}
To do this, with $\epsilon\in(0,\frac{1}{10})$, we observe that
\begin{align*}
\skp{g_5\divideontimes g_6}{\phi_{\epsilon}}=\int_{\setR^d}\abs{y}^{-d+1-s}\left(\int_{\setR^d}\abs{x}^{s-\delta-1}x_1\phi_{\epsilon}(x+y)\,dx\right)\,dy.
\end{align*}
For fixed $\kappa\in(0,1/100)$, consider $\epsilon\in(0,1/100)$ with $2\epsilon\leq\kappa$ and define
\begin{align*}
&\mathcal{G}_{\kappa}=\{(x,y)\in\setR^{2d}:\abs{y}\leq\kappa\},\\
&\mathcal{H}_{\kappa}=\{(x,y)\in\setR^{2d}:\abs{e_1-y}\leq\kappa\},\\
&\mathcal{I}_{\kappa}=\setR^{2d}\setminus(\mathcal{G}_{\kappa}\cup\mathcal{H}_{\kappa}).
\end{align*}
We first estimate the integration on $\mathcal{G}_{\kappa}$. Using $\support\phi\subset B_{\epsilon}(e_1)$, we have
\begin{align*}
&\iint_{\mathcal{G}_{\kappa}}\abs{y}^{-d+1-s}\abs{x}^{s-\delta-1}x_1\phi_{\epsilon}(x+y)\,dx\,dy\\
&\,\,=\frac{\overline{c}}{\epsilon^d}\int_{\abs{y}\leq\kappa}\abs{y}^{-d+1-s}\int_{\setR^d}\abs{x-y-e_1}^{s-\delta-1}(x_1-y_1-1)\overline{\phi}\left(\frac{x}{\epsilon}\right)\,dx\,dy\\
&\,\,\lesssim \int_{\abs{y}\leq\kappa}\abs{y}^{-d+1-s}\dashint_{\abs{x}\leq\epsilon}\overline{\phi}\left(\frac{x}{\epsilon}\right)\,dx\,dy\lesssim \kappa^{1-s}.
\end{align*}

On the set $\mathcal{H}_{\kappa}$, using $\support\phi\subset B_1(0)$ and $\{\abs{e_1-y}\leq\kappa\}\subset\{\abs{x-y-e_1}\leq 2\kappa\}$ for each $x\in B_{\epsilon}(0)$, we compute
\begin{align*}
&\iint_{\mathcal{H}_{\kappa}}\abs{y}^{-d+1-s}\abs{x}^{s-\delta-1}x_1\phi_{\epsilon}(x+y)\,dx\,dy\\
&\quad=\dfrac{\overline{c}}{\epsilon^d}\int_{\abs{e_1-y}\leq\kappa}\abs{y}^{-d+1-s}\int_{\setR^d}\abs{x-y-e_1}^{s-\delta-1}(x_1-y_1-1)\overline{\phi}\left(\frac{x}{\epsilon}\right)\,dx\,dy\\
&\quad\lesssim \dashint_{\abs{x}\leq\epsilon}\left(\int_{\abs{e_1-y}\leq\kappa}\abs{x-y-e_1}^{s-\delta}\,dy\right)\overline{\phi}\left(\frac{x}{\epsilon}\right)\,dx\\
&\quad\lesssim \dashint_{\abs{x}\leq\epsilon}\left(\int_{\abs{x-y-e_1}\leq\kappa}\abs{x-y-e_1}^{s-\delta}\,dy\right)\overline{\phi}\left(\frac{x}{\epsilon}\right)\,dx\lesssim \kappa^{d+s-\delta}.
\end{align*}

On the set $\mathcal{I}_{\kappa}$, the integrand $\abs{y}^{-d+1-s}\abs{x}^{s-\delta-1}x_1\phi_{\epsilon}(x+y)$ are no longer singular, so we can apply the argument of the proof of \cite[Page 714, Theorem 7]{Eva10}. In detail, we compute
\begin{align*}
&\iint_{\mathcal{I}_{\kappa}}\abs{y}^{-d+1-s}\abs{x}^{s-\delta-1}x_1\phi_{\epsilon}(x+y)\,dx\,dy\\
&=\int_{\substack{\abs{e_1-y}\geq\kappa\\\cap\abs{y}\geq\kappa}}\int_{\setR^d}\abs{y}^{-d+1-s}\abs{x-y-e_1}^{s-\delta-1}(x_1-y_1-1)\overline{\phi}\left(\frac{x}{\epsilon}\right)\,dx\,dy\\
&\quad\,\,\underset{\epsilon\rightarrow 0}{\longrightarrow}\,\,-\int_{\substack{\abs{e_1-y}\geq\kappa\\\cap\abs{y}\geq\kappa}}\abs{y}^{-d+1-s}\abs{y+e_1}^{s-\delta-1}(y_1+1)\,dy.
\end{align*}

Since $\kappa\in(0,1/100)$ is arbitrary chosen and $\epsilon$ is already sent to zero, summing up above three estimates yields the conclusion.
\end{proof}

Now we have the following.
\begin{lemma}\label{lem:Fourier4}
We have 
\begin{align}\label{eq:Fourier4}
\begin{split}
&\mathcal{F}[g_5](\xi)=-\pi^{-\frac{d}{2}-s+\delta}i\dfrac{\Gamma\left(\frac{d+s-\delta+1}{2}\right)}{\Gamma\left(\frac{-s+\delta+1}{2}\right)}\abs{\xi}^{-d-s+\delta-1}\xi_1:=\mathcal{G}_5(\xi)\\&\quad \text{and} \quad \mathcal{F}[g_6](\xi)=\pi^{\frac{d}{2}+s-1}\dfrac{\Gamma(\frac{-s+1}{2})}{\Gamma\left(\frac{d+s-1}{2}\right)}\abs{\xi}^{s-1}:=\mathcal{G}_6(\xi)
\end{split}
\end{align}
in the distributional sense.
\end{lemma}
\begin{proof}
We compute
\begin{align*}
\mathcal{F}[g_5](\xi)&=\dfrac{1}{s-\delta+1}\mathcal{F}[\partial_1\abs{x}^{s-\delta+1}]\\
&=\dfrac{2\pi i\xi_1}{s-\delta+1}\mathcal{F}[\abs{x}^{s-\delta+1}]\\
&=\dfrac{2\pi i\xi_1}{s-\delta+1}\dfrac{(2\pi)^{-s+\delta-1}\Gamma\left(\frac{d+s-\delta+1}{2}\right)}{\pi^{\frac{d}{2}}2^{-s+\delta-1}\Gamma\left(\frac{-s+\delta-1}{2}\right)}\abs{\xi}^{-d-s+\delta-1}\\
&=-\pi^{-\frac{d}{2}-s+\delta}i\dfrac{\Gamma\left(\frac{d+s-\delta+1}{2}\right)}{\Gamma\left(\frac{-s+\delta+1}{2}\right)}\abs{\xi}^{-d-s+\delta-1}\xi_1
\end{align*}
and
\begin{align*}
\mathcal{F}[g_6](\xi)=\pi^{\frac{d}{2}+s-1}\dfrac{\Gamma(\frac{-s+1}{2})}{\Gamma\left(\frac{d+s-1}{2}\right)}\abs{\xi}^{s-1}.
\end{align*}
This completes the proof.
\end{proof}

Now with $\oldphi_{\epsilon}(x):=\frac{1}{\epsilon^d}\overline{\phi}\left(\frac{x}{\epsilon}\right)$ with $\epsilon\in(0,1)$ where $\overline{\phi}$ is a standard mollifier in $\setR^d$, we obtain the following lemma.

\begin{lemma}\label{lem:Fourier40}
The multiplication of distributions $\mathcal{G}_5\cdot \mathcal{G}_6$ is well-defined in $x\in\setR^d$ in the sense that for any $\phi\in \mathcal{S}(\setR^d)$,
\begin{align*}
\skp{\mathcal{G}_5\cdot \mathcal{G}_6}{\phi}=\lim_{\epsilon\rightarrow 0}\bigskp{\skp{\mathcal{G}_5(y)}{\oldphi_{\epsilon}(\cdot-y)}_y\mathcal{G}_6(\cdot)}{\phi(\cdot)}.
\end{align*}
Moreover, we have
\begin{align}\label{eq:g50g60}
\lim_{\epsilon\rightarrow 0}\bigskp{\skp{\mathcal{G}_5(y)}{\oldphi_{\epsilon}(\cdot-y)}_y\mathcal{G}_6(\cdot)}{\phi(\cdot)}=\skp{\mathcal{G}_5\mathcal{G}_6}{\phi},
\end{align}
where 
\begin{align*}
\mathcal{G}_5\mathcal{G}_6=-\pi^{\delta-1}i\dfrac{\Gamma\left(\frac{d+s-\delta+1}{2}\right)\Gamma\left(\frac{-s+1}{2}\right)}{\Gamma\left(\frac{-s+\delta+1}{2}\right)\Gamma\left(\frac{d+s-1}{2}\right)}\abs{\xi}^{-d+\delta-2}\xi_1-ic_1(d,s)(\partial_1\delta_0)\identity_{\{\delta=0\}}
\end{align*}
with $\partial_1\delta_0$ being the partial derivative with respect to $x_1$ of Dirac delta distribution concentrated at 0, and
\begin{align*}
c_1(d,s)=\frac{\pi^{\frac{d}{2}+1-2s}(d+s-1)}{d\Gamma\left(\frac{d}{2}\right)}.
\end{align*}
\end{lemma}
\begin{proof}
Let us choose $\kappa\geq 2\epsilon$ and consider
\begin{align*}
&\biggabs{\int_{\abs{x}\leq\kappa}\left(\int_{\setR^d}\left[\abs{y}^{-d-s+\delta-1}y_1-\abs{x}^{-d-s+\delta-1}x_1\right]\oldphi_{\epsilon}(x-y)\,dy\right)\abs{x}^{s-1}\phi(x)\,dx}\\
&\quad=\biggabs{\int_{\abs{x}\leq\kappa}\abs{x}^{-d+\delta-2}x_1\phi(x)\,dx}\\
&\quad\quad+\biggabs{\overline{c}\int_{\setR^d}\left(\int_{\abs{x}\leq\kappa}\abs{x}^{s-1}\phi(x)\abs{x-y}^{-d-s+\delta-1}(x_1-y_1)\,dx\right)\oldphi_{\epsilon}(y)\,dy}\\
&\quad=:\abs{I_1}+\abs{\overline{c}I_2}.
\end{align*}
Using the fact that $x\mapsto \abs{x}^{-d+\delta-2}x_1\frac{\phi(x)-\phi(-x)}{2}$ is odd, and $\abs{\phi(x)-\phi(-x)}\lesssim \abs{x}$, note that
\begin{align}\label{eq:I14}
\begin{split}
\abs{I_1}&=\biggabs{\int_{\abs{x}\leq\kappa}\abs{x}^{-d+\delta-2}x_1\frac{\phi(x)-\phi(-x)}{2}\,dx}\\
&=\int_{\abs{x}\leq\kappa}\abs{x}^{-d+\delta}\,dx\lesssim\kappa^{\delta}.
\end{split}
\end{align} 

For $I_2$, using the fact that the map
\begin{align*}
x\mapsto \abs{x}^{s-1}\dfrac{\phi(x)+\phi(-x)}{2}\int_{\setR^d}\abs{x-y}^{-d-s+\delta-1}(x_1-y_1)\oldphi_{\epsilon}(y)\,dy
\end{align*}
is odd since $\oldphi_{\epsilon}$ is even, we write
\begin{align}\label{eq:I2}
I_2&=\int_{\setR^d}\left(\int_{\abs{x}\leq\kappa}\abs{x}^{s-1}\frac{\phi(x)-\phi(-x)}{2}\abs{x-y}^{-d-s+\delta-1}(x_1-y_1)\,dx\right)\oldphi_{\epsilon}(y)\,dy.
\end{align}

Now assume that $\delta>0$. The case of $\delta=0$ will be considered later. We write
\begin{align*}
I_2&=\int_{\abs{x}\leq\kappa}\left(\int_{\setR^d}\abs{y}^{-d-s+\delta-1}y_1\oldphi_{\epsilon}(x-y)\,dy\right)\abs{x}^{s-1}\frac{\phi(x)-\phi(-x)}{2}\,dx\\
&=\int_{\abs{x}\leq 2\epsilon}\left(\int_{\setR^d}\abs{y}^{-d-s+\delta-1}y_1\oldphi_{\epsilon}(x-y)\,dy\right)\abs{x}^{s-1}\frac{\phi(x)-\phi(-x)}{2}\,dx\\
&\quad+\int_{2\epsilon\leq \abs{x}\leq\kappa}\left(\int_{\setR^d}\abs{y}^{-d-s+\delta-1}y_1\oldphi_{\epsilon}(x-y)\,dy\right)\abs{x}^{s-1}\frac{\phi(x)-\phi(-x)}{2}\,dx.
\end{align*}

Let us estimate the first term in the right-hand side. We have
\begin{align*}
&\int_{\abs{x}\leq 2\epsilon}\left(\int_{\setR^d}\abs{y}^{-d-s+\delta-1}y_1\oldphi_{\epsilon}(x-y)\,dy\right)\abs{x}^{s-1}\frac{\phi(x)-\phi(-x)}{2}\,dx\\
&=\int_{\abs{x}\leq 2\epsilon}\left(\int_{\abs{y}\leq\frac{1}{2}\abs{x}}\abs{y}^{-d-s+\delta-1}y_1\left(\oldphi_{\epsilon}(x-y)-\oldphi_{\epsilon}(x)\right)\,dy\right)\abs{x}^{s-1}\frac{\phi(x)-\phi(-x)}{2}\,dx\\
&\quad+\int_{\abs{x}\leq 2\epsilon}\left(\int_{\abs{y}\geq\frac{1}{2}\abs{x}}\abs{y}^{-d-s+\delta-1}y_1\oldphi_{\epsilon}(x-y)\,dy\right)\abs{x}^{s-1}\frac{\phi(x)-\phi(-x)}{2}\,dx,
\end{align*}
where we have used the fact that
\begin{align*}
\oldphi_{\epsilon}(x)\int_{\abs{y}\leq\frac{1}{2}\abs{x}}\abs{y}^{-d-s+\delta-1}y_1\,dy=0.
\end{align*}

Using $\abs{\oldphi_{\epsilon}(x-y)-\oldphi_{\epsilon}(x)}\lesssim \epsilon^{-1-d}\abs{y\cdot\nabla\overline{\phi}}$, we further compute
\begin{align*}
&\int_{\abs{x}\leq 2\epsilon}\left(\int_{\abs{y}\leq\frac{1}{2}\abs{x}}\abs{y}^{-d-s+\delta-1}y_1\left(\oldphi_{\epsilon}(x-y)-\oldphi_{\epsilon}(x)\right)\,dy\right)\abs{x}^{s-1}\frac{\phi(x)-\phi(-x)}{2}\,dx\\
&\quad\lesssim\frac{1}{\epsilon}\dashint_{\abs{x}\leq 2\epsilon}\left(\int_{\abs{y}\leq\frac{1}{2}\abs{x}}\abs{y}^{-d-s+\delta+1}\,dy\right)\abs{x}^{s}\,dx\lesssim \epsilon^{\delta}\lesssim \kappa^{\delta},
\end{align*}
where we have used $\abs{\phi(x)-\phi(-x)}\lesssim \abs{x}$. For the other term, using $\support\oldphi_{\epsilon}\subset B_{\epsilon}(0)$, we have
\begin{align*}
&\int_{\abs{x}\leq 2\epsilon}\left(\int_{\abs{y}\geq\frac{1}{2}\abs{x}}\abs{y}^{-d-s+\delta-1}y_1\oldphi_{\epsilon}(x-y)\,dy\right)\abs{x}^{s-1}\frac{\phi(x)-\phi(-x)}{2}\,dx\\
&\quad\lesssim \int_{\abs{x}\leq 2\epsilon}\abs{x}^{s}\epsilon^{\delta-s}\left(\int_{\abs{y}\leq\epsilon}\oldphi_{\epsilon}(y)\,dy\right)\,dx\lesssim \epsilon^{\delta}\lesssim \kappa^{\delta}.
\end{align*}
Note that the implicit constants in the above estimates do not depend on $\epsilon$. Also, using $\support\oldphi_{\epsilon}\subset B_{\epsilon}(0)$ and $\abs{\phi(x)-\phi(-x)}\lesssim \abs{x}$, we can estimate
\begin{align*}
&\int_{2\epsilon\leq \abs{x}\leq\kappa}\left(\int_{\setR^d}\abs{y}^{-d-s+\delta-1}y_1\oldphi_{\epsilon}(x-y)\,dy\right)\abs{x}^{s-1}\frac{\phi(x)-\phi(-x)}{2}\,dx\\
&\quad=\int_{2\epsilon\leq \abs{x}\leq\kappa}\left(\int_{\abs{x-y}\leq\epsilon}\abs{y}^{-d-s+\delta-1}y_1\oldphi_{\epsilon}(x-y)\,dy\right)\abs{x}^{s-1}\frac{\phi(x)-\phi(-x)}{2}\,dx\\
&\quad=\overline{c}\int_{2\epsilon\leq \abs{x}\leq\kappa}\left(\dashint_{\abs{y}\leq\epsilon}\abs{x-y}^{-d-s+\delta-1}(x_1-y_1)\overline{\phi}\left(\frac{y}{\epsilon}\right)\,dy\right)\abs{x}^{s-1}\frac{\phi(x)-\phi(-x)}{2}\,dx\\
&\underset{\epsilon\rightarrow 0}{\longrightarrow} \overline{c}\int_{\abs{x}\leq\kappa}\abs{x}^{-d+\delta-s}\abs{x}^{s-1}\frac{\phi(x)-\phi(-x)}{2}\,dx\\
&\quad\leq \int_{\abs{x}\leq\kappa}\abs{x}^{-d+\delta}\,dx\lesssim \kappa^{\delta}.
\end{align*}

Then merging the estimates, we have
\begin{align*}
\biggabs{\int_{\abs{x}\leq\kappa}\left(\int_{\setR^d}(\mathcal{G}_1(y)-\mathcal{G}_1(x))\oldphi_{\epsilon}(x-y)\,dy\right)\mathcal{G}_2(x)\phi(x)\,dx}\lesssim\kappa^{\delta}.
\end{align*}

Now similar to the proof of Lemma \ref{lem:A2}, since on the set $\abs{x}\geq\kappa$ with $\kappa\geq 2\epsilon$, both $\mathcal{G}_5(x)$ and $\mathcal{G}_6(x)$ are no longer singular, by Lebesgue differentiation theorem we have
\begin{align*}
&\lim_{\epsilon\rightarrow 0}\biggabs{\int_{\abs{x}\geq\kappa}\left(\int_{\setR^d}(\mathcal{G}_5(y)-\mathcal{G}_5(x))\oldphi_{\epsilon}(x-y)\,dy\right)\mathcal{G}_6(x)\phi(x)\,dx}=0.
\end{align*}
Then together with \eqref{eq:I14}, we obtain 
\begin{align*}
&\lim_{\epsilon\rightarrow 0}\bigabs{\bigskp{\skp{\mathcal{G}_5(y)}{\oldphi_{\epsilon}(\cdot-y)}_y\mathcal{G}_6(\cdot)}{\phi(\cdot)}-\skp{\mathcal{G}_5\mathcal{G}_6}{\phi}}\\
&\quad\leq\lim_{\epsilon\rightarrow 0}\biggabs{\int_{\abs{x}\leq\kappa}\left(\int_{\setR^d}(\mathcal{G}_5(y)-\mathcal{G}_5(x))\oldphi_{\epsilon}(x-y)\,dy\right)\mathcal{G}_6(x)\phi(x)\,dx}\\
&\quad +\lim_{\epsilon\rightarrow 0}\biggabs{\int_{\abs{x}\geq\kappa}\left(\int_{\setR^d}(\mathcal{G}_5(y)-\mathcal{G}_5(x))\oldphi_{\epsilon}(x-y)\,dy\right)\mathcal{G}_6(x)\phi(x)\,dx}\lesssim \kappa^{\delta}.
\end{align*}
Since $\kappa$ was arbitrary, we get \eqref{eq:g50g60} in case of $\delta>0$.

In case of $\delta=0$, using change of variables, from \eqref{eq:g50g60} we estimate
\begin{align*}
I_2&=\dashint_{B_{\epsilon}(0)}\left(\int_{\abs{x}\leq\kappa}\abs{x}^{s-1}\dfrac{\phi(x)-\phi(-x)}{2}\abs{x-y}^{-d-s-1}(x_1-y_1)\,dx\right)\overline{\phi}\left(\frac{y}{\epsilon}\right)\,dy\\
&=\dashint_{B_{\frac{\epsilon}{\kappa}}(0)}\left(\int_{\abs{x}\leq 1}\kappa^{d}\kappa^{s-1}\abs{x}^{s-1}\phi(\kappa x)\kappa^{-d-s}\abs{x-y}^{-d-s-1}(x_1-y_1)\,dx\right)\overline{\phi}\left(\frac{\kappa y}{\epsilon}\right)\,dy\\
&\underset{\epsilon\rightarrow 0}{\longrightarrow}\dfrac{1}{\kappa}\int_{\abs{x}\leq 1}\abs{x}^{s-1}\phi(\kappa x)\abs{x}^{-d-s-1}x_1\,dx\\
&\quad=\dfrac{1}{\kappa}\int_{\abs{x}\leq 1}\phi(\kappa x)\abs{x}^{-d-2}x_1\,dx.
\end{align*}
Estimating further, using the fact that $x\mapsto\frac{\phi(\kappa x)+\phi(-\kappa x)}{2}\abs{x}^{-d-2}x_1$ is odd, we see that
\begin{align*}
\dfrac{1}{\kappa}\int_{\abs{x}\leq 1}\phi(\kappa x)\abs{x}^{-d-2}x_1\,dx&=\dfrac{1}{\kappa}\int_{\abs{x}\leq 1}\abs{x}^{-d-2}x_1\frac{\phi(\kappa x)-\phi(-\kappa x)}{2}\,dx\\
&=\dfrac{1}{\kappa}\int_{\abs{x}\leq 1}\abs{x}^{-d-1}x_1\frac{\phi(\kappa x)-\phi(-\kappa x)}{2\kappa\abs{x}}\,dx\\
&\underset{\kappa\rightarrow 0}{\longrightarrow}\int_{\abs{x}\leq 1}\abs{x}^{-d-1}x_1x\cdot\nabla\phi(0)\,dx\\
&\quad=\partial_1\phi(0)\int_{\abs{x}\leq 1}\abs{x}^{-d-1}x_1^2\,dx.
\end{align*}
Here, using \cite{Fol01}, we calculate
\begin{align*}
\int_{\abs{x}\leq 1}\abs{x}^{-d-1}x_1^2\,dx&=\int_{\abs{x}\leq 1}\abs{x}^{-d+1}\frac{x_1^2}{\abs{x}^2}\,dx\\
&=\int_{0}^{1}\left(\int_{S^d}\kappa_1^2\,d\kappa\right) r^{-d+1}r^{d-1}\,dr=\int_{S^d}\kappa_1^2\,d\kappa=\frac{\pi^{\frac{d}{2}}}{\Gamma\left(\frac{d+2}{2}\right)}.
\end{align*}
Thus we have obtained
\begin{align*}
\lim_{\kappa\rightarrow 0}\lim_{\epsilon\rightarrow 0}I_2=\frac{\pi^{\frac{d}{2}}}{\Gamma\left(\frac{d+2}{2}\right)}\partial_1\phi(0)=\frac{\pi^{\frac{d}{2}}}{\Gamma\left(\frac{d+2}{2}\right)}\skp{\partial_1\delta_0}{\phi}
\end{align*}
so that recalling the coefficients of $\abs{\xi}^{-d-s+\delta-1}\xi_1$ and $\abs{\xi}^{s-1}$ in \eqref{eq:Fourier4}, finally we get
\begin{align*}
&\mathcal{G}_5\mathcal{G}_6=\pi^{1-2s+\delta}i\dfrac{\Gamma\left(\frac{d+s-\delta+1}{2}\right)\Gamma\left(\frac{-s+1}{2}\right)}{\Gamma\left(\frac{-s+\delta+1}{2}\right)\Gamma\left(\frac{d+s-1}{2}\right)}\left(\abs{\xi}^{-d+\delta-2}\xi_1+\frac{\pi^{\frac{d}{2}}}{\Gamma\left(\frac{d+2}{2}\right)}(\partial_1\delta_0)\identity_{\{\delta=0\}}\right)\\
&=\pi^{1-2s+\delta}i\dfrac{\Gamma\left(\frac{d+s-\delta+1}{2}\right)\Gamma\left(\frac{-s+1}{2}\right)}{\Gamma\left(\frac{-s+\delta+1}{2}\right)\Gamma\left(\frac{d+s-1}{2}\right)}\abs{\xi}^{-d+\delta-2}\xi_1\!+\!\frac{\pi^{\frac{d}{2}+1-2s}(d+s-1)i}{d\Gamma\left(\frac{d}{2}\right)}(\partial_1\delta_0)\identity_{\{\delta=0\}}.
\end{align*}
The proof is complete.
\end{proof}

With the help of Lemma \ref{lem:Fourier4} and \ref{lem:Fourier40}, we have the following theorem.
\begin{theorem}\label{thm:Fourierconv4}
We have 
\begin{align*}
\mathcal{F}[g_5\divideontimes g_6]=\mathcal{F}[g_5]\mathcal{F}[g_6]-ic_1(d,s)(\partial_1\delta_0)\identity_{\{\delta=0\}}
\end{align*}
in the distributional sense, where $\mathcal{F}[g_5]\mathcal{F}[g_6]$ means the usual multiplication between $\mathcal{F}[g_5]$ and $\mathcal{F}[g_6]$, $\partial_1\delta_0$ means the partial derivative with respect to $x_1$ of Dirac delta distribution concentrated at 0, and $c_1(d,s)$ is defined in Lemma \ref{lem:Fourier40}.
\end{theorem}
\begin{proof}
For any $\phi\in\mathcal{S}(\setR^d)$, note that
\begin{align*}
\skp{\mathcal{F}[g_5\divideontimes g_6]}{\phi}&=\skp{\mathcal{F}[g_5]\cdot\mathcal{F}[g_6]}{\phi}\\
&=\lim_{\epsilon\rightarrow 0}\skp{(\mathcal{F}[g_5]\divideontimes\oldphi_{\epsilon})\mathcal{F}[g_6]}{\phi}\\
&=\lim_{\epsilon\rightarrow 0}\bigskp{\skp{\mathcal{F}[g_5](y)}{\oldphi_{\epsilon}(x-y)}_y\mathcal{F}[g_6](x)}{\phi(x)}_x\\
&=\skp{\mathcal{F}[g_5]\mathcal{F}[g_6]+ic_1(d,s)(\partial_1\delta_0)\identity_{\{\delta=0\}}}{\phi},
\end{align*}
where for the first equality since $g_5,g_6\in\mathcal{S}'(\setR^d)$, we have used \cite[Page 96]{Vla02} together with Lemma \ref{lem:conv4} and Lemma \ref{lem:conv40}, for the second equality since $\mathcal{F}[g_5]\divideontimes\oldphi_{\epsilon}\in\theta_M$ as in \cite[Page 84]{Vla02} ($\theta_M$ is defined in \eqref{eq:thetaM}) and $\mathcal{F}[g_6]$ in $\mathcal{S}'(\setR^d)$, $(\mathcal{F}[g_5]\divideontimes\oldphi_{\epsilon})\mathcal{F}[g_6]\in\mathcal{S}'(\setR^d)$ holds as in \cite[Page 79]{Vla02}, for the third equality again since $\mathcal{F}[g_5]\divideontimes\oldphi_{\epsilon}\in\theta_M$, \cite[Page 84]{Vla02} is used, and for the last equality Lemma \ref{lem:Fourier40} together with Lemma \ref{lem:Fourier4} is used.
\end{proof}

Then it is worth mentioning the following corollary.
\begin{corollary}\label{cor:Fourierconv40}
	We have 
\begin{align*}
(\mathcal{F}\circ\mathcal{F})[g_5\divideontimes g_6]=\mathcal{F}\left[\mathcal{F}[g_5]\mathcal{F}[g_6]-ic_1(d,s)(\partial_1\delta_0)\identity_{\{\delta=0\}}\right].
\end{align*}
\end{corollary}
\begin{proof}
Since $\phi\in\mathcal{S}(\setR^d)$ implies $\mathcal{F}[\phi]\in\mathcal{S}(\setR^d)$, by Theorem \ref{thm:Fourierconv4} we have
\begin{align*}
\skp{\mathcal{F}[g_5\divideontimes g_6]}{\mathcal{F}[\phi]}&=\skp{\mathcal{F}[g_5]\cdot\mathcal{F}[g_6]}{\mathcal{F}[\phi]}\\
&=\skp{\mathcal{F}[g_5]\mathcal{F}[g_6]+ic_1(d,s)(\partial_1\delta_0)\identity_{\{\delta=0\}}}{\mathcal{F}[\phi]},
\end{align*}
which proves the conclusion.
\end{proof}

\subsection{Convolution theorem for $f_4(d,s,\delta)$}\label{sec:f4}

We define the distribution $g_7$, where its convolution with $g_6$ will recover $f_4(d,s,\delta)$. Recall that here we also assume $\delta\in[0,\tfrac{d}{2})$.
\begin{definition}\label{def:g7}
For $\phi\in \mathcal{S}(\setR^d)$, we define the distribution $g_7$ in $\setR^d$ as
\begin{align*}
\skp{g_7}{\phi}:=\skp{\abs{x}^{-1-\delta}\widehat{x_1}}{\phi}=\displaystyle\int_{\setR^d}\abs{x}^{-\delta-2}x_1\phi(x)\,dx.
\end{align*}
\end{definition}

Note that the above distribution is well-defined with $\phi\in \mathcal{S}(\setR^d)$, and $g_7$ can be understood as locally integrable functions in $\setR^d$. The main goal of this subsection is to show the following relation (see Theorem \ref{thm:Fourierconv5}):
\begin{align}\label{eq:g7}
\mathcal{F}[g_7\divideontimes g_6]=\mathcal{F}[g_7]\mathcal{F}[g_6]
\end{align}
in the distributional sense, where $\divideontimes$ is a convolution of two distributions in $\mathcal{S}'(\setR^d)$ introduced in \cite[Page 96]{Vla02}, and $\mathcal{F}[g_7]\mathcal{F}[g_6]$ means the usual multiplication between $\mathcal{F}[g_7]$ and $\mathcal{F}[g_6]$. The well-definedness of the left-hand side term of \eqref{eq:g7} will be proved in Lemma \ref{lem:conv5} and Lemma \ref{lem:conv50}, and the well-definedness of the right-hand side term of \eqref{eq:g7} will be proved in Lemma \ref{lem:Fourier50}. Overall, proofs are similar to those of the previous subsection, but for the completeness we give proofs.

To show \eqref{eq:g7}, first we prove the following lemma.
\begin{lemma}\label{lem:conv5}
For any $\phi\in \mathcal{S}(\setR^d)$, $\skp{g_6(y)}{\skp{g_7(x)}{\phi(x+y)}_x}_y$ is well-defined, i.e.,
\begin{align*}
\skp{g_6(y)}{\skp{g_7(x)}{\phi(x+y)}_x}_y=\int_{\setR^d}\abs{y}^{-d+1-s}\left(\int_{\setR^d}\abs{x}^{-\delta-2}x_1\phi(x+y)\,dx\right)\,dy
\end{align*}
exists for any $\phi\in\mathcal{S}(\setR^d)$.
\end{lemma}
\begin{proof}
Let us choose $\phi\in \mathcal{S}(\setR^d)$ and $R\geq2$. Note that for each $y\in\setR^d$, we see $\abs{x}^{-\delta-2}x_1\phi(x+y)\in L^1(\setR^d)$ since $\phi\in\mathcal{S}(\setR^d)$. When $\abs{y}\leq R$, we obtain
\begin{align*}
\Biggabs{\int_{\setR^d}\abs{x}^{-\delta-2}x_1\phi(x+y)\,dx}\lesssim R^{d-\delta-1}+\int_{\setR^d\setminus B_R(0)}\abs{x}^{-\delta-1}\abs{\phi(x+y)}\,dx.
\end{align*}
Here, if $-y\in B_{\frac{3}{4}R}(0)$, then using $\abs{\phi(x+y)}\lesssim\abs{x+y}^{-d-s}$, we obtain
\begin{align*}
\int_{\setR^d\setminus B_R(0)}\abs{x}^{-\delta-1}\abs{\phi(x+y)}\,dx\lesssim\int_{\setR^d\setminus B_R(0)}\abs{x}^{-\delta-1}\abs{x+y}^{-d-s}\,dx\lesssim R^{-1-s}.
\end{align*}
On the other hand, if $-y\not\in B_{\frac{3}{4}R}(0)$ so that $B_{\frac{1}{4}R}(-y)\cap B_{\frac{1}{2}R}(0)=\emptyset$, then employing $\abs{\phi(x+y)}\lesssim \min\{1,\abs{x+y}^{-d-s}\}$,
\begin{align*}
&\int_{\setR^d\setminus B_R(0)}\abs{x}^{-\delta-1}\abs{\phi(x+y)}\,dx\\
&\quad\lesssim \int_{B_{\frac{1}{4}R}(-y)}\abs{x}^{-\delta-1}\,dx+\int_{\setR^d\setminus (B_{\frac{1}{2}R}(0)\cup B_{\frac{1}{4}R}(-y))}\abs{x}^{-\delta-1}\abs{x+y}^{-d-s}\,dx\lesssim R^{d-1}
\end{align*}
holds, since $R\geq 2$. Hence if $\abs{y}\leq R$, we obtain
\begin{align}\label{eq:y<R5}
\Biggabs{\int_{\setR^d}\abs{x}^{-\delta-2}x_1\phi(x+y)\,dx}\lesssim R^{d+s}.
\end{align}

Now we consider the case of $\abs{y}\geq R$. Note that
\begin{align}\label{eq:x.delta5}
\begin{split}
&\int_{\setR^d}\abs{x}^{-\delta-2}x_1\phi(x+y)\,dx\\
&=\int_{\setR^d}\abs{x-y}^{-\delta-2}(x_1-y_1)\phi(x)\,dx\,dy\\
&=\int_{B_\frac{\abs{y}}{2}(0)}\abs{x-y}^{-\delta-2}(x_1-y_1)\phi(x)\,dx\\
&\quad+\int_{B_\frac{\abs{y}}{2}(y)}\abs{x-y}^{-\delta-2}(x_1-y_1)\phi(x)\,dx\\
&\quad+\int_{\setR^d\setminus (B_\frac{\abs{y}}{2}(0)\cup B_\frac{\abs{y}}{2}(y))}\abs{x-y}^{-\delta-2}(x_1-y_1)\phi(x)dx.
\end{split}
\end{align}
Here, we obtain
\begin{align*}
\int_{B_\frac{\abs{y}}{2}(0)}\abs{x-y}^{-\delta-2}(x_1-y_1)\phi(x)\,dx\lesssim \abs{y}^{-1}\int_{B_\frac{\abs{y}}{2}(0)}\phi(x)\,dx\lesssim\abs{y}^{-1}.
\end{align*}
For another term, using $\abs{\phi(x)}\lesssim \abs{x}^{-d}\eqsim \abs{y}^{-d}$ for $x\in B_\frac{\abs{y}}{2}(y)$, we estimate
\begin{align*}
\int_{B_\frac{\abs{y}}{2}(y)}\abs{x-y}^{-\delta-2}(x_1-y_1)\phi(x)\,dx\lesssim \abs{y}^{d-\delta-1}\abs{y}^{-d}\lesssim \abs{y}^{-1}.
\end{align*}
For the other term, using $\abs{\phi(x)}\lesssim \abs{x}^{-d-1}$, there holds
\begin{align*}
\int_{\setR^d\setminus (B_\frac{\abs{y}}{2}(0)\cup B_\frac{\abs{y}}{2}(y))}\abs{x-y}^{-\delta-2}(x_1-y_1)\phi(x)\,dx\lesssim\abs{y}^{-\delta-1}\abs{y}^{-1}\lesssim \abs{y}^{-2}.
\end{align*}
Summing up, with \eqref{eq:y<R5} we see that
\begin{align}\label{eq:dist'5}
\begin{split}
\int_{\setR^d}\abs{x}^{-\delta-2}x_1\phi(x+y)\,dx\lesssim \,R^{d+s}\cdot\indicator_{\{\abs{y}\leq R\}}+\abs{y}^{-1}\indicator_{\{\abs{y}\geq R\}}.
\end{split}
\end{align}

Then together with $\abs{y}^{-d+1-s}>0$, we estimate as follows for $L\geq 2R$: 
\begin{align*}
&\Biggabs{\int_{\abs{y}\leq L}\abs{y}^{-d+1-s}\left(\int_{\setR^d}\abs{x}^{-\delta-2}x_1\phi(x+y)\,dx\right)\,dy}\\
&\lesssim \int_{\abs{y}\leq R}R^{d}\abs{y}^{-d+1-s}\,dy+\int_{R\leq\abs{y}\leq L}\abs{y}^{-1}\abs{y}^{-d+1-s}\,dy\\
&\lesssim R^{d+1-s}+R^{-s}\lesssim R^{d+1}.
\end{align*}
Also, for $\widetilde{L}>L$, we have
\begin{align*}
&\Biggabs{\int_{L\leq\abs{y}\leq\widetilde{L}}\abs{y}^{-d+1-s}\left(\int_{\setR^d}\abs{x}^{-\delta-2}x_1\phi(x+y)\,dx\right)\,dy}\\
&\lesssim \int_{L\leq\abs{y}\leq \widetilde{L}}\abs{y}^{-1}\abs{y}^{-d+1-s}\,dy\lesssim L^{-s}.
\end{align*}
This proves that
\begin{align*}
\int_{\abs{y}\leq L}\abs{y}^{-d+1-s}\left(\int_{\setR^d}\abs{x}^{-\delta-2}x_1\phi(x+y)\,dx\right)\,dy
\end{align*}
is a Cauchy sequence for $L$. Then we see that
\begin{align*}
\skp{g_6(y)}{\skp{g_7(x)}{\phi(x+y)}_x}_y=\int_{\setR^d}\abs{y}^{-d+1-s}\left(\int_{\setR^d}\abs{x}^{-\delta-2}x_1\phi(x+y)\,dx\right)\,dy
\end{align*}
exists, so we prove the conclusion. 
\end{proof}

Recalling $\eta_{k}(\cdot)$ in \eqref{eq:eta}, and $\xi_j(\cdot,\cdot)$ in \eqref{eq:xi}, we have the following lemma. Its proof is similar to Lemma \ref{lem:conv10} and \ref{lem:conv3} so we skip the proof.

\begin{lemma}\label{lem:conv50}
The convolution of distributions $g_7\divideontimes g_6$ is well-defined in $x\in\setR^d$ in the sense that for any $\phi\in \mathcal{S}(\setR^d)$, 
\begin{align}\label{eq:kj5}
\begin{split}
\skp{g_7\divideontimes g_6}{\phi}&=\skp{g_6\divideontimes g_7}{\phi}\\
&:=\lim_{k\rightarrow\infty}\lim_{j\rightarrow\infty}\bigskp{(\eta_kg_6)(y)}{\skp{g_7(x)}{\xi_j(x,y)\phi(x+y)}_x}_y.
\end{split}		
\end{align}
Moreover, we have
\begin{align*}
\skp{g_7\divideontimes g_6}{\phi}=\bigskp{g_6(y)}{\skp{g_7(x)}{\phi(x+y)}_x}_y.
\end{align*}
\end{lemma}

\begin{lemma}\label{lem:welldef.f4}
The function $f_4(d,s,\delta)$ is well-defined in the sense that
\begin{align*}
\lim_{R\rightarrow \infty}\int_{\frac 1R \leq \abs{h} \leq R}\dfrac{\abs{e_1-h}^{-\delta-2}(1-h_1)}{\abs{h}^{d-1+s}}\,dh=c(d,s,\delta)\in[0,\infty).
\end{align*}
\end{lemma}
\begin{proof}
We define
\begin{align*}
U_R&\coloneqq \int_{\frac 1R \leq \abs{h} \leq R}\dfrac{\abs{e_1-h}^{-\delta-2}(1-h_1)}{\abs{h}^{d-1+s}}\,dh.
\end{align*}
Now, let $R > r \geq 2$. Then
\begin{align*}
U_R- U_r
&= \int_{\frac 1R \leq \abs{h} < \frac 1r} \dfrac{\abs{e_1-h}^{-\delta-2}(1-h_1)}{\abs{h}^{d-1+s}}\,dh
\\
&\quad + \int_{r < \abs{h} \leq R}\dfrac{\abs{e_1-h}^{-\delta-2}(1-h_1)}{\abs{h}^{d-1+s}}\,dh
\coloneqq \mathrm{I} + \mathrm{II}.
\end{align*}
If $\frac 1R \leq \abs{h} < \frac 1r$, then $\abs{e_1-h}\eqsim 1$, so we can see that
\begin{align*}
\mathrm{I} &\leq\int_{\frac 1R \leq \abs{h} < \frac 1r} \dfrac{\abs{e_1-h}^{-\delta-1}}{\abs{h}^{d-1+s}}\,dh\lesssim r^{s-1}.
\end{align*}
Also, for $\mathrm{II}$, we estimate
\begin{align*}
&\int_{r < \abs{h} \leq R}\dfrac{\abs{e_1-h}^{-\delta-2}(1-h_1)}{\abs{h}^{d-1+s}}\,dh\lesssim\int_{r < \abs{h} \leq R}\dfrac{\abs{h}^{-\delta-1}}{\abs{h}^{d-1+s}}\,dh\lesssim r^{-\delta-s}
\end{align*}
for $\abs{h}\geq r\geq 2$. This proves that $U_R$ is a Cauchy sequence for $R \to \infty$. Thus $f_4(d,s,\delta) = \lim_{R \to \infty} U_R$ exists. This proves the lemma.
\end{proof}

\begin{corollary}\label{cor:welldef.F4}
  For all $x \in \setR^d$ the limit
  \begin{align*}
  F_1(x):=\lim_{R \to \infty} \int_{\frac 1R \leq \abs{h} \leq R}\dfrac{\abs{x-h}^{-\delta-1}\widehat{x_1-h_1}}{\abs{h}^{d-1+s}}\,dh
  \end{align*}
  exists. Moreover, $F_1(\cdot)\in L^1(B_1(0))$ and
  \begin{align*}
  \abs{F_1(x)} \lesssim \abs{x}^{-s-\delta}.
  \end{align*}
\end{corollary}
\begin{proof}
For $x\in\setR^d\setminus\{0\}$, by change of variables, we can write
\begin{align*}
\int_{\frac 1R \leq \abs{h} \leq R}\dfrac{\abs{x-h}^{-\delta-1}\widehat{x_1-h_1}}{\abs{h}^{d-1+s}}\,dh=\abs{x}^{-s-\delta}\int_{\frac{\abs{x}}{R} \leq \abs{h} \leq R\abs{x}}\dfrac{\abs{e_1-h}^{-\delta-1}\widehat{1-h_1}}{\abs{h}^{d-1+s}}\,dh.
\end{align*}
Applying a similar argument as in the proof of Lemma \ref{lem:welldef.f4} yields the conclusions. When $x=0$, then $F_1(0)=0$ since the map $h\mapsto \abs{h}^{-d-s-\delta}h_1$ is odd. This completes the proof.
\end{proof}

Now we prove the following lemma, which rigorously establishes the connection between $f_4(d,s,\delta)$ and $g_7\divideontimes g_6$.
\begin{lemma}\label{lem:A50}
We have
$f_4(d,s,\delta)=\lim_{\epsilon\rightarrow 0}\skp{g_7\divideontimes g_6}{\phi_{\epsilon}}$.
\end{lemma}

\begin{proof}
To do this, with $\epsilon\in(0,\frac{1}{10})$, we observe that
\begin{align*}
\skp{g_5\divideontimes g_6}{\phi_{\epsilon}}=\int_{\setR^d}\abs{y}^{-d+1-s}\left(\int_{\setR^d}\abs{x}^{-\delta-2}x_1\phi_{\epsilon}(x+y)\,dx\right)\,dy.
\end{align*}
For fixed $\kappa\in(0,1/100)$, consider $\epsilon\in(0,1/100)$ with $2\epsilon\leq\kappa$ and define
\begin{align*}
&\mathcal{G}_{\kappa}=\{(x,y)\in\setR^{2d}:\abs{y}\leq\kappa\},\\
&\mathcal{H}_{\kappa}=\{(x,y)\in\setR^{2d}:\abs{e_1-y}\leq\kappa\},\\
&\mathcal{I}_{\kappa}=\setR^{2d}\setminus(\mathcal{G}_{\kappa}\cup\mathcal{H}_{\kappa}).
\end{align*}
We first estimate the integration on $\mathcal{G}_{\kappa}$. Using $\support\phi\subset B_{\epsilon}(e_1)$, we have
\begin{align*}
&\iint_{\mathcal{G}_{\kappa}}\abs{y}^{-d+1-s}\abs{x}^{-\delta-2}x_1\phi_{\epsilon}(x+y)\,dx\,dy\\
&\,\,=\frac{\overline{c}}{\epsilon^d}\int_{\abs{y}\leq\kappa}\abs{y}^{-d+1-s}\int_{\setR^d}\abs{x-y-e_1}^{-\delta-2}(x_1-y_1-1)\overline{\phi}\left(\frac{x}{\epsilon}\right)\,dx\,dy\\
&\,\,\lesssim \int_{\abs{y}\leq\kappa}\abs{y}^{-d+1-s}\dashint_{\abs{x}\leq\epsilon}\overline{\phi}\left(\frac{x}{\epsilon}\right)\,dx\,dy\lesssim \kappa^{1-s}.
\end{align*}

On the set $\mathcal{H}_{\kappa}$, using $\support\phi\subset B_1(0)$ and $\{\abs{e_1-y}\leq\kappa\}\subset\{\abs{x-y-e_1}\leq 2\kappa\}$ for each $x\in B_{\epsilon}(0)$, we compute
\begin{align*}
&\iint_{\mathcal{H}_{\kappa}}\abs{y}^{-d+1-s}\abs{x}^{-\delta-2}x_1\phi_{\epsilon}(x+y)\,dx\,dy\\
&\quad=\dfrac{\overline{c}}{\epsilon^d}\int_{\abs{e_1-y}\leq\kappa}\abs{y}^{-d+1-s}\int_{\setR^d}\abs{x-y-e_1}^{-\delta-2}(x_1-y_1-1)\overline{\phi}\left(\frac{x}{\epsilon}\right)\,dx\,dy\\
&\quad\lesssim \dashint_{\abs{x}\leq\epsilon}\left(\int_{\abs{e_1-y}\leq\kappa}\abs{x-y-e_1}^{-\delta-1}\,dy\right)\overline{\phi}\left(\frac{x}{\epsilon}\right)\,dx\\
&\quad\lesssim \dashint_{\abs{x}\leq\epsilon}\left(\int_{\abs{x-y-e_1}\leq\kappa}\abs{x-y-e_1}^{-\delta-1}\,dy\right)\overline{\phi}\left(\frac{x}{\epsilon}\right)\,dx\lesssim \kappa^{d-\delta-1}.
\end{align*}

On the set $\mathcal{I}_{\kappa}$, the integrand $\abs{y}^{-d+1-s}\abs{x}^{-\delta-2}x_1\phi_{\epsilon}(x+y)$ are no longer singular, so we can apply the argument of the proof of \cite[Page 714, Theorem 7]{Eva10}. In detail, we compute
\begin{align*}
&\iint_{\mathcal{I}_{\kappa}}\abs{y}^{-d+1-s}\abs{x}^{-\delta-2}x_1\phi_{\epsilon}(x+y)\,dx\,dy\\
&=\int_{\substack{\abs{e_1-y}\geq\kappa\\\cap\abs{y}\geq\kappa}}\int_{\setR^d}\abs{y}^{-d+1-s}\abs{x-y-e_1}^{-\delta-2}(x_1-y_1-1)\overline{\phi}\left(\frac{x}{\epsilon}\right)\,dx\,dy\\
&\quad\,\,\underset{\epsilon\rightarrow 0}{\longrightarrow}\,\,-\int_{\substack{\abs{e_1-y}\geq\kappa\\\cap\abs{y}\geq\kappa}}\abs{y}^{-d+1-s}\abs{y+e_1}^{-\delta-2}(y_1+1)\,dy.
\end{align*}

Since $\kappa\in(0,1/100)$ is arbitrary chosen and $\epsilon$ is already sent to zero, summing up above three estimates yields the conclusion.
\end{proof}

Now we have the following.
\begin{lemma}\label{lem:Fourier5}
We have 
\begin{align}\label{eq:Fourier5}
\begin{split}
&\mathcal{F}[g_7](\xi)=-\pi^{-\frac{d}{2}-\delta+1}i\dfrac{\Gamma\left(\frac{d+\delta}{2}\right)}{\Gamma\left(\frac{2-\delta}{2}\right)}\abs{\xi}^{-d+\delta}\xi_1:=\mathcal{G}_7(\xi)
\end{split}
\end{align}
in the distributional sense.
\end{lemma}
\begin{proof}
For $\delta>0$, we compute
\begin{align*}
\mathcal{F}[g_5](\xi)&=\dfrac{1}{-\delta}\mathcal{F}[\partial_1\abs{x}^{-\delta}]\\
&=-\dfrac{2\pi i\xi_1}{\delta}\mathcal{F}[\abs{x}^{-\delta}]\\
&=-\dfrac{2\pi i\xi_1}{\delta}\dfrac{(2\pi)^{-\delta}\Gamma\left(\frac{d+\delta}{2}\right)}{\pi^{\frac{d}{2}}2^{-\delta}\Gamma\left(\frac{-\delta}{2}\right)}\abs{\xi}^{-d+\delta}\\
&=-\pi^{-\frac{d}{2}-\delta+1}i\dfrac{\Gamma\left(\frac{d+\delta}{2}\right)}{\Gamma\left(\frac{2-\delta}{2}\right)}\abs{\xi}^{-d-\delta}\xi_1.
\end{align*}
When $\delta=0$, we still obtain the above equality using the following relation:
\begin{align*}
\mathcal{F}[\abs{h}^{-2}h_1](\xi)&=\mathcal{F}[\partial_1(\log\abs{h})]=2\pi i\xi_1\mathcal{F}[\log\abs{h}]\\
&=2\pi i\xi_1\left(-\frac{\Gamma\left(\frac{d}{2}\right)}{2\pi^{\frac{d}{2}}}\frac{1}{\abs{\xi}^d}+\left(\log2+\frac{\psi\left(\frac{d}{2}\right)-\gamma}{2}\right)\delta_0\right)\\
&=-i\pi^{1-\frac{d}{2}}\Gamma\left(\frac{d}{2}\right)\abs{\xi}^{-d}\xi_1,
\end{align*}
where $\psi(\cdot)$ is the digamma function, $\gamma$ is the Euler-Mascheroni constant, and $\delta_0$ is the Dirac delta distribution concentrated at 0, and for the third equality we have used \cite[Remark 1.2]{CheWet19}, and for the last line we have used $\xi_1\delta\equiv 0$. This completes the proof.
\end{proof}

Now with $\oldphi_{\epsilon}(x):=\frac{1}{\epsilon^d}\overline{\phi}\left(\frac{x}{\epsilon}\right)$ with $\epsilon\in(0,1)$ where $\overline{\phi}$ is a standard mollifier in $\setR^d$, we obtain the following lemma.

\begin{lemma}\label{lem:Fourier50}
The multiplication of distributions $\mathcal{G}_7\cdot \mathcal{G}_6$ is well-defined in $x\in\setR^d$ in the sense that for any $\phi\in \mathcal{S}(\setR^d)$,
\begin{align*}
\skp{\mathcal{G}_7\cdot \mathcal{G}_6}{\phi}=\lim_{\epsilon\rightarrow 0}\bigskp{\skp{\mathcal{G}_7(y)}{\oldphi_{\epsilon}(\cdot-y)}_y\mathcal{G}_6(\cdot)}{\phi(\cdot)}.
\end{align*}
Moreover, we have
\begin{align}\label{eq:g7g6}
\lim_{\epsilon\rightarrow 0}\bigskp{\skp{\mathcal{G}_7(y)}{\oldphi_{\epsilon}(\cdot-y)}_y\mathcal{G}_6(\cdot)}{\phi(\cdot)}=\skp{\mathcal{G}_7\mathcal{G}_6}{\phi},
\end{align}
where 
\begin{align*}
\mathcal{G}_7\mathcal{G}_6=\pi^{s-\delta}i\dfrac{\Gamma\left(\frac{d+\delta}{2}\right)\Gamma\left(\frac{1-s}{2}\right)}{\Gamma\left(\frac{2-\delta}{2}\right)\Gamma\left(\frac{d+s-1}{2}\right)}\abs{\xi}^{-d+s+\delta-1}\xi_1
\end{align*}
means the usual multiplication.
\end{lemma}
\begin{proof}
Let us choose $\kappa\geq 2\epsilon$ and consider
\begin{align*}
&\biggabs{\int_{\abs{x}\leq\kappa}\left(\int_{\setR^d}\left[\abs{y}^{-d-\delta}y_1-\abs{x}^{-d-\delta}x_1\right]\oldphi_{\epsilon}(x-y)\,dy\right)\abs{x}^{s-1}\phi(x)\,dx}\\
&\quad=\biggabs{\int_{\abs{x}\leq\kappa}\abs{x}^{-d-\delta+s-1}x_1\phi(x)\,dx}\\
&\quad\quad+\biggabs{\overline{c}\int_{\setR^d}\left(\int_{\abs{x}\leq\kappa}\abs{x}^{s-1}\phi(x)\abs{x-y}^{-d-\delta}(x_1-y_1)\,dx\right)\oldphi_{\epsilon}(y)\,dy}\\
&\quad=:\abs{I_1}+\abs{\overline{c}I_2}.
\end{align*}

Using the fact that $x\mapsto \abs{x}^{-d-\delta+s-1}x_1\frac{\phi(x)-\phi(-x)}{2}$ is odd, and $\abs{\phi(x)-\phi(-x)}\lesssim \abs{x}$, note that
\begin{align}\label{eq:I15}
\begin{split}
\abs{I_1}&=\biggabs{\int_{\abs{x}\leq\kappa}\abs{x}^{-d-\delta+s-1}x_1\frac{\phi(x)-\phi(-x)}{2}\,dx}\\
&=\int_{\abs{x}\leq\kappa}\abs{x}^{-d-\delta+s+1}\,dx\lesssim\kappa^{1+s-\delta}.
\end{split}
\end{align}

For $I_2$, using the fact that the map
\begin{align*}
x\mapsto \abs{x}^{s-1}\dfrac{\phi(x)+\phi(-x)}{2}\int_{\setR^d}\abs{x-y}^{-d-\delta}(x_1-y_1)\oldphi_{\epsilon}(y)\,dy
\end{align*}
is odd since $\oldphi_{\epsilon}$ is even, we write
\begin{align}\label{eq:I2}
I_2&=\int_{\setR^d}\left(\int_{\abs{x}\leq\kappa}\abs{x}^{s-1}\frac{\phi(x)-\phi(-x)}{2}\abs{x-y}^{-d-\delta}(x_1-y_1)\,dx\right)\oldphi_{\epsilon}(y)\,dy.
\end{align}

We write
\begin{align*}
I_2&=\int_{\abs{x}\leq\kappa}\left(\int_{\setR^d}\abs{y}^{-d-\delta}y_1\oldphi_{\epsilon}(x-y)\,dy\right)\abs{x}^{s-1}\frac{\phi(x)-\phi(-x)}{2}\,dx\\
&=\int_{\abs{x}\leq 2\epsilon}\left(\int_{\setR^d}\abs{y}^{-d-\delta}y_1\oldphi_{\epsilon}(x-y)\,dy\right)\abs{x}^{s-1}\frac{\phi(x)-\phi(-x)}{2}\,dx\\
&\quad+\int_{2\epsilon\leq \abs{x}\leq\kappa}\left(\int_{\setR^d}\abs{y}^{-d-\delta}y_1\oldphi_{\epsilon}(x-y)\,dy\right)\abs{x}^{s-1}\frac{\phi(x)-\phi(-x)}{2}\,dx.
\end{align*}

Let us estimate the first term in the right-hand side. We have
\begin{align*}
&\int_{\abs{x}\leq 2\epsilon}\left(\int_{\setR^d}\abs{y}^{-d-\delta}y_1\oldphi_{\epsilon}(x-y)\,dy\right)\abs{x}^{s-1}\frac{\phi(x)-\phi(-x)}{2}\,dx\\
&=\int_{\abs{x}\leq 2\epsilon}\left(\int_{\abs{y}\leq\frac{1}{2}\abs{x}}\abs{y}^{-d-\delta}y_1\left(\oldphi_{\epsilon}(x-y)-\oldphi_{\epsilon}(x)\right)\,dy\right)\abs{x}^{s-1}\frac{\phi(x)-\phi(-x)}{2}\,dx\\
&\quad+\int_{\abs{x}\leq 2\epsilon}\left(\int_{\abs{y}\geq\frac{1}{2}\abs{x}}\abs{y}^{-d-\delta}y_1\oldphi_{\epsilon}(x-y)\,dy\right)\abs{x}^{s-1}\frac{\phi(x)-\phi(-x)}{2}\,dx,
\end{align*}
where we have used the fact that
\begin{align*}
\oldphi_{\epsilon}(x)\int_{\abs{y}\leq\frac{1}{2}\abs{x}}\abs{y}^{-d-\delta}y_1\,dy=0.
\end{align*}

Using $\abs{\oldphi_{\epsilon}(x-y)-\oldphi_{\epsilon}(x)}\lesssim \epsilon^{-1-d}\abs{y\cdot\nabla\overline{\phi}}$, we further compute
\begin{align*}
&\int_{\abs{x}\leq 2\epsilon}\left(\int_{\abs{y}\leq\frac{1}{2}\abs{x}}\abs{y}^{-d-\delta}y_1\left(\oldphi_{\epsilon}(x-y)-\oldphi_{\epsilon}(x)\right)\,dy\right)\abs{x}^{s-1}\frac{\phi(x)-\phi(-x)}{2}\,dx\\
&\quad\lesssim\frac{1}{\epsilon}\dashint_{\abs{x}\leq 2\epsilon}\left(\int_{\abs{y}\leq\frac{1}{2}\abs{x}}\abs{y}^{-d-\delta+2}\,dy\right)\abs{x}^{s}\,dx\lesssim \epsilon^{s+1-\delta}\lesssim \kappa^{s+1-\delta},
\end{align*}
where we have used $\abs{\phi(x)-\phi(-x)}\lesssim \abs{x}$. For the other term, using $\support\oldphi_{\epsilon}\subset B_{\epsilon}(0)$, we have
\begin{align*}
&\int_{\abs{x}\leq 2\epsilon}\left(\int_{\abs{y}\geq\frac{1}{2}\abs{x}}\abs{y}^{-d-\delta}y_1\oldphi_{\epsilon}(x-y)\,dy\right)\abs{x}^{s-1}\frac{\phi(x)-\phi(-x)}{2}\,dx\\
&\quad\lesssim \int_{\abs{x}\leq 2\epsilon}\abs{x}^{s}\epsilon^{1-\delta}\left(\int_{\abs{y}\leq\epsilon}\oldphi_{\epsilon}(y)\,dy\right)\,dx\lesssim \epsilon^{s+1-\delta}\lesssim \kappa^{s+1-\delta}.
\end{align*}
Note that the implicit constants in the above estimates do not depend on $\epsilon$. Also, using $\support\oldphi_{\epsilon}\subset B_{\epsilon}(0)$ and $\abs{\phi(x)-\phi(-x)}\lesssim \abs{x}$, we can estimate
\begin{align*}
&\int_{2\epsilon\leq \abs{x}\leq\kappa}\left(\int_{\setR^d}\abs{y}^{-d-\delta}y_1\oldphi_{\epsilon}(x-y)\,dy\right)\abs{x}^{s-1}\frac{\phi(x)-\phi(-x)}{2}\,dx\\
&\quad=\int_{2\epsilon\leq \abs{x}\leq\kappa}\left(\int_{\abs{x-y}\leq\epsilon}\abs{y}^{-d-\delta}y_1\oldphi_{\epsilon}(x-y)\,dy\right)\abs{x}^{s-1}\frac{\phi(x)-\phi(-x)}{2}\,dx\\
&\quad=\overline{c}\int_{2\epsilon\leq \abs{x}\leq\kappa}\left(\dashint_{\abs{y}\leq\epsilon}\abs{x-y}^{-d-\delta}(x_1-y_1)\overline{\phi}\left(\frac{y}{\epsilon}\right)\,dy\right)\abs{x}^{s-1}\frac{\phi(x)-\phi(-x)}{2}\,dx\\
&\underset{\epsilon\rightarrow 0}{\longrightarrow} \overline{c}\int_{\abs{x}\leq\kappa}\abs{x}^{-d-\delta+1}\abs{x}^{s-1}\frac{\phi(x)-\phi(-x)}{2}\,dx\\
&\quad\leq \int_{\abs{x}\leq\kappa}\abs{x}^{-d+s-\delta+1}\,dx\lesssim \kappa^{s+1-\delta}.
\end{align*}

Then merging the estimates, we have
\begin{align*}
\biggabs{\int_{\abs{x}\leq\kappa}\left(\int_{\setR^d}(\mathcal{G}_1(y)-\mathcal{G}_1(x))\oldphi_{\epsilon}(x-y)\,dy\right)\mathcal{G}_2(x)\phi(x)\,dx}\lesssim\kappa^{s+1-\delta}.
\end{align*}
Note that the implicit constants in the above estimates do not depend on $\epsilon$.

Now similar to the proof of Lemma \ref{lem:A2}, since on the set $\abs{x}\geq\kappa$ with $\kappa\geq 2\epsilon$, both $\mathcal{G}_7(x)$ and $\mathcal{G}_6(x)$ are no longer singular, by Lebesgue differentiation theorem we have
\begin{align*}
&\lim_{\epsilon\rightarrow 0}\biggabs{\int_{\abs{x}\geq\kappa}\left(\int_{\setR^d}(\mathcal{G}_7(y)-\mathcal{G}_7(x))\oldphi_{\epsilon}(x-y)\,dy\right)\mathcal{G}_6(x)\phi(x)\,dx}=0.
\end{align*}
Then together with \eqref{eq:I15}, we obtain 
\begin{align*}
&\lim_{\epsilon\rightarrow 0}\bigabs{\bigskp{\skp{\mathcal{G}_7(y)}{\oldphi_{\epsilon}(\cdot-y)}_y\mathcal{G}_6(\cdot)}{\phi(\cdot)}-\skp{\mathcal{G}_7\mathcal{G}_6}{\phi}}\\
&\quad\leq \lim_{\epsilon\rightarrow 0}\biggabs{\int_{\abs{x}\leq\kappa}\left(\int_{\setR^d}(\mathcal{G}_7(y)-\mathcal{G}_7(x))\oldphi_{\epsilon}(x-y)\,dy\right)\mathcal{G}_6(x)\phi(x)\,dx}\\
&\quad\quad+\lim_{\epsilon\rightarrow 0}\biggabs{\int_{\abs{x}\geq\kappa}\left(\int_{\setR^d}(\mathcal{G}_7(y)-\mathcal{G}_7(x))\oldphi_{\epsilon}(x-y)\,dy\right)\mathcal{G}_6(x)\phi(x)\,dx}\lesssim \kappa^{s+1-\delta}.
\end{align*}
Since $\kappa$ was arbitrary, we get \eqref{eq:g7g6}.
\end{proof}

With the help of Lemma \ref{lem:Fourier5} and \ref{lem:Fourier50}, we have the following theorem.
\begin{theorem}\label{thm:Fourierconv5}
We have 
\begin{align*}
\mathcal{F}[g_7\divideontimes g_6]=\mathcal{F}[g_7]\mathcal{F}[g_6]
\end{align*}
in the distributional sense, where $\mathcal{F}[g_7]\mathcal{F}[g_6]$ means the usual multiplication between $\mathcal{F}[g_7]$ and $\mathcal{F}[g_6]$.
\end{theorem}
\begin{proof}
For any $\phi\in\mathcal{S}(\setR^d)$, note that
\begin{align*}
\skp{\mathcal{F}[g_7\divideontimes g_6]}{\phi}&=\skp{\mathcal{F}[g_7]\cdot\mathcal{F}[g_6]}{\phi}\\
&=\lim_{\epsilon\rightarrow 0}\skp{(\mathcal{F}[g_7]\divideontimes\oldphi_{\epsilon})\mathcal{F}[g_6]}{\phi}\\
&=\lim_{\epsilon\rightarrow 0}\bigskp{\skp{\mathcal{F}[g_7](y)}{\oldphi_{\epsilon}(x-y)}_y\mathcal{F}[g_6](x)}{\phi(x)}_x\\
&=\skp{\mathcal{F}[g_7]\mathcal{F}[g_6]}{\phi},
\end{align*}
where for the first equality since $g_7,g_6\in\mathcal{S}'(\setR^d)$, we have used \cite[Page 96]{Vla02} together with Lemma \ref{lem:conv5} and Lemma \ref{lem:conv50}, for the second equality since $\mathcal{F}[g_7]\divideontimes\oldphi_{\epsilon}\in\theta_M$ as in \cite[Page 84]{Vla02} ($\theta_M$ is defined in \eqref{eq:thetaM}) and $\mathcal{F}[g_6]$ in $\mathcal{S}'(\setR^d)$, $(\mathcal{F}[g_7]\divideontimes\oldphi_{\epsilon})\mathcal{F}[g_6]\in\mathcal{S}'(\setR^d)$ holds as in \cite[Page 79]{Vla02}, for the third equality again since $\mathcal{F}[g_7]\divideontimes\oldphi_{\epsilon}\in\theta_M$, \cite[Page 84]{Vla02} is used, and for the last equality Lemma \ref{lem:Fourier50} together with Lemma \ref{lem:Fourier5} is used.
\end{proof}

Then it is worth mentioning the following corollary.
\begin{corollary}\label{cor:Fourierconv50}
	We have $(\mathcal{F}\circ\mathcal{F})[g_7\divideontimes g_6]=\mathcal{F}\left[\mathcal{F}[g_7]\mathcal{F}[g_6]\right]$.
\end{corollary}
\begin{proof}
Since $\phi\in\mathcal{S}(\setR^d)$ implies $\mathcal{F}[\phi]\in\mathcal{S}(\setR^d)$, by Theorem \ref{thm:Fourierconv5} we have
\begin{align*}
\skp{\mathcal{F}[g_7\divideontimes g_6]}{\mathcal{F}[\phi]}=\skp{\mathcal{F}[g_7]\cdot\mathcal{F}[g_6]}{\mathcal{F}[\phi]}=\skp{\mathcal{F}[g_7]\mathcal{F}[g_6]}{\mathcal{F}[\phi]},
\end{align*}
which proves the conclusion.
\end{proof}

\printbibliography
\end{document}